\documentclass{aims}

\usepackage{txfonts}
\def\typeofarticle{Research article}
\def\currentvolume{x}
\def\currentissue{x}
\def\currentyear{xxxx}

\def\ppages{xxx--xxx}
\def\DOI{10.3934/mine.xxxx.x.xxx}
\def\Received{xxxx}
\def\Accepted{xxxx}
\def\Published{xxxx}

\numberwithin{equation}{section}

\makeatletter
\renewcommand{\@biblabel}[1]{#1\hfill \hspace{-0.2cm}}

\makeatother

\usepackage{amsfonts}
\usepackage{amssymb}
\usepackage{amsthm}
\usepackage{amsmath}
\usepackage{mathrsfs}
\usepackage{epsfig}
\usepackage{tikz,fp,ifthen}
\usepackage{float}
\usepackage{graphicx}
\usepackage{enumitem}
\usepackage{esint}
\usepackage{mathtools}
\mathtoolsset{showonlyrefs}

\theoremstyle{plain}
\newtheorem{thm}{Theorem}[section]
\newtheorem{lem}[thm]{Lemma}
\newtheorem{proposition}[thm]{Proposition}
\newtheorem{cor}[thm]{Corollary}
\newtheorem*{theorem}{Theorem}

\theoremstyle{remark}

\newtheorem{remark}[thm]{Remark}

\theoremstyle{definition}
\newtheorem{definition}[thm]{Definition}

\begin{document}

\title{Global existence and stability for the modified Mullins--Sekerka and surface diffusion flow}

\author{Serena Della Corte\affil{1}, Antonia Diana\affil{2} and Carlo Mantegazza\affil{3,}\corrauth}


\shortauthors{the Author(s)}

\address{%
  \addr{\affilnum{1}}{Delft Institute of Applied Mathematics, Delft University of Technology, The Netherlands}
  \addr{\affilnum{2}}{Scuola Superiore Meridionale, Universit\`{a} degli Studi di Napoli Federico II, Italy}
    \addr{\affilnum{3}}{Dipartimento di Matematica e Applicazioni ``Renato Caccioppoli'' \& Scuola Superiore Meridionale, Universit\`{a} degli Studi di Napoli Federico II, Italy}}

\corraddr{Email: c.mantegazza@sns.it}

\def\OO{{\mathrm{O}}}
\def\II{{\mathrm{I}}}
\newcommand{\T}{\mathbb{T}}
\newcommand{\R}{\mathbb{R}}
\newcommand{\N}{\mathbb{N}}
\def\Z{\mathbb Z}
\newcommand{\SSS}{\mathbb S}
\newcommand{\BBB}{B}
\newcommand{\RRR}{{\mathrm R}}
\newcommand{\Ric}{{\mathrm {Ric}}}
\newcommand{\Scal}{{\RRR}}
\newcommand{\LB}{\Delta_t}
\newcommand{\Lebes}{\mathscr{L}}
\DeclarePairedDelimiter{\normainf}{\lVert}{\rVert_{\infty}}
\DeclarePairedDelimiter{\norma}{\lVert}{\rVert}
\DeclarePairedDelimiter{\abs}{\lvert}{\rvert}
\newcommand{\dmu}{d \mu}
\newcommand{\Htilde}{\widetilde{H}}
\newcommand{\Wduep}{W^{2,p}}
\newcommand{\Tort}{T^{\perp}}
\newcommand{\de}{\delta}
\newcommand{\pa}{\partial}
\newcommand{\C}{\mathbb{C}}
\newcommand{\Hdue}{\mathcal{H}^2}
\newcommand{\varempty}{\O}
\newcommand{\HHH}{\mathrm{H}} 
\newcommand{\vol}{\mathrm{Vol}}
\newcommand{\grad}{\nabla}
\newcommand{\A}{\mathcal A}
\newcommand{\J}{\mathcal J}
\newcommand{\eps}{\varepsilon}
\newcommand{\medint}{-\kern -,375cm\int}
\newcommand{\medintinrigo}{-\kern -,315cm\int}
\newcommand{\hdueM}{\mathfrak{h}^{2, \alpha}_M(F, U)}
\newcommand{\CunoM}{\mathfrak{C}^{1}_M(F, U)}
\newcommand{\gij}{g_{ij}}
\newcommand{\Id}{\mathrm{Id}}
\newcommand{\I}{\mathrm{I}}
\newcommand{\Xtilde}{\widetilde{X}}
\newcommand{\bigvert}{\biggl\vert}
\newcommand{\wto}{\rightharpoonup}
\newcommand{\beq}{\begin{equation}}
\newcommand{\eeq}{\end{equation}}
\def\Div{\operatorname*{div}\nolimits}
\newcommand{\dist}{\mathrm{dist\,}}

\begin{abstract}
In this survey we present the state of the art about the asymptotic behavior and stability of the \emph{modified Mullins--Sekerka flow} and the {\em surface diffusion flow} of smooth sets, mainly due to E.~Acerbi, N.~Fusco, V.~Julin and M.~Morini. First we discuss in detail the properties of the nonlocal Area functional under a volume constraint, of which the two flows are the gradient flow with respect to suitable norms, in particular, we define the {\em strict stability} property for a critical set of such functional and we show that it is a necessary and sufficient condition for minimality under $W^{2,p}$--perturbations, holding in any dimension. Then, we show that, in dimensions two and three, for initial sets sufficiently ``close'' to a smooth {\em strictly stable critical} set $E$, both flows exist for all positive times and asymptotically ``converge'' to a translate of $E$.
\end{abstract}


\maketitle

\section{Introduction}

Geometric evolutions are a fascinating topic naturally arising from the study of dynamical models in physics and material sciences. Concrete examples are, for instance, the analysis of the behavior in time of the interfaces surfaces in phase changes of materials or in the flows of immiscible fluids. From the mathematical point of view, they describe the motion of geometric objects or structures, usually driven by systems of partial differential equations.\\ 
In this work we rethink, expand the details and present in a unified treatment the results of E.~Acerbi, N.~Fusco, V.~Julin and M.~Morini~\cite{AcFuMoJu,AcFuMo} about two of the most recent of such geometric motions, namely, the \emph{modified Mullins--Sekerka flow} and the {\em surface diffusion flow}. 

Both flows deal with an evolution in time of smooth subsets $E_t$ of an open set $\Omega\subseteq\R^n$, with $d(E_t,\partial\Omega)>0$, for every $t$ in a time interval $[0,T)$, such that their boundaries $\pa E_t$, which are smooth hypersurfaces, move with some ``outer'' normal velocity $V_t$ that, in the first case, is obtained as solution of the following ``mixed'' system
\begin{equation}\label{msfSistema}\tag{mMSF}
\begin{cases}
V_t=[\partial_{\nu_t} w_t]  & \text{on  } \partial E_t \\
\Delta w_t=0 & \text{in  } \Omega \setminus \partial E_t\\
w_t=\HHH_t + 4 \gamma v_t & \text{on  } \partial E_t\\
-\Delta v_t = u_{E_t} - \fint_{\Omega} u_{E_t} \,dx & \text{in  } \Omega\,\,\text{(distributionally)} 
\end{cases}
\end{equation}
where $\gamma$ is a nonnegative parameter, $v,w:[0,T)\times\overline{\Omega}\to\R$ are continuous functions such that, setting $w_t=w(t,\cdot)$ and $v_t=v(t,\cdot)$, the functions $v_t$ and $w_t$ are smooth in $\Omega \setminus \partial E_t$, for every $t\in[0,T)$; the functions $\nu_t,\HHH_t$ are the ``outer'' normal and the relative mean curvature of $\partial E_t $ and $u_{E_t}= 2  \chi_{\text{\raisebox{-.5ex}{$\scriptstyle E_t$}}} -1$; finally, the velocity of the motion is given by $[\partial_{\nu_t} w_t]$ which denotes $\partial_{\nu_t}w_t^{+}- \partial_{\nu_t}w_t^{-}$, that, is the ``jump'' of the normal derivative of $w_t$ on $\pa E_t$, where $w_t^{+}$ and $w_t^{-}$ are the restrictions of $w_t$ to  $\Omega \setminus\overline{E}_t$ and $E_t$, respectively.\\
The resulting motion, called {\em modified Mullins--Sekerka flow}~\cite{MS} (see also~\cite{Crank,Gurtin1} and~\cite{EscherSi4} for a very clear and nice introduction to such flow), arises as a singular limit of a nonlocal version of the Cahn--Hilliard equation~\cite{alikakos,pego,Le}, to describe phase separation in diblock copolymer melts (see also~\cite{OK}). It has been also called {\em Hele--Shaw model}~\cite{XChen}, or {\em Hele--Shaw model with surface tension}~\cite{EscherSi1,EscherSi2,EscherSi3}. We mention that the adjective ``modified'' comes from the introduction of the parameter $\gamma > 0$ in the system~\eqref{msfSistema}, while choosing $\gamma=0$ we have the original flow proposed by Mullins and Sekerka in~\cite{MS}. 

In the second case, we will say that a flow of sets $E_t$ as above, is a solution of the {\em surface diffusion flow} if the normal velocity is pointwise given by
\begin{equation}\label{sdf}\tag{SDF}
V_t = \LB \HHH_t  \qquad\text{on $\pa E_t $,}
\end{equation}
where $\LB$ is the Laplacian of the hypersurface $\pa E_t$, for all $t \in [0,T)$. Such flow was first proposed by Mullins in~\cite{Mullins} to study thermal grooving in material sciences (see also~\cite{escmaysim} for a nice presentation), in particular, in the physically relevant case of three--dimensional space, it describes the evolution of interfaces between solid phases of a system, which are studied  in a variety of physical settings including phase transitions, epitaxial deposition and grain
growth (see for instance~\cite{GurJab} and the references therein).

Notice that, while in this latter case, the velocity flow is immediately well defined, the system~\eqref{msfSistema} is clearly undetermined as it is, since the behavior of the functions $w_t$ and $v_t$ is not prescribed on the boundary of $\Omega$ (which is also possibly not bounded). By simplicity, we will consider flows in the whole Euclidean space and we assume that all the functions and sets involved are periodic with respect to the standard lattice $\Z^n$ of $\R^n$. It is then clear that this is equivalent to ``ambient'' the problem in the $n$--dimensional ``flat'' torus $\mathbb{T}^n = \R^n/ \Z^n $, hence in the sequel we will assume $\Omega=\mathbb{T}^n$, modifying the definitions above accordingly. Another possibility would be asking that $\Omega\subseteq\R^n$ is bounded, the moving sets do not ``touch'' the boundary of $\Omega$ and that all the functions $w_t$ and $v_t$ are subject to homogeneous (zero) Neumann boundary conditions on $\partial \Omega$ (see Subsection~\ref{Neucase}).

A very important property of these geometric flows is that both are the {\em gradient flow} of a functional, which clearly gives a natural ``energy'', decreasing in time during the evolution (the velocity $V_t$ is minus the gradient, that is, the {\em Euler--Lagrange equation} of a functional).\\
Precisely, in any dimension $n\in\N$, the modified Mullins--Sekerka flow is the $H^{-1/2}$--{gradient flow} of the following {\em nonlocal Area functional}
 \begin{equation}
 \label{eq:J}
 J(E)= \A(\pa E) + \gamma \int_{\T^n} \int_{\T^n} G(x,y) u_E(x) u_E(y) \, dx \, dy \,,
 \end{equation}
under the constraint that the volume $\vol(E)=\mathscr{L}^n(E)$ is fixed, where (here and in the whole paper),
$$
\A(\pa E)= \int_{\partial E} \, d\mu
$$
is the classical {\em Area functional} that gives the {\em area} of the $(n-1)$--dimensional smooth boundary of any sets $E$ ($\mu$ is the ``canonical'' measure associated to the Riemannian metric on $\pa E$ induced by metric of $\T^n$ coming from the scalar product of $\R^n$, which coincides with the $n$--dimensional Hausdorff measure $\mathcal{H}^n$) and $G$ is the Green function of $\T^n$ (see~\cite{Le}, for details).\\
Similarly, the surface diffusion flow can be regarded as the $H^{-1}$--gradient flow of the Area functional $\A$ with fixed volume.

Then, it clearly follows that, in both cases, the volume of the evolving sets $\vol(E_t)$ is constant in time, while neither convexity (see~\cite{Conv} and~\cite{Ito}) is maintained, nor there holds the so--called ``comparison property'' asserting that if two initial sets are one contained into the other, they stay so during the two flows. This is due to the lack of the {\em maximum principle} for parabolic equations or systems of order larger than two. We remind that such properties are shared by the more famous {\em mean curvature flow}, which is also a gradient flow of the Area functional (without the constraint on the volume), but with respect to the $L^2$--norm (see~\cite{Man}, for instance).

Parametrizing the moving smooth surfaces $\partial E_t$ by some maps (embeddings) $\psi_t:M\to\T^n$ such that $\psi_t(M)=\partial E_t$, where $M$ is a fixed smooth, compact $(n-1)$--dimensional differentiable manifold and $\nu_t$ is the outer unit normal vector to $\partial E_t$ as above, the evolution laws~\eqref{msfSistema} and~\eqref{sdf} can be respectively expressed as 
$$
\frac{\partial}{\partial t}\psi_t=V_t \nu_t=[\partial_{\nu_t} w_t] \nu_t\,,
$$
and
$$
\frac{\partial}{\partial t}\psi_t=(\Delta_t\HHH_t)\nu_t \,.
$$
Due to the parabolic nature (not actually so explicit in the first case) of these systems of PDEs, it is known that for every smooth initial set $E_0$ in $\T^n$, with boundary described by $\psi_0:M\to\T^n$, both flows with such initial data exist unique and are smooth in some positive time interval $[0,T)$. 
Indeed, such short time existence and uniqueness results were proved by Escher and Simonett~\cite{EscherSi1,EscherSi2,EscherSi3} and independently by Chen, Hong and Yi~\cite{chenhong} for the modified Mullins--Sekerka flow and by Escher, Mayer and Simonett in~\cite{escmaysim} for the surface diffusion flow of a smooth compact hypersurface in domains of the Euclidean space of any dimension. With minor modifications, their proof can be adapted to get the same conclusion also for smooth initial hypersurfaces of $\T^n$.
 
The aim of this work is to show that, in dimensions two and three, for initial data sufficiently ``close'' to a smooth {\em strictly stable critical} set $E$ for the relative ``energy'' functional (the nonlocal or the usual Area functional) under a volume constraint, the flows exist for all positive times and asymptotically converge {\em in some sense} to a ``translate'' of $E$.\\
The notions of criticality and stability are as usual defined in terms of first and second variations of $J$ and $\A$. We say that a smooth subset $E \subseteq \T^n$ is \emph{critical} for $J$ (or for $\A$, simply choosing $\gamma=0$ in formula~\eqref{eq:J}) if for any smooth one--parameter family of diffeomorphisms $\Phi_t:\T^n\to\T^n$, such that $\mathrm{Vol}(\Phi_{t}(E))=\mathrm{Vol}(E)$, for $t\in(-\eps,\eps)$ and $\Phi_0=\mathrm{Id}$ ($E_t=\Phi_t(E)$ will be called {\em volume--preserving variation} of $E$), we have 
$$
\frac{d}{dt} J(\Phi_t(E))\Bigl|_{t=0}=0 \, .
$$
We will see that this condition is equivalent to the existence of a constant $\lambda \in \R$ such that
\[
\HHH+ 4 \gamma v_E = \lambda \qquad \text{on $\partial E$},
\]
where $\HHH$ is the mean curvature of $\pa E$ and  $v_E$ is the potential defined as
\begin{equation}
v_E(x)=\int_{\T^n} G(x,y)u_E(y) dy \, ,
\end{equation}
with $G$ the Green function of the torus $\T^n$ and $u_E= \chi_{\text{\raisebox{-.5ex}{$\scriptstyle E$}}} - \chi_{\text{\raisebox{-.5ex}{$\scriptstyle \T^n \setminus E$}}}$.\\
The second variation of $J$ at a critical set $E$, leading to the central notion of {\em stability}, is more involved and, differently by the original papers, we will compute it with the tools and methods of differential/Riemannian geometry (like the first variation). We will see that at a critical set $E$, the second variation of $J$ (the second derivative at $t=0$ of $J(E_t)$) along a volume--preserving variation $E_t=\Phi_t(E)$ only depends on the normal component $\varphi$  on $\partial E$ of the {\em infinitesimal generator} field $X=\frac{\partial\Phi_t}{\partial t} \bigl|_{t=0}$ of the variation. The volume constraint on the admissible deformations of $E$ implies that the functions $\varphi$ must have zero integral on $\pa E$, hence it is natural to define a quadratic form $\Pi_E$ on such space of functions which is related to the second variation of $J$ by the following equality
\begin{equation}\label{PI0}
\Pi_E(\varphi)=\frac{d^2}{dt^2} J(\Phi_t(E))\Bigr \vert_{t=0}
\end{equation}
where $E_t=\Phi_t(E)$ is a volume--preserving variation of $E$ such that 
$$
\Bigl\langle\nu_E \,\Bigr\vert\frac{ \pa \Phi_t}{ \pa t}\Bigr \vert_{t=0}\Bigr\rangle=\varphi
$$
on $\pa E$, with $\nu_E$ the {\em outer unit normal vector} of $\pa E$.\\
Because of the obvious {\em translation invariance} of the functional $J$, it is easy to see (by means of the formula~\eqref{PI0}) that the form $\Pi_E$ vanishes on the finite dimensional vector space given by the functions $\psi= \langle\nu_E\vert \eta\rangle$, for every vector $\eta\in\R^n$. We underline that the presence of such ``natural'' degenerate subspace of the quadratic form $\Pi_E$ (or, equivalently, the translation invariance of $J$) is the main reason of several technical difficulties.\\
We then say that a smooth critical set $E \subseteq \T^n$ is {\em strictly stable} if 
\begin{equation}
\Pi_E( \varphi ) > 0 
\end{equation}
for all non--zero functions $\varphi:\pa E\to\R$, with zero integral and $L^2$--orthogonal to every function $\psi= \langle\nu_E\vert \eta\rangle$.

Then, the heuristic idea is that in a region around a strictly stable critical set $E$, we have a ``potential well'' for the ``energy'' $J$ (and the set $E$ is a local minimum) and, defining a suitable notion of ``closedness'', if one set starts close enough to $E$, during its evolution by (minus) the gradient of such energy, it cannot ``escape'' the well and asymptotically converges to a set of (local) minimal energy, which must be a translate of $E$. That is, the strict stability of $E$ implies a ``dynamical'' stability in a neighborhood.

At the moment, this conclusion, that we state precisely below, can be shown only in dimension at most three, because of missing estimates in higher dimensions (see the discussion at the beginning of Section~\ref{globalex}). When $n>3$ this and several other questions on these flows remain open. Anyway, this is sufficient for the application to some physically relevant models, since the evolution laws~\eqref{msfSistema} and~\eqref{sdf} describe, respectively, pattern--forming processes such as the solidification in pure liquids and the evolution of interfaces between solid phases of a system, driven by surface diffusion of atoms under the action of a chemical potential (see for instance~\cite{GurJab} and the references therein). In this paper, we will only deal with the three--dimensional case, but we underline that all the results and arguments hold, without relevant modifications, also in the two--dimensional situation of $\T^2=\R^2/\Z^2$, where the moving boundaries of the sets are curves.\\ 
Moreover, we mention here that all the results also hold in a bounded open subset $\Omega$ of $\R^2$ or $\R^3$, for moving sets which do not ``touch'' the boundary of $\Omega$, imposing that the functions $w_t$ and $v_t$ in the definition of the modified Mullins--Sekerka flow satisfy a {\em zero Neumann boundary condition} (as we mentioned above), instead than choosing the ``toric ambient'' (see Subsection~\ref{Neucase} for more details).

\begin{theorem}[Theorem~\ref{existence} and Remark~\ref{existence+}]
Let $E\subseteq\T^3$ be a smooth strictly stable critical set for the nonlocal Area functional under a volume constraint and $N_\eps$ a suitable tubular neighborhood of $\pa E$. For every $\alpha\in (0,1/2)$ there exists $M>0$ such that, if $E_0$ is a smooth set satisfying
\begin{itemize}
\item $\vol( E_0)= \vol( E )$, 
\item $\vol( E_0\triangle E)  \leq M$, 
\item the boundary of $E_0$ is contained in $N_\eps$ and can be represented as
\begin{equation}
\pa E_0= \{ y+ \psi_{E_{0}} (y) \nu_E(y) \, : \, y \in \pa E \},
\end{equation}
for some function $\psi _{E_0} : \pa E \to \R$ such that $ \norma { \psi _ {E_0}}_{ C^{1,\alpha} ( \pa E)} \leq M$,\\
\item there holds\ \vspace{-15pt}
$$ 
\int_{\T^3} \vert \nabla w_{E_0 }\vert^2\,dx \leq M\,,
$$
where $w_0=w_{E_0}$ is the function relative to $E_0$, as in system~\eqref{msfSistema},
\end{itemize}
then, there exists a unique smooth solution $E_t$ of the modified Mullins--Sekerka flow (with parameter $\gamma\geq 0$) starting from $E_0$, which is defined for all $t\geq0$. Moreover, $E_t\to E+\eta$ exponentially fast in $C^k$ as $t\to +\infty$, for every $k\in\N$, for some $\eta\in \R^3$, with the meaning that the functions 
$\psi_{\eta, t} : \pa E+ \eta \to \R$ representing $\pa E_t$ as ``normal graphs'' on $\pa E + \eta$, that is,
$$
\pa E_t= \{ y+ \psi_{\eta,t} (y) \nu_{E+\eta}(y) \, : \, y \in \pa E+\eta \},
$$
satisfy for every $k\in\N$, the estimates
$$
\Vert \psi_{\eta, t}\Vert_{C^k(\pa E + \eta)}\leq C_ke^{-\beta_k t}
$$
for every $t\in[0,+\infty)$, for some positive constants $C_k$ and $\beta_k$.
\end{theorem}

\begin{theorem}[Theorem~\ref{existence2} and Remark~\ref{existence2+}]
Let $E\subseteq\T^3$ be a strictly stable critical set for the Area functional under a volume constraint and let $N_\eps$ be a tubular neighborhood of $\pa E$. For every $\alpha\in (0,1/2)$ there exists $M>0$ such that, if $E_0$ is a smooth set satisfying
\begin{itemize}
\item $\vol( E_0)= \vol( E )$, 
\item $\vol( E_0\triangle E)  \leq M$, 
\item the boundary of $E_0$ is contained in $N_\eps$ and can be represented as
\begin{equation}
\pa E_0= \{ y+ \psi_{E_{0}} (y) \nu_E(y) \, : \, y \in \pa E \} \, ,
\end{equation}
for some function $\psi _{E_0} : \pa F \to \R$ such that $ \norma { \psi _ {E_0}}_{ C^{1,\alpha} ( \pa E)} \leq M$,
\item there holds \ \vspace{-10pt}
$$
\int_{\pa E_0} \vert \nabla \HHH_0\vert^2\, \dmu_0 \leq M \, ,
$$
\end{itemize}
then there exists a unique smooth solution $E_t$ of the surface diffusion flow starting from $E_0$, which is defined for all $t\geq0$. Moreover, $E_t\to E+\eta$ exponentially fast in $C^k$ as $t\to +\infty$, for some $\eta\in \R^3$, with the same meaning as above.
\end{theorem}

We remark that the line of the proof in~\cite{AcFuMoJu} that we are going to present, is based on suitable energy identities and compactness arguments to establish these global existence and exponential stability results. This was actually a completely new approach to manage the translation invariance of the functional $J$, in previous literature dealt with by means of semigroup techniques. 

Summarizing, the work is organized as follows: in Section~\ref{nonlocsec} we study the nonlocal Area functional (constrained or not) and we compute its first and second variation, then we discuss the notions of criticality, stability and local minimality of a set and their mutual relations, in this context. In Section~\ref{msfsdf} we introduce the modified Mullins--Sekerka and the surface diffusion flow and we analyze their basic properties. Section~\ref{globalex} is devoted to show the two main theorems above, while finally in Section~\ref{classification}, we discuss the classification of the stable and strictly stable critical sets (to whom then the two stability results apply).

\section{The nonlocal Area functional}\label{nonlocsec}

We start by introducing the \emph{nonlocal Area functional} and its basic properties. 

In the following we denote by $\mathbb{T}^n$ the $n$--dimensional flat torus of unit volume which is defined as the Riemannian quotient of $\mathbb{R}^n$ with respect to the equivalence relation $x \sim y \iff x-y \in \mathbb{Z}^n$, with $\mathbb{Z}^n$ the standard integer lattice of $\R^n$. Then, the functional space $W^{k,p}(\mathbb{T}^n)$, with $k \in \mathbb{N}$ and $p \ge 1$, can be identified with the subspace of $W^{k,p}_{\mathrm{loc}}(\R^n)$ of the functions that are $1$--periodic with respect to all coordinate directions. A set $E \subseteq \mathbb{T}^n$ is of class $C^k$ (or smooth) if its ``$1$--periodic extension'' to $\mathbb{R}^n$ is of class $C^k$ (or smooth,) which means that its boundary is locally a graph of a function of class $C^k$ around every point. We will denote with $\vol(E)={\mathscr L}^n(E)$ the volume of $E\subseteq\T^n$.

Given a smooth set $E\subseteq \T^n$, we consider the associated potential 
\begin{equation}\label{potential1}
v_E(x)=\int_{\T^n} G(x,y)u_E(y) dy \, ,
\end{equation}
where $G$ is the Green function (of the Laplacian) of the torus $\T^n$ and $u_E= \chi_{\text{\raisebox{-.5ex}{$\scriptstyle E$}}} - \chi_{\text{\raisebox{-.5ex}{$\scriptstyle \T^n \setminus E$}}}$. More precisely, $G$ is the (distributional) solution of
\begin{equation}\label{G2}
-\Delta_x G(x,y)=\delta_y-1\quad\text{in $\T^n$} \quad \text{with}\quad \int_{\T^n} G(x,y)\, dx=0,
\end{equation}
for every fixed $y\in\T^n$, where $\delta_y$ denotes the Dirac delta measure at $y \in \T^n$ (the $n$--torus $\T^n$ has unit volume).\\
By the properties of the Green function, $v_E$ is then the unique solution of
\begin{equation}
\begin{cases}\label{potential}
{\displaystyle{-\Delta v_E= u_E- m \qquad \text{in $\T^n$ (distributionally)}}}\\
{\displaystyle{\int_{\T^n} v_E(x) \, dx=0}}
\end{cases}
\end{equation}
where $m= \vol(E) - \vol(\T^n \setminus E)=2\vol(E) - 1$.
\begin{remark}
By standard elliptic regularity arguments (see~\cite{gt}, for instance), $v_E \in W^{2,p}(\T^n)$ for all $p\in [1, +\infty)$. More precisely, there exists a constant $C=C(n,p)$ such that $\|v_E\|_{W^{2,p}(\T^n)}\leq C$, for all $E\subseteq\T^n$ such that $\vol(E)-\vol(\T^n\setminus E)=m$.
\end{remark}

Then, we define the following {\em nonlocal Area functional} (see~\cite{KnZe,MuZa,StTo}, for instance).

\begin{definition}[Nonlocal Area functional]\label{NAFdef}
Given $\gamma \ge 0$, the \emph{nonlocal Area functional} $J$ is defined as
\begin{equation}\label{area}
J(E)= \A(\pa E) + \gamma \int_{\T^n} |\nabla {v_E}(x)|^2 \, dx,
\end{equation}
for every smooth set $E \subseteq \T^n$, where the function $v_E:\T^n\to\R$ is given by 
formulas~\eqref{potential1}--\eqref{potential} and 
$$
\A(\pa E)= \int_{\partial E} \, d\mu
$$
is the \emph{Area functional} , where $\mu$ is the ``canonical'' measure associated to the Riemannian metric on $\partial E$  induced by the metric tensor of $\T^n$, coming from the scalar product of $\R^n$ (it is easy to see that $\mu$ coincides with the  $(n-1)$--dimensional Hausdorff measure restricted to $\partial E$).
\end{definition}

{\em Since the nonlocal Area functional is defined adding to the Area functional a constant $\gamma\geq 0$ times a nonlocal term, all the results of this section will also hold for the Area functional, taking $\gamma=0$.}

\medskip

Multiplying by $v_E$ both sides of the first equation in system~\eqref{potential} and integrating by parts (and using also the second equation), we obtain
\begin{align}\label{G1}
\int_{\T^n} |\nabla  v_E(x)|^2 \, dx
=&\,- \int_{\T^n} v_E(x) \Delta v_E(x) \, dx \nonumber\\
=&\,\int_{\T^n} v_E(x) (u_E(x) - m) \, dx\nonumber\\
=&\,\int_{\T^n} v_E(x) u_E(x) \, dx\nonumber\\
=&\,\int_{\T^n}\int_{\T^n} G(x,y) u_E(x) u_E(y) \, dx\,dy,
\end{align}
hence, the functional $J$ can be also written in the useful form
$$
J(E)= \A(\pa E) + \gamma \int_{\T^n} \int_{\T^n} G(x,y)u_E(x)u_E(y) \, dx\,dy .
$$
\subsection{First and second variation}\label{sec1.2}\ \vskip.3em

We start by computing the \emph{first variation} of the functional $J$.

\begin{definition}\label{admissiblevar}
Let $E \subseteq \T^n$ be a smooth set. Given a smooth map $\Phi:(-\eps,\eps)\times\T^n\to\T^n$, for $\eps>0$, such that $\Phi_t=\Phi(t,\cdot):\T^n\to\T^n$ is a one--parameter family of diffeomorphism with $\Phi_0=\mathrm{Id}$, we say that $E_t=\Phi_t(E)$ is the {\em variation} of $E$ associated to $\Phi$ (or to $\Phi_t$). If moreover there holds $\mathrm{Vol}(E_t)=\mathrm{Vol}(E)$ for every $t\in(-\eps,\eps)$, we call $E_t$ a {\em volume--preserving} variation of $E$.\\
The vector field $X\in C^\infty(\T^n; \R^n)$ defined as $X=\frac{\pa \Phi_t}{\pa t}\,\bigr\vert_{t=0}$, is called the {\em infinitesimal generator} of the variation $E_t$. 
\end{definition}

\begin{remark} 
As we are going to consider only smooth sets $E$, it is easy to see that this definition of variation is equivalent to have a family of diffeomorphisms $\Phi_t$ of $E$ only, indeed these latter can always be extended to the whole $\T^n$. Moreover, as the relevant objects are actually the boundaries of the sets $E$ and in view of the sequel, we could even consider only smooth ``deformations'' of $\partial E$. We chose the above definition since it is easier and more convenient for the computations that are following.
\end{remark}

\begin{definition}
Given a variation $E_t$ of $E$, coming from the one--parameter family of diffeomorphism $\Phi_t$, the \emph{first variation of $J$ at $E$ with respect to $\Phi_t$} is given by
\begin{equation}
\label{first variation:def}
\frac{d}{dt}J(E_t)\Bigl|_{t=0} \, .
\end{equation}
We say that $E$ is a {\em critical set} for $J$, if all the first variations relative to variations $E_t$ of $E$ are zero.\\
We say that $E$ is a {\em critical set} for $J$ under a volume constraint, if all the first variations relative to volume--preserving variations $E_t$ of $E$ are zero.
\end{definition}

It is clear that if $E$ is a minimum for $J$ (under a volume constraint), then it is a critical set for $J$ (under a volume constraint).
We are now going to compute the first variation of $J$ and see that it depends only on the restriction to $\partial E$ of the infinitesimal generator $X$ of the variation $E_t$ of $E$.
   
We briefly recall some ``geometric'' notations and results about the (Riemannian) geometry of the hypersurfaces in $\R^n$, referring to~\cite{gahula,Man,petersen2} for instance.

\smallskip

{\em In the whole work, we will adopt the convention of summing over the repeated indices.}

\smallskip

Given any smooth immersion $\psi:M \to \T^n$ of the smooth, $(n-1)$--dimensional, compact manifold $M$, representing a hypersurface $\psi(M)$ of $\T^n$, considering local coordinates around any $p\in M$, we have local bases of the tangent space $T_p M$, which can be identified with the $(n-1)$--dimensional hyperplane $d\psi_p(T_pM)$ of $\R^n\approx T_{\psi(p)}\T^n$ which is tangent to $\psi(M)$ at $\psi(p)$, and of the cotangent space $T_p^{*}M$, respectively given by vectors $\bigl\{\frac{\partial\,}{\partial x_i}\bigr\}$ and 1--forms $\{dx_j\}$. So we denote the vectors on $M$ by  $X=X^i\frac{\partial\,}{\partial x_i}$ and the 1--forms by $\omega=\omega_jdx_j$, where the indices refer to the chosen local coordinate chart of $M$. With the above identification, we have clearly $\frac{\partial\,}{\partial x_i}\approx \frac{\partial\psi}{\partial x_i}$, for every $i\in\{1,\dots,n-1\}$.

The manifold $M$ gets in a natural way a metric tensor $g$, pull--back via the map $\psi$ of the metric tensor of $\T^n$, coming from the standard scalar product of $\R^n$ (as $\T^n\approx\R^n/\Z^n$), hence, turning it into a Riemannian manifold $(M,g)$. Then, the components of $g$ in a local chart are
$$
g_{ij}=\left\langle\frac{\pa \psi}{\pa x_i}\,\right\vert\left.\!\frac{\pa \psi}{\pa x_j}\right\rangle
$$
and the ``canonical'' measure $\mu$, induced on $M$ by the metric $g$ is then given by $\mu=\sqrt{\det g_{ij}}\,{\Lebes}^{n-1}$, where ${\Lebes}^{n-1}$ is the standard Lebesgue measure on $\R^{n-1}$. 

Thus, supposing that $M$ has a {\em global} coordinate chart, we can write the Area functional on the hypersurface $\psi(M)$ in the following way, 
\begin{equation}\label{areachart}
\A(\psi (M))=\int_{M} d\mu = \int_{M} \sqrt{\det g_{ij}(x)} \, dx\,.
\end{equation}
When this is not the case (as it is usual), we need several local charts $(U_k,\varphi_k)$ and a subordinated partitions of unity $f_k:M\to[0,1]$ (that is, the compact support of $f_k:M\to[0,1]$ is contained in the open set $U_k\subseteq M$, for every $k\in\mathcal{I}$), then
\begin{equation}
\A(\psi (M))=\int_{M} d\mu =\sum_{k\in\mathcal I}\int_{M} f_k\,d\mu=\sum_{k\in\mathcal I}\int_{U_k} f_k(x)\sqrt{\det g_{ij}^k(x)} \, dx\,,
\end{equation}
where $g^k_{ij}$ are the coefficients of the metric $g$ in the local chart $(U_k,\varphi_k)$.

\smallskip

{\em In order to work with coordinates, in the computations with integrals in this section we will assume that all the hypersurfaces have a global coordinate chart, by simplicity. All the results actually hold also in the general case by using partitions of unity as above.}

\smallskip

The induced Levi--Civita covariant derivative on $(M,g)$ of a vector field $X$ and of a 1--form $\omega$ 
are respectively given by
$$
\nabla _jX^i=\frac{\partial X^i}{\partial
  x_j}+\Gamma^{i}_{jk}X^k\,, \qquad \nabla _j\omega_i=\frac{\partial \omega_i}{\partial
  x_j}-\Gamma^k_{ji}\omega_k\,, 
$$
where $\Gamma^{i}_{jk}$ are the Christoffel symbols of the connection $\nabla$, expressed
by the formula
$$
\Gamma^{i}_{jk}=\frac{1}{2} g^{il}\Bigl(\frac{\partial\,}{\partial
    x_j}g_{kl}+\frac{\partial\,}{\partial
    x_k}g_{jl}-\frac{\partial\,}{\partial
    x_l}g_{jk}\Bigr)\,.
$$
Moreover, the gradient $\nabla  f$ of a function, the divergence $\Div X$ of a tangent vector field and the Laplacian $\Delta f$ at a point $p \in M$, are defined respectively by
$$
g(\nabla  f(p) , v)=df_p(v)\qquad\forall v\in T_p M\,,
$$
\begin{equation}\label{divformcar}
\Div X={\mathrm {tr}} \nabla  X=\nabla _iX^i=\frac{\partial X^i}{\partial x_i}+\Gamma^{i}_{ik}X^k
\end{equation}
(in a local chart) and $\Delta f =\Div\nabla f$. 
We then recall that by the {\em divergence theorem} for compact manifolds (without boundary), there holds
\begin{equation}\label{divteo}
\int_{M}\Div X\,d\mu=0\,,
\end{equation}
for every tangent vector field $X$ on $M$, which in particular implies
\begin{equation}\label{corollariodivteo}
\int_{M}\Delta f\,d\mu=0\,,
\end{equation}
for every smooth function $f:M \to\R$.

Assuming that we have a globally defined unit {\em normal} vector field $\nu:M\to\R^n$ to $\varphi(M)$ (this will hold in our situation where the hypersurfaces will be boundaries of smooth sets $E\subseteq\T^n$, hence we will always consider $\nu$ to be the {\em outer unit normal vector} at every point of $\pa E$), we define the {\em second fundamental form} $B$ which is a symmetric bilinear form given, in a local charts, by its components
\begin{equation}\label{secform}
h_{ij} = - \biggl \langle \frac{\pa^2 \psi}{\pa x_i \pa x_j}\,\biggr\vert \,\nu  \biggr \rangle
\end{equation}
and whose trace is the {\em mean curvature} $\HHH= g^{ij} h_{ij}$ of the hypersurface (with these choices, the standard sphere of $\R^n$ has positive mean curvature).\\
The symmetry properties of the covariant derivative of $B$ are given by the {\em Codazzi--Mainardi equations}
\begin{equation}\label{codazzi}
\nabla _ih_{jk}=\nabla _jh_{ik}=\nabla _kh_{ij}\,.
\end{equation}
In the sequel, the following {\em Gauss--Weingarten relations} will be fundamental,
\begin{equation}\label{GW}
\frac{\pa^2\psi}{\pa x_i\pa
  x_j}=\Gamma_{ij}^k\frac{\pa\psi}{\pa
  x_k}- h_{ij}\nu\qquad\qquad\frac{\pa\nu}{\pa x_j}=
h_{jl}g^{ls}\frac{\pa\psi}{\pa x_s}\,,
\end{equation}
which imply
\begin{equation}\label{lap}
\Delta\psi=g^{ij}\Bigl(\frac{\partial^2\psi}{\partial x_i\partial
  x_j}-\Gamma_{ij}^k\frac{\partial\psi}{\partial
  x_k}\Bigr)=-g^{ij}h_{ij}\nu=-\HHH\nu\,.
\end{equation}
Moreover, we have the formula
\begin{equation}\label{Deltanu}
\Delta \nu = \nabla \HHH -|B|^2\nu\,,
\end{equation}
indeed, computing in {\em normal coordinates} at a point $p\in M$, 
\begin{align*}
\Delta\nu=&\,g^{ij}\Bigl(\frac{\partial^2\nu}{\partial x_i\partial x_j}-\Gamma_{ij}^k\frac{\partial\nu}{\partial  x_k}\Bigr)\\
=&\,g^{ij}\frac{\partial}{\partial x_i}\Bigl(h_{jl}g^{ls}\frac{\pa\psi}{\pa x_s}\Bigr)\\
=&\,g^{ij}\nabla_i h_{jl}g^{ls}\frac{\pa\psi}{\pa x_s}+g^{ij}h_{jl}g^{ls}\frac{\partial^2\psi}{\partial x_i\pa x_s}\\
=&\,g^{ij}\nabla_l h_{ij}g^{ls}\frac{\pa\psi}{\pa x_s}-g^{ij}h_{jl}g^{ls}h_{is}\nu\\
=&\,\nabla\HHH  -|B|^2\nu\,,
\end{align*}
since all $\Gamma_{ij}^k$ and $\frac{\partial}{\partial x_i}g^{jk}$ are zero at $p\in M$ in such coordinates and we used Codazzi--Mainardi equations~\eqref{codazzi}.

\medskip

{\em In the following, when it is clear by the context, we will write $\nabla $, $\Div$ and $\Delta$ for both the Riemannian operators on a hypersurface and the standard operators of $\T^n\approx\R^n/\Z^n$, but these latter will be instead denoted by $\nabla ^{\T^n}$, $\Div^{\!\T^n}$ and $\Delta^{\!\T^n}$ when they will be computed at a point of a hypersurface, in order to avoid any possibility of misunderstanding.}

\medskip

\begin{thm}[First variation of the functional $J$]\label{first var}
Let $E\subseteq \T^n$ a smooth set and $\Phi:(-\eps,\eps)\times \T^n\to\T^n$ a smooth map giving a variation $E_t=\Phi_t(E)$ with infinitesimal generator $X \in C^\infty (\T^n; \R^n)$. Then, 
\begin{equation}\label{eqcar2222}
\frac{d}{dt} J(E_t)\Bigl|_{t=0}=\int_{\partial E} (\HHH+ 4 \gamma v_E) \langle X \vert \nu_E\rangle \, d\mu
\end{equation}
where $\nu_E$ is the outer unit normal vector and $\HHH$ the mean curvature of the boundary 
$\partial E$ (as defined above, relative to $\nu_E$), while the function $v_E:\T^n\to \R$ is the potential associated to $E$, defined by formulas~\eqref{potential1}--\eqref{potential}.\\
In particular, the first variation of the functional $J$ depends only on the normal component of the restriction of the infinitesimal generator $X$ to $\pa E$.\\
Clearly, when $\gamma=0$ we get the well known first variation of the Area functional at a smooth set $E$,
\begin{equation}
\frac{d}{dt}\A(\pa E_t)\Bigl|_{t=0}= \int_{\partial E} \HHH\langle X \vert \nu_E\rangle \, d\mu\,.
\end{equation}
\end{thm}
\begin{proof}
We start by computing the derivative of the Area functional term of $J$. We let $\psi_t:\partial E\to\T^n$ be the embedding given by 
$$
\psi_t(x)=\Phi(t,x)\,,
$$
for $x\in\partial E$ and $t\in(-\eps,\eps)$, then $\psi_t(\partial E)=\partial E_t$ and $\partial_t\psi_t\bigl\vert_{t=0}=X$ at every point of $\partial E$, moreover $\psi_0$ is simply the inclusion map of $\pa E$ in $\T^n$.\\
Denoting by $g_{ij}=g_{ij}(t)$ the induced metrics (via $\psi_t$, as above) on the smooth hypersurfaces $\partial E_t$ and setting $\psi=\psi_0$, in a local chart we have 
\begin{align*}
\frac{\partial\,}{\partial t}g_{ij}\,\Bigr\vert_{t=0}\,
=&\,\left.\frac{\partial\,}{\partial
  t}\left\langle\frac{\partial\psi_t}{\partial
    x_i}\,\right\vert\left.\frac{\partial\psi_t}{\partial
    x_j}\right\rangle\right\vert_{t=0}\\
\,=&\,\left\langle\frac{\partial X}{\partial
    x_i}\,\right\vert\left.\frac{\partial\psi}{\partial
    x_j}\right\rangle+\left\langle\frac{\partial X}{\partial
    x_j}\,\right\vert\left.\frac{\partial\psi}{\partial
    x_i}\right\rangle\\
\,=&\,\frac{\partial\,}{\partial x_i}\left\langle
  X\,\left\vert\,\frac{\partial\psi}{\partial
    x_j}\right\rangle\right.
+\frac{\partial\,}{\partial x_j}\left\langle
  X\,\left\vert\,\frac{\partial\psi}{\partial
    x_i}\right\rangle\right.
-2\left\langle
  X\,\left\vert\,\frac{\partial^2\psi}{{\partial
    x_i}{\partial x_j}}\right\rangle\right.\\
\,=&\,\frac{\partial\,}{\partial x_i}\left\langle
  X_\tau\,\left\vert\,\frac{\partial\psi}{\partial
    x_j}\right\rangle\right.
+\frac{\partial\,}{\partial x_j}\left\langle
  X_\tau\,\left\vert\,\frac{\partial\psi}{\partial
    x_i}\right\rangle\right.
-2\Gamma_{ij}^k\left\langle
  X_\tau\,\left\vert\,\frac{\partial\psi}{\partial
    x_k}\right\rangle\right. +2h_{ij}\langle X\,\vert\,\nu_E\rangle\,,
\end{align*}
where we used the Gauss--Weingarten relations~\eqref{GW} in the last step and we denoted with $X_\tau= X-\langle X \vert \nu_E \rangle\nu_E$ the ``tangential part'' of the vector field $X$ along the hypersurface $\partial E$ (seeing $T_x\pa E$ as a hyperplane of $\R^n\approx T_x\T^n$).\\
Letting $\omega$ be the $1$--form defined by $\omega(Y)=g(X_\tau,Y)$, 
this formula can be rewritten as
\begin{equation}
\frac{\partial}{\partial t}g_{ij}\Bigr \vert_{t=0}=
\frac{\partial\omega_j}{\partial x_i}
+\frac{\partial\omega_i}{\partial x_j}
- 2\Gamma_{ij}^k\omega_k + 2h_{ij} \langle X| \nu_E  \rangle = \nabla _i\omega_j+\nabla _j\omega_i+ 2h_{ij} \langle X| \nu_E \rangle \,.\label{derg2}
\end{equation}
Hence, by the formula
\begin{equation}\label{detform}
\frac{d}{dt}\det A(t)=\det A(t)\,{\mathrm{tr}}\,[A^{-1}(t)\circ A'(t)]\,,
\end{equation}
holding for any $n\times n$ squared matrix $A(t)$ dependent on $t$, we get
\begin{align}
\frac{\partial\,}{\partial t}\sqrt{\det g_{ij}}\,\Bigr\vert_{t=0}\,
=&\,\frac{\sqrt{\det g_{ij}}\,
g^{ij}\left.\!\!\frac{\partial\,}{\partial t}g_{ij}\right\vert_{t=0}}{2}\\
=&\,\frac{\sqrt{\det g_{ij}}\,
g^{ij}\bigl(\nabla _i\omega_j+ 
\nabla _j\omega_i+2h_{ij}\langle X\,\vert\,\nu_E\rangle\bigr)}{2}\\
=&\,\sqrt{\det g_{ij}}\bigl(\Div\!X_\tau +\HHH\langle X\,\vert\,\nu_E\rangle\bigr)\,,\label{dermu2}
\end{align}
where the divergence is the (Riemannian) one relative to the hypersurface $\pa E$.
Then, we conclude (recalling the discussion after formula~\eqref{areachart})
\begin{align}
\frac{\partial\,}{\partial t}{\A}(\pa E_t) \,\Bigr\vert_{t=0}=\frac{\partial\,}{\partial t}{\A}(\psi_t(\pa E)) \,\Bigr\vert_{t=0}\,
=&\,\frac{\partial\,}{\partial t}\,\int_{\pa E}\,d\mu_t\,\Bigr\vert_{t=0}\nonumber\\
=&\,\frac{\partial\,}{\partial t}\,\int_{\pa E}\sqrt{\det g_{ij}}\,dx\,\Bigr\vert_{t=0}\nonumber\\
=&\,\int_{\pa E}\frac{\partial\,}{\partial t}\sqrt{\det g_{ij}}\,\Bigr\vert_{t=0}\,dx\nonumber\\
=&\,\int_{\pa E}\!\bigl(\Div\!X_\tau\!+\!\HHH\langle X\,\vert\,\nu_E\rangle\bigr)\sqrt{\det g_{ij}}\,dx\nonumber\\
=&\,\int_{\pa E}\!\bigl(\Div\!X_\tau\!+\!\HHH\langle X\,\vert\,\nu_E\rangle\bigr)\,d\mu\nonumber\\
=&\,\int_{\pa E}\HHH\langle X\,\vert\,\nu_E\rangle\,d\mu\label{local}
\end{align}
where in the last step we applied the divergence theorem, that is, formula~\eqref{divteo}, on $\pa E$.

In order to compute the derivative of the nonlocal term, we set
\begin{equation}\label{eqc0}
v(t,x)= v_{E_t}(x)= \int_{\T^n} G(x,y) u_{E_t}(x) \, dy = \int_{E_t} G(x,y) \, dy - \int_{E_t^c} G(x,y) \, dy,
\end{equation}
where $E_t^c = \T^n \setminus E_t$.
Then,
\begin{align}
\frac{d}{dt} \Bigl( \int_{\T^n} |\nabla  v_{E_t}(x)|^2 \, dx \Bigr) \Bigl|_{t=0} &= \frac{d}{dt} \Bigl( \int_{\T^n} |\nabla  v(t,x)|^2 \, dx\Bigl)\Bigl|_{t=0}\\
&= 2 \int_{\T^n} \nabla  v_E(x) \frac{\partial}{\partial t} \nabla  v(t,x)\Bigl |_{t=0} \, dx\\
&= 2 \int_{\T^n} (u_E(x) - m) \frac{\partial}{\partial t} v(t,x) \Bigl |_{t=0} \, dx,
\end{align}
where in the last equality we used the fact that $- \Delta v_E = u_E - m$ and we integrated by parts. Now, we note that \begin{equation}\label{eqc1}
\frac{\partial}{\partial t} v(t,x) = \frac{\partial}{\partial t} \Bigr (\int_{E_t} G(x,y) \, dy \Bigl) -  \frac{\partial}{\partial t} \Bigl (\int_{E_t^c} G(x,y) \, dy \Bigr),
\end{equation}
and, by a change of variable,
\begin{equation}
\label{eqqq50}
\frac{\partial}{\partial t} \Bigl (\int_{E_t} G(x,y) \, dy \Bigr) \Bigl |_{t=0}= \frac{\partial}{\partial t} \Bigl (\int_{E} G(x,\Phi(t,z)) J\Phi(t,z) \, dz \Bigr) \Bigl |_{t=0},
\end{equation}
where $J\Phi(t,\cdot)$ is the Jacobian of $\Phi(t, \cdot)$. Then, as $J\Phi(t,z)=\det [d\Phi(t,z)]$, using again formula~\eqref{detform}, we have
\begin{align}
\frac{\pa }{\pa t} J\Phi(t,z)\,\Bigr\vert_{t=0}= &\,J\Phi(t,z)\,{\mathrm{tr}}\,\Bigl[d\Phi(t,z)^{-1}\circ \frac{\pa }{\pa t} d\Phi(t,z)\Bigl]\,\Bigr\vert_{t=0}\\
= &\,J\Phi(t,z)\,{\mathrm{tr}}\,\Bigl[d\Phi(t,z)^{-1}\circ d\frac{\pa }{\pa t} \Phi(t,z)\Bigl]\,\Bigr\vert_{t=0}\\
= &\,{\mathrm{tr}}\,dX(z)\\
=&\,\Div X(z)\,,
\end{align}
by the definition of $X$ and being $\Phi(0,z)=z$. Thus, carrying the time derivative inside the integral in equation~\eqref{eqqq50}, we obtain
\begin{align}
\frac{\partial}{\partial t} \Bigl (\int_{E_t} G(x,y) \, dy \Bigr) \Bigl |_{t=0}
=&\,\int_E \bigl(\langle\nabla _y G(x,y) \vert X(y)\rangle + G(x,y) \Div\!X(y)\bigr)\,dy \\
=&\,\int_E \Div_y \bigl(G(x,y) X(y)\bigr) \, dy\\ 
=&\,\int_{\partial E} G(x,y) \langle X(y)\vert \nu_E(y)\rangle \, d\mu(y)\,.\label{eqc2}
\end{align}
By a very analogous computation we get
\begin{equation}\label{eqc3}
-\frac{\partial}{\partial t} \Bigl (\int_{E_t^c} G(x,y) \, dy \Bigr) \Bigl |_{t=0}= \int_{\partial E} G(x,y)\langle X(y) \vert \nu_E(y) \rangle \, d\mu(y)\,,
\end{equation}
then, using equalities~\eqref{potential1} and~\eqref{G2}, we conclude
\begin{align}
\frac{d}{dt}\int_{\T^n} |\nabla  v_{E_t}(x) |^2 \, dx\,\biggr|_{t=0}
=&\, 4 \int_{\T^n} (u_E(x) - m) \Bigl ( \int_{\pa E}G(x,y)\langle X(y) \vert \nu_E(y)\rangle \, d\mu(y)\Bigr)\,dx\nonumber\\
=&\,4 \int_{\pa E} \Bigl (\int_{\T^n}G(x,y)(u_E(x)-m) \, dx \Bigr) \langle X(y) \vert \nu_E(y)\rangle \, d\mu(y)\nonumber\\
=&\,4 \int_{\pa E} v_E(y)  \langle X(y) \vert \nu_E(y)\rangle \, d\mu(y)\, .\label{nonlocal}
\end{align}
Combining formulas~\eqref{local} and~\eqref{nonlocal}, we finally obtain formula~\eqref{eqcar2222}.
\end{proof}

Given a smooth set $E$ and any vector field $X\in C^\infty(\T^n; \R^n)$, considering the associated smooth flow $\Phi: (-\eps,\eps)\times \T^n \to \T^n$, defined by the system
\begin{equation}\label{varflow}
\begin{cases}
\frac{\pa \Phi}{\pa t}(t,x)=X(\Phi(t,x)), \\ \Phi(0, x)=x
\end{cases}
\end{equation}
for every $x \in \T^n$ and $t \in (-\eps, \eps)$, for some $\eps>0$, we have a variation $E_t=\Phi_t(E)$ with infinitesimal generator $X$. We call this variation the {\em special variation} associated to $X$. Moreover, given any smooth vector field $\overline{X}\in C^\infty(\pa E; \R^n)$, it can be extended easily to a smooth vector field $X\in C^\infty(\T^n; \R^n)$ with $X\vert_{\pa E}=\overline{X}$.

Hence, if $E$ is a critical set for $J$ there holds 
\begin{equation*}
\int_{\partial E} (\HHH+ 4 \gamma v_E) \langle X \vert \nu_E\rangle \, d\mu=0\,,
\end{equation*}
for every $X\in C^\infty(\T^n; \R^n)$. Choosing a smooth vector field $X\in C^\infty(\T^n; \R^n)$ with $X\vert_{\pa E}=(\HHH+ 4 \gamma v_E)\nu_E$, we then obtain the following corollary.

\begin{cor}\label{critcor}
A smooth set $E \subseteq \T^n$ is a critical set for $J$ if and only if the function $\HHH + 4 \gamma v_E$ is zero on $\pa E$.
When $\gamma =0$, we recover the classical condition $\HHH=0$ for a {\em minimal surface} in $\R^n$.
\end{cor}

It is less easy to characterize the infinitesimal generators of the volume--preserving variations of $E$, in order to find an analogous criticality condition on a set $E$, for the functional $J$ under a volume constraint.\\
Given $\Phi:(-\eps,\eps)\times\T^n\to\T^n$ such that $\vol(\Phi_t(E))=\vol(E_t)=\vol(E)$ for all $t \in(-\eps,\eps)$, we let $X_t\in C^\infty(\T^n; \R^n)$ be the family of the vector fields (well) defined by the formula
$$
X_t(\Phi(t,z))=\frac{\pa \Phi}{\pa t}(t,z),
$$
for every $t\in(-\eps,\eps)$ and $z\in\T^n$, hence, if $t=0$, the vector field $X=X_0$ is the infinitesimal generator of the volume--preserving variation $E_t$. Then, by changing variables, we have
\begin{equation}\label{eqc999}
0= \frac{d}{dt} \vol(E_t)=\frac{d}{dt}\int_{E_t}\,dx=\frac{d}{dt}\int_E J\Phi(t,z)\,dz= \int_E \frac{\partial}{\partial t} J\Phi(t,z)\,dz\,.
\end{equation}
As $J\Phi(t,z)=\det [d\Phi(t,z)]$, by means of formula~\eqref{detform}, we obtain
\begin{equation}
\frac{\pa }{\pa t} J\Phi(t,z)= J\Phi(t,z)\,{\mathrm{tr}}\,[d\Phi(t,z)^{-1}\circ dX_t(\Phi(t,z))\circ d\Phi(t,z)],
\end{equation}
since, by the definition of $X_t$ above, 
$$
\frac{\pa }{\pa t} d\Phi(t,z)=d\,\frac{\pa \Phi}{\pa t}(t,z)=d[X_t(\Phi(t,z))]=dX_t(\Phi(t,z))\circ d\Phi(t,z).
$$
Being the trace of a matrix invariant by conjugation, we conclude
\begin{equation}\label{tJac}
\frac{\pa }{\pa t} J\Phi(t,z)= J\Phi(t,z)\,{\mathrm{tr}}\,[dX_t(\Phi(t,z))]= J\Phi(t,z)\Div\!X_t(\Phi(t,z)),
\end{equation}
hence, by equality~\eqref{eqc999} and the divergence theorem (in $\T^n$), it follows
\begin{equation}\label{eqc1000}
0=\int_E \Div\!X_t(\Phi(t,z))J\Phi(t,z)\,dz=\int_{E_t} \Div\!X_t(x)\,dx=\int_{\pa E} \langle X_t\circ\Phi_t \vert \nu_{E_t}\rangle \, d\mu_t\,,
\end{equation}
where $\nu_{E_t}$ is the outer unit normal vector and $\mu_t$ the canonical Riemannian measure of the smooth hypersurface $\partial E_t$, given by the embedding $\psi_t=\Phi_t:\pa E\to\T^n$. Thus, letting $t=0$,
\begin{equation}\label{eqc1000bis}
\frac{d}{dt} \vol(E_t)\Bigr\vert_{t=0}=\int_{\pa E} \langle X\vert \nu_E\rangle \, d\mu=0
\end{equation}
and we conclude that if $X\in C^\infty(\T^n; \R^n)$ is the infinitesimal generator of a volume--preserving variation for $E$, its normal component $\varphi=\langle X \vert \nu_E\rangle$ on $\pa E$ has zero integral (with respect to the measure $\mu$).\\
Conversely, we have the following lemma whose proof is postponed after Lemma~\ref{lemma1}, since the arguments in the two proofs are very similar.

\begin{lem}\label{vector field}
Let $\varphi:\partial E\to\R$ a smooth function with zero integral with respect to the measure $\mu$ on $\pa E$. Then, there exists a smooth vector field $X\in C^\infty(\T^n; \R^n)$ such that $\varphi=\langle X \vert \nu_E\rangle$, $\Div\! X=0$ in a neighborhood of $\pa E$ and the flow $\Phi$ defined by system~\eqref{varflow} having $X$ as infinitesimal generator, gives a volume--preserving variation $E_t=\Phi_t(E)$ of $E$.
\end{lem}

Hence, with this characterization of the infinitesimal generators of the volume--preserving variations for $E$, by Theorem~\ref{first var} we have that $E$ is a critical set for the functional $J$ under a volume constraint if and only if
\begin{equation*}
\int_{\partial E} (\HHH+ 4 \gamma v_E) \langle X \vert \nu_E\rangle \, d\mu=0\,,
\end{equation*}
for every $X\in C^\infty(\T^n; \R^n)$ such that $\langle X \vert \nu_E\rangle$ has zero integral on $\pa E$. By Lemma~\ref{vector field}, this is similarly to say that
$$
\int_{\pa E}(\HHH + 4 \gamma v_E) \varphi \, d\mu =0 \,,
$$
for all $\varphi \in C^{\infty}(\pa E)$ such that $\int_{\pa E} \varphi \, d\mu =0$, which is equivalent to the existence of a constant $\lambda \in \R$ such that
\begin{equation}\label{EL}
\HHH + 4 \gamma v_E= \lambda \qquad \text{on $\pa E$.}
\end{equation}

\begin{remark} The parameter $\lambda$ may be clearly interpreted as a \emph{Lagrange multiplier} associated with the volume constraint for $J$.
\end{remark}

\begin{proposition}\label{critprop}
A smooth set $E \subseteq \T^n$ is a critical set for $J$ under a volume constraint if and only if the function $\HHH + 4 \gamma v_E$ is constant on $\pa E$.
When $\gamma =0$, we recover the classical \emph{constant mean curvature} condition for hypersurfaces in $\R^n$.
\end{proposition}

Now we deal with the {\em second variation} of the functional $J$.

\begin{definition}
Given a variation $E_t$ of $E$, coming from the one--parameter family of diffeomorphism $\Phi_t$, the \emph{second variation of $J$ at $E$ with respect to $\Phi_t$} is given by
\begin{equation}\label{variazioneseconda}
\frac{d^2}{dt^2} J(E_t)\Bigl|_{t=0}\,.
\end{equation}
\end{definition}

In the following proposition we compute the second variation of the Area functional.
Then, we do the same for the nonlocal term of $J$ and we conclude with the second variation of the functional $J$.

\begin{proposition}[Second variation of $\A$]\label{secondvarA}
Let $E\subseteq \T^n$ a smooth set and $\Phi:(-\eps,\eps)\times \T^n\to\T^n$ a smooth map giving a variation $E_t=\Phi_t(E)$ with infinitesimal generator $X \in C^\infty (\T^n; \R^n)$. Then, 
\begin{align}
\frac{d^2}{dt^2} \A(\pa E_t)\Bigr|_{t=0}=&\, \int_{\pa E}\bigl(|\nabla  \langle X\vert \nu_E\rangle|^2- \langle X\vert \nu_E\rangle^2|B|^2\bigr)\, d\mu\\
&+\int_{\pa E}\HHH\bigl(\HHH \langle X\vert \nu_E\rangle^2+\langle Z\vert\nu_E\rangle-2\langle X_\tau |\nabla \langle X\vert \nu_E\rangle\rangle+ B(X_\tau, X_\tau)\bigr)\, d\mu\,,
\end{align}
where $X_\tau=X-\langle X|\nu_E\rangle\nu_E$ is the {\em tangential part} of $X$ on $\partial E$, $B$ and $\HHH$ are respectively the second fundamental form and the mean curvature of $\pa E$, and
\begin{equation}\label{Zdef}
Z= \frac{\pa^2\Phi}{\pa t^2}(0, \cdot )=\frac{\pa}{\pa t}[X_t(\Phi (t, \cdot))]\,\Bigr\vert_{t=0}=\frac{\pa X_t}{\pa t}\,\Bigr\vert_{t=0}+ d X(X)\,,
\end{equation}
where, for every $t\in(-\eps,\eps)$, the vector field $X_t\in C^\infty(\T^n; \R^n)$ is defined by the formula
$$
X_t(\Phi(t,z))=\frac{\pa \Phi}{\pa t}(t,z),
$$
for every $z\in\T^n$, hence, $X_0=X$.
\end{proposition}
\begin{proof}
We let $\psi_t=\Phi(t,\cdot)\vert_{\partial E}$. By arguing as in the first part of the proof of Theorem~\ref{first var} (without taking $t=0$), we have
$$
\frac{d}{dt}\A(\pa E_t) = \int_{\pa E} \HHH_t \langle X_t\circ\Phi_t\vert \nu_{E_t}\rangle \, d\mu_t,
$$
where $\HHH_t$ is the mean curvature of $\pa E_t$.
Consequently, we have
\begin{equation}
\frac{d^2}{dt^2} \A(\pa E_t)\,\Bigr|_{t=0}= \frac{d}{dt} \int_{\pa E} \HHH_t \langle X_t\circ\Phi_t \vert \nu_{E_t}\rangle \sqrt{\det g_{ij}} \, dx\, \Bigr |_{t=0}
\end{equation}
where $g_{ij}=g_{ij}(t)$.\\
In order to simplify the notation in the following computations, we drop the subscripts, that is, we let $\HHH(t,\cdot) = \HHH_t$, $\nu(t,\cdot)=\nu_{E_t}$, $\varphi(t,\cdot)=\langle X_t\circ\Phi_t\vert \nu_{E_t}\rangle$, $\psi(t,\cdot)= \psi_t$ and $X(t,\cdot)=X_t\circ\Phi_t$ (by a little abuse of notation, since $X$ is already the infinitesimal generator of the variation).\\
We then need to compute the derivatives
\begin{equation}\label{1}
\frac{\pa \HHH}{\pa t}\, \Bigr |_{t=0}\qquad\qquad\text{ and }\qquad\qquad
\frac{\pa}{\pa t}\langle X\vert \nu\rangle\, \Bigr |_{t=0}
\end{equation}
since we already know, by formula~\eqref{dermu2}, that
\begin{equation}
\label{3}
\frac{\pa }{\pa t} \sqrt{\det g_{ij}}\, \Bigr |_{t=0}=\bigl(\Div\!X_{\tau} +\HHH \varphi\bigr)\sqrt{\det g_{ij}}\, \Bigr |_{t=0}\,,
\end{equation}
hence, this derivative gives the following contribution to the second variation,
\begin{equation}
\label{contr 3}
\int_{\pa E}( \varphi \HHH \Div\!X_{\tau} + \varphi^2 \HHH^2) \, d\mu\,.
\end{equation}
Then, we compute (recalling formula~\eqref{Zdef})
$$
\frac{\pa \langle X \vert \nu\rangle}{\pa t}\, \biggr |_{t=0}= \left \langle \frac{\pa X}{\pa t} \biggl\vert \nu\right \rangle\biggr |_{t=0} +\left \langle X \biggl\vert \frac{\pa \nu}{\pa t}\right \rangle\biggr |_{t=0}= \left \langle  Z\vert \nu\right \rangle+\left \langle X \biggl\vert \frac{\pa \nu}{\pa t}\right \rangle\biggr |_{t=0}
$$
and using the fact that $\frac{\pa \nu}{\pa t}\bigr |_{t=0}$ is tangent to $\pa E$, in a local coordinate chart we obtain 
$$
\left \langle X \biggl\vert \frac{\pa \nu}{\pa t}\right\rangle \biggr |_{t=0}=X_{\tau}^p \left\langle \frac{\pa \psi}{\pa x_p}\biggl \vert \frac{\pa \nu}{\pa t}\right\rangle \biggr |_{t=0},
$$
where in the last inequality we used the notation $X_{\tau}=X_\tau ^p \frac{\pa \psi}{\pa x_p}$. Notice that, $\bigl \langle \frac{\pa \psi}{\pa x_p}\bigl \vert \nu \bigr\rangle=0$ for every $p\in\{1,\dots,n-1\}$ and $t\in(-\eps,\eps)$, hence, using the Gauss--Weingarten relations~\eqref{GW},
\begin{align}
0&= \frac{\pa}{\pa t} \left \langle \frac{\pa \psi}{\pa x_p}\biggl\vert \nu\right\rangle\biggr |_{t=0}=\left\langle\frac{\pa X}{\pa x_p}\biggl\vert\nu\right\rangle+\left\langle \frac{\pa \psi}{\pa x_p}\bigvert \frac{\pa \nu}{\pa t}\right\rangle \biggr |_{t=0}\\
&=\frac{\pa}{\pa x_p}\langle X\vert \nu\rangle - \left\langle X \bigvert \frac{\pa \nu}{\pa x_p}\right\rangle+\left\langle\frac{\pa \psi}{\pa x_p}\bigvert \frac{\pa \nu}{\pa t}\right\rangle \biggr |_{t=0}\\
&=\frac{\pa \varphi}{\pa x_p}-\left\langle X_\tau \bigvert \frac{\pa \nu}{\pa x_p}\right\rangle + \left\langle\frac{\pa \psi}{\pa x_p}\bigvert \frac{\pa \nu}{\pa t}\right\rangle \biggr |_{t=0}\\
&=\frac{\pa \varphi}{\pa x_p}- X_{\tau}^q\left\langle\frac{\pa \psi}{\pa x_q}\bigvert \frac{\pa \nu}{\pa x_p}\right\rangle+ \left\langle\frac{\pa \psi}{\pa x_p}\bigvert \frac{\pa \nu}{\pa t}\right\rangle \biggr |_{t=0}\\
&=\frac{\pa \varphi}{\pa x_p}- X_{\tau}^q \left\langle\frac{\pa \psi}{\pa x_q}\bigvert h_{pl}g^{li} \frac{\pa \psi}{\pa x_i}\right\rangle+ \left\langle\frac{\pa \psi}{\pa x_p}\bigvert \frac{\pa \nu}{\pa t}\right\rangle \biggr |_{t=0}\\
&=\frac{\pa \varphi}{\pa x_p} - X_{\tau} ^q h_{pl}g^{li}g_{qi}+\left\langle\frac{\pa \psi}{\pa x_p}\bigvert \frac{\pa \nu}{\pa t}\right\rangle \biggr |_{t=0}
\end{align}
and we can conclude that
\begin{equation}\label{eqcar3333}
\left\langle \frac{\pa \psi}{\pa x_p} \bigvert \frac{\pa \nu}{\pa t}\right\rangle \biggr |_{t=0} = -\frac{\pa \varphi}{\pa x_p}+ X_{\tau}^q h_{pq}\,,
\end{equation}
where $h_{pq}$ are the components of the second fundamental form $B$ of $\pa E$ in the local chart.
Thus, we obtain the following identity 
\begin{align}
\frac{\pa}{\pa t}\langle X\vert \nu\rangle\, \Bigr |_{t=0}&=\langle Z\vert \nu\rangle + X_{\tau}^p \left\langle\frac{\pa \psi}{\pa x_p} \bigvert \frac{\pa \nu}{\pa t}\right\rangle \biggr |_{t=0}\\
&=\langle Z \vert \nu\rangle - \frac{\pa \varphi}{\pa x_p}X_{\tau}^p +X_{\tau}^p X_{\tau}^q h_{pq}\\
&=\langle Z \vert \nu\rangle -\langle X_{\tau}| \nabla  \langle X\vert \nu\rangle\rangle +B(X_\tau,X_\tau)\label{ZZZ}
\end{align}
and the relative contribution to the second variation is given by 
\begin{equation}
\int_ {\pa E}\HHH \bigr(\langle Z \vert \nu \rangle -\langle X_{\tau}| \nabla  \langle X\vert \nu\rangle\rangle+B(X_\tau,X_\tau)\bigl)\,d\mu\,.
\end{equation}
Now we conclude by computing the first derivative in~\eqref{1}. To this aim, we note that 
$$
\HHH=-\left\langle \frac{\pa^2 \psi}{\pa x_i \pa x_j}\bigvert \nu\right\rangle g^{ij}
$$
hence, we need the following terms
\begin{equation}
\label{1A}
\frac{\pa g^{ij}}{\pa t}\, \Bigr |_{t=0}
\end{equation}
\begin{equation}
\label{1B}
\left\langle \frac{\pa^2 \psi}{\pa x_i \pa x_j} \bigvert \frac{\pa \nu}{\pa t}\right\rangle \biggr |_{t=0}
\end{equation}
\begin{equation}
\label{1C}
\left\langle \frac{\pa}{\pa t} \frac{\pa^2 \psi}{\pa x_i \pa x_j} \bigvert \nu\right\rangle\biggr |_{t=0} \, .
\end{equation}
We start with the term~\eqref{1A}, recalling that
\begin{equation}
\frac{\pa g_{ij}}{\pa t}\, \Bigr |_{t=0}=\nabla _i \omega_j + \nabla _j \omega_i +2 h_{ij} \langle X\vert \nu\rangle
\end{equation}
by equation~\eqref{derg2}, where $\omega$ is the $1$--form defined by $\omega(Y)=g(X_\tau,Y)$.\\
Using the fact that $g_{ij}g^{jk}=0$, we obtain
\begin{equation}
0= \frac{\pa g_{ij}}{\pa t}\, \Bigr |_{t=0} g^{jk}+g_{ij}\frac{\pa g^{jk}}{\pa t}\, \Bigr |_{t=0}= g^{jk}\bigr(\nabla _i \omega_j + \nabla _j \omega_i +2h_{ij} \langle X\vert \nu \rangle \bigl)+ g_{ij}\frac{\pa g^{jk}}{\pa t}\, \Bigr |_{t=0}
\end{equation}
then,
\begin{equation}\label{1Acalc}
\frac{\pa g^{pk}}{\pa t}\, \Bigr |_{t=0}=- g^{jp} g^{ik}\Bigr(\nabla _i \omega_j + \nabla _j \omega_i +2 h_{ij} \langle X\vert \nu\rangle \Bigl)=-\nabla  ^{p}X_\tau ^k -\nabla ^k X_\tau ^p - 2 h^{pk}\varphi\,.
\end{equation}
We then proceed with the computation of the term~\eqref{1B}, by means of equation~\eqref{eqcar3333},
\begin{equation}
\left\langle\frac{\pa^2 \psi}{\pa x_i \pa x_j}\bigvert \frac{\pa \nu}{\pa t}\right\rangle \biggr |_{t=0}= \Gamma_{ij}^k \left\langle\frac{\pa \psi}{\pa x_k} \bigvert \frac{\pa \nu}{\pa t}\right\rangle \biggr |_{t=0}=\Gamma_{ij}^k \bigr(-\frac{\pa \varphi}{\pa x_k} + X_{\tau}^q h_{qk} \bigl)
\end{equation}
and finally we compute the term~\eqref{1C},
\begin{equation}
\left\langle\frac{\pa}{\pa t}\frac{\pa^2 \psi}{\pa x_i \pa x_j}  \bigvert \nu\right\rangle= \left\langle\frac{\pa^2 X}{\pa x_i \pa x_j} \bigvert \nu\right\rangle\biggr |_{t=0}=\left\langle\frac{\pa^2(\varphi\nu)}{\pa x_i \pa x_j}\bigvert \nu\right\rangle + \left\langle\frac{\pa^2X_\tau }{\pa x_i \pa x_j} \bigvert \nu\right\rangle.
\end{equation}
We have
\begin{align}
\left\langle \frac{\pa^2(\varphi\nu)}{\pa x_i \pa x_j} \bigvert \nu\right\rangle&= \frac{\pa^2 \varphi}{\pa x_i \pa x_j} +\left\langle\frac{\pa^2 \nu}{\pa x_i \pa x_j} \bigvert \nu \right\rangle \varphi \\
&= \frac{\pa^2 \varphi}{\pa x_i \pa x_j} +\left\langle\frac{\pa}{\pa x_i}\bigl(h_{jl}g^{lp}\frac{\pa \psi}{\pa x_p}\bigr)\bigvert \nu\right\rangle \varphi\\
&=\frac{\pa^2 \varphi}{\pa x_i \pa x_j} + h_{jl}g^{lp}\left\langle\frac{\pa^2 \psi}{\pa x_i \pa x_j} \bigvert \nu \right\rangle \varphi\\
&=\frac{\pa^2 \varphi}{\pa x_i \pa x_j} + \varphi h_{jl}{g^{lp}}h_{ip}
\end{align}
and
\begin{align}
\left\langle \frac{\pa^2 X_\tau}{\pa x_i \pa x_j} \bigvert \nu\right\rangle&= \frac{\pa}{\pa x_i}\left\langle \frac{\pa X_\tau}{\pa x_j}\bigvert \nu\right\rangle - \left\langle\frac{\pa X_\tau}{\pa x_j} \bigvert \frac{\pa \nu}{\pa x_i}\right\rangle\\
&= \frac{\pa}{\pa x_i}\left\langle\frac{\pa}{\pa x_j}\Bigr(X_\tau^p \frac{\pa \psi}{\pa x_p}\Bigl)\bigvert \nu \right\rangle- \left\langle \frac{\pa X_\tau}{\pa x_j} \bigvert \frac{\pa \nu}{\pa x_i}\right\rangle\\
&=\frac{\pa}{\pa x_i}\biggr[X_\tau^p \left\langle\frac{\pa^2 \psi}{\pa x_j \pa x_p}\bigvert\nu \right\rangle\biggl] - \left\langle\frac{\pa X_\tau}{\pa x_j}\bigvert \frac{\pa \nu}{\pa x_i}\right\rangle\\
&=-\frac{\pa}{\pa x_i}\bigr(X_\tau^p h_{pj}\bigl) - \left\langle\frac{\pa X_\tau}{\pa x_j} \bigvert \frac{\pa \nu}{\pa x_i}\right\rangle\\
&=-\frac{\pa}{\pa x_i} \bigr(X_\tau^p h_{pj}\bigl)- \left\langle\frac{\pa}{\pa x_j}\Bigr(X_\tau^p \frac{\pa \psi}{\pa x_p}\Bigl)\bigvert \frac{\pa \nu}{\pa x_i}\right\rangle\\
&=-\frac{\pa}{\pa x_i} \bigr(X_\tau^p h_{pj}\bigl) - X_{\tau}^p \left\langle\frac{\pa^2 \psi}{\pa x_j \pa x_p}\bigvert \frac{\pa \nu}{\pa x_i}\right\rangle-\frac{\pa X_\tau^p}{\pa x_j}\left\langle\frac{\pa \psi}{\pa x_p} \bigvert \frac{\pa \nu}{\pa x_i}\right\rangle\\
&=-\frac{\pa}{\pa x_i} \bigr(X_\tau^p h_{pj}\bigl)-X_\tau^p \Gamma_{jp}^k \left\langle\frac{\pa \psi}{\pa x_k}\bigvert \frac{\pa \nu}{\pa x_i}\right\rangle-\frac{\pa X_\tau^p}{\pa x_j} \left\langle\frac{\pa \psi}{\pa x_p}\bigvert \frac{\pa \nu}{\pa x_i}\right\rangle\\
&=-\frac{\pa}{\pa x_i} \bigr(X_\tau^p h_{pj}\bigl) - X_\tau^p \Gamma_{jp}^kh_{il} g^{lq} g_{kq}-\frac{\pa X^p}{\pa x_j}h_{il}g^{lq}g_{pq}\\
&=-\frac{\pa}{\pa x_i}\bigr(X_\tau^p h_{pj}\bigl)- X_\tau ^p \Gamma_{jp}^k h_{ik} - \frac{\pa X^k}{\pa x_j} h_{ik}.
\end{align}
Hence, we finally get
\begin{align}
\frac{\pa \HHH}{\pa t}\, \Bigr |_{t=0}=&\,-2 h_{ij} \nabla ^i X_\tau ^j - 2  \langle X\vert \nu\rangle|B|^2- g^{ij}\frac{\pa^2 \varphi}{\pa x_i \pa x_j} + g^{ij}\Gamma_{ij}^k \frac{\pa \varphi}{\pa x_k}\\
&\,+ |B|^2 \langle X\vert \nu\rangle - g^{ij} \Gamma_{ij}^k h_{kq} X_\tau^q
+ g^{ij} \frac{\pa}{\pa x_i} (X_\tau ^p h_{pj}) + h_{ij}\nabla ^i X^j_j\\
=&\,-|B|^2 \langle X\vert \nu\rangle - h_{ij} \nabla ^i X_\tau^j - \Delta \varphi\\
&\,+ g^{ij}\Bigr[\frac{\pa}{\pa x_i}\Bigr(X_\tau^p h_{pj}\Bigl)-\Gamma_{ij}^k \Bigr(X_\tau^p h_{pk}\Bigl)\Bigl]\\
=&\,-\varphi |B|^2 - \Delta \varphi - h_{ij}\nabla ^i X_\tau^j + g^{ij}\nabla _{i}(X_\tau^p h_{pj})\\
=&\,-\varphi|B|^2 -\Delta \varphi - h_{ij}\nabla ^i X_\tau^j + \Div(X_\tau^p h_{pj})\\
=&\,-\varphi|B|^2 - \Delta \varphi + \langle X_\tau \vert \Div\!B\rangle\\
=&\,-\varphi|B|^2 -\Delta \varphi + \langle X_\tau \vert \nabla  \HHH\rangle\, ,\label{derH}
\end{align}
where in the last equality we used the Codazzi--Mainardi equations (see~\cite{Man}). We conclude that the contribution of the first term in~\eqref{1} is then
\begin{equation}
\int_{\pa E} \varphi \bigr(-\varphi |B|^2 - \Delta \varphi + \langle X_\tau \vert \nabla  \HHH\rangle \bigl) \, d\mu.
\end{equation}
Putting all these contributions together, we obtain the second variation of the Area functional,
\begin{align}
\frac{d^2}{d t^2}\A(\partial E_t)\,\Bigr\vert_{t=0}&= \int_{\pa E} \Bigl[-\varphi \Delta \varphi - \varphi^2 |B|^2 +\varphi \langle X_\tau \vert \nabla  \HHH\rangle +\varphi \HHH \Div\!X_\tau + \varphi^2\HHH^2\\
&\qquad\qquad + \HHH\bigl(\langle Z \vert \nu \rangle- \langle X_\tau \vert \nabla  \varphi\rangle + B(X_\tau,X_\tau)\bigr)\Bigr] \, d\mu\,.
\end{align}
Integrating by parts, we have
\begin{equation}
\int_{\pa E} \varphi \langle X_\tau \vert \nabla  \HHH\rangle\,d\mu= -\int_{\pa E}\bigl[\HHH \langle X_\tau \vert \nabla  \varphi\rangle + \HHH \varphi \Div\!X_\tau\bigr]\,d\mu
\end{equation}
and we can conclude
\begin{equation}
\frac{d^2}{dt^2}\A(\partial E_t)\,\Bigr\vert_{t=0}= \int_{\pa E}\Bigl[ |\nabla  \varphi|^2 - \varphi^2 |B|^2 + \varphi^2 \HHH^2+ \HHH(\langle Z \vert \nu \rangle - 2 \langle X_\tau \vert \nabla  \varphi\rangle + B(X_\tau, X_\tau))\Bigr]\, d\mu\,,
\end{equation}
which is the formula we wanted.
\end{proof}

\begin{proposition}[Second variation of the nonlocal term]\label{second var F}
Let $E\subseteq \T^n$, $\Phi$, $E_t$, $X$, $X_\tau$, $X_t$, $\HHH$, $B$ and $Z$ as in the previous proposition. Then, setting
\begin{equation}
N(t)=\int_{\T^n}|\nabla  v_{E_t}(x)|^2\, dx\,,
\end{equation}
where $v_{E_t}:\T^n\to\R$ is the function defined by formulas~\eqref{potential1}--\eqref{potential} and $\partial_{\nu_E} v_E=\langle\nabla^{\T^n}\!v_{E}|\nu_E\rangle$, the following formula holds
\begin{align}
\frac{d^2}{dt^2}N(t)\Bigr|_{t=0}=&\,
8 \int_{\pa E} \int_{\pa E} G(x,y) \langle X(x) \vert \nu_E (x) \rangle \langle X(y) \vert \nu_{E} (y)\rangle \, d\mu(x) d\mu(y)\\
&\,+4\int_{\pa E} \Bigl[v_E\bigl(\HHH\langle X\,\vert\,\nu_E\rangle^2+\langle Z \vert \nu_E\rangle-2\langle X_{\tau}| \nabla  \langle X\vert \nu_E\rangle\rangle+B(X_\tau,X_\tau)\bigr)\\
&\,\,\,\qquad\qquad +\partial_{\nu_E} v_E\langle X|\nu_E\rangle^2 \,\Bigr]\,d\mu\,,\label{N''}
\end{align}
giving the second variation of the nonlocal term of $J$.
\end{proposition}
\begin{proof}
By arguing as in the second part of the proof of Theorem~\ref{first var} (equations~\eqref{eqc1}--\eqref{nonlocal}), we have
\begin{equation}
\frac{d}{dt} N(t)=4 \int_{\pa E} v_{E_t} \langle X_t\circ\Phi_t\vert \nu_{E_t}\rangle \, d\mu_t=4 \int_{\pa E} v_{E_t} \langle X_t\circ\Phi_t\vert \nu_{E_t}\rangle \sqrt{\det g_{ij}} \, dx\,.
\end{equation}
Setting $v(t,x)=v_{E_t}(x)$, $v_t=\frac{\pa v}{\pa t}(0,\cdot)$, $v_i=\frac{\pa v}{\pa x_i}(0,\cdot)$ and adopting the same notation of the proof of the previous proposition, that is, we let $\HHH(t,\cdot) = \HHH_t$, $\nu(t,\cdot)=\nu_{E_t}$ and $X(t,\cdot)=X_t\circ\Phi_t$, we have
\begin{align}
\frac{d^2}{dt^2} N(t)\Bigr|_{t=0}=&\,4 \frac{d}{dt}\int_{\pa E} v\langle X\vert \nu\rangle \sqrt{\det g_{ij}} \, dx\,\Bigr|_{t=0}\\
=&\,4 \int_{\pa E} \Bigl[v_t\langle X\vert \nu\rangle+v_iX^i\langle X\vert \nu\rangle+v\langle X\vert \nu\rangle \Div\!X_{\tau}\\
&\,\,\,\qquad\quad+v\HHH\langle X\vert \nu\rangle^2+v\frac{\pa}{\pa t}\langle X\vert \nu\rangle\,\Bigr|_{t=0}\, \Bigr] \, d\mu\\
=&\,4 \int_{\pa E} \Bigl[v_t\langle X\vert \nu\rangle+v_iX^i\langle X\vert \nu\rangle+ v\langle X\vert \nu\rangle \Div\!X_{\tau}\\
&\,\,\,\qquad\quad+v\bigl(\HHH\langle X\vert \nu\rangle^2+\langle Z \vert \nu\rangle -\langle X_{\tau}| \nabla  \langle X\vert \nu\rangle\rangle +B(X_\tau,X_\tau)\bigr)\,\Bigr]\, d\mu\,,
\end{align}
by formulas~\eqref{dermu2} and~\eqref{ZZZ}. Then, integrating by parts the divergence, we obtain
\begin{align}
\frac{d^2}{dt^2} N(t)\Bigr|_{t=0}=&\,4 \int_{\pa E} \Bigl[v_t\langle X\vert \nu\rangle+v_iX^i\langle X\vert \nu\rangle-\langle\nabla v\vert X_\tau\rangle\langle X\vert \nu\rangle\\
&\,\,\,\qquad\quad+v\bigl(\HHH\langle X\vert \nu\rangle^2+\langle Z \vert \nu\rangle -2\langle X_{\tau}| \nabla  \langle X\vert \nu\rangle\rangle +B(X_\tau,X_\tau)\bigr)\,\Bigr]\, d\mu\\
=&\,4 \int_{\pa E} \Bigl[v_t\langle X\vert \nu\rangle+\partial_\nu v\langle X|\nu\rangle^2\\
&\,\,\,\qquad\quad+v\bigl(\HHH\langle X\vert \nu\rangle^2+\langle Z \vert \nu\rangle -2\langle X_{\tau}| \nabla  \langle X\vert \nu\rangle\rangle +B(X_\tau,X_\tau)\bigr)\,\Bigr]\, d\mu
\end{align}
where $\partial_\nu v=\langle\nabla^{\T^n}\!\!v\vert\nu\rangle$.\\
Now, by equations~\eqref{eqc1}--\eqref{eqc3}, there holds
\begin{equation}\label{v'}
v_t(0,x)=2 \int_{\pa E} G(x,y)\left\langle X(y) \vert \nu (y)\right\rangle \, d\mu(y)\,,
\end{equation}
hence, substituting this expression for $v_t$ in the equation above we have formula~\eqref{N''}.
\end{proof}

Putting together Propositions~\ref{secondvarA} and~\ref{second var F}, we then obtain the second variation of the nonlocal Area functional $J$.

\begin{thm}[Second variation of the functional $J$] \label{secondvar}
Let $E\subseteq \T^n$ a smooth set and $\Phi:(-\eps,\eps)\times \T^n\to\T^n$ a smooth map giving a variation $E_t$ with infinitesimal generator $X \in C^\infty (\T^n; \R^n)$.  Then,
\begin{align}
\frac{d^2 }{dt^2}J(E_t)\Bigl|_{t=0} =&\,\int_{\pa E}\bigl(|\nabla  \langle X\vert  \nu_E\rangle|^2-\langle X\vert \nu_E\rangle^2|B|^2\bigr)\, d\mu\nonumber\\
& \,+8\gamma \int_{\pa E}  \int_{\pa E}G(x,y)\langle X\vert  \nu_E(x)\rangle \langle X\vert \nu_E(y)\rangle \,d\mu(x)\,d\mu(y)\nonumber\\
&\,+4\gamma\int_{\pa E}\pa_{\nu_E} v_E \langle X\vert  \nu_E\rangle ^2\,d\mu+R\,, \label{IIfinale}
\end{align}
with the ``remainder term'' $R$ given by
\begin{align}
R=&\,\int_{\pa E}(\HHH+4\gamma v_E)\bigl(\HHH\langle X\vert \nu\rangle^2+\langle Z \vert \nu\rangle -2\langle X_{\tau}| \nabla  \langle X\vert \nu\rangle\rangle +B(X_\tau,X_\tau)\bigr)\, d\mu\\
=&\,\int_{\pa E}(\HHH+4\gamma v_E)\Bigl[\langle X\vert \nu_E\rangle\Div^{\!\T^n}\!\!X
-\Div\bigl(\langle X\vert \nu_E\rangle X_\tau\bigr)+\Bigl\langle\frac{\partial X_t}{\partial t}\Bigr\vert_{t=0}\Bigr\vert\nu_E\Bigr\rangle\,\Bigr]\, d\mu
\end{align}
where $\nu_E$ is the outer unit normal vector to $\pa E$, $X_\tau=X-\langle X|\nu_E\rangle\nu_E$ is the {\em tangential part} of $X$ on $\partial E$, $v_{E}:\T^n\to\R$ is the function defined by formulas~\eqref{potential1}--\eqref{potential}, $\partial_{\nu_E} v_E=\langle\nabla^{\T^n}\!v_{E}|\nu_E\rangle$, $B$ and $\HHH$ are respectively the second fundamental form and the mean curvature of $\pa E$, the vector field $X_t\in C^\infty(\T^n; \R^n)$ is defined by the formula
$X_t(\Phi(t,z))=\frac{\pa \Phi}{\pa t}(t,z)$ for every $t\in(-\eps,\eps)$ and $z\in\T^n$, and
\begin{equation}
Z= \frac{\pa^2\Phi}{\pa t^2}(0, \cdot )=\frac{\pa}{\pa t}[X_t(\Phi (t, \cdot))]\,\Bigr\vert_{t=0}=\frac{\pa X_t}{\pa t}\,\Bigr\vert_{t=0}+ d X(X)\,.
\end{equation}
\end{thm}
\begin{proof}
Formula~\eqref{IIfinale} and the first equality for $R$ follows simply adding (after multiplying the nonlinear term by $\gamma$) the expressions for 
$\frac{d^2}{dt^2}\A (\pa E_t)\,\bigr|_{t=0}$
and $\frac{d^2}{dt^2}\int_{\T^n}|\nabla \, v_{E_t}|^2\, dx\,\bigr|_{t=0}$ we found in Propositions~\ref{secondvarA} and~\ref{second var F}.\\
If now we show that
\begin{align}
\HHH \langle X| &\, \nu_E  \rangle ^2+\langle Z | \nu_E  \rangle -2 \langle X_\tau | \nabla \langle X| \nu_E  \rangle\rangle +
B(X_\tau, X_\tau)\\
&=\langle X| \nu_E  \rangle\Div^{\!\T^n}\!\!X-\Div (\langle X| \nu_E  \rangle X_\tau)+\Bigl\langle\frac{\partial X_t}{\partial t}\Bigr\vert_{t=0}\Bigr\vert\nu_E\Bigr\rangle\,,\label{claim}
\end{align}
we clearly obtain the second expression for $R$.\\
We note that, being every derivative of $\nu_E$ a tangent vector field,
\begin{align}
\langle X_\tau | \nabla  \langle X| \nu_E  \rangle\rangle
&=\langle \nu_E| dX(X_\tau)\rangle +\langle X| \langle X_\tau|\nabla \nu_E\rangle\rangle \\
&=\langle \nu_E| dX(X_\tau)\rangle +\langle X_\tau| \langle X_\tau|\nabla \nu_E\rangle\rangle \\
&=\langle \nu_E | dX(X_\tau)\rangle +B (X_\tau,X_\tau )\,,
\end{align}
by the Gauss--Weingarten relations~\eqref{GW}.\\
Therefore, since $Z-\frac{\partial X_t}{\partial t}\bigr\vert_{t=0}=dX(X)$, we have
\begin{align}
\HHH \langle X  &\vert \nu_E  \rangle ^2+\langle Z | \nu_E  \rangle -2 \langle X_\tau | \nabla \langle X| \nu_E  \rangle\rangle +
B(X_\tau, X_\tau) -\Bigl\langle\frac{\partial X_t}{\partial t}\Bigr\vert_{t=0}\Bigr\vert\nu\Bigr\rangle\nonumber\\
& = \HHH \langle X|  \nu_E  \rangle ^2+ \langle \nu_E | dX(X) \rangle - \langle X_\tau |\nabla \langle X| \nu_E  \rangle\rangle -
\langle \nu_E |  dX(X_\tau)\rangle \nonumber\\
& = \HHH \langle X|  \nu_E  \rangle ^2+\langle \nu_E | dX(\langle X| \nu_E  \rangle \nu_E)\rangle -\langle X_\tau | \nabla \langle X| \nu_E  \rangle\rangle \nonumber\\
& = \HHH \langle X|  \nu_E  \rangle ^2+\langle X| \nu_E  \rangle\langle \nu_E | dX(\nu_E)\rangle +\langle X| \nu_E  \rangle\Div\!X_\tau  -\Div (\langle X| \nu_E  \rangle X_\tau)\,.\label{eqcar5555}
\end{align}
Now we notice that, choosing an orthonormal basis $e_1,\dots,e_{n-1},e_n=\nu_E$ of $\R^n$ at a point $p\in\partial E$ and letting $X=X^ie_i$, we have
$$
\langle e_i|\nabla^\top\!X^i\rangle=\bigl\langle e_i\bigr|\nabla^{\T^n}\!X^i-\langle \nabla^{\T^n}\!X^i|\nu_E\rangle\nu_E\bigr\rangle=\Div^{\!\T^n}\!\!X-\langle \nu_E | dX(\nu_E)\rangle\,,
$$
where the symbol $\nabla^{\top}\!f$ denotes the projection on the tangent space to $\partial E$ of the gradient $\nabla^{\T^n}\!\!f$ of a function, called {\em tangential gradient} of $f$ and coincident with the gradient operator of $\pa E$ applied to the restriction of $f$ to the hypersurface, while $\langle e_i|\nabla^\top\!X^i\rangle$ is called {\em tangential divergence} of $X$, usually denoted with $\Div^{\!\top}\!\!X$ and coincident with the (Riemannian) divergence of $\pa E$ if $X$ is a tangent vector field, as we will see below (see~\cite{Si}). Moreover, if we choose a local parametrization of $\partial E$ such that $\frac{\pa\psi}{\pa x_i}(p)=e_i$, for $i\in\{1,\dots,n-1\}$, we have $e_i^j=\frac{\pa\psi^j}{\pa x_i}=g^{ij}=\delta_{ij}$ at $p$ and 
\begin{align}
\langle e_i|\nabla ^\top\!X^i\rangle=\Div^{\!\top}\!\!X=&\,\langle e_i|\nabla ^\top\!X^i_\tau\rangle+\langle e_i|\nabla ^\top\!(\langle X|\nu_E\rangle\nu_E^i)\rangle\\
=&\,\langle e_i|\nabla  X^i_\tau\rangle+\langle X|\nu_E\rangle \langle e_i|\nabla^{\T^n}\!\!\nu_E^i\rangle\\
=&\,\langle e_i|\nabla  X^i_\tau\rangle+\langle X|\nu_E\rangle\frac{\pa\psi^j}{\pa x_i} h_{jl}g^{ls}\frac{\pa\psi^i}{\pa x_s}\\
=&\,\nabla _{e_i}X^i_\tau+\langle X|\nu_E\rangle h_{ii}\\
=&\,\Div\!X_\tau+\langle X|\nu_E\rangle\HHH\,,
\end{align}
where we used again the Gauss--Weingarten relations~\eqref{GW} and the fact that the covariant derivative of a tangent vector field along a hypersurface of $\R^n$ can be obtained by differentiating in $\R^n$ (a local extension of) the vector field and projecting the result on the tangent space to the hypersurface (see~\cite{gahula}, for instance). Hence, we get
$$
\langle \nu_E | dX(\nu_E)\rangle=\Div^{\!\T^n}\!\!X-\langle e_i|\nabla ^\top\!X^i\rangle=\Div^{\!\T^n}\!\!X-\Div\!X_\tau-\langle X|\nu_E\rangle\HHH
$$
and equation~\eqref{claim} follows by substituting this left term in formula~\eqref{eqcar5555}.
\end{proof}

\begin{remark} We are not aware of the presence in literature of this ``geometric'' line in deriving the (first and) second variation of $J$, moreover, in~\cite[Theorem~2.6, Step~3, equation~2.67]{ChSt}, this latter is obtained only at a critical set, while in~\cite[Theorem~3.6]{CaMoMo} the methods are strongly ``analytic'' and in our opinion less straightforward. These two papers are actually the ones on which is based the computation in~\cite[Theorem~3.1]{AcFuMo} of the second variation of $J$ at a general smooth set $E\subseteq\T^n$. Anyway, in this last paper, the variations of $E$ are all {\em special} variations, that is, they are given by the flows in system~\eqref{varflow}, indeed, the term with the time derivative of $X_t$ is missing (see formulas~3.1 and~7.2 in~\cite{AcFuMo}).

Notice that the second variation in general does not depend only on the normal component $\langle X|\nu_E\rangle$ of the restriction to $\pa E$ of the infinitesimal generator $X$ of a variation $\Phi$ (this will anyway be true at a critical set $E$, see below), due to the presence of the $Z$--term and of $B (X_\tau,X_\tau )$ depending also on the tangential component of $X$ and of its behavior around $\pa E$. Even if we restrict ourselves to the special variations coming from system~\eqref{varflow}, with a {\em normal} infinitesimal generator $X$, which imply that all the vector fields $X_t$ are the same and coinciding with $X$, hence $Z=dX(X)$ and $X_\tau=0$, the second variation still depends also on the behavior of $X$ in a neighborhood of $\pa E$ (as $Z$). However, there are very particular case in which it depend only on $\langle X|\nu_E\rangle$, for instance when the variation is special and $X$ is normal with zero divergence (of $\T^n$) on $\pa E$ (in particular, if $\Div^{\!\T^n}\!\!X=0$ in a neighborhood of $\pa E$ or in the whole $\T^n$), as it can be seen easily by the second form of the remainder term $R$ in the above theorem.
\end{remark}

We see now how the second variation behaves at a critical set of $J$. 

\begin{cor}\label{IIcritcor}
If $E \subseteq \T^n$ is a critical set for $J$, there holds
\begin{align}
\frac{d^2 }{dt^2}J(E_t)\Bigl|_{t=0} =&\,\int_{\pa E}\bigl(|\nabla  \langle X\vert  \nu_E\rangle|^2-\langle X\vert \nu_E\rangle^2|B|^2\bigr)\, d\mu\nonumber\\
&\,+8\gamma \int_{\pa E}  \int_{\pa E}G(x,y)\langle X\vert  \nu_E(x)\rangle \langle X\vert \nu_E(y)\rangle \,d\mu(x)\,d\mu(y)\nonumber\\
&\,+4\gamma\int_{\pa E}\pa_{\nu_E} v_E \langle X\vert  \nu_E\rangle ^2\,d\mu\,,
\end{align}
for every variation $E_t$ of $E$, hence, the second variation of $J$ at $E$ depends only on the normal component of the restriction of the infinitesimal generator $X$ to $\pa E$, that is, on $\langle X| \nu_E  \rangle$.\\
When $\gamma=0$ we get the well known second variation of the Area functional at a smooth set $E$ such that $\pa E$ is a minimal surface in $\R^n$,
\begin{equation}
\frac{d^2 }{dt^2}\A(\pa E_t)\Bigl|_{t=0} =\int_{\pa E}\bigl(|\nabla  \langle X\vert  \nu_E\rangle|^2-\langle X\vert \nu_E\rangle^2|B|^2\bigr)\, d\mu\,.
\end{equation}
\end{cor}
\begin{proof} The thesis follows immediately, recalling that there holds $\HHH+4\gamma v_E=0$, by Corollary~\ref{critcor}, hence the remainder term $R$ in formula~\eqref{IIfinale} is zero.
\end{proof}

Finally, we see that the second variation has the same form (that is, $R=0$) also for $J$ under a volume constraint, at a critical set. 

\begin{proposition}\label{IIcritprop}
If $E \subseteq \T^n$ is a critical set for $J$ under a volume constraint, there holds
\begin{align}
\frac{d^2 }{dt^2}J(E_t)\Bigl|_{t=0} =&\,\int_{\pa E}\bigl(|\nabla  \langle X\vert  \nu_E\rangle|^2-\langle X\vert \nu_E\rangle^2|B|^2\bigr)\, d\mu\nonumber\\
& \,+8\gamma \int_{\pa E}  \int_{\pa E}G(x,y)\langle X\vert  \nu_E(x)\rangle \langle X\vert \nu_E(y)\rangle \,d\mu(x)\,d\mu(y)\nonumber\\
&\,+4\gamma\int_{\pa E}\pa_{\nu_E} v_E \langle X\vert  \nu_E\rangle ^2\,d\mu\,,
\end{align}
for every volume--preserving variation $E_t$ of $E$, hence, the second variation of $J$ at $E$ depends only on the normal component of the restriction of the infinitesimal generator $X$ to $\pa E$, that is, on $\langle X| \nu_E  \rangle$.\\
When $\gamma=0$ we get the second variation of the Area functional under a volume constraint, at a smooth set $E$ such that $\pa E$ has constant mean curvature,
\begin{equation}
\frac{d^2 }{dt^2}\A(\pa E_t)\Bigl|_{t=0} =\int_{\pa E}\bigl(|\nabla  \langle X\vert  \nu_E\rangle|^2-\langle X\vert \nu_E\rangle^2|B|^2\bigr)\, d\mu\,.
\end{equation}
\end{proposition}
\begin{proof} By Proposition~\ref{critprop}, the function $\HHH+4\gamma v_E$ is equal to a constant $\lambda\in\R$ on $\pa E$, then the remainder term $R$ in formula~\eqref{IIfinale} becomes
$$
R=\lambda\int_{\pa E}\bigl(\HHH\langle X\vert \nu\rangle^2+\langle Z \vert \nu\rangle -2\langle X_{\tau}| \nabla  \langle X\vert \nu\rangle\rangle +B(X_\tau,X_\tau)\bigr)\, d\mu\,.
$$
Computing, in the same hypotheses and notations of Proposition~\ref{second var F}, the second derivative of the (constant) volume of $E_t$, by equations~\eqref{eqc999}--\eqref{eqc1000} we have (recalling formulas~\eqref{dermu2},~\eqref{ZZZ} and using the divergence theorem)
\begin{align}
0=&\, \frac{d^2}{dt^2}\vol(E_t)\,\Bigr\vert_{t=0}=\frac{d}{dt}\int_{E_t} \Div\!X_t(x) \,dx\,\Bigr\vert_{t=0} =\frac{d}{dt}\int_{\pa E} \langle X\vert \nu_{E_t}\rangle \, d\mu_t\,\Bigr\vert_{t=0}\\
=&\,\int_{\pa E}\Bigl[\Div\!X_\tau\langle X\,\vert\,\nu_E\rangle+\HHH\langle X\,\vert\,\nu_E\rangle^2+\langle Z \vert \nu_E\rangle -\langle X_{\tau}| \nabla  \langle X\vert \nu_E\rangle\rangle+B(X_\tau,X_\tau)\,\Bigr]\,d\mu\\
=&\,\int_{\pa E}\Bigl[\HHH\langle X\,\vert\,\nu_E\rangle^2+\langle Z \vert \nu_E\rangle -2\langle X_{\tau}| \nabla  \langle X\vert \nu_E\rangle\rangle+B(X_\tau,X_\tau)\,\Bigr]\,d\mu\,,\label{sec der vol}
\end{align}
hence $R=0$ and we are done.
\end{proof}

\begin{remark}\label{rem999}
Notice that by the previous computation and relation~\eqref{claim}, it follows 
\begin{equation}\label{eqcar6666}
\frac{d^2}{dt^2}\vol(E_t)\,\Bigr\vert_{t=0}=\int_{\pa E}\Bigl[\langle X\vert \nu_E\rangle\Div^{\!\T^n}\!\!X+\Bigl\langle\frac{\partial X_t}{\partial t}\Bigr\vert_{t=0}\Bigr\vert\nu\Bigr\rangle\,\Bigr]\, d\mu=0\,,
\end{equation}
for every volume--preserving variation $E_t$ of $E$. Hence, if we restrict ourselves to the special (volume--preserving) variations coming from system~\eqref{varflow}, as in~\cite{AcFuMo}, we have
$$
\frac{d^2}{dt^2}\vol(E_t)\,\Bigr\vert_{t=0}=\int_{\pa E}\langle X\vert \nu_E\rangle\Div^{\!\T^n}\!\!X\, d\mu=0\,,
$$
indeed, for such variations we have $X_t=X$, for every $t\in(-\eps,\eps)$. One can clearly use equality~\eqref{eqcar6666} to show the above proposition, as the term $R$ reduces (using the second form in Theorem~\ref{secondvar}) to
$$
R=\lambda\int_{\pa E}\Bigl[\langle X\vert \nu_E\rangle\Div^{\!\T^n}\!\!X\,+\Bigl\langle\frac{\partial X_t}{\partial t}\Bigr\vert_{t=0}\Bigr\vert\nu\Bigr\rangle\,\Bigr]\, d\mu\,,
$$
by the divergence theorem.\\
Moreover, we see that if we have a special variation generated by a vector field $X$ such that $\Div^{\!\T^n}\!\!X=0$ on $\pa E$, then 
$\frac{d^2}{dt^2}\vol(E_t)\,\bigr\vert_{t=0}=0$ and if $E$ is a critical set, $R=0$. This is then true for the special volume--preserving variations coming from Lemma~\ref{vector field} and when $X$ is a constant vector field, hence the associated special variation $E_t$ is simply a translation of $E$ (clearly, in this case $J(E_t)$ is constant and the first and second variations are zero).
\end{remark}

\subsection{Stability and $W^{2,p}$--local minimality}\label{stabsec}\ \vskip.3em

By Proposition~\ref{IIcritprop}, the second variation of the functional $J$ under a volume constraint at a smooth critical set $E$ is a quadratic form in the normal component on $\partial E$ of the infinitesimal generator $X\in C^\infty(\T^n; \R^n)$ of a volume--preserving variation, that is, on $\varphi=\langle X|\nu_E\rangle$. This and the fact that the infinitesimal generators of the volume--preserving variations are ``characterized'' by having zero integral of such normal component on $\pa E$, by Lemma~\ref{vector field} and the discussion immediately before, motivate the following definition.

\begin{definition}\label{Pi}
Given any smooth open set $E\subseteq \T^n$ we define the space of (Sobolev) functions (see~\cite{Aubin})
\begin{equation} 
\Htilde^1( \partial E) = \Bigl\{\varphi:\partial E\to\R\, :\, \varphi \in H^1(\partial E)\,\, \text{ and }\, \int_{\partial E} \varphi \,\dmu = 0 \Bigr\},
\end{equation}
and the quadratic form $\Pi _E : \Htilde^1(\pa E) \to \R$ as
\begin{align}
\Pi_E(\varphi)=&\, \int_{\pa E}\bigl(|\nabla  \varphi|^2- \varphi^2|B|^2\bigr)\, d\mu+8\gamma \int_{\pa E}  \int_{\pa E}G(x,y)\varphi(x)\varphi(y)\,d\mu(x)\,d\mu(y)\nonumber\\
&\,+4\gamma\int_{\pa E}\pa_{\nu_E} v_E \varphi ^2\,d\mu\,,\label{Pieq}
\end{align}
with the notations of Theorem~\ref{secondvar}.
\end{definition}

\begin{remark}\label{rm:potential}
Letting for $\varphi\in \Htilde^1(\pa E)$,
$$
v_\varphi(x)=\int_{\pa E} G(x, y)\varphi(y)\, d\mu(y)\,,
$$
it follows (from the properties of the Green's function) that $v_\varphi$ satisfies distributionally $-\Delta v_{\varphi}=\varphi\mu$ in 
$\T^n$, indeed,
\begin{align*}
\int_{\T^n}\langle\nabla v_\varphi(x)|\nabla \psi(x)\rangle\,dx
=&\,-\int_{\T^n}v_\varphi(x)\Delta\psi(x)\,dx\\
=&\,-\int_{\T^n}\int_{\pa E} G(x, y)\varphi(y)\Delta\psi(x)\,d\mu(y)dx\\
=&\,-\int_{\pa E} \varphi(y)\int_{\T^n}G(x, y)\Delta\psi(x)\,dx\,d\mu(y)\\
=&\,-\int_{\pa E} \varphi(y)\int_{\T^n}\Delta G(x, y)\psi(x)\,dx\,d\mu(y)\\
=&\,\int_{\pa E} \varphi(y)\Bigl[\psi(y)-\int_{\T^n}\psi(x)\,dx\Bigr]\,d\mu(y)\\
=&\,\int_{\pa E}\varphi(y)\psi(y)\,d\mu(y)\,,
\end{align*}
for all $\psi\in C^\infty(\T^n)$, since $\int_{\pa E}\varphi(y)\,d\mu(y)=0$. Therefore, taking $\psi=v_\varphi$, we have
$$
\int_{\T^n}|\nabla v_\varphi(x)|^2\,dx=\int_{\pa E}\varphi(y)v_\varphi(y)\,d\mu(y)\,,
$$
hence, the following identity holds
\begin{equation}\label{eqcar777}
\int_{\pa E}  \int_{\pa E}G(x,y)\varphi(x)\varphi(y)\,d\mu(x)d\mu(y)=
\int_{\pa E}\varphi(y)\, v_{\varphi}(y)\, d\mu(y)=\int_{\T^n}|\nabla v_{\varphi}(x)|^2\, dx\,,
\end{equation}
and we can write
\begin{equation}
\Pi_E(\varphi)=\int_{\pa E}\Bigl(|\nabla  \varphi|^2- \varphi^2|B|^2\Bigr)\, d\mu
+8\gamma \int_{\T^n}|\nabla v_{\varphi}|^2\, dx+4\gamma\int_{\pa E}\pa_{\nu_E} v_E \varphi ^2\,d\mu\,,\label{Pieq2}
\end{equation}
for every $\varphi\in\Htilde^1(\partial E)$.
\end{remark}

\begin{definition}\label{admvector}
Given any smooth open set $E\subseteq \T^n$, we say that a smooth vector field $X\in C^\infty(\T^n; \R^n)$ is {\em admissible for $E$} if the function $\varphi:\pa E\to\R$ given by $\varphi=\langle X \vert \nu_E\rangle$ belongs to $\Htilde^1( \partial E)$, that is, has zero integral on $\pa E$.
\end{definition}

\begin{remark}\label{remcarlo888}
Clearly, if $X\in C^\infty(\T^n; \R^n)$ is the infinitesimal generator of a volume--preserving variation for $E$, then $X$ is admissible, by the discussion after Corollary~\ref{critcor}. 
\end{remark}

\begin{remark}\label{remcarlo}
By what we said above, if $E$ is a smooth critical set for $J$ under a volume constraint, we can from now on consider only the special variations $E_t=\Phi_t(E)$ associated to admissible vector fields $X$, given by the flow $\Phi$ defined by system~\eqref{varflow}, hence
\begin{equation}
\frac{d}{dt} J(E_t) \Bigl|_{t=0} =\int_{\pa E} \langle X \vert \nu_E \rangle \, \dmu=0
\end{equation}
and
\begin{equation}\label{PI}
\frac{d^2}{dt^2} J(E_t)\Bigr \vert_{t=0}=\Pi_E(\langle X \vert \nu_E \rangle)
\end{equation}
where $\Pi_E$ is the quadratic form defined by formula~\eqref{Pieq}.
\end{remark}

We notice that every constant vector field $X=\eta\in \R^n$ is clearly admissible, as
$$
\int_{\pa E} \langle \eta\, \vert \nu_E \rangle \, \dmu=\int_{E} \Div\eta\,dx=0
$$
and the associated flow is given by $ \Phi(t,x) = x + t \eta$, then, by the translation invariance of the functional $J$, we have $J(E_t)=J(E)$ and
\begin{equation}
0=\frac{d^2}{dt^2} J(E_t) \Bigr \vert_{t=0}= \Pi_E (\langle \eta\, \vert \nu_E\rangle)\,,
\end{equation}
that is, the form $\Pi_E$ is zero on the vector subspace 
\begin{equation}\label{T}
T(\pa E) = \bigl\{\langle \eta\,  | \nu_E  \rangle\, : \, \eta\in \R^n \bigr\}\subseteq \Htilde^1(\pa E)
\end{equation}
of dimension clearly less than or equal to $n$.
We split 
\begin{equation}\label{decomp}
\Htilde^1(\pa E) = T(\pa E) \oplus \Tort(\pa E)\,,
\end{equation}
where $\Tort (\pa E)\subseteq \Htilde^1(\pa E)$ is the vector subspace $L^2$--orthogonal to $T(\pa E)$ (with respect to the measure $\mu$ on $\partial E$), that is,
\begin{align}
\Tort(\pa E)
=&\,\Bigl \{\varphi \in \Htilde^1(\pa E) \, : \, \int_{\pa E} \varphi \nu_E \, \dmu = 0\Bigr \}\\
=&\,\Bigl \{\varphi \in H^1(\pa E) \, : \,\int_{\pa E} \varphi \, \dmu = 0 \,\,\text{ and }\, \int_{\pa E} \varphi \nu_E \, \dmu = 0\Bigr \}\label{Tort} 
\end{align}
and we give the following ``stability'' conditions.

\begin{definition}[Stability]\label{str stab}
We say that a critical set $E \subseteq\T^n$ for $J$ under a volume constraint is {\em stable} if
\begin{equation}\label{stabile}
\Pi_E (\varphi) \geq 0 \qquad\text{for all $\varphi \in \Htilde^1(\pa E)$}
\end{equation}
and {\em strictly stable} if moreover
\begin{equation}\label{strettstabile}
\Pi_E (\varphi) > 0 \qquad\text{for all $\varphi \in \Tort(\pa E) \setminus \{0 \}$.}
\end{equation}
\end{definition}

\begin{remark}\label{rembase0}
Introducing the symmetric bilinear form associated (by polarization) to $\Pi_E$ on $\Htilde^1(\pa E)$, 
$$
b_E(\varphi,\psi)=\frac{\Pi_E (\varphi+\psi)-\Pi_E (\varphi-\psi)}{4}
$$
at a critical set $E \subseteq\T^n$, it can be seen that actually $T(\pa E)$ is a degenerate vector subspace of $\Htilde^1(\pa E)$ for $b_E$, that is, $b_E(\varphi,\psi)=0$ for every $\varphi\in\Htilde^1(\pa E)$ and $\psi\in T(\pa E)$. Indeed, we observe that by formula~\eqref{potential1} and the properties of the Green function, we get
\begin{align}
\nabla v_E(x)&=\int_{\T^n} \nabla_x G(x,y) u_E(x) \, dy\\
& = \int_E \nabla_x G(x,y) \, dy - \int_{E^c}\nabla_x G(x,y) \, dy\\
& = -\int_E \nabla_y G(x,y) \, dy +\int_{E^c}\nabla_y G(x,y) \, dy\\
&= -2 \int_{\partial E}  G(x,y) \nu_E(y) \, d \mu(y)\,,\label{eqcar501}
\end{align}
where in the last passage we applied the divergence theorem.\\
By means of formula~\eqref{Deltanu}
$$
\Delta \nu_E = \nabla \HHH -|B|^2\nu_E\,,
$$
since $E$ (being critical) satisfies $\HHH + 4 \gamma v_E= \lambda$ for some constant $\lambda\in\R$, we have
\begin{align}
-\Delta \nu_E- |\BBB|^2\nu_E&=\nabla(4 \gamma v_E - \lambda)\\
&=\nabla^{\T^n}\!(4 \gamma v_E - \lambda)-\partial_{\nu_E}(4 \gamma v_E - \lambda)\\
&= -4\gamma (\partial_{\nu_E} v_E) \nu_E -8\gamma \int_{\partial E}  G(x,y) \nu_E(y) \, d \mu(y)
\end{align}
on $\pa E$, by formula~\eqref{eqcar501}.\\
This equation can be written as $L(\nu_i)= 0$, for every $i\in\{1,\dots,n\}$, where $L$ is the self--adjoint, linear operator defined as
$$
L(\varphi) = -\Delta \varphi - |\BBB|^2\varphi  + 4\gamma \partial_{\nu_E} v_E \varphi + 8\gamma \int_{\partial E}  G(x,y) \varphi(y) \, d \mu(y)\,,
$$
which clearly satisfies 
$$
b_E(\varphi,\psi)=\int_{\pa E}\langle L(\varphi)\vert \psi\rangle\,\dmu\qquad\text{ and }\qquad\Pi_E(\varphi) =\int_{\pa E}\langle L(\varphi)\vert \varphi\rangle\,\dmu\,.
$$
Then, if we ``decompose'' a smooth function $\varphi \in \widetilde H^1(\pa E)$ as $\varphi = \psi + \langle \eta \vert  \nu_E\rangle$, for some $\eta\in \R^n$ and $\psi \in T^\perp(\partial E)$, we have (recalling formula~\eqref{Pieq})
\begin{align*}
\Pi_E(\varphi) =&\,\int_{\pa E}\langle L(\varphi)\vert \varphi\rangle\,\dmu\\
=&\,\int_{\pa E}\langle L(\psi)\vert \psi \rangle\,\dmu+ 2\int_{\pa E} \langle L(\langle\eta \vert  \nu_E\rangle) \vert \psi\rangle\,\dmu
+\int_{\pa E}\langle L(\langle\eta \vert  \nu_E\rangle)\vert \langle \eta \vert  \nu_E \rangle   \rangle\,\dmu\\
=&\,\Pi_E(\psi)\,. 
\end{align*}
By approximation with smooth functions, we conclude that this equality holds for every function in $\widetilde H^1(\pa E)$.\\
The initial claim about the form $b_E$ then easily follows by its definition. Moreover, if $E$ is a strictly stable critical set there holds
\begin{equation} \label{uusi stability}
\Pi_E(\varphi)>0 \qquad \text{for every}\, \varphi \in \widetilde H^1(\pa E) \setminus T(\pa E).
\end{equation}
\end{remark}

\begin{remark}\label{rembase}
We observe that there exists an orthonormal frame $\{e_1, \dots,  e_n \}$ of $\R^n$ such that
\begin{equation}\label{ort}
\int_{\pa E} \langle \nu_E |  e_i \rangle \langle \nu_E | e_j \rangle \, \dmu=0,
\end{equation}
for all $i \ne j$, indeed, considering the symmetric $n\times n$--matrix $A= (a_{ij})$ with components $a_{ij}= \int_{\pa E} \nu_E^i \nu_E^j \, \dmu$, where $\nu_E^i=\langle\nu_E|\eps_i\rangle$ for some basis $\{\eps_1,\dots,\eps_n\}$ of $\R^n$, we have
\begin{equation}
\int_{\pa E} (O\nu_E)_i (O \nu_E)_j \, \dmu = (OAO^{-1})_{ij},
\end{equation}
for every $O \in SO(n)$. Choosing $O$ such that $OAO^{-1}$ is diagonal and setting $e_i=O^{-1}\eps_i$, relations~\eqref{ort} are clearly satisfied. \\
Hence, the  functions $\langle \nu_E | e_i \rangle$ which are not identically zero are an orthogonal basis of $T(\pa E)$. We set 
\begin{equation}\label{IIeq}
\II_E=\bigl\{i\in\{1,\dots,n\}\,:\,\text{$\langle\nu_E|e_i\rangle$ is not identically zero}\bigr\}
\end{equation}
and
\begin{equation}\label{OOeq}
\OO_E=\mathrm{Span}\{e_i\,:\,i\in \II_E\},
\end{equation}
then, given any $\varphi \in \Htilde^1(\pa E)$, its projection on $\Tort (\pa E)$ is 
\begin{equation}\label{projection}
\pi(\varphi)= \varphi - \sum_{i\in\II_E} \frac{\int_{\pa E} \varphi \langle \nu_E |  e_i \rangle \, \dmu}{\norma{\langle \nu_E |  e_i\rangle}_{L^2(\partial E)}^2}\langle \nu_E | e_i \rangle\,.
\end{equation}
\end{remark}

\smallskip

{\em From now on we will extensively use Sobolev spaces on smooth hypersurfaces. Most of their properties hold as in $\R^n$, standard references are~\cite{AdamsFournier} in the Euclidean space and~\cite{Aubin} when the ambient is a manifold.}

\bigskip

Given a smooth set $E \subseteq \T^n$, for $\eps>0$ small enough, we let ($d$ is the ``Euclidean'' distance on $\T^n$) 
\begin{equation}
\label{tubdef}
N_\eps=\{x \in \T^n \, : \, d(x,\pa E)<\eps\}
\end{equation}
to be a {\em tubular neighborhood} of $\pa E$ such that the {\em orthogonal projection map} $\pi_E:N_\eps\to \pa E$ giving the (unique) closest point on $\pa E$ and the {\em signed distance function} $d_E:N_\eps\to\R$ from $\pa E$
\begin{equation}\label{sign dist}
d_E(x)=
\begin{cases}
d(x, \pa E) &\text{if $x \notin E$}\\
-d(x, \pa E) &\text{if $x \in E$}
\end{cases} 
\end{equation}
are well defined and smooth in $N_\eps$ (for a proof of the existence of such tubular neighborhood and of all the subsequent properties, see~\cite{ManMen} for instance). Moreover, for every $x\in N_\eps$, the projection map is given explicitly by 
\begin{equation}\label{eqcar2050}
\pi_E(x)=x-\nabla d^2_E(x)/2=x-d_E(x)\nabla d_E(x)
\end{equation}
and the unit vector $\nabla d_E(x)$ is orthogonal to $\pa E$ at the point $\pi_E(x)\in\partial E$, indeed actually 
\begin{equation}\label{eqcar410bis}
\nabla d_E(x)=\nabla d_E(\pi_E(x))=\nu_E(\pi_E(x))\,,
\end{equation}
which means that the integral curves of the vector field $\nabla d_E$ are straight segments orthogonal to $\pa E$.\\ 
This clearly implies that the map 
\begin{equation}\label{eqcar410}
\partial E\times (-\eps,\eps)\ni (y,t)\mapsto L(y,t)=y+t\nabla d_E(y)=y+t\nu_E(y)\in N_\eps
\end{equation}
is a smooth diffeomorphism with inverse
$$
N_\eps\ni x\mapsto L^{-1}(x)=(\pi_E(x),d_E(x))\in \partial E\times (-\eps,\eps)\,.
$$
Moreover, denoting with $JL$ its  Jacobian (relative to the hypersurface $\pa E$), there holds
\begin{equation}\label{eqcar411}
0<C_1\leq JL(y,t)\leq C_2
\end{equation}
on $\pa E\times(-\eps,\eps)$, for a couple of constants $C_1,C_2$, depending on $E$ and $\eps$.

By means of such tubular neighborhood of a smooth set $E\subseteq\T^n$ and the map $L$, we can speak of ``$W^{k,p}$--closedness'' (or ``$C^{k,\alpha}$--closedness'') to $E$ of another smooth set $F\subseteq\T^n$, asking that for some $\delta>0$ ``small enough'', we have  $\vol(E\triangle F) < \delta$ and that $\partial F$ is contained in a tubular neighborhood $N_\eps$ of $E$, as above, described by  
$$
\pa F=\{y+\psi(y)\nu_E(y) \, : \, y\in \pa E\},
$$
for a smooth function $\psi:\partial E\to\R$ with $\norma{\psi}_{W^{k,p}(\pa E)}< \delta$ (resp. $\norma{\psi}_{C^{k,\alpha}(\pa E)}< \delta$). That is, we are asking that the two sets $E$ and $F$ differ by a set of small measure and that their boundaries are ``close'' in $W^{k,p}$ (or $C^{k,\alpha}$) as graphs.

Notice that
$$
\psi(y)=\pi_2\circ L^{-1}\bigl(\pa E\cap \{y+\lambda \nu_E(y)\,:\,\lambda\in\R\}\bigr)\,,
$$
where $\pi_2:\pa E\times(-\eps,\eps)\to\R$ is the projection on the second factor.\\
Moreover, given a sequence of smooth sets $F_i\subseteq\T^n$, we will write $F_i\to E$ in $W^{k,p}$ (resp. $C^{k,\alpha}$) if for every $\delta>0$, there hold  $\vol(F_i\triangle E) < \delta$, the smooth boundary $\partial F_i$ is contained in some $N_\eps$, relative to $E$ and it is described by 
$$
\pa F_i=\{y+\psi_i(y)\nu_E(y) \, : \, y\in \pa E\},
$$
for a smooth function $\psi_i:\partial E\to\R$ with $\norma{\psi_i}_{W^{k,p}(\pa E)}< \delta$ (resp. $\norma{\psi_i}_{C^{k,\alpha}(\pa E)}< \delta$), for every $i\in\N$ large enough. 

\medskip

{\em From now on, in all the rest of the work, we will refer to the volume--constrained nonlocal Area functional $J$ (and Area functional $\A$), sometimes without underlining the presence of such constraint, by simplicity. Moreover, with $N_\eps$ we will always denote a suitable tubular neighborhood of a smooth set, with the above properties.}

\begin{definition}\label{min}
We say that a smooth set $E \subseteq \T^n$ is a {\em local minimizer} for the functional $J$ (for the Area functional $\A$) if there exists $\delta >0$ such that 
$$
J(F)\ge J(E) \qquad (\A(F) \ge \A(E))
$$
for all smooth sets $F\subseteq \T^n$ with $\vol(F)=\vol(E)$ and $\vol(E\triangle F)< \delta$.

We say that a smooth set $E \subseteq \T^n$ is a {\em $W^{2,p}$--local minimizer} if there exists $\delta >0$ and a tubular neighborhood $N_\eps$ of $E$, as above, such that 
$$
J(F) \geq J(E) \qquad (\A(F) \ge \A(E))
$$
for all smooth sets $F \subseteq \T^n$ with $\vol (F)= \vol(E)$, $\vol(E \triangle F)<\delta$ and $\partial F$ contained in $N_\eps$, described by  
$$
\pa F=\{y+\psi(y)\nu_E(y) \, : \, y\in \pa E\},
$$
for a smooth function $\psi:\partial E\to\R$ with $\norma{\psi}_{W^{2,p}(\pa E)}< \delta$.\\
Clearly, any local minimizer is a $W^{2,p}$--local minimizer.
\end{definition}
 
We immediately show a {\em necessary} condition for $W^{2,p}$--local minimizers.

\begin{proposition}
Let the smooth set $E \subseteq \T^n$ be a $W^{2,p}$--local minimizer of $J$, then $E$ is a critical set and 
$$
\Pi_E(\varphi)\ge 0 \qquad \qquad \text{for all $\varphi \in \Htilde^1(\pa E)$,}
$$
in particular, $E$ is stable.
\end{proposition}
\begin{proof}
If $E$ is a $W^{2,p}$--local minimizer of $J$, given any $\varphi \in C^\infty(\pa E)\cap\Htilde^1(\pa E)$, we consider the admissible vector field $X\in C^\infty(\T^n; \R^n)$ given by Lemma~\ref{vector field} and the associated flow $\Phi$. Then, the variation  $E_t=\Phi_t(E)$ of $E$ is volume--preserving, that is, $\vol(E_t)=\vol(E)$ and for every $\delta > 0$, there clearly exists a tubular neighborhood $N_\eps$ of $E$ and $\overline{\eps}>0$ such that for $t \in (-\overline{\eps}, \overline{\eps})$ we have
$$
\vol(E\triangle E_t) < \delta
$$
and
$$\pa E_t=\{y+\psi(y)\nu_E(y) \, : \, y\in \pa E\}\subseteq N_\eps$$
for a smooth function $\psi:\partial E\to\R$ with $\norma{\psi}_{W^{2,p}(\pa E)}< \delta$. 
Hence, the $W^{2,p}$--local minimality of $E$ implies
$$
J(E) \le J(E_t),
$$
for every $t \in (-\overline{\eps}, \overline{\eps})$.
It follows
$$
0=\frac{d}{dt}J(E_t)\Bigl|_{t=0}= \int_{\partial E} (\HHH+ 4 \gamma v_E) \varphi \, d\mu,
$$
by Theorem~\ref{first var}, which implies that $E$ is a critical set, by the subsequent discussion and
$$
0\leq\frac{d^2}{dt^2}J(E_t)\Bigl|_{t=0}=\Pi_E(\varphi),
$$
by Proposition~\ref{IIcritprop} and Remark~\ref{remcarlo}.\\
Then, the thesis easily follows by the density of $C^\infty(\pa E)\cap\Htilde^1(\pa E)$ in $\Htilde^1(\pa E)$ (see~\cite{Aubin}, for instance) and the definition of $\Pi_E$, formula~\eqref{Pieq}.
\end{proof}

The rest of this section will be devoted to show that the strict stability (see Definition~\ref{str stab}) is a {\em sufficient} condition for the $W^{2,p}$--local minimality. Precisely, we will prove the following theorem.

\begin{thm}\label{W2pMin}
Let $p>\max\{2, n-1\}$ and $E\subseteq\T^n$ a smooth strictly stable critical set for the nonlocal Area functional $J$ (under a volume constraint), with $N_\eps$ a tubular neighborhood of $\pa E$ as in formula~\eqref{tubdef}. Then, there exist constants $\delta,C>0$ such that
\begin{equation}
J(F)\ge J(E) + C[\alpha(E,F)]^2,
\end{equation}
for all smooth sets $F\subseteq \T^n$ such that $\vol(F)=\vol(E)$, $\vol(F\triangle E)<\delta$, $\pa F \subseteq N_{\varepsilon}$ and
\begin{equation}
\pa F= \{y+\psi(y)\nu_E(y)\, : \, y \in \pa E\},
\end{equation}
for a smooth function $\psi$ with $\norma{\psi}_{W^{2,p}(\pa E)} < \delta$, where the ``distance'' $\alpha(E,F)$ is defined as 
\beq\label{alpha}
\alpha(E,F)= \min_{\eta\in \R^n} \vol(E \triangle (F+\eta)).
\eeq
As a consequence, $E$ is a $W^{2,p}$--local minimizer of $J$. Moreover, if $F$ is $W^{2,p}$--close enough to $E$ and $J(F)=J(E)$, then $F$ is a translate of $E$, that is, $E$ is locally the unique $W^{2,p}$--local minimizer, up to translations.
\end{thm}

\begin{remark}
We could have introduced the definitions of {\em strict} local minimizer or {\em strict} $W^{2,p}$--local minimizer for the nonlocal Area functional, by asking that the inequalities $J(F)\leq J(E)$ in Definition~\ref{min} are equalities if and only if $F$ is a translate of $E$. With such notion, the conclusion of this theorem is that $E$ is actually a strict $W^{2,p}$--local minimizer (with a ``quantitative'' estimate of its minimality).
\end{remark}

\begin{remark}\label{L1min}
With some extra effort, it can be proved that in the same hypotheses of Theorem~\ref{W2pMin}, the set $F$ is actually a local minimizer (see~\cite{AcFuMo}). Since in the analysis of the modified Mullins--Sekerka and surface diffusion flow in the next sections we do not need such a stronger result, we omitted to prove it.
\end{remark}

For the proof of this result we need some technical lemmas. We underline that most of the difficulties are due to the presence of the degenerate subspace $T(\partial E)$ of the form $\Pi_E$ (where it is zero), related to the translation invariance of the nonlocal Area functional (recall the discussion after Definition~\ref{Pi}).

In the next key lemma we are going to show how to construct volume--preserving variations (hence, admissible smooth vector fields) ``deforming'' a set $E$ to any other smooth set with the same volume, which is $W^{2,p}$--close enough. By the same technique we will also prove Lemma~\ref{vector field} immediately after, whose proof was postponed from Subsection~\ref{sec1.2}.

\begin{lem}\label{lemma1}
Let $E \subseteq \T^n$ be a smooth set and $N_\eps$ a tubular neighborhood of $\pa E$ as above, in formula~\eqref{tubdef}. For all $p >n-1$, there exist constants $\delta, C>0$ such that if $\psi \in C^\infty (\pa E)$ and $\norma{\psi}_{\Wduep(\pa E)} \leq \delta$, then there exists a vector field $X\in C^\infty(\T^n; \R^n)$ with $\Div\!X=0$ in $N_\eps$ and the associated flow $\Phi$, defined by system~\eqref{varflow}, satisfies 
\begin{equation}\label{flowin}
\Phi(1,y)= y+ \psi(y)\nu_E(y)\,, \qquad\text{for all $y \in \pa E$.}
\end{equation}
Moreover, for every $t \in [0,1]$
\begin{equation}\label{normflow}
\norma{\Phi(t, \cdot) - \Id}_{\Wduep (\pa E)} \leq C \norma{\psi }_{\Wduep(\pa E)} \, .
\end{equation}
Finally, if $\vol(E_1) = \vol(E)$, then the variation $E_t=\Phi_t(E)$ is volume--preserving, that is, $\vol(E_t) = \vol (E)$ for all $t \in [-1,1]$ and the vector field $X$ is admissible.
\end{lem}
\begin{proof}
We start considering the vector field $\Xtilde\in C^\infty (N_\eps;\R^n)$ defined as 
\begin{equation}\label{eqcar400}
\Xtilde (x) = \xi(x) \nabla  d_E (x)
\end{equation}
for every $x \in N_\eps$, where $ d_E:N_\eps\to\R$ is the signed distance function from $E$ and $\xi:N_\eps\to\R$ is the function defined as follows: for all $y \in \pa E$, we let $f_y : (-\eps , \eps) \to \R$ to be the unique solution of the ODE
$$
\begin{cases}
f'_y (t) + f_y(t) \Delta  d_E (y+t \nu_E (y))=0 \\
f_y(0)= 1
\end{cases}
$$
and we set 
$$ 
\xi(x)=\xi (y+ t \nu_E (y)) = f_y(t) =\exp \Bigl ( - \int _0 ^t {\Delta d_E (y+ s \nu_E(y)) \, ds } \Bigr ), 
$$
recalling that the map $(y,t)\mapsto x=y+t\nu_E(y)$ is a smooth diffeomorphism between $\partial E\times (-\eps,\eps)$ and $N_\eps$ (with inverse $x\mapsto (\pi_E(x),d_E(x))$, where $\pi_E$ is the orthogonal projection map on $E$, defined by formula~\eqref{eqcar2050}). Notice that the function $f$ is always positive, thus the same holds for $\xi$ and $\xi=1$, $\nabla d_E=\nu_E$, hence $\widetilde{X}=\nu_E$ on $\partial E$.

Our aim is then to prove that the smooth vector field $X$ defined by 
\begin{equation}\label{field}
X(x)=  \int _0 ^{\psi(\pi_E (x))} {\frac{ds}{\xi(\pi_E(x)+ s \nu_E(\pi_E(x)))} } \, \widetilde{X}(x)
\end{equation}
for every $x \in N_\eps$ and extended smoothly to all $\T^n$, satisfies all the properties of the statement of the lemma. 

\smallskip

\noindent \textbf{Step ${\mathbf 1}$.} We saw that $\Xtilde\vert_{\pa E} = \nu_E$, now we show that $\Div\!\Xtilde =0$ and analogously $\Div\!X=0$ in $N_\eps$.\\
Given any $x=y+t \nu_E(y)\in N_\eps$, with $y\in\pa E$, we have
\begin{align} 
\Div\!\Xtilde (x) &= \Div [\xi (x) \nabla  d_E(x)]\\
&= \langle\nabla\xi (x)| \nabla  d_E(x)\rangle+\xi (x) \Delta d_E(x)\\
&= \frac{\partial}{\partial t}[\xi (y+t \nu_E (y))]+ \xi (y+t \nu_E (y)) \Delta d_E (y+ t \nu_E(y)) \\
& = f'_y (t) + f_y (t) \Delta d_E ( y+ t \nu_E (y))\\
&= 0,
\end{align}
where we used the fact that $f_y'(t)=\langle \nabla \xi(y+t\nu_E(y))\vert \nu_E(y)\rangle$ and $\nabla d_E(y+t \nu_E (y))=\nu_E(y)$, by formula~\eqref{eqcar410bis}.\\
Since the function
$$
x \mapsto \theta(x)=\int_0 ^{\psi(\pi_E(x))} \frac{ds}{\xi (\pi_E(x) + s \nu (\pi_E (x))}
$$
is clearly constant along the segments $t\mapsto x+t\nabla d_E(x)$, for every $x\in N_\eps$, it follows that 
$$
0=\frac{\partial}{\partial t}\bigl[\theta(x+t\nabla d_E(x))\bigr]\,\Bigr\vert_{t=0}=\langle \nabla\theta(x)|\nabla d_E(x)\rangle,
$$
hence,
$$
\Div\!X= \langle \nabla\theta|\nabla d_E\rangle\xi + \theta \Div\!\Xtilde = 0.
$$

\smallskip

\noindent \textbf{Step ${\mathbf 2}$.} Recalling that $\psi \in C^\infty (\pa E)$ and $p>n-1$, we have
$$ 
\norma{\psi}_{L^\infty (\pa E)} \leq\norma{\psi}_{C^1(\pa E)} \leq C_E \norma{\psi}_{\Wduep (\pa E)},
$$
by Sobolev embeddings (see~\cite{Aubin}). Then, we can choose $\delta < {\eps}/{C_E}$ such that for all $x \in \pa E$ we have that $x\pm\psi(x) \nu_E (x) \in N_\eps$.\\
To check that equation~\eqref{flowin} holds, we observe that
\begin{equation}
\theta(x)=\int_0 ^{\psi (\pi_E(x))} {\frac{ds}{\xi (\pi_E(x)+ s \nu_E (\pi_E(x)))} }
\end{equation}
represents the time needed to go from $\pi_E(x)$ to $\pi_E(x)+ \psi (\pi_E(x)) \nu_E (\pi_E (x))$ along the trajectory of the vector field $\Xtilde$, which is the segment connecting $\pi_E(x)$ and $\pi_E(x)+ \psi (\pi_E(x)) \nu_E (\pi_E (x))$, of length $\psi (\pi_E(x))$, parametrized as 
$$
s\mapsto \pi_E(x)+ s\psi (\pi_E(x)) \nu_E (\pi_E (x)),
$$
for $s\in[0,1]$ and  which is traveled with velocity $\xi (\pi_E(x)+ s \nu_E (\pi_E(x)))=f_{\pi_E(x)}(s)$. Therefore, by the above definition of $X=\theta\Xtilde$ and the fact that the function $\theta$ is constant along such segments, we conclude that
$$
\Phi(1, y) - \Phi (0,y) = \psi(y) \nu_E (y)\,,
$$
that is, $\Phi(1,y)= y+ \psi (y) \nu_E (y)$, for all $y\in \pa E$.

\smallskip

\noindent \textbf{Step ${\mathbf 3}$.} To establish inequality~\eqref{normflow}, we first show that 
\begin{equation}\label{normfield}
\norma{X}_{\Wduep(N_\eps)} \leq C \norma{\psi}_{\Wduep(\pa E)}
\end{equation}
for a constant $C>0$ depending only on $E$ and $\eps$.
This estimate will follow from the definition of $X$ in equation~\eqref{field} and the definition of $\Wduep$--norm, that is,
$$
\norma{X}_{\Wduep(N_\eps)} = \norma{X}_{L^p (N_\eps)} + \norma{ \nabla  X}_{L^p (N_\eps)} + \norma{ \nabla  ^2 X }_{L^p (N_\eps)} \, .
$$ 
As $|\nabla d_E|=1$ everywhere and the positive function $ \xi $ satisfies $0 < C_1 \leq \xi \leq C_2$ in $N_\eps$, for a pair of constants $C_1$ and $C_2$, we have
\begin{align}
\norma{X} _{L^p (N_\eps)} ^ p & =  \int_{N_\eps} \biggl \vert {\int_0 ^{\psi(\pi_E(x))} \frac{ds}{\xi ( \pi_E(x) + s \nu_E (\pi_E(x)))} \, \xi(x) \nabla  d_E (x) } \biggr \vert^p \, dx  \nonumber \\
& \leq \norma{\xi}^p_{L^\infty (N_\eps)} \int_{N_\eps} \biggl \vert \int_0 ^{\psi(\pi_E(x))} \frac{ds}{\xi ( \pi_E(x) + s \nu_E (\pi_E(x)))} \biggr \vert^p \, dx \nonumber \\
& \leq \frac{C_2^p}{C_1^p}\int_{N_\eps} \abs{\psi(\pi_E(x))}^p \, dx \nonumber \\
& =\frac{C_2^p}{C_1^p}\int_{\pa E}\int_{-\eps}^\eps \abs{\psi(\pi_E(y+t\nu_E(y)))}^p JL(y,t) \, dt\,d\mu(y)\\
& =\frac{C_2^p}{C_1^p}\int_{\pa E} \abs{\psi(y)}^p\int_{-\eps}^\eps JL(y,t) \, dt\,d\mu(y)\\
& \leq C\int_{\pa E} \abs{\psi(y)}^p \, d\mu(y)\\
& = C\norma{\psi}_{L^p(\pa E)} ^p \, ,\label{normX}
\end{align}
where $L:\partial E\times (-\eps,\eps)\to N_\eps$ the smooth diffeomorphism defined in formula~\eqref{eqcar410} and $JL$ its Jacobian. Notice that the constant $C$ depends only on $E$ and $\eps$.
 
Now we estimate the $L^p$--norm of $\nabla  X$. We compute 
\begin{align}
\nabla X =&\,\frac{\nabla  \psi (\pi_E(x)) d\pi_E(x)}{ \xi(\pi_E(x) + \psi(\pi_E(x)) \nu_E (\pi_E(x)))} \, \xi (x) \nabla  d_E(x) \\
&\,-\biggl [\int_0 ^{\psi (\pi_E (x))}\frac{\nabla  \xi (\pi_E(x) + s \nu_E (\pi_E(x)))}{\xi ^ 2 (\pi_E(x) + s \nu_E (\pi_E(x)))}
d\pi_E(x)\,\mathrm{Id}\, ds\biggr]\, \xi (x) \nabla  d_E (x) \\
&\,-\biggl [\int_0 ^{\psi (\pi_E (x))}\frac{\nabla  \xi (\pi_E(x) + s \nu_E (\pi_E(x)))}{\xi ^ 2 (\pi_E(x) + s \nu_E (\pi_E(x)))}
d\pi_E(x)s\,d\nu_E (\pi_E (x))\, ds\biggr]\, \xi (x) \nabla  d_E (x)\\
&\,+ \int_0 ^{\psi (\pi_E (x))} \frac{ds}{\xi  (\pi_E(x) + s \nu_E (\pi_E(x)))} \bigl ( \nabla  \xi (x) \nabla  d_E (x) + \xi (x)\nabla^2 d_E (x) \bigr)
\end{align}
and we deal with the integrals in the three terms as before, changing variable by means of the function $L$. That is, since all the functions $d\pi_E$, $d\nu_E$, $\nabla^2 d_E$, $\xi$, $1/\xi$, $\nabla\xi$ are bounded by some constants depending only on $E$ and $\eps$, we easily get (the constant $C$ could vary from line to line)
\begin{align*}
\norma{\nabla  X}_{L^p (N_\eps)}^p \leq 
&\,C\int_{N_\eps} \vert \nabla  \psi (\pi_E(x))\vert^p \, dx+C\int_{N_\eps}  \vert\psi(\pi_E (x))\vert^p \, dx\\
=&\,C\int_{\pa E}\int_{-\eps}^\eps \abs{\nabla\psi(\pi_E(y+t\nu_E(y)))}^p\,JL(y,t) \, dt\,d\mu(y)\\
&\,+C\int_{\pa E}\int_{-\eps}^\eps \abs{\psi(\pi_E(y+t\nu_E(y)))}^p\,JL(y,t) \, dt\,d\mu(y)\\
=&\, C\int_{\pa E} \bigl(\abs{\psi(y)}^p+\abs{\nabla\psi(y)}^p\bigr)\,\int_{-\eps}^\eps JL(y,t) \, dt\,d\mu(y)\\
& \leq C\norma{\psi}_{L^p(\pa E)} ^p +C\norma{\nabla\psi}_{L^p(\pa E)} ^p\\
& \leq C\norma{\psi}_{W^{1,p}(\pa E)} ^p\,.
\end{align*}
A very analogous estimate works for $\norma{ \nabla ^2 X} _{L^p (N_\eps)}^p$ and we obtain also
\begin{equation}\label{normgrad2X}
\norma{\nabla   ^2 X}_{L^p (N_\eps)}^p \leq C\norma{\psi}_{\Wduep (\pa E)}^p\,,
\end{equation}
hence, inequality~\eqref{normfield} follows with $C=C(E,\eps)$.

Applying now Lagrange theorem to every component of $\Phi (\cdot,y)$ for any $y\in\pa E$ and $t\in[0,1]$, we have
$$
\Phi_i(t,y)-y_i=\Phi_i(t,y) - \Phi_i(0, y)= t X^i(\Phi(s,y))\,,
$$
for every $i\in\{1,\dots,n\}$, where $s=s(y,t)$ is a suitable value in $(0,1)$. Then, it clearly follows
\begin{equation}\label{normfi}
\norma{\Phi(t,\cdot)- \Id}_{L^\infty (\pa E)}\leq C\norma{X} _{L^\infty (N_\eps)}\leq C\norma{X}_{\Wduep(N_\eps)}
\leq C \norma{\psi}_{\Wduep(\pa E)}
\end{equation}
by estimate~\eqref{normfield}, with $C=C(E,\eps)$ (notice that we used Sobolev embeddings, being $p>n-1$, the dimension of $\pa E$).\\
Differentiating the equations in system~\eqref{varflow}, we have (recall that we use the convention of summing over the repeated indices)
\begin{equation}\label{flowdiff}
\begin{cases}
\frac{\pa}{\pa t} \nabla^i\Phi_j(t,y) = \nabla^k  X^j(\Phi (t,y)) \nabla^i\Phi_k(t,y) \\
\nabla  ^i \Phi_j (0,y)=\delta_{ij}
\end{cases}
\end{equation}
for every $i,j\in\{1,\dots,n\}$. It follows,
\begin{align}
\frac{\pa}{\pa t} \bigl\vert \nabla^i\Phi_j(t,y)-\delta_{ij}\bigl\vert^2\leq&\,2\bigl|(\nabla^i\Phi_j(t,y)-\delta_{ij})\nabla^k  X^j(\Phi (t,y)) \nabla^i\Phi_k(t,y)\bigr|\\
\leq&\,2\Vert\nabla X\Vert_{L^\infty(N_\eps)}\bigl\vert \nabla^i\Phi_j(t,y)-\delta_{ij}\bigl\vert^2+2\Vert\nabla X\Vert_{L^\infty(N_\eps)}\bigl|\nabla^i\Phi_j(t,y)-\delta_{ij}\bigr|
\end{align}
hence, for almost every $t\in[0,1]$, where the following derivative exists,
$$
\frac{\pa}{\pa t} \bigl\vert \nabla^i\Phi_j(t,y)-\delta_{ij}\bigl\vert\leq C\Vert\nabla X\Vert_{L^\infty(N_\eps)}\bigl(\bigl\vert \nabla^i\Phi_j(t,y)-\delta_{ij}\bigl\vert+1\bigr)\,.
$$
Integrating this differential inequality, we get
$$
\bigl\vert \nabla^i\Phi_j(t,y)-\delta_{ij}\bigl\vert\leq e^{tC\Vert\nabla X\Vert_{L^\infty(N_\eps)}}-1\leq e^{C\Vert X\Vert_{W^{2,p}(N_\eps)}}-1,
$$
as $t\in[0,1]$, where we used Sobolev embeddings again.
Then, by inequality~\eqref{normfield}, we estimate
\begin{equation}
\sum_{1\leq i,j\leq n}\norma{\nabla^i\Phi_j(t,\cdot) - \delta_{ij}}_{L^\infty (\pa E)}
\leq C\bigl(e^{C\norma{\psi}_{\Wduep (\pa E)}}-1\bigr)\leq C\norma{\psi}_{\Wduep (\pa E)},\label{normgradfi}
\end{equation}
as $\norma{\psi}_{\Wduep(\pa E)} \leq \delta$, for any $t \in [0,1]$ and $y\in\pa E$, with $C=C(E,\eps,\delta)$.\\
Differentiating equations~\eqref{flowdiff}, we obtain
\begin{equation}\label{flowdiffdiff}
\begin{cases}
\frac{\pa}{\pa t} \nabla^\ell\nabla^i\Phi_j (t,y) = \nabla^s\nabla^kX^j(\Phi(t,y))\nabla^i\Phi_k(t,y)\nabla^\ell\Phi_s(t,y)\\
\phantom{\frac{\pa}{\pa t} \nabla^\ell\nabla^i\Phi_j (t,y) =} +\nabla^kX^j (\Phi (t,y)) \nabla^\ell\nabla^i\Phi_k(t,y) \\
\nabla^\ell\nabla^i\Phi (0,y)= 0
\end{cases}
\end{equation}
(where we sum over $s$ and $k$), for every $t \in [0,1]$, $y\in\pa E$ and $i,j,\ell\in\{1,\dots,n\}$.\\
This is a linear {\em non--homogeneous} system of ODEs such that, if we control $C\norma{\psi}_{\Wduep (\pa E)}$, the smooth coefficients in the right side multiplying the solutions $\nabla^\ell\nabla^i\Phi_j (\cdot,y)$ are uniformly bounded  (as in estimate~\eqref{normgradfi}, Sobolev embeddings then imply that $\nabla X$ is bounded in $L^\infty$ by $C\norma{\psi}_{\Wduep (\pa E)}$). Hence, arguing as before, for almost every $t\in[0,1]$ where the following derivative exists, there holds
\begin{align}
\frac{\pa}{\pa t} \bigl\vert \nabla^2\Phi(t,y)\bigl\vert\leq&\,C\Vert\nabla X\Vert_{L^\infty(N_\eps)}\bigl\vert \nabla^2\Phi(t,y)\bigl\vert+C\vert\nabla^2 X(\Phi(t,y))\vert\\
\leq&\,C\delta\bigl\vert \nabla^2\Phi(t,y)\bigl\vert+C\vert\nabla^2 X(\Phi(t,y))\vert\,,
\end{align}
by inequality~\eqref{normfield} (notice that inequality~\eqref{normgradfi} gives an $L^\infty$--bound on $\nabla\Phi$, {\em not only} in $L^p$, which is crucial).
Thus, by means of Gronwall's lemma (see~\cite{stampvidpic}, for instance), we obtain the estimate 
$$
\bigl\vert \nabla^2\Phi(t,y)\bigl\vert\leq C\int_0^t \vert\nabla^2 X(\Phi(s,y))\vert e^{C\delta(t-s)}\,ds\leq C\int_0^t \vert\nabla^2 X(\Phi(s,y))\vert\,ds\,,
$$
hence,
\begin{align}
\norma{\nabla  ^2 \Phi ( t, \cdot) }_{L^p( \pa E)}^p 
\leq&\, C\int_{\pa E}\Bigl(\int_0^t \vert\nabla^2 X(\Phi(s,y))\vert\,ds\Bigr)^p\,d\mu(y)\\
\leq&\, C\int_0^t\int_{\pa E} \vert\nabla^2 X(\Phi(s,y))\vert^p\,d\mu(y)ds\\
=&\, C\int_{N_\eps} \vert\nabla^2 X(x)\vert^pJL^{-1}(x)\,dx\\
\leq&\, C \norma{\nabla  ^2 X}_{L^p(N_\eps)}^p\\
\leq&\, C\norma{X}_{W^{2,p}(N_\eps)}^p\\
\leq&\, C\norma{\psi}_{\Wduep (\pa E)}^p\,,\label{normgrad2fi}
\end{align}
by estimate~\eqref{normfield}, for every $t\in[0,1]$, with $C=C(E,\eps,\delta)$.\\
Clearly, putting together inequalities~\eqref{normfi},~\eqref{normgradfi} and~\eqref{normgrad2fi}, we get the 
estimate~\eqref{normflow} in the statement of the lemma.

\smallskip

\noindent \textbf{Step $\mathbf{4}$.} Finally, computing as in formula~\eqref{sec der vol} and Remark~\ref{rem999}, we have
$$
\frac{d^2}{dt^2} \vol(E_t) = \int_{\pa E} \langle X| \nu_{E_t} \rangle \Div^{\T^n}\!\!X \, \dmu_t ,
$$ 
for every $t\in[-1,1]$, hence, since by Step~$1$ we know that $ \Div^{\T^n}\!\!X=0$ in $N_\eps$ (which contains each $\pa E_t$), we conclude that $ \frac{d^2}{dt^2} \vol(E_t) =0$ for all $t \in [-1,1]$, that is, the function $ t \mapsto \vol(E_t)$ is linear.\\
If then $\vol(E_1)=\vol(E)= \vol(E_0)$, it follows that $\vol (E_t)= \vol(E) $, for all $t \in [-1,1]$ which implies that $X$ is admissible, by Remark~\ref{remcarlo888}. 
\end{proof}

With an argument similar to the one of this proof, we now prove Lemma~\ref{vector field}.

\begin{proof}[Proof of Lemma~\ref{vector field}]
Let $\varphi:\partial E\to\R$ a $C^\infty$ function with zero integral, then we define the following smooth vector field in $N_\eps$,
\begin{equation}
X(x)=  \varphi(\pi_E(x))\Xtilde(x),
\end{equation}
where $\Xtilde$ is the smooth vector field defined by formula~\eqref{eqcar400} and we extend it to a smooth vector field $X\in C^\infty(\T^n; \R^n)$ on the whole $\T^n$. Clearly, by the properties of $\Xtilde$ seen above, 
$$
\langle X(y) \vert \nu_E(y)\rangle=
\varphi(y)\langle\Xtilde(y) \vert \nu_E(y)\rangle=\varphi(y)
$$
for every $y\in\pa E$.\\
As the function $x \mapsto \varphi(\pi_E(x))$ is constant along the segments $t\mapsto x+t\nabla d_E(x)$, for every $x\in N_\eps$, it follows, as in Step~$1$ of the previous proof, that $\Div\! X=0$ in $N_\eps$. Then, arguing as in Step~$4$, the flow $\Phi$ defined by system~\eqref{varflow} having $X$ as infinitesimal generator, gives a variation $E_t=\Phi_t(E)$ of $E$ such that the function $ t \mapsto \vol(E_t)$ is linear, for $t$ in some interval $(-\delta,\delta)$. Since, by equation~\eqref{eqc1000bis}, there holds
$$
\frac{d}{dt} \vol(E_t)\,\Bigr\vert_{t=0}=\int_{\pa E} \langle X \vert \nu_{E}\rangle \, d\mu=\int_{\pa E} \varphi\, d\mu=0,
$$
such function $ t \mapsto \vol(E_t)$ must actually be constant.\\
Hence, $\vol (E_t)= \vol(E) $, for all $t\in(-\delta,\delta)$ and the variation $E_t$ is volume--preserving.
\end{proof}

The next lemma gives a technical estimate needed in the proof of Theorem~\ref{W2pMin}.

\begin{lem}\label{lemmastima}
Let $p>\max\{2, n-1\}$ and $E \subseteq \T^n$ a strictly stable critical set for the (volume--constrained) functional $J$. Then, in the hypotheses and notation of Lemma~\ref{lemma1}, there exist constants $\delta, C>0$ such that if $\norma{\psi}_{\Wduep (\pa E)} \leq \delta$ then $|X| \leq C |\langle X \vert \nu_{E_t}\rangle |$ on $\pa E_t$ and
\begin{equation}\label{gradtangX}
\norma{\nabla  X}_{L^2(\pa E_t)} \leq C \norma{ \langle X | \nu_{E_t} \rangle}_{H^1(\pa E_t)}
\end{equation}
(here $\nabla$ is the covariant derivative along $E_t$), for all $t \in [0,1]$, where $X \in C^{\infty}(\T^n ; \R^n)$ is the smooth vector field defined in formula~\eqref{field}.
\end{lem}
\begin{proof}
Fixed $\eps >0$, from inequality~\eqref{normflow} it follows that there exist $\delta >0$ such that if $\norma{\psi}_{\Wduep (\pa E)} \leq \delta$ there holds
$$
\abs{\nu_{E_t}(\Phi(t,y))-\nu_E(y)} \leq \eps
$$
for every $y\in\pa E$, hence, as $\nabla d_E=\nu_E$ on $\pa E$, we have
$$
\abs{\nabla d_E(\Phi^{-1}(t,x))  - \nu_{E_t}(x)}= \abs{\nu_E(\Phi^{-1}(t,x))  - \nu_{E_t}(x)} \leq \eps
$$
for every $x\in \pa E_t$. Then, if $\norma{\psi}_{\Wduep (\pa E)}$ is small enough, $\Phi^{-1}(t,\cdot)$ is close to the identity, thus
$$
\abs{\nabla d_E(\Phi^{-1}(t,x))  - \nabla d_E(x)} \leq \eps
$$
on $\pa E_t$ and we conclude
\begin{equation}\label{normanorm2}
\norma{\nabla d_E  - \nu_{E_t}}_{L^\infty (\pa E_t)} \leq 2 \eps \, .
\end{equation}
Moreover, using again the inequality~\eqref{normflow} and following the same argument above, we also obtain 
\begin{equation}\label{gradientnu}
\norma{\nabla^2 d_E  - \nabla \nu_{E_t}}_{L^\infty (\pa E_t)} \leq 2 \eps \, .
\end{equation}
We estimate $X_{\tau _t}=X-\langle X \vert \nu _{E_t} \rangle \nu_{E_t}$ (recall that $X=\langle X \vert \nabla d_E \rangle \nabla d_E$),
\begin{align}
\abs{X_{\tau _t}} & = \abs{X- \langle X \vert \nu_{E_t} \rangle \nu_{E_t}}\\
&= \abs{\langle X \vert \nabla d_E \rangle \nabla d_E -\langle X \vert \nu_{E_t} \rangle \nu_{E_t}} \\
& =  \abs{\langle X \vert \nabla d_E \rangle \nabla d_E -\langle X \vert \nu_{E_t} \rangle \nabla d_E + \langle X \vert \nu_{E_t} \rangle \nabla d_E -\langle X \vert \nu_{E_t} \rangle \nu_{E_t}} \\
&\leq \abs{ \langle X \vert (\nabla d_E - \nu_{E_t} ) \rangle \nabla d_E  }+ \abs{\langle  X \vert \nu_{E_t} \rangle  ( \nabla d_E -\nu_{E_t} )}\\
&\leq2 \abs{X}\,\abs{\nabla d_E - \nu_{E_t}}\\
& \leq 4 \eps \abs{X}\,,
\end{align}
then
$$
\abs{X_{\tau _t}} \leq 4 \eps \abs{X_{\tau _t} + \langle X \vert \nu _{E_t} \rangle \nu_{E_t}} \leq 4 \eps \abs{X_{\tau _t}} + \abs{\langle X \vert \nu _{E_t} \rangle}\,,
$$
hence,
\begin{equation}\label{Xtau}
\abs{X_{\tau _t}} \leq C \abs{ \langle X \vert \nu _{E_t} \rangle} \,.
\end{equation}
We now estimate the covariant derivative of $X_{\tau _t}$ along $\pa E_t$, that is,
\begin{align}
\abs{\nabla X_{\tau _t}} 
=&\, \abs{\nabla X- \nabla ( \langle X \vert \nu_{E_t} \rangle \nu_{E_t})} \\
=&\,  \abs{ \nabla  (\langle X \vert \nabla d_E \rangle \nabla d_E ) - \nabla  (\langle X \vert \nu_{E_t} \rangle \nu_{E_t}) } \\
=&\, |\nabla  (\langle X \vert \nabla d_E \rangle \nabla d_E) - \nabla  ( \langle X \vert \nu_{E_t} \rangle \nabla d_E )+\nabla  (\langle X \vert \nu_{E_t} \rangle \nabla d_E )- \nabla  (\langle X \vert \nu_{E_t} \rangle \nu_{E_t})|\\
\leq&\,  \abs{ \nabla  (\langle X \vert (\nabla d_E - \nu_{E_t} ) \rangle \nabla d_E ) }+\abs{\nabla  (\langle  X \vert \nu_{E_t} \rangle  ( \nabla d_E -\nu_{E_t} )) }\\
\leq&\,  C \eps \bigl [ \abs{ \nabla  X } + \abs{ \nabla  \langle X \vert \nu_{E_t} \rangle} \bigr] + C \abs{X} \bigl [ \abs{\nabla(\nabla d_E) } + \abs{ \nabla  \nu_{E_t} } \bigr] \\
\leq&\,  C \eps \bigl [ \abs{ \nabla  ( \langle X \vert \nu_{E_t} \rangle \nu_{E_t} + X_{\tau _t}) } + \abs{ \nabla  \langle X \vert \nu_{E_t} \rangle} \bigr ] + C \bigl(\abs{\langle X \vert \nu_{E_t} \rangle} + \abs{X_{\tau _t} }\bigr)\, \bigl [ \abs{ \nabla^2 d_E } + \abs{ \nabla  \nu_{E_t} } \bigr ] 
\end{align}
hence, using inequality~\eqref{Xtau} and arguing as above, there holds
\begin{equation}\label{gradtau}
\abs{\nabla X_{\tau _t}} \leq C \abs{ \nabla  \langle X \vert \nu_{E_t} \rangle } + C \abs{\langle X\vert \nu_{E_t} \rangle} \bigl [\abs{  \nabla^2 d_E } + \abs{ \nabla  \nu_{E_t} } \bigr ]  \, .
\end{equation}
Then, we get
\begin{align}
\norma{\nabla X_{\tau _t}}_{L^2 (\pa E_t)}^2\leq 
&\, C \norma{ \nabla  \langle X \vert \nu_{E_t} \rangle }_{L^2 (\pa E_t)}^2 + C \int_{\pa E_t} \abs{\langle X \vert \nu_{E_t}\rangle}^2 \bigl [\abs{  \nabla^2 d_E }  + \abs{ \nabla  \nu_{E_t} } \bigr ]^2 \, \dmu 
\\
\leq &\,   C \norma{\langle X \vert \nu_{E_t} \rangle }_{H^1 (\pa E_t)}^2+ C \norma{\langle X \vert \nu_{E_t}\rangle}^2_{L^{\frac{2p}{p-2}} (\pa E_t)} \bigl \Vert \abs{  \nabla^2 d_E }  + \abs{ \nabla  \nu_{E_t} } \bigr\Vert^2 _{L^p(\pa E_t)} 
\\
\leq &\, C \ \norma{\langle X \vert \nu_{E_t} \rangle}^2 _{H^1 (\pa E_t)}
\end{align}
where in the last inequality we used as usual Sobolev embeddings, as $p> \max\{2, n-1\}$ and the fact that $\norma{\nabla \nu_{E_t}}_{L^p(\pa E_t)}$ is bounded by the inequality~\eqref{gradientnu} (as $\norma{\nabla^2 d_E}_{L^p(\pa E_t)}$).\\ 
Considering the covariant derivative of $X=X_{\tau_t}+\langle X \vert \nu _{E_t} \rangle \nu_{E_t}$, by means of this estimate, the trivial one 
$$
\norma{ \nabla  \langle X \vert \nu_{E_t} \rangle } _{L^2 (\pa E_t)} \leq  \norma{ \langle X \vert \nu_{E_t} \rangle}  _{H^1(\pa E_t)}
$$
and inequality~\eqref{Xtau}, we obtain estimate~\eqref{gradtangX}.
\end{proof}

We now show that any smooth set $E$ sufficiently $W^{2,p}$--close to another smooth set $F$, can be ``translated'' by a vector $\eta\in\R^n$ such that $\pa E-\eta=\{y+\varphi(y)\nu_F(y) \, :\, y\in\partial F\}$, for a function $\varphi\in C^\infty(\pa F)$ having a suitable small ``projection'' on $T(\pa F)$ (see the definitions and the discussion after Remark~\ref{remcarlo}).

\begin{lem}\label{Lemma 3.8}
Let $p>n-1$ and $F\subseteq\T^n$ a smooth set with a tubular neighborhood $N_\eps$ as above, in formula~\eqref{tubdef}. For any $\tau>0$ there exist constants $\delta,C>0$  such that if another smooth set $E\subseteq \T^n$ satisfies $\vol(E\triangle F)<\delta$ and $\partial E=\{y+\psi(y)\nu_F(y) \,:\,y\in\partial F\}\subseteq N_\eps$ for a function $\psi\in C^\infty(\R)$ with $\norma{\psi}_{W^{2,p}(\partial F)}<\delta$, then there exist $\eta\in\R^n$ and $\varphi\in C^\infty(\pa F)$ with the following properties:
$$
\partial E-\eta=\{y+\varphi(y)\nu_F(y) \,:\,y\in\partial F\}\subseteq N_\eps\,,
$$
$$
\abs{\eta}\leq C\norma{\psi}_{W^{2,p}(\partial F)},\qquad \norma{\varphi}_{W^{2,p}(\partial F)}\leq C\norma{\psi}_{W^{2,p}(\partial F)}
$$
and 
$$
\Bigl \vert \int_{\partial F}\varphi\nu_F \, \dmu\, \Bigr \vert \leq\tau\norma{\varphi}_{L^2(\partial F)}\,.
$$ 
\end{lem}
\begin{proof}
We let $d_{F}$ to be the signed distance function from $\pa F$. We underline that, throughout the proof, the various constants will be all independent of $\psi:\pa F\to\R$.\\
We recall that in Remark~\ref{rembase} we saw that there exists an orthonormal basis $\{e_1, \dots, e_n\}$ of $\R^n$ such that the functions $\langle \nu_F \vert e_i \rangle$ are orthogonal in $L^2(\pa F)$, that is,
\begin{equation}\label{ort2}
\int_{\pa F} \langle \nu_{F} \vert  e_i \rangle \langle \nu_{F} \vert e_j \rangle \, \dmu=0,
\end{equation}
for all $i \ne j$ and we let $\II_F$ to be the set of the indices $i\in \{1,\dots, n\}$ such that 
$\norma{\langle \nu_F \vert e_i \rangle }_{L^2(\pa F)}>0$. Given a smooth function $\psi:\pa F\to \R$, we set $\eta=\sum_{i=1}^n\eta_ie_i$, where 
\begin{equation}\label{uno}
\eta_i  =
\begin{cases}
\frac{1}{\norma{ \langle \nu_F \vert e_i \rangle}^2_{L^2(\pa F)}}\int_{\pa F}\psi(x) \langle \nu_F(x) \vert e _i \rangle \,\dmu \quad & \text{if $i\in\II_F$,}\\
\eta_i=0\quad & \text{otherwise.}
\end{cases}
\end{equation}
Note that, from H\"older inequality, it follows
\begin{equation}\label{due}
\abs{\eta}\leq C_1\norma{\psi}_{L^2(\pa F)}\,.
\end{equation}

\smallskip

\noindent\textbf{Step $\mathbf 1$.} Let $T_\psi:\pa F\to\pa F$ be the map
$$
T_\psi(y)=\pi_{F}(y+\psi(y)\nu_F(y)-\eta)\,.
$$ 
It is easily checked that there exists $\eps_0>0$ such that if 
\begin{equation}\label{trans1}
\norma{\psi}_{W^{2,p}(\pa F)}+\abs{\eta}\leq \eps_0\leq 1\,,
\end{equation}
then $T_{\psi}$ is a smooth diffeomorphism, moreover, 
\begin{equation}\label{trans2}
\Vert JT_\psi-1\Vert_{L^\infty(\pa F)}\leq C\norma{\psi}_{C^1(\pa F)}
\end{equation}
(here $JT_{\psi}$ is the Jacobian relative to $\pa F$) and
\begin{equation}\label{trans3}
\norma{T_\psi-\Id}_{W^{2,p}(\pa F)}+\norma{T_\psi^{-1}-\Id}_{W^{2,p}(\pa F)}\leq C(\norma{\psi}_{W^{2,p}(\pa F)}+\abs{\eta})\,.
\end{equation}
Therefore, setting $\widehat E=E-\eta$, we have 
$$
\pa\widehat E=\{z+\varphi(z)\nu_F(z) \, :\,z\in\partial F\}\,,
$$
for some function $\varphi$ which is linked to $\psi$ by the following relation: for all $y\in\pa F$, we let $z=z(y)\in\pa F$ such that  
$$
y+\psi(y)\nu_F(y)-\eta=z+\varphi(z)\nu_F(z)\,,
$$
then,
$$
T_\psi(y)=\pi_F(y+\psi(y)\nu_F(y)-\eta)=\pi_F(z+\varphi(z)\nu_F(z))=z,
$$
that is, $y=T_\psi^{-1}(z)$ and 
\begin{align*}
\varphi(z)=&\,\varphi(T_\psi(y))\\
=&\,d_F(z+\varphi(z)\nu_F(z))\\
=&\,d_F(y+\psi(y)\nu_F(y)-\eta)\\
=&\,d_F(T_\psi^{-1}(z)+\psi(T_\psi^{-1}(z))\nu_F(T_\psi(y))-\eta).
\end{align*}
Thus, using inequality~\eqref{trans3}, we have
\begin{equation}\label{quattro}
\norma{\varphi}_{W^{2,p}(\pa F)}\leq C_2\bigl(\norma{\psi}_{W^{2,p}(\pa F)}+\abs{\eta}\bigr),
\end{equation}
for some constant $C_2>1$. We now estimate
\begin{align}
\int_{\pa F}\varphi(z)\nu_F(z)\,\dmu (z) &=\int_{\pa F}\varphi(T_{\psi}(y))\nu_F(T_{\psi}(y))JT_\psi(y)\,\dmu  (y) \\
&= \int_{\pa F}\varphi(T_{\psi}(y))\nu_F(T_{\psi}(y))\, \dmu (y) +R_1\,,\label{cinque}
\end{align}
where
\begin{equation}
\abs{R_1}=\biggl \vert \int_{\pa F}\varphi(T_{\psi}(y))\nu_F(T_{\psi}(y))\,[JT_\psi(y)-1]\, \dmu(y) \biggr \vert\leq C_3\norma{ \psi}_{C^1(\pa F)}\norma{\varphi}_{L^2(\pa F)}\,,\label{sei}
\end{equation}
by inequality~\eqref{trans2}.\\
On the other hand,
\begin{align}
\int_{\pa F}&\varphi(T_{\psi}(y))\nu_F(T_{\psi}(y))\, \dmu(y)\\
&=\int_{\pa F}\bigl[y+\psi(y)\nu_F(y)-\eta-T_{\psi}(y)\bigr]\,\dmu(y) \\
&=\int_{\pa F}\bigl[y+\psi(y)\nu_F(y)-\eta-\pi_{F}(y+\psi(y)\nu_F(y)-\eta)\bigr]\, \dmu(y)\\
&=\int_{\pa F}\bigl\{\psi(y)\nu_F(y)-\eta+\bigl[\pi_{F}(y)-\pi_{F}(y+\psi(y)\nu_F(y)-\eta)\bigr]\bigr\}\, \dmu(y)\\
&=\int_{\pa F}(\psi(y)\nu_F(y)-\eta)\, \dmu(y) +R_2\,, \label{sette}
\end{align}
where
\begin{align}
R_2&=\int_{\pa F}\bigl[\pi_{F}(y)-\pi_{F}(y+\psi(y)\nu_F(y)-\eta)\bigr]\, \dmu(y)\\
&=-\int_{\pa F} \dmu(y) \int_0^1\nabla \pi_{F}(y+t(\psi(y)\nu(y)-\eta))(\psi(y)\nu_F(y)-\eta)\,dt \\
&=-\int_{\pa F}\nabla \pi_{F}(y)(\psi(y)\nu_F(y)-\eta)\, \dmu(y)+R_3\,. \label{otto}
\end{align}
In turn, recalling inequality~\eqref{due}, we get
\begin{equation}
\abs{R_3}\leq\int_{\pa F} \dmu(y) \int_0^1\vert \nabla \pi_{F}(y+t(\psi(y)\nu_F(y)-\eta))-\nabla \pi_{F}(y)\vert\, \vert \psi(y)\nu_F(y)-\eta \vert \,dt\leq C_4\norma{\psi}^2_{L^2(\pa F)}\,.\label{nove}
\end{equation}
Since in $N_\eps$, by equation~\eqref{eqcar2050}, we have $\pi_{F}(x)=x-d_{F}(x)\nabla  d_{F}(x)$, it follows
$$
\frac{\pa {\pi_{F}^i}}{\pa x_j}(x)=\delta_{ij}-\frac{\pa d_{F}}{\pa x_i}(x)\frac{\pa d_{F}}{\pa x_j}(x)-d_{F}(x)\frac{\pa^2 d_{F}}{\pa x_i\pa x_j}(x),
$$
thus, for all  $y\in\pa F$, there holds
$$
\frac{\pa{\pi_{F}^i}}{\pa x_j}(y)=\delta_{ij}-\frac{\pa d_{F}}{\pa x_i}(y)\frac{\pa d_{F}}{\pa x_j}(y)\,.
$$
From this identity and equalities~\eqref{cinque},~\eqref{sette} and~\eqref{otto}, we conclude
$$
\int_{\pa F}\varphi(z)\nu_F(z)\,\dmu (z)=\int_{\pa F}\bigl[\psi(x)\nu_F(x)- \langle \eta \, \vert\, \nu_F(x)\rangle \nu_F(x)\bigr]\, \dmu (x)+R_1+R_3\,.
$$
As the integral at the right--hand side vanishes by relations~\eqref{ort2} and~\eqref{uno}, estimates~\eqref{sei} and~\eqref{nove} imply
\begin{align}
\Bigl \vert \int_{\pa F}\varphi(y)\nu_F(y)\, \dmu (y)\Bigr \vert &\leq C_3\norma{\psi}_{C^1(\pa F)}\norma{\varphi}_{L^2(\pa F)}+C_4\norma{\psi}^2_{L^2(\pa F)} \\ &\leq C\norma{\psi}_{C^1(\pa F)}\bigl(\norma{\varphi}_{L^2(\pa F)}+\norma{\psi}_{L^2(\pa F)}\bigr) \\
&\leq C_5 \norma{\psi}_{W^{2,p}(\pa F)}^{1-\vartheta}\norma{\psi}_{L^2(\pa F)}^{\vartheta}\bigl(\norma{\varphi}_{L^2(\pa F)}+\norma{\psi}_{L^2(\pa F)}\bigr) \,,\label{dieci}
\end{align}
where in the last passage we used a well--known interpolation inequality, with $\vartheta\in(0,1)$  depending only on $p>n-1$ (see~\cite[Theorem~3.70]{Aubin}).

\smallskip

\noindent \textbf{Step $\mathbf 2$.} The  previous estimate does not allow to conclude directly, but we have to rely on the following iteration procedure.
Fix any number $K>1$ and assume that  $\delta \in (0,1)$ is such that (possibly considering a smaller $\tau$) 
\begin{equation}\label{eta}
\tau+\delta<\eps_0/2,\qquad C_2\delta (1+2C_1)\leq\tau,\qquad2C_5\delta^\vartheta K\leq\tau\,.
\end{equation}
Given $\psi$, we set $\varphi_0=\psi$  and we denote by  $\eta^1$ the vector defined as in~\eqref{uno}.  We set $E_1=E-\eta^1$  and denote by $\varphi_1$  the function   such that $\pa E_1=\{x+\varphi_1(x)\nu_F(x) \,: \, x \in \pa F\}$.  As before, $\varphi_1$  satisfies 
$$
y+\varphi_0(y)\nu_F(y)-\eta^1=z+\varphi_1(z)\nu_F(z)\,.
$$
Since $\norma{\psi}_{W^{2,p}(\pa F)}\leq\delta$ and  $|\eta|\leq C_1\norma{\psi}_{L^2(\pa F)}$, by inequalities~\eqref{due},~\eqref{quattro} and~\eqref{eta} we have
\begin{equation}\label{undici}
\norma{\varphi_1}_{W^{2,p}(\pa F)}\leq C_2\delta(1+C_1)\leq\tau\,.
\end{equation}
Using again that  $\norma{\psi}_{W^{2,p}(\pa F)}<\delta<1$,  by estimate~\eqref{dieci} we obtain
$$
\Bigl \vert \int_{\pa F}\varphi_1(y)\nu_F(y)\, \dmu(y) )\Bigr \vert \leq C_5\norma{\varphi_0}_{L^2(\pa F)}^{\vartheta}\bigl(\norma{\varphi_1}_{L^2(\pa F)}+\norma{\varphi_0}_{L^2(\pa F)}\bigr)\,,
$$
where we have  $\norma{\varphi_0}_{L^2(\pa F)}\leq\delta$.\\
We now distinguish two cases.\\
If  $\norma{\varphi_0}_{L^2(\pa F)}\leq K\norma{\varphi_1}_{L^2(\pa F)}$,  from the previous inequality and~\eqref{eta}, we get 
\begin{align}
\Bigl \vert \int_{\pa F}\varphi_1(y)\nu_F(y)\, \dmu(y)\Bigr \vert  & \leq C_5\delta^{\vartheta}\bigl(\norma{\varphi_1}_{L^2(\pa F)}+\norma{\varphi_0}_{L^2(\pa F)}\bigr) \nonumber \\ & \, \leq 2C_5\delta^{\vartheta}K\norma{\varphi_1}_{L^2(\pa F)} \nonumber \\& \, \leq\delta\norma{\varphi_1}_{L^2(\pa F)}\,,
\end{align}
thus, the conclusion follows with $\eta=\eta^1$.\\
In the other case,
\begin{equation}\label{dodici}
\norma{\varphi_1}_{L^2(\pa F)}\leq\frac{\norma{\varphi_0}_{L^2(\pa F)}}{K}\leq\frac{\delta}{K}\leq \delta\,.
\end{equation}
We then repeat the whole procedure: we denote by $\eta^2$  the vector defined as in formula~\eqref{uno} with $\psi$ replaced by $\varphi_1$, we set $E_2=E_1-\eta^2=E-\eta^1-\eta^2$ and we consider the corresponding $\varphi_2$ which satisfies
$$
w+\varphi_2(w)\nu_F(w)=z+\varphi_1(z)\nu_F(z)-\eta^2=y+\varphi_0(y)\nu_F(y)-\eta^1-\eta^2\,.
$$
Since 
$$
\norma{\varphi_0}_{W^{2,p}(\pa F)}+\abs{\eta^1+\eta^2}\leq \delta +C_1\delta+C_1\norma{\varphi_1}_{L^2(\pa F)}\leq \delta+C_1\delta\Bigl(1+\frac1K\Bigr)\leq C_2\delta(1+2C_1)\leq\tau\,,
$$
the map $T_{\varphi_0}(y)=\pi_{F}(y+\varphi_0(y)\nu_F(y)-(\eta^1+\eta^2))$ is a diffeomorphism, thanks to formula~\eqref{trans1} (having chosen $\tau$ and $\delta$ small enough).\\
Thus, by applying inequalities~\eqref{quattro} (with $\eta=\eta^1+\eta^2$),~\eqref{due},~\eqref{eta} and~\eqref{dodici}, we get
$$
\norma{\varphi_2}_{W^{2,p}(\pa F)}\leq C_2\bigl(\norma{\varphi_0}_{W^{2,p}(\pa F)}+\abs{ \eta^1+\eta^2} \bigr)\leq C_2\delta\Bigl(1+C_1+\frac{C_1}{K}\Bigr)\leq\tau\,,
$$
as $K>1$, analogously to conclusion~\eqref{undici}. On the other hand, by estimates~\eqref{due},~\eqref{undici} and~\eqref{dodici}, 
$$
\norma{\varphi_1}_{W^{2,p}(\pa F)}+\eta^2\leq C_2\delta(1+C_1)+C_1\frac{\delta}{K}\leq 
C_2\delta(1+2C_1)\leq \tau\,,
$$
hence, also the map $T_{\varphi_1}(x)=\pi_{F}(x+\varphi_1(x)\nu_F(x)-\eta^2)$ is a diffeomorphism
satisfying inequalities~\eqref{trans1} and~\eqref{trans2}. Therefore, arguing as before, we obtain
$$
\Bigl \vert \int_{\pa F}\varphi_2(y)\nu_F(y)\,\dmu(y)\Bigr \vert \leq C_5\norma{\varphi_1}_{L^2(\pa F)}^{\vartheta}\bigl(\norma{\varphi_2}_{L^2(\pa F)}+\norma{\varphi_1}_{L^2(\pa F)}\bigr)\,.
$$
Since $\norma{\varphi_1}_{L^2(\pa F)}\leq\delta$ by inequality~\eqref{dodici}, if $\norma{\varphi_1}_{L^2(\pa F)}\leq K\norma{\varphi_2}_{L^2(\pa F)}$  the conclusion follows  with $\eta=\eta^1+\eta^2$. Otherwise, we iterate the procedure observing that 
$$
\norma{\varphi_2}_{L^2(\pa F)}\leq\frac{\norma{\varphi_1}_{L^2(\pa F)}}{K}\leq\frac{\norma{\varphi_0}_{L^2(\pa F)}}{K^2}\leq\frac{\delta}{K^2}\,.
$$
This construction leads to three (possibly finite) sequences $\eta^n$, $E_n$ and $\varphi_n$ such that
$$
\begin{cases}
E_n=E-\eta^1-\dots-\eta^n, \qquad \abs{\eta^n} \leq\frac{C_1\delta}{K^{n-1}}\\
\norma{\varphi_n}_{W^{2,p}(\pa F)}\leq C_2\bigl(\norma{\varphi_0}_{W^{2,p}(\pa F)}+\abs{\eta^1+\dots+\eta^n}\bigr)\leq C_2\delta(1+2C_1)\\
\norma{\varphi_n}_{L^2(\pa F)} \leq\frac{\delta}{K^n}\\
\partial E_n=\{x+\varphi_n(x)\nu_F(x) \, :\,x\in\pa F \}\
\end{cases}
$$
If for some $n\in\N$ we have $\norma{\varphi_{n-1}}_{L^2(\pa F)}\leq K\norma{\varphi_n}_{L^2(\pa F)}$,
the construction stops, since, arguing as before,  
$$
\Bigl \vert \int_{\pa F}\varphi_n(y)\nu_F(y)\,\dmu(y)\Bigr \vert \leq\delta\norma{\varphi_n}_{L^2(\pa F)}
$$
 and the conclusion follows with  $\eta=\eta^1+\dots+\eta^n$ and $\varphi=\varphi_n$. Otherwise, the  iteration continues indefinitely and we get the thesis with 
$$
\eta=\sum_{n=1}^\infty\eta^n, \qquad\varphi=0\,, 
$$
(notice that the series is converging), which actually means that  $E=\eta+F$.
\end{proof}

We are now ready to show the main theorem of this first part of the work.

\begin{proof}[Proof of Theorem~\ref{W2pMin}]\ \\
\noindent \textbf{Step $\mathbf{1}$.} We first want to see that 
\beq\label{m0}
m_0= \inf \left\{\Pi_E(\varphi) \, : \, \varphi \in T^\perp(\pa E), \norma{\varphi}_{H^1(\pa E)}=1 \right\} >0.
\eeq
To this aim, we consider a minimizing sequence $\varphi_i$ for the above infimum and we assume that $\varphi_i \rightharpoonup \varphi_0$ weakly in $H^1 (\pa E)$, then $\varphi_0\in T^\perp(\pa E)$ (since it is a closed subspace of $H^1 (\pa E)$) and if $\varphi_0 \ne 0$, there holds
$$m_0= \lim_{i \to +\infty} \Pi_E(\varphi_i) \ge \Pi_E(\varphi_0)>0$$ 
due to the strict stability of $E$ and the lower semicontinuity of $\Pi_E$ (recall formula~\eqref{Pieq} and the fact that the weak convergence in $H^1(\pa E)$ implies strong convergence in $L^2(\pa E)$ by Sobolev embeddings).
On the other hand, if instead $\varphi_0=0$, again by the strong convergence of $\varphi_i \to \varphi_0$ in $L^2(\pa E)$, by looking at formula~\eqref{Pieq}, we have 
$$
m_0= \lim_{i \to\infty} \Pi_E(\varphi_i)=\lim_{i\to\infty} \int_{\pa E} |\nabla \varphi_i|^2 \, \dmu = \lim_{i\to\infty}\Vert\varphi_i\Vert_{H^1(\pa E)}^2=1
$$
since $\Vert\varphi_i \Vert_{L^2(\pa E)}\to0$.

\smallskip

\noindent \textbf{Step ${\mathbf{2}}$.} Now we show that there exists a constant $\delta_1 >0$ such that if $E$ is like in the statement and $\pa F= \{y+ \psi(y)\nu_E(y) \, : \, y \in \pa E \}$, with $\norma{\psi}_{W^{2,p}(\pa E)} \le \delta_1$, and $\vol(F)=\vol(E)$, then
\begin{equation}\label{3.38}
\inf\left\{\Pi_F(\varphi) \, : \, \varphi \in \Htilde^1(\pa F), \norma{\varphi}_{H^1(\pa F)}=1, \Bigl | \int_{\pa F} \varphi \nu_F \, \dmu \Bigr| \le \delta_1\right\}\ge \frac{m_0}{2}.
\end{equation}
We argue by contradiction assuming that there exists a sequence of sets $F_i$ with $\pa F_i=\{y+ \psi_i(y) \nu_E(y) \, : \, y \in \pa E\}$ with $\norma{\psi_i}_{W^{2,p}(\pa E)} \to 0$ and $\vol(F_i)=\vol(E)$, and a sequence of functions $\varphi_i \in \Htilde^1(\pa F_i)$ with $\norma{\varphi_i}_{H^1(\pa F_i)}=1$ and $\int_{\pa F_i} \varphi_i \nu_{F_i} \, \dmu_i \to 0$, such that
\begin{equation}
\Pi_{F_i}(\varphi_i)<\frac{m_0}{2}.
\end{equation}
We then define the following sequence of smooth functions
 \beq
 \label{eqcar1030}\widetilde{\varphi}_i(y)= \varphi_i(y+ \psi_i(y)\nu_E(y)) - \fint_{\pa E} \varphi_i (y + \psi_i(y) \nu_E(y)) \, \dmu(y)
 \eeq
which clearly belong to  $\Htilde^1(\pa E)$. Setting $\theta_i(y)=y+ \psi_i(y)\nu_E(y)$, as $p>\max \{2, n-1\}$, by the Sobolev embeddings, $\theta_i\to\mathrm{Id}$ in $C^{1,\alpha}$ and $\nu_{F_i} \circ \theta_i \to \nu_E$ in $C^{0,\alpha}(\pa E)$, hence, the sequence $\widetilde{\varphi}_i$ is bounded in $H^1(\pa E)$ and if $\{e_k\}$ is the special orthonormal basis found in Remark~\ref{rembase}, we have $\langle \nu_{F_i}\circ\theta_i \vert  e_k\rangle \to \langle \nu_E \vert e_k\rangle$ uniformly for all $k\in\{1, \dots, n\}$. Thus, 
\beq
\int_{\pa E} \widetilde{\varphi}_i \langle \nu_E \vert \eps_i \rangle \, \dmu \to 0,
\end{equation}
as $i\to\infty$, indeed, 
$$
\int_{\pa E} \widetilde{\varphi}_i \langle \nu_E \vert  e_k \rangle \, \dmu-\int_{\pa E} \widetilde{\varphi}_i \langle \nu_{F_i}\circ\theta_i \vert  e_k \rangle \, \dmu\to 0
$$
and
$$
\int_{\pa E} \widetilde{\varphi}_i \langle \nu_{F_i}\circ\theta_i \vert  e_k \rangle \, \dmu=\int_{\pa F_i}\varphi_i \langle \nu_{F_i} \vert  e_k \rangle \,J\theta_i^{-1} \dmu_i \to 0,
$$
as the Jacobians (notice that $J\theta_i$ are Jacobians ``relative'' to the hypersurface $\pa E$) $J\theta_i^{-1}\to 1$ uniformly and we assumed $\int_{\pa F_i} \varphi_i \nu_{F_i} \, \dmu_i \to 0$.\\
Hence, using expression~\eqref{projection},  for the projection map $\pi$ on $T^\perp(\pa E)$, it follows 
\begin{equation}
\norma{\pi (\widetilde{\varphi}_i)-\widetilde{\varphi}_i}_{H^1(\pa E)} \to 0
\end{equation}
as $i\to\infty$ and 
\begin{equation}\label{norm pi}
\lim_{i\to\infty}\norma{\pi(\widetilde{\varphi}_i)}_{H^1(\pa E)}=\lim_{i\to\infty}\norma{\widetilde{\varphi}_i}_{H^1(\pa E)}=\lim_{i\to\infty}\norma{\varphi_i}_{H^1(\pa F_i)}=1,
\end{equation}
since $\norma{\varphi_i}_{W^{2,p}(\pa E)} \to 0$, thus $\norma{\varphi_i}_{C^{1,\alpha}(\pa E)} \to 0$, by looking at the definition of the functions $\widetilde{\varphi}_i$ in formula~\eqref{eqcar1030}.\\ 
Note now that the $W^{2,p}$-- convergence of $F_i$ to $E$ (the second fundamental form $B_{\pa F_i}$ of $\pa F_i$ is ``morally'' the Hessian of $\varphi_i$) implies
\begin{equation}\label{conv B}
\BBB_{\pa F_i} \circ \theta_i \to \BBB_{\pa E} \qquad \text{in $L^p(\pa E)$}\,,
\end{equation}
as $i\to\infty$, then, by Sobolev embeddings again (in particular $H^1(\pa E)\hookrightarrow L^q(\pa E)$ for any $q\in[1,2^*)$, with $2^*=2(n-1)/(n-3)$ which is larger than $2$) and the $W^{2,p}$--convergence of $F_i$ to $E$, we get
\begin{equation}\label{conv B2}
\int_{\pa F_i}|B_{\pa F_i}|^2 \varphi_i^2 \, \dmu_i - \int_{\pa E} |B_{\pa E}|^2 \widetilde{\varphi}_i ^2 \, \dmu \to 0\,.
\end{equation}
Standard elliptic estimates for the problem~\eqref{potential} (see~\cite{Ev}, for instance) imply the convergence of the potentials
\begin{equation}\label{proj2bis}
v_{F_i}\to v_E\ \ \text{in $C^{1,\beta}(\T^n)$ for all $\beta\in(0,1)$,}
\end{equation}
for $i\to\infty$, hence arguing as before, 
$$
\int_{\pa F_i}\pa_{\nu_{F_i}} v_{F_i}\varphi_i^2\,d\mu_i-\int_{\pa E}\pa_{\nu_E} v_E\widetilde{\varphi}_i^2\,d\mu\to 0\,.
$$
Setting, as in Remark~\ref{rm:potential},
$$
v_{E,\widetilde{\varphi}_i}(x)=\int_{\pa E} G(x, y)\widetilde{\varphi}_i(y)\, d\mu(y)=\int_{\pa E} G(x, y)\varphi_i(\theta_i(y))\, d\mu(y)-m_i\int_{\pa E} G(x, y)\, d\mu(y)\,,
$$
where $m_i=\fint_{\pa E} \varphi_i(y + \psi_i(y) \nu_E(y)) \, \dmu(y)\to 0$, as $i\to\infty$, and
$$
v_{F_i,\varphi_i}(x)=\int_{\pa F_i} G(x,z)\varphi_i(z)\, d\mu_i(z)=\int_{\pa E} G(x,\theta_i(y))\varphi_i(\theta_i(y))J\theta_i(y)\,d\mu(y)\,,
$$
it is easy to check (see~\cite[pages~537--538]{AcFuMo}, for details) that
$$
\int_{\T^n}|\nabla v_{F_i,\varphi_i}|^2\, dx-\int_{\T^n}|\nabla v_{E,\widetilde{\varphi}_i}|^2\, dx\to0\,.
$$
Finally, recalling expression~\eqref{Pieq2}, we conclude
\begin{equation}
\Pi_{F_i}(\varphi_i)-\Pi_E(\widetilde{\varphi}_i) \to 0\,,
\end{equation}
since we have
$$
\norma{\varphi_i}_{L^2(\pa F_i)}-\norma{\widetilde{\varphi}_i}_{L^2(\pa E)}\to 0\,,
$$
which easily follows again by looking at the definition of the functions $\widetilde{\varphi}_i$ in formula~\eqref{eqcar1030} and taking into account that $\norma{\varphi_i}_{C^{1,\alpha}(\pa E)} \to 0$, hence limits~\eqref{norm pi} imply
$$
\norma{\nabla \varphi_i}_{L^2(\pa F_i)}-\norma{\nabla\widetilde{\varphi}_i}_{L^2(\pa E)}\to 0\,.
$$
By the previous conclusion $\norma{\pi(\widetilde{\varphi}_i)-\widetilde{\varphi}_i}_{H^1(\pa E)}\to 0$ and Sobolev embeddings, it this then straightforward, arguing as above, to get also
$$
\Pi_E(\widetilde{\varphi}_i)-\Pi_E(\pi(\widetilde{\varphi}_i))\to 0,
$$
hence, 
\begin{equation}
\Pi_{F_i}(\varphi_i)-\Pi_E(\pi(\widetilde{\varphi}_i)) \to 0.
\end{equation}
Since we assumed that $\Pi_{F_i}(\varphi_i)<{m_0}/{2}$, we conclude that for $i\in\N$, large enough there holds
\begin{equation}
\Pi_E(\pi(\widetilde{\varphi}_i))\leq \frac{m_0}{2}<m_0,
\end{equation}
which is a contradiction to Step~$1$, as $\pi(\widetilde{\varphi}_i)\in T^\perp(\pa E)$.

\smallskip

\noindent \textbf{Step ${\mathbf{3}}$.} Let us now consider $F$ such that $\vol(F)= \vol(E)$, $\vol(F\triangle E)<\delta$ and
\begin{equation}
\pa F= \{y + \psi(y) \nu_E (y)\, : \, y \in \pa E \}\subseteq N_\eps,
\end{equation}
with $\norma{\psi}_{W^{2,p}(\pa E)} \le \delta $ where $\delta>0$ is smaller than $\delta_1$ given by Step~$2$.\\ 
Taking a possibly smaller $\delta>0$, we consider the field $X$ and the associated flow $\Phi$ found in Lemma~\ref{lemma1}. Hence, $\Div X=0$ in $N_\eps$ and $\Phi(1,y)= y+ \psi(y)\nu_E(y)$, for all $y \in \pa E$, that is, $\Phi(1,\pa E)=\pa F\subseteq N_\eps$ which implies $E_1=\Phi_1(E)=F$ and $\vol(E_1)=\vol(F)=\vol(E)$. Then the special variation $E_t=\Phi_t(E)$ is volume--preserving, for $t \in [-1,1]$ and the vector field $X$ is admissible, by the last part of such lemma.\\
By Lemma~\ref{Lemma 3.8}, choosing an even smaller $\delta>0$ if necessary, possibly replacing $F$ with a translate $F-\sigma$ for some $\eta\in \R^n$ if needed, we can assume that
\begin{equation}\label{assumption}
\left| \int_{\pa E} \psi \, \nu_E \, \dmu \right| \le \frac{\delta_1}{2} \norma{\psi}_{L^2(\pa E)}.
\end{equation}
We now claim that
\begin{equation}\label{claim1}
\left|\int_{\pa E} \langle X \vert \nu_{E_t} \rangle \nu_{E_t} \, \dmu_t \right| \le \delta_1 \norma{\langle X \vert \nu_{E_t} \rangle}_{L^2(\pa E_t)} \qquad \forall t \in [0,1].
\end{equation}
To this aim, we write
\begin{align}
\int_{\pa E} \langle X \vert \nu_{E_t} \rangle \nu_{E_t} \, \dmu_t&= \int_{\pa E} \langle X\circ \Phi_t\vert \nu_{E_t}\circ\Phi_t\rangle(\nu_{E_t}\circ\Phi_t)\, J\Phi_t \, \dmu\\
&= \int_{\pa E} \langle X\circ \Phi_t \vert \nu_{E} \rangle \nu_E \, \dmu + R_1\\
&= \int_{\pa E}\langle X(x) \vert \nu_E \rangle \nu_E \, \dmu + R_1+R_2\\
&= \int_{\pa E} \psi \nu_E \, \dmu + R_1+ R_2+ R_3
\end{align}
with appropriate $R_1, R_2$ and $R_3$ (see below). \\
By the definition of $X$ in formula~\eqref{field} (in the proof of Lemma~\ref{lemma1}),  the bounds $0<C_1\leq \xi\leq C_2$ and $\norma{J(\pi_E\circ\Phi_t)^{-1}}_{L^\infty(\pa E)}\leq C_3$ (by inequality~\eqref{normflow} and Sobolev embeddings, as $p>\max\{2,n-1\}$, we have $\norma{\Phi(t, \cdot) - \Id}_{C^{1,\alpha}(\pa E)} \leq C \norma{\psi }_{\Wduep(\pa E)}\leq C\delta$), the following inequality holds
\begin{align}
\int_{\pa E} |X(\Phi(t,x))| \, \dmu&=\, \int_{\pa E} \biggl| \int_0^{\psi(\pi_E (\Phi(t,x)))} \frac{\xi(\Phi(t,x)) \nabla  d_E(\Phi(t,x))}{\xi(\Phi(t,x)+ s \nu(\pi_E(\Phi(t,x))))}\, ds \biggr| \, \dmu \nonumber \\
&\,\le C \int_{\pa E} \left|\psi(\pi_E(\Phi(t,x))) \right|\, \dmu \nonumber \\
&\,=\int_{\pa E} |\psi(z)| J(\pi_E \circ \Phi_t)^{-1}(z)\, \dmu(z) \nonumber\\
&\le\, C \norma{\psi}_{L^2(\pa E)}.\label{inequality1}
\end{align}
for every $t\in[0,1]$.\\
We want now to prove that for every $\overline{\varepsilon}>0$, choosing a suitably small $\delta>0$ we have the estimate
\begin{equation}\label{eqcar1011}
|R_1|+|R_2|+|R_3| \le \eps \norma{\psi}_{L^2(\pa E)}.
\end{equation}
First,
\begin{align}
R_1&= \int_{\pa E} \langle X\circ\Phi_t\vert \nu_{E_t}\circ\Phi\rangle \nu_{E_t}\circ\Phi_t [J\Phi_t-1]\, \dmu \\
&\quad+\int_{\pa E} \langle X\circ\Phi_t \vert \nu_{E_t}\circ\Phi_t \rangle \nu_{E_t}\circ\Phi_t \, \dmu-\int_{\pa E} \langle X\circ\Phi_t,\nu_E \rangle \nu_E \, \dmu\\
&=\,\int_{\pa E} \langle X\circ\Phi_t\vert\nu_{E_t}\circ\Phi_t\rangle \nu_{E_t}\circ\Phi_t\, [J\Phi_t-1] \, \dmu+ \int_{\pa E} \langle X\circ\Phi_t\vert \nu_{E_t}\circ\Phi_t- \nu_E\rangle \nu_E\,\dmu \\
&\quad+ \int_{\pa E} \langle X\circ\Phi_t\vert\nu_{E_t}\circ\Phi_t\rangle(\nu_{E_t}\circ\Phi_t-\nu_E) \, \dmu\\
&\leq\,\int_{\pa E} |X\circ\Phi_t|\, \Vert J\Phi_t-1\Vert_{L^\infty(\pa E)} \, \dmu+ \int_{\pa E} |X\circ\Phi_t|\,\norma{\nu_E-\nu_{E_t}\circ\Phi_t}_{L^\infty (\pa E)}\,\dmu\,,
\end{align}
then, since by equality~\eqref{flowin}, it follow that for every $t\in[0,1]$ the two terms
\begin{equation}
\norma{\nu_E-\nu_{E_t}\circ\Phi(t,x)}_{L^\infty (\pa E)}\qquad \text{and}\qquad \norma{J\Phi_t-1}_{L^\infty(\pa E)}
\end{equation}
can be made (uniformly in $t\in[0,1]$) small as we want, if $\delta>0$ is small enough, by using inequality~\eqref{inequality1}, we obtain
\begin{equation}\label{R1}
|R_1|\le \overline{\eps} \norma{\psi}_{L^2(\pa E)}/3.
\end{equation}
Then we estimate, by means of inequality~\eqref{flowin} and where $s=s(t,y)\in[t,1]$,
\begin{align}
|R_2| &\le \int_{\pa E} |X(\Phi(t,x))-X(\Phi(1,x))| + |X(\Phi(1,x))-X(x)| \, \dmu \\ \nonumber
&\le \int_{\pa E} |X(\Phi(t,x))-X(\Phi(1,x))| + \norma{\nabla  X}_{L^2(N_\eps)} \norma{\psi}_{L^2(\pa E)}\\ \nonumber
&=\int_{\pa E} (1-t)|\nabla X(\Phi_s(y))| \left|\frac{\pa \Phi_s}{\pa t}(y)\right| \, \dmu(y) + \norma{\nabla X}_{L^2(N_\varepsilon)}\norma{\psi}_{L^2(\pa E)}\\
& \le \int_{\pa E} |\nabla  X(\Phi(s,x))| |\Phi(t,x)-\Phi(1,x)| + \norma{\nabla  X}_{L^2(N_\eps)} \norma{\psi}_{L^2(\pa E)}\\ \nonumber
& \le C \norma{\nabla  X}_{L^\infty (N_\eps)} C \norma{\psi}_{L^2(\pa E)} + \norma{\nabla  X}_{L^2(N_\eps)}\norma{\psi}_{L^2(\pa E)},
\end{align}
where in the last inequality we use equation~\eqref{inequality1}. Hence, using equality~\eqref{normfield} and Sobolev embeddings, as $p>\max\{2,n-1\}$, we get
$$
|R_2|  \le C \norma{\psi}_{\Wduep(\pa E)}\norma{\psi}_{L^2(\pa E)},
$$
then, since $\norma{\psi}_{W^{2,p}(\pa E)} < \delta$, we obtain
\begin{equation}
|R_2|< \overline{\varepsilon} \norma{\psi}_{L^2(\pa E)}/3,
\end{equation}
if $\delta_2$ is small enough.\\
Arguing similarly, recalling the definition of $X$ given by formula~\eqref{field}, we also obtain $|R_3| \le \overline{\varepsilon} \norma{\psi}_{L^2(\pa E)}$, hence estimate~\eqref{eqcar1011} follows. We can then conclude that, for $\delta>0$ small enough, we have
\begin{equation}
\left| \int_{\pa E} \langle X \vert \nu_{E_t} \rangle \nu_{E_t} \, \dmu_t \right|\le \left|\int_{\pa E} \psi \nu_E \, \dmu \right| + \overline{\eps} \norma{\psi}_{L^2(\pa E)}\le \Bigl(\frac{\delta_1}{2} + \overline{\eps}\Bigr) \norma{\psi}_{L^2(\pa E)}
\end{equation}
for any $t\in[0,1]$, where in the last inequality we used the assumption~\eqref{assumption}, thus choosing $\overline{\varepsilon}=\delta_1/4$ we get
$$
\biggl| \int_{\pa E} \langle X \vert \nu_{E_t} \rangle \nu_{E_t} \, \dmu_t \biggr| \le \frac{3\delta_1}{4}\norma{\psi}_{L^2(\pa E)}.
$$
Along the same line, it is then easy to prove that
\begin{equation}\label{3.46}
\norma{\langle X \vert \nu_{E_t}\rangle}_{L^2(\pa E_t)} \ge (1-\eps)\norma{\psi}_{L^2(\pa E)},
\end{equation}
for any $t\in[0,1]$, hence claim~\eqref{claim1} follows.\\
As a consequence, since $\langle X \vert \nu_{E_t}\rangle\in\Htilde^1(\pa E_t)$, being $X$ admissible for $E_t$ (recalling computation~\ref{eqc1000}) and $\pa E_t$ can be described as a graph over $\pa E$ with a function with small norm in $W^{2,p}(\pa E)$ (by estimate~\eqref{normflow} of Lemma~\ref{lemma1}), we can apply Step~$2$ with $F=E_t$ to the function $\langle X \vert \nu_{E_t}\rangle/\Vert\langle X \vert \nu_{E_t}\rangle\Vert_{H^1(\pa E_t)}$, concluding 
\begin{equation}\label{eqcar1020}
\Pi_{E_t}(\langle X \vert \nu_{E_t}\rangle)\geq\frac{m_0}{2}\Vert\langle X \vert \nu_{E_t}\rangle\Vert_{H^1(\pa E_t)}.
\end{equation}

By means of Lemma~\ref{lemmastima}, for $\delta>0$ small enough, we now show the following inequality on $\pa E_t$ (here $\Div$ is the divergence operator and $X_{\tau _t}=X-\langle X \vert \nu _{E_t} \rangle \nu_{E_t}$ is a tangent vector field on $\pa E_t$), for any $t\in[0,1]$,
\begin{align}
\norma{\Div (X_{\tau_t} \langle X \vert \nu_{E_t} \rangle )}_{L^{\frac{p}{p-1}}(\pa E_t)} 
=&\, \norma{\Div X_{\tau_t} \langle X \vert \nu_{E_t}\rangle + \langle X_{\tau_t}\vert \nabla\langle X  \vert  \nu_{E_t}\rangle\rangle }_{L^{\frac{p}{p-1}}(\pa E_t)} \\
\le&\, C\norma{\nabla  X_{\tau_t}}_{L^2(\pa E_t)} \norma{\langle X \vert \nu_{E_t} \rangle}_{L^{\frac{2p}{p-2}}(\pa E_t)}\\
&\,+ C\norma{X_{\tau_t}}_{L^{\frac{2p}{p-2}}(\pa E_t)}\norma{\nabla  \langle X \vert \nu_{E_t}\rangle}_{L^2(\pa E_t)}\\
\le&\, C\norma{X}_{H^1(\pa E_t)}\norma{X}_{L^{\frac{2p}{p-2}}(\pa E_t)}\\
\le&\, C\norma{X}^2_{H^1(\pa E_t)}\\
\le&\, C\norma{\langle X \vert \nu_{E_t} \rangle}^2_{H^1(\pa E_t)},\label{claim2}
\end{align}
where we used the Sobolev embedding $H^1(\pa E_t)\hookrightarrow L^{\frac{2p}{p-2}}(\pa E_t)$, as $p>\max\{2,n-1\}$.\\

Then, we compute (here $X_{\tau_t}$ is the tangent component of $X$, $\HHH_t$ is the mean curvature and $v_{E_t}$ the potential relative to $E_t$ defined by formula~\eqref{potential1})
\begin{align}
J(F)-J(E)=&\,J(E_1)-J(E)\\
=&\,\int_0^1 (1-t)\frac{d^2}{dt^2}J(E_t)\, dt \\
=&\,\int_0^1(1-t)\bigl(\Pi_{E_t}(\langle X \vert \nu_{E_t}\rangle)+R_t\bigr)\,dt \\
=&\,\int_0^1(1-t)\Pi_{E_t}(\langle X \vert \nu_{E_t}\rangle)\,dt \\
&\,-\int_0^1(1-t)\int_{\pa E}(4\gamma v_{E_t} + \HHH_t) \Div(X_{\tau_t}\langle X \vert \nu_{E_t} \rangle)\, \dmu_t\, dt.
\end{align}
by Theorem~\ref{secondvar} and the definition of $\Pi_{E_t}$ in formula~\eqref{Pieq}, considering the second form of the remainder term $R_t$, relative to $E_t$ and taking into account that $\Div X=0$ in $N_\eps$ and that $X_t=X$, as the variation is special.\\
Hence, by estimate~\eqref{eqcar1020}, we have (recall that $4\gamma v_E+\HHH=4\gamma v_{E_0} +\HHH_0=\lambda$ constant, as $E$ is a critical set)
\begin{align}
J(F)-J(E) \ge&\, \frac{m_0}{2} \int_0^1 (1-t) \norma{\langle X \vert \nu_{E_t} \rangle}^2_{H^1(\pa E_t)} \, dt\\
&\,- \int_0^1 (1-t) \int_{\pa E_t}(\HHH_t+4\gamma v_{E_t})\, \Div(X_{\tau_t}\langle X \vert \nu_{E_t} \rangle)\, \dmu_t \, dt\\
=&\,\frac{m_0}{2} \int_0^1 (1-t) \norma{\langle X \vert \nu_{E_t} \rangle}^2_{H^1(\pa E_t)} \, dt\\
&\,- \int_0^1 (1-t) \int_{\pa E_t}(\HHH_t + 4 \gamma v_{E_t} -\lambda) \Div(X_{\tau_t}\langle X \vert \nu_{E_t} \rangle)\, \dmu_t\, dt\\
\geq&\,\frac{m_0}{2} \int_0^1 (1-t) \norma{\langle X \vert \nu_{E_t} \rangle}^2_{H^1(\pa E_t)} \, dt\\
&\,- \int_0^1 (1-t) \norma{\HHH_t + 4 \gamma v_{E_t} - \lambda}_{L^p(\pa E_t)} \norma{\Div(X_{\tau_t}\langle X \vert \nu_{E_t} \rangle)}_{L^\frac{p}{p-1}(\pa E_t)}\,dt\\
\geq&\,\frac{m_0}{2} \int_0^1 (1-t) \norma{\langle X \vert \nu_{E_t} \rangle}^2_{H^1(\pa E_t)} \, dt\\
&\,- C\int_0^1 (1-t) \norma{\HHH_t+ 4 \gamma v_{E_t} - \lambda}_{L^p(\pa E_t)}\norma{\langle X \vert \nu_{E_t} \rangle}^2_{H^1(\pa E_t)}\,dt, 
\end{align}
by estimate~\eqref{claim2}. 
If $\delta>0$ is sufficiently small, as $E_t$ is $W^{2,p}$--close to $E$ (recall the definition of $v_{E_t}$ in formula~\eqref{potential1}), we have 
$$
\norma{\HHH_t + 4 \gamma v_{E_t} - \lambda}_{L^p(\pa E_t)}<m_0/4C\,,
$$
hence,
$$
J (F)-J(E) \ge \frac{m_0}{4} \int_0^1 (1-t) \norma{\langle X \vert \nu_{E_t} \rangle}^2_{H^1(\pa E_t)} \, dt.
$$
Then, we can conclude the proof of the theorem with the following series of inequalities, holding for a suitably small $\delta>0$ as in the statement,
\begin{align}
J(F)&\ge J(E) + \frac{m_0}{2} \int_0^1 (1-t) \norma{\langle X \vert \nu_{E_t} \rangle}^2_{H^1(\pa E_t)} \, dt\\
& \ge J(E) + C \norma{\langle X \vert \nu_E\rangle}^2_{L^2(\pa E)}\\
& \ge J(E) + C \norma{\psi}^2_{L^2(\pa E)}\\
& \ge J(E) + C[\vol(E \triangle F)]^2\\
& \ge J(E)+ C[\alpha(E,F)]^2,
\end{align}
where the first inequality is due to the $W^{2,p}$--closedness of $E_t$ to $E$, the second one by the very expression~\eqref{field} of the vector field $X$ on $\pa E$,
\begin{equation}
|\langle X(y)\vert\nu_E(y)\rangle|= \Bigl|\int _0 ^{\psi(y)}\frac{ds}{\xi(y+ s \nu_E(y))}\,\Bigr|\leq C|\psi(y)|,
\end{equation}
the third follows by a straightforward computation (involving the map $L$ defined by formula~\eqref{eqcar410} and its Jacobian), as $\pa E$ is a ``normal graph'' over $\pa F$ with $\psi$ as ``height function'', finally the last one simply by the definition of the ``distance'' $\alpha$, recalling that we possibly translated the ``original'' set $F$ by a vector $\eta\in\R^n$, at the beginning of this step.
\end{proof}

We conclude this section by proving two propositions that will be used later. The first one says that when a set is sufficiently $W^{2,p}$--close to a strictly stable critical set of the functional $J$, then the quadratic form~\eqref{Pieq} remains uniformly positive definite (on the orthogonal complement of its degenerate subspace, see the discussion at the end of the previous subsection). 

\begin{proposition}\label{2.6}
Let $p>\max \{2,n-1\}$ and $E\subseteq\T^n$ be a smooth strictly stable critical set with $N_\eps$ a tubular neighborhood of $\pa E$, as in formula~\eqref{tubdef}. Then, for every $\theta\in (0,1]$ there exist $\sigma_\theta,\delta>0$ such that if a smooth set $F \subseteq \T^n$ is $W^{2,p}$--close to $E$, that is, $\vol(F\triangle E)<\delta$ and $\pa F\subseteq N_\eps$ with
\begin{equation}
\pa F= \{y + \psi (y) \nu_E(y) \, : \, y \in \pa E\}
\end{equation}
for a smooth $\psi$ with $\norma{\psi}_{W^{2,p}(\pa E)}<\delta$, there holds
\begin{equation}\label{2.13}
\Pi_F(\varphi) \geq \sigma_\theta \norma{\varphi}^2_{H^1(\pa F)},
\end{equation}
for all $\varphi \in \Htilde^1(\pa F)$ satisfying
\begin{equation}
\min_{\eta\in\OO_E} \norma{\varphi - \langle \eta \vert \nu_F \rangle }_{L^2(\pa F)} \geq \theta \norma{\varphi}_{L^2(\pa F)},
\end{equation}
where $\OO_E$ is defined by formula~\eqref{OOeq}.
\end{proposition}
\begin{proof}\ \\
\noindent \textbf{Step ${\mathbf{1}}$.} We first show that for every  $\theta\in (0, 1]$ there holds
\begin{equation}
m_\theta=\inf\Bigl\{\Pi_E(\varphi)\,:\, \varphi\in \Htilde^1(\pa E)\,, \|\varphi\|_{H^1(\pa E)}=1\,\,\text{ and }\min_{\eta\in \OO_E}\|\varphi-\langle\eta\vert\nu_E\rangle\|_{L^2(\pa E)}\geq \theta\|\varphi\|_{L^2(\pa E)}\Bigr\}>0\,.\label{c0}
\end{equation}
Indeed, let $\varphi_i$ be a minimizing sequence for this infimum and assume that $\varphi_i\rightharpoonup \varphi_0\in \Htilde^1(\pa E)$ weakly in $H^1(\pa E)$.\\
If $\varphi_0\neq 0$, as the weak convergence in $H^1(\pa E)$ implies strong convergence in $L^2(\pa E)$ by Sobolev embeddings, for every $\eta\in\OO_E$ we have
$$
\|\varphi_0-\langle\eta\vert\nu_E\rangle\|_{L^2(\pa E)}
=\lim_{i\to\infty}\|\varphi_i-\langle\eta\vert\nu_E\rangle\|_{L^2(\pa E)}\geq \lim_{i\to\infty}\theta\|\varphi_i\|_{L^2(\pa E)}
=\theta\|\varphi_0\|_{L^2(\pa E)},
$$
hence,
$$
\min_{\eta\in\OO_E}\|\varphi_0-\langle\eta\vert\nu_E\rangle\|_{L^2(\pa E)}\geq \theta\|\varphi_0\|_{L^2(\pa E)}>0,
$$
thus, we conclude  $\varphi_0\in \Htilde^1(\pa E)\setminus T(\pa E)$ and
$$
m_\theta=\lim_{i\to\infty} \Pi_E(\varphi_i)\geq \Pi_E(\varphi_0)>0\,,
$$
where the last inequality follows from estimate~\eqref{uusi stability} in Remark~\ref{rembase0}.\\
If $\varphi_0= 0$, then again by the strong convergence of $\varphi_i\to\varphi_0$ in $L^2(\pa E)$, by looking at formula~\eqref{Pieq}, we have
$$
m_\theta=\lim_{i\to\infty}\Pi_E(\varphi_i)=\lim_{i\to\infty}\int_{\pa E}|\nabla \varphi_i|^2 \, d\mu=\lim_{i\to\infty}\Vert\varphi_i\Vert_{H^1(\pa E)}^2=1
$$
since $\Vert\varphi_i\Vert_{L^2(\pa E)}\to0$.
 
\smallskip

\noindent \textbf{Step ${\mathbf{2}}$.}  In order to finish the proof it is enough to show the existence of some $\delta>0$ 
such that if $\vol(F\triangle E)<\delta$ and $\pa F=\bigl\{y+\psi(y) \nu_E(y):\, y\in \pa E\bigr\}$ with  $\|\psi\|_{W^{2,p}(\pa E)}<\delta$, then
\begin{align}
\inf\Bigl\{&\,\Pi_F(\varphi)\,:\, \varphi\in \Htilde^1(\pa F)\,, \|\varphi\|_{H^1(\pa F)}=1\,\,\,\text{ and }\min_{\eta\in \OO_E}\|\varphi-\langle\eta\vert\nu_F\rangle\|_{L^2(\pa F)}\geq\theta\|\varphi\|_{L^2(\pa F)}\Bigr\}\\
&\,\geq\sigma_\theta=\frac{1}{2}\min \{m_{\theta/2},1\}\,,\label{c2unif}
\end{align}
where $m_{\theta/2}$ is defined by formula~\eqref{c0}, with $\theta/2$ in place of $\theta$.\\
Assume by contradiction that there exist a sequence of smooth sets $F_i\subseteq\T^n$, with $\pa F_i=\{y+\psi_i(y) \nu_E
(y):\, y\in \pa E\}$ and $\|\psi_i\|_{W^{2,p}(\pa E)}\to 0$, and a sequence  $\varphi_i\in \Htilde^1(\pa F_i)$, with $\|\varphi_i\|_{H^1(\pa F_i)}=1$ and $\min_{\eta\in\OO_E}\|\varphi_i-\langle\eta\vert\nu_{F_i}\rangle\|_{L^2(\pa F_i)}\geq \theta\|\varphi_i\|_{L^2(\pa F_i)}$, such that 
\begin{equation}\label{liminf}
\Pi_{F_i}(\varphi_i)<\sigma_\theta\leq  m_{\theta/2}/2\,.
\end{equation}
Let us suppose first that $\lim_{i\to\infty}\|\varphi_i\|_{L^2(\pa F_i)}=0$ and observe that by  Sobolev embeddings  $\|\varphi_i\|_{L^q(\pa F_i)}\to0$ for some $q>2$, thus, since the functions $\psi_i$ are uniformly bounded in $W^{2,p}(\pa E)$ for $p>\max\{2,n-1\}$, recalling formula~\eqref{Pieq}, it is easy to see that
$$
\lim_{i\to\infty}\Pi_{F_i}(\varphi_i)=\lim_{i\to\infty}\int_{\pa F_i}|\nabla \varphi_i|^2 \, d\mu_i=\lim_{i\to\infty}\Vert\varphi_i\Vert_{H^1(\pa F_i)}^2=1\,,
$$
which is a contradiction with assumption~\eqref{liminf}.\\
Hence, we may assume that 
\begin{equation}\label{lim positive}
\lim_{i\to\infty}\|\varphi_i\|_{L^2(\pa F_i)} >0. 
\end{equation}
The idea now is to write every $\varphi_i$ as a function on $\pa E$. We define the functions $\widetilde{\varphi}_i(\pa E)\to\R$, given by 
$$
\widetilde\varphi_i(y)=\varphi_i\bigl(y+\psi_i(y) \nu_E(y)\bigr)- \fint_{\pa E}\varphi_i(y+\psi_i(y) \nu_E(y))\, d\mu(y)\,,
$$
for every $y\in \pa E$.\\
As $\psi_i\to 0$ in $W^{2,p}(\pa E)$, we have in particular that 
\begin{equation}\label{sothat}
\widetilde\varphi_i\in \Htilde^1(\pa E)\,, \qquad \|\widetilde\varphi_i\|_{H^1(\pa E)}\to 1\, \qquad\text{and}
\qquad \frac{\|\widetilde\varphi_i\|_{L^2(\pa E)}}{\| \varphi_i\|_{L^2(\pa F_i)}}\to 1\,,
\end{equation}
moreover, note also that $\nu_{F_i}(\cdot +\psi_i(\cdot)\nu_E(\cdot))\to \nu_E$ in $W^{1,p}(\pa E)$ and thus in $C^{0,\alpha}(\pa E)$ for a suitable $\alpha\in (0,1)$, depending on $p$, by Sobolev embeddings. Using this fact and taking into account the third limit above and inequality~\eqref{lim positive}, one can easily show that 
$$
\liminf_{i\to\infty}\frac{\min_{\eta\in\OO_E}\|\widetilde\varphi_i-\langle\eta\vert\nu_{E}\rangle\|_{L^2(\pa E)}}{\|\widetilde\varphi_i\|_{L^2(\pa E)}}\geq 
\liminf_{i\to\infty}\frac{\min_{\eta\in\OO_E}\|\varphi_i-\langle\eta\vert\nu_{F_i}\rangle\|_{L^2(\pa F_i)}}{\|\varphi_i\|_{L^2(\pa E_i)}}\geq \theta\,.
$$
Hence, for $i\in\N$ large enough, we have 
$$
\|\widetilde\varphi_i\|_{H^1(\pa E)}\geq 3/4\qquad\text{and}\qquad \min_{\eta\in\OO_E}\|\widetilde\varphi_i-\langle\eta\vert\nu_{E}\rangle\|_{L^2(\pa E)}\geq \frac{\theta}2\|\widetilde\varphi_i\|_{L^2(\pa E)}\,,
$$
then, in turn, by Step~$1$, we infer
\begin{equation}\label{bystep1}
\Pi_E(\widetilde\varphi_i)\geq \frac{9}{16}m_{\theta/2}\,.
\end{equation}
Arguing now exactly like in the final part of Step~$2$ in the proof of Theorem~\ref{W2pMin}, we have that all the terms of $\Pi_{F_i}(\varphi_i)$ are asymptotically close to the corresponding terms of $\Pi_{E}(\widetilde\varphi_i)$, thus
$$
\Pi_{F_i}(\varphi_i)-\Pi_{E}(\widetilde\varphi_i)\to 0\,,
$$
which is a contradiction, by inequalities~\eqref{liminf} and~\eqref{bystep1}. This establishes inequality~\eqref{c2unif} and concludes the proof.
\end{proof}

The following final result of this section states the fact that close to a strictly stable critical set there are no other smooth critical sets (up to translations).

\begin{proposition}\label{prop:nocrit}
Let $p$ and $E\subseteq\T^n$ be as in Proposition~\ref{2.6}. Then, there exists $\delta>0$ such that if $E'\subseteq\T^n$ is a smooth critical set with $\vol(E')=\vol(E)$, $\vol(E\triangle E')<\delta$, $\pa E'\subseteq N_\eps$ and
\begin{equation}
\pa E'= \{y + \psi (y) \nu_E(y) \, : \, y \in \pa E\}
\end{equation}
for a smooth $\psi$ with $\norma{\psi}_{W^{2,p}(\pa E)}<\delta$, then $E'$ is a translate of $E$.
\end{proposition}
\begin{proof}
In Step~$3$ of the proof of Theorem~\ref{W2pMin}, it is shown that under these hypotheses on $E$ and $E'$, if $\delta>0$ is small enough, we may find a small vector $\eta\in \R^n$  and a volume--preserving variation $E_t$ such that$E_0=E$, $E_1=E'-\eta$ and 
$$
\frac{d^2 }{dt^2}J(E_t)\geq C[\vol(E \triangle (E'-\eta))]^2\,,
$$
for all $t\in [0,1]$, where $C$ is a positive constant independent of $E'$.\\
Assume that $E'$ is a smooth critical set as in the statement, which is not a translate of $E$, then $\frac{d }{dt}J(E_t)\bigl|_{t=0}=0$, but from the above formula it follows $\frac{d}{dt}J(E_t)\bigl|_{t=1}>0$, which implies that $E'-\eta$ cannot be critical, hence neither $E'$, which is a contradiction. Indeed, $s\mapsto E_{1-s}$ is a volume--preserving variation for $E'-\eta$ and
$$
\frac{d}{ds}J(E_{1-s})\Bigl|_{s=0}=-\frac{d}{dt}J(E_t)\Bigl|_{t=1}<0\,,
$$
showing that $E'-\eta$ is not critical.
\end{proof}

\section{The modified Mullins--Sekerka and the surface diffusion flow}\label{msfsdf}

We start with the notion of smooth flow of sets. 

\begin{definition}\label{def:smoothflow}
Let $E_t\subseteq \T^n$ for $t\in [0,T)$ be a one-parameter family of sets, then we say that it is a {\em smooth flow}
if there exists a smooth {\em reference set} $F\subseteq\T^n$ and a map $\Psi\in C^{\infty}([0,T)\times\T^n;\T^n)$ such that $\Psi_t=\Psi(t,\cdot)$ is a smooth diffeomorphism from $\T^n$ to $\T^n$ and $E_t=\Psi_t(F)$, for all $t\in [0, T)$.
\end{definition}

The \emph{velocity} of the motion of any point $x =\Psi_t(y)$ of the set $E_t$, with $y \in F$, is then given by
\beq
X_t(x)=X_t(\Psi_t(y))=\frac{\pa \Psi_t}{\pa t} (y).
\eeq
(notice that, in general, the smooth vector field $X_t$, defined in the whole $\T^n$ by $X_t(\Psi_t(z))=\frac{\pa \Psi_t}{\pa t} (z)$ for every $z\in\T^n$, is not independent of $t$).\\
When $x \in \pa E_t$, we define the \emph{outer normal velocity} of the flow of the boundaries $\partial E_t$, which are smooth hypersurfaces of $\T^n$, as
\beq
V_t(x)=\langle X_t(x)\vert \nu_{E_t}(x)\rangle,
\eeq
for every $t \in [0,T)$, where $\nu_{E_t}$ is the outer normal vector to $E_t$.

\medskip

{\em For more clarity and to simplify formulas and computations, from now on we will denote with 
$$
\int_{\pa E_t} f\,d\mu_t\qquad\qquad\text{ the integral }\qquad\qquad\int_{\pa F} f\circ \Phi_t\,d\mu_t\,,
$$
for every $f:\pa E_t\to\R$, where in the second integral $\mu_t$ is the canonical Riemannian measure induced on the hypersurface $\pa E_t$, parametrized by $\Phi_t\vert_{\pa F}$, by the flat metric of $\T^n$ (coinciding with the Hausdorff $(n-1)$--dimensional measure). Moreover, in the same spirit we set $\nu_t=\nu_{E_t}$.}

\medskip

Before giving the definition of the {\em modified Mullins--Sekerka flow} (first appeared in~\cite{MS} -- see also~\cite{Crank,Gurtin1} and~\cite{EscherSi4} for a very clear and nice introduction to such flow), we need some notation. Given a smooth set $E\subseteq\T^n$ and $\gamma\geq0$, we denote by $w_E$ the unique solution in $H^1(\T^n)$ of the following problem
\begin{equation}\label{WE}
\begin{cases}
\Delta w_E=0 & \text{in }\T^n\setminus \pa E\\
w_E= \HHH + 4\gamma v_{E} & \text{on } \, \partial E,
\end{cases}
\end{equation}
where $v_E$ is the potential introduced in~\eqref{potential} and $\HHH$ is the mean curvature of $\pa E$. Moreover, we denote by $w^+_E$ and $w^-_E$ the restrictions $w_E|_{E^c}$ and ${w_E}{|_{E}}$, respectively. Finally, denoting as usual by $\nu_E$ the outer unit normal to $E$, we set
$$
[\pa_{\nu_E} w_E]=\pa_{\nu_E}w^+_E-\pa_{\nu_E}w^-_E=-(\pa_{\nu_{E^c}}w^+_E+\pa_{\nu_E}w^-_E)\,.
$$
that is the ``jump'' of the normal derivative of $w_E$ on $\pa E$.

\begin{definition}\label{MSF def}
Let $E\subseteq \T^n$ be a smooth set. We say that a smooth flow $E_t$ such that $E_0=E$, is a {\em modified Mullins--Sekerka flow with parameter $\gamma\geq0$}, on the time interval $[0, T)$ and with initial datum $E$, if the outer normal velocity $V_t$ of the moving boundaries $\pa E_t$ is given by
\begin{equation}\label{MSnl}
V_t= [\partial_{\nu_t} w_{t}] \quad\text{ on $\partial E_t$ for all $t\in [0, T)$,}
\end{equation}
where $w_t=w_{E_t}$ (with the above definitions) and we used the simplified notation 
$\partial_{\nu_t} w_{t}$ in place of $\partial_{\nu_{E_t}} w_{E_t}$. 
\end{definition}

\begin{remark}
The adjective ``modified'' comes from the introduction of the parameter $\gamma \ge 0$ in the problem, while considering $\gamma=0$ we have the original flow proposed by Mullins and Sekerka in~\cite{MS} (see also~\cite{Crank,Gurtin1}), which has been also called {\em Hele--Shaw model}~\cite{XChen}, or {\em Hele--Shaw model with surface tension}~\cite{EscherSi1,EscherSi2,EscherSi3}, which arises as a singular limit of a nonlocal version of the Cahn--Hilliard equation~\cite{alikakos,pego,Le}, to describe phase separation in diblock copolymer melts (see also~\cite{OK}). 
\end{remark}
Parametrizing the smooth hypersurfaces $M_t=\partial E_t$ of $\T^n$ by some smooth embeddings $\psi_t:M\to\T^n$ such that $\psi_t(M)=\partial E_t$ (here $M$ is a fixed smooth differentiable $(n-1)$--dimensional manifold and the map $(t,p)\mapsto\psi(t,p)=\psi_t(p)$ is smooth), the geometric evolution law~\eqref{MSnl} can be expressed equivalently as
\begin{equation}\label{msf2perp}
\Bigl\langle\frac{\partial\psi_t}{\partial t}\,\Bigl\vert\,\nu_t\Bigr\rangle=[\pa_{\nu_t} w_t],
\end{equation}
where we denoted by $\nu_t$ the outer unit normal to $M_t=\pa E_t$.\\
Moreover, as the moving hypersurfaces $M_t=\pa E_t$ are compact, it is always possible to smoothly reparametrize them with maps (that we still call) $\psi_t$ such that 
\begin{equation}\label{msf2}
\frac{\partial\psi_t}{\partial t}=[\pa_{\nu_t} w_t]\nu_t\,,
\end{equation}
in describing such flow. This follows by the {\em invariance by tangential perturbations of the velocity}, shared by the flow due to its geometric nature and can be proved following the line in Section~1.3 of~\cite{Man}, where the analogous property is shown in full detail for the (more famous) mean curvature flow. Roughly speaking, the tangential component of the velocity of the points of the moving hypersurfaces, does not affect the global ``shape'' during the motion.

Like the nonlocal Area functional $J$ (see Definition~\ref{NAFdef}), the flow is obviously invariant by translations, or more generally under any isometry of $\T^n$ (or $\R^n$). Moreover, if $\psi:[0,T)\times M\to\T^n$ is a modified Mullins--Sekerka flow of hypersurfaces, in the sense of equation~\eqref{msf2perp} and $\Phi:[0,T)\times M\to M$ is a time--dependent family of smooth diffeomorphisms of $M$, then it is easy to check that the reparametrization $\widetilde{\psi}:[0,T)\times M\to\T^n$ defined as $\widetilde{\psi}(t,p)=\psi(t,\Phi(t,p))$ is still a modified Mullins--Sekerka flow (again in the sense of equation~\eqref{msf2perp}). This property can be reread as ``the flow is invariant under reparametrization'', suggesting that the really relevant objects are actually the subsets $M_t=\psi_t(M)$ of $\T^n$.

We show now that the volume of the sets $E_t$ is preserved during the evolution. We remark that instead, other geometric properties shared for instance by the mean curvature flow (see~\cite[Chapter~2]{Man}), like convexity are not necessarily maintained (see~\cite{Conv}), neither there holds the so--called ``comparison property'' asserting that if two initial sets are one contained in the other, they stay so during the two respective flows.

This volume--preserving property can be easily proved, arguing as in the
computation leading to equation~\eqref{eqc1000}. Indeed, if $E_t=\Psi_t(F)$ is a modified Mullins--Sekerka flow, described by $\Psi\in C^{\infty}([0,T) \times \T^n; \T^n)$, with an associated smooth vector field $X_t$ as above, we have 
\begin{align}\label{volumepreserving}
\frac{d}{dt} \vol( E_t)=&\,\int_F\frac{\partial}{\partial t} J\Psi_t(y)\,dy=\int_F\Div X_t(\Psi(t,y))J\Psi(t,y)\,dy\\
=&\,\int_{ E_t} \Div X_t(x)\,dx=\int_{\partial E_t} \langle X_t\vert\nu_t\rangle \, d\mu_t=\int_{\pa E_t}V_t\, \dmu_t\\
=&\, \int_{\pa E_t} [ \pa_{\nu_t} w_t ] \, \dmu_t = \int_{\pa E_t} \bigl(\pa_{\nu_t} w_t^+ - \pa_{\nu_t} w_t^- \bigr)\, \dmu_t= 0\,,
\end{align}
where the last equality follows from the divergence theorem and the fact that $w_t$ is harmonic in $\T^n \setminus \pa E_t$.

Another important property of the modified Mullins--Sekerka flow is that it can be regarded as the $H^{-1/2}$--gradient flow of the functional $J$ under the constraint that the volume is fixed, that is, the outer normal velocity $V_t$ is minus such $H^{-1/2}$--gradient of the functional $J$ (see~\cite{Le}).\\
For any smooth set $E \subseteq \T^n$, we let the space $\widetilde{H}^{-1/2}(\pa E)\subseteq L^{2}(\pa E)$ to be the dual of $\widetilde{H}^{1/2}(\pa E)$ (the functions in $H^{1/2}(\pa E)$ with zero integral) with the Gagliardo $H^{1/2}$--seminorm  (see~\cite{AdamsFournier,Dem,NePaVa,RuSi}, for instance)
$$
\Vert u\Vert_{\widetilde{H}^{1/2}(\pa E)}^2=[u]_{H^{1/2}(\pa E)}^2=\int_{\pa E}\int_{\pa E}\frac{\vert u(x)-u(y)\vert^2}{\vert x-y\vert^{n+1}}\,d\mu(x) d\mu(y)
$$
(it is a norm for $\widetilde{H}^{1/2}(\pa E)$ since the functions in it have zero integral) and the pairing between $\widetilde{H}^{1/2}(\pa E)$ and $\widetilde{H}^{-1/2}(\pa E)$ simply being the integral of the product of the functions on $\pa E$.\\
We define the linear operator $\Delta_{\pa E}$ on the smooth functions $u$ with zero integral on $\pa E$ as follows: we consider the unique smooth solution $w$ of the problem
\begin{equation}
\begin{cases}
\Delta w=0 & \text{in } \T^n\setminus\pa E\\
w=u & \text{on } \, \partial E
\end{cases}
\end{equation}
and we denote by $w^+$ and $w^-$ the restrictions $w|_{E^c}$ and $w|_E$, respectively, then we 
set 
$$
\Delta_{\pa E}u=\pa_{\nu}w^+-\pa_{\nu}w^-= [\pa_{\nu} w]\,,
$$
which is another smooth function on $\pa E$ with zero integral. Then, we have
$$
\int_{\T^n}\vert\nabla w\vert^2\,dx=\int_{E\cup E^c}\Div(w\nabla w)\,dx=-\int_{\pa E}u\Delta_{\pa E}u\,d\mu
$$
and such quantity turns out to be a norm equivalent to the one given by the Gagliardo seminorm on $\widetilde{H}^{1/2}(\pa E)$ above (this is related to the theory of trace spaces for which we refer to~\cite{AdamsFournier, Gagliardo}), see~\cite{Le}. Hence, it induces the dual norm 
\begin{equation}
\Vert v \Vert^2_{\widetilde{H}^{-1/2}(\pa E)}=\int_{\pa E} v (- \Delta_{\pa E} )^{-1} v \, \dmu
\end{equation}
for every smooth function $v\in\widetilde{H}^{-1/2}(\pa E)$. By polarization, we have the $\widetilde{H}^{-1/2}(\pa E)$--scalar product between a pair of smooth functions $u,v:\pa E\to\R$ with zero integral, 
\begin{equation}
\langle u|v \rangle_{\widetilde{H}^{-1/2}(\pa E)}=\int_{\pa E} u (- \Delta_{\pa E} )^{-1} v \, \dmu\,.
\end{equation}
This scalar product, extended to the whole space $\widetilde{H}^{-1/2}(\pa E)$, makes it a Hilbert space (see~\cite{Gar}), hence, by {\em Riesz representation theorem}, there exists a function $\nabla_{\pa E}^{\widetilde{H}^{-1/2}}\!\!J\in\widetilde{H}^{-1/2}(\pa E)$ such that, for every smooth function $v\in\Htilde^{-1/2}(\pa E)$, there holds
\begin{equation}\label{gradient}
\int_{\pa E} v (\HHH + 4 \gamma v_E) \, \dmu =\delta J_{\pa E}(v)= \langle v|\nabla_{\pa E}^{\widetilde{H}^{-1/2}}\!\!J\rangle_{\widetilde{H}^{-1/2} (\pa E)} = \int_{\pa E} v(- \Delta_{\pa E})^{-1} \nabla_{\pa E}^{\widetilde{H}^{-1/2}}\!\!J\, \dmu \,,
\end{equation}
by Theorem~\ref{first var}, where $v_E$ is the potential introduced in~\eqref{potential} and $\HHH$ is the mean curvature of $\pa E$.\\
Then, by the {\em fundamental lemma of calculus of variations}, we conclude 
$$
(-\Delta_{\pa E})^{-1}\nabla_{\pa E}^{\widetilde{H}^{-1/2}}\!\!J=\HHH+ 4 \gamma v_E+c\,,
$$
for a constant $c\in\R$, that is, recalling the definition of $w_E$ in problem~\eqref{WE} and of the operator $\Delta_{\pa E}$ above,
$$
\nabla_{\pa E}^{\widetilde{H}^{-1/2}}\!\!J=-\Delta_{\pa E}(\HHH+ 4 \gamma v_E)= -[\pa_{\nu_E} w_E]\,.
$$
It clearly follows that the outer normal velocity of the moving boundaries $V_t=[\pa_{\nu_t} w_t]$ is minus the $\widetilde{H}^{-1/2}$--gradient of the volume--constrained functional $J$.

\medskip

We deal now with the {\em surface diffusion flow}.

\begin{definition}\label{def:SDsol}
Let $E\subseteq \T^n$ be a smooth set. We say that a smooth flow $E_t=\Phi_t(E)$, for $t \in[0, T)$, with $E_0=E$, is a {\em surface diffusion flow} starting from $E$ if the outer normal velocity $V_t$ of the moving boundaries $\pa E_t$ is given by
\begin{equation}\label{SD}
V_t= \Delta_t\HHH_t \quad\text{ for all $t\in[0, T)$}
\end{equation}
where $\Delta_t$ is the (rough) Laplacian associated to the hypersurface $\pa E_t$, with the Riemannian metric induced by $\T^n$ (that is, by $\R^n$).
\end{definition}

Such flow was first proposed by Mullins in~\cite{Mullins} to study thermal grooving in material sciences and first analyzed mathematically more in detail in~\cite{escmaysim}. In particular, in the physically relevant case of three--dimensional space, it describes the evolution of interfaces between solid phases of a system, driven by surface diffusion of atoms under the action of a chemical potential (see for instance~\cite{GurJab}).

With the same argument used for the modified Mullins--Sekerka flow, representing the smooth hypersurfaces $\pa E_t$ in $\T^n$ with a family of smooth embeddings $\psi_t:M\to \T^n$, we can describe the flow as 
\begin{equation}\label{sdf2perp}
\Bigl\langle\frac{\partial\psi_t}{\partial t} \Bigl\vert  \nu_t\Bigr\rangle=\Delta_t\HHH_t\,
\end{equation}
and also simply as
\begin{equation}\label{sdf2}
\frac{\partial\psi_t}{\partial t}=(\Delta_t\HHH_t)\nu_t\,.
\end{equation}

\begin{remark}
This is actually the more standard way to define the surface diffusion flow, in the more general situation of smooth and possibly {\em immersed--only} hypersurfaces (usually in $\R^n$), without being the boundary of any set.
\end{remark}

By means of equation~\eqref{lap}, the system~\eqref{sdf2} can be rewritten as
\begin{equation}\label{paraeq}
\frac{\partial\psi_t}{\partial t}=-\Delta_t\Delta_t\psi_t+\text{ lower order terms}
\end{equation}
and it can be seen that it is a fourth order, {\em quasilinear} and {\em degenerate}, parabolic system of PDEs. Indeed, it is quasilinear, as the coefficients (as second order partial differential operator) of the Laplacian associated to the induced metrics $g_t$ on the evolving hypersurfaces, that is, 
$$
\Delta_t\psi_t(p)=\Delta_{g_t(p)}\psi_t(p)=g^{ij}_t(p)\nabla^{g_t(p)}_i\nabla^{g_t(p)}_j\psi_t(p)
$$
depend on the first order  derivatives of $\psi_t$, as $g_t$ (and the coefficient of $\Delta_t\Delta_t$ on the third order derivatives). Moreover, the operator at the right hand side of system~\eqref{sdf2} is degenerate, as its symbol (the symbol of the linearized operator) admits zero eigenvalues due to the invariance of the Laplacian by diffeomorphisms.

Arguing as in computation~\eqref{volumepreserving}, using the equation~\eqref{SD} in place of~\eqref{MSnl}, it can be seen that also the surface diffusion flow of boundaries of sets is volume--preserving. Moreover, analogously to the modified Mullins--Sekerka flow (see the discussion above), it does not preserve convexity (see~\cite{Ito}), nor the embeddedness (in the ``stand--alone'' formulation of motion of hypersurfaces, as in formula~\eqref{sdf2}, see~\cite{gigaito1}), indeed it also does not have a ``comparison principle'', while it is invariant by isometries of $\T^n$, reparametrizations and tangential perturbations of the velocity of the motion. In addition, it can be regarded as the $\widetilde{H}^{-1}$--gradient flow of the volume--constrained Area functional, in the following sense (see~\cite{Gar}, for instance). For a smooth set $E \subseteq \T^n$, we let the space $\widetilde{H}^{-1}(\pa E)\subseteq L^{2}(\pa E)$ to be the dual of $\widetilde{H}^1(\pa E)$  with the norm $\Vert u\Vert_{\widetilde{H}^1(\pa E)}=\int_{\pa E}\vert\nabla u\vert^2\,d\mu$ and the pairing between $\widetilde{H}^1(\pa E)$ and  $\widetilde{H}^{-1}(\pa E)$ simply being the integral of the product of the functions on $\pa E$.\\
Then, it follows easily that the norm of a smooth function $v\in\widetilde{H}^{-1}(\pa E)$ is given by 
\begin{equation}
\Vert v \Vert^2_{\widetilde{H}^{-1}(\pa E)}= \int_{\pa E} v (- \Delta )^{-1} v \, \dmu 
= \int_{\pa E} \langle \nabla(- \Delta )^{-1} v \vert   \nabla(- \Delta )^{-1} v \rangle \, \dmu
\end{equation}
and, by polarization, we have the $\widetilde{H}^{-1}(\pa E)$--scalar product between a pair of smooth functions $u,v:\pa E\to\R$ with zero integral, 
\begin{equation}
\langle u\vert  v \rangle_{\widetilde{H}^{-1}(\pa E)}= \int_{\pa E} \langle \nabla(- \Delta )^{-1} u \vert   \nabla(- \Delta )^{-1} v \rangle \, \dmu=
\int_{\pa E} u (- \Delta )^{-1} v \, \dmu\,,
\end{equation}
integrating by parts.\\
This scalar product, extended to the whole space $\widetilde{H}^{-1}(\pa E)$, make it a Hilbert space, hence, by {\em Riesz representation theorem}, there exists a function $\nabla_{\pa E}^{\widetilde{H}^{-1}}\!\!\A\in\widetilde{H}^{-1}(\pa E)$ such that, for every smooth function $v\in\widetilde{H}^{-1}(\pa E)$, there holds
$$
\int_{\pa E} v\HHH \, \dmu =\delta\A_{\pa E}(v)=  \langle v \vert  \nabla_{\pa E}^{\widetilde{H}^{-1}}\!\!\A\rangle_{\widetilde{H}^{-1} (\pa E)} =  \int_{\pa E} v(- \Delta)^{-1}  \nabla_{\pa E}^{\widetilde{H}^{-1}}\!\!\A\, \dmu \,,
$$
by Theorem~\ref{first var} (with $\gamma = 0$).\\
Then, by the {\em fundamental lemma of calculus of variations}, we conclude 
$$
(-\Delta)^{-1}\nabla_{\pa E}^{\widetilde{H}^{-1}}\!\!\A=\HHH+c\,,
$$
for a constant $c\in\R$, that is,
$$
\nabla_{\pa E}^{\widetilde{H}^{-1}}\!\!\A= - \Delta\HHH\,.
$$
It clearly follows that the outer normal velocity of the moving boundaries of a surface diffusion flow $V_t=\Delta_t\HHH_t$ is minus the $\widetilde{H}^{-1}$--gradient of the volume--constrained functional $\A$.

\begin{remark}\label{localcarlo}
It is interesting to notice that the ({\em unmodified}, that is, with $\gamma=0$) Mullins--Sekerka flow is the $H^{-1/2}$--gradient flow and the surface diffusion flow the $H^{-1}$--gradient flow of the Area functional on the boundary of the sets, {\em under a volume constraint}, while considering the {\em unconstrained} Area functional, its $L^2$--gradient flow is the mean curvature flow.

It follows that, in a way, the unmodified Mullins--Sekerka flow, representing the moving hypersurfaces as of smooth embeddings $\psi_t:M\to \T^n$, can be described as
\begin{equation}\label{sdf3}
\frac{\partial\psi_t}{\partial t}=(\Delta_t^{1/2}\HHH_t)\nu_t=-\Delta_t^{3/2}\psi_t+\text{ lower order terms,}
\end{equation}
showing its parabolic nature (differently by the surface diffusion flow, in this case the equation is {\em nonlocal}, due to the fractional Laplacian involved, even if the functional is still simply the Area, hence implying that the flow depends only on the hypersurface) -- again quasilinear and degenerate -- and suggesting the problem of analyzing (and eventually generalizing the existing results) the nonlocal evolutions of hypersurfaces given by the laws
\begin{equation}\label{sdf3*}
\frac{\partial\psi_t}{\partial t}=(\Delta_t^s\HHH_t)\nu_t=-\Delta_t^{s+1}\psi_t+\text{ lower order terms,}
\end{equation}
when $s>0$, arising from considering, as above, the $H^{-s}$--gradient of the Area functional on the boundary of the sets (under a volume constraint).\\
Up to our knowledge, these flows are not present in literature and it would be also interesting to compare them to the {\em fractional mean curvature flows} arising considering the gradient flows associated to the {\em fractional Area functionals} on the boundary of a set (in this case such functionals are ``strongly'' nonlocal), see~\cite{Imbert,JLamanna} and references therein, for instance.
\end{remark}

\subsection{Short time existence}\ \vskip.3em

To state the short time existence and uniqueness results for the two flows, we give the following definition which is actually fundamental for the discussion of the global existence in the next section.

\begin{definition}
Given a smooth set $E\subseteq\T^n$ and a tubular neighborhood $N_\eps$ of $\pa E$, as in formula~\eqref{tubdef}, for any $M\in(0,\eps/2)$ (recall the discussion in Subsection~\ref{stabsec} about the notion of ``closedness'' of sets), we denote by $\mathfrak{C}^1_M(E)$, the class of all smooth sets $F\subseteq E\cup N_\eps$ such that $\vol(F\triangle E)\leq M$ and  
\begin{equation}\label{front}
\pa F=\{x+\psi_F(x)\nu_{E}(x):\, x\in \pa E \}\,,
\end{equation}
for some $\psi_F\in C^\infty(\pa E)$, with $\norma{\psi_F}_{C^1(\pa E)}\leq M$ (hence, $\pa F\subseteq N_\eps$). For every $k\in\N$ and $\alpha\in (0,1)$, we also denote by $\mathfrak{C}^{k,\alpha}_M(E)$ the collection of sets $F\in \mathfrak{C}^1_M(E)$ such that $\norma{\psi_F}_{C^{k,\alpha}(\pa E)}\leq M$.
\end{definition}

The following existence/uniqueness theorem of classical solutions for the modified Mullins--Sekerka flow was proved by Escher and Simonett~\cite{EscherSi1,EscherSi2,EscherSi3} and independently by Chen, Hong and Yi~\cite{chenhong} (see also~\cite{EsNi}). The original version deals with the flow in domains of $\R^n$, but it can be easily adapted to hold also when the ambient is the flat torus $\T^n$.

\begin{thm}\label{th:EscNis}
Let $E\subseteq\T^n$ be a smooth set and $N_\varepsilon$ a tubular neighborhood of $\pa E$, as in formula~\eqref{tubdef}. Then, for every $\alpha\in (0,1)$ and $M\in(0,\eps/2)$ small enough, there exists $T=T(E,M,\alpha)>0$ such that if $E_0\in \mathfrak{C}^{2,\alpha}_M(E)$ there exists a unique smooth modified Mullins--Sekerka flow with parameter $\gamma\geq0$, starting from $E_0$, in the time interval $[0, T)$.
\end{thm}

We now state the analogous result (and also of dependence on the initial data) for the surface diffusion flow starting from a smooth hypersurface, proved by Escher, Mayer and Simonett in~\cite{escmaysim}, which should be expected by the explicit parabolic nature of the system~\eqref{sdf2}, as shown by the formula~\eqref{paraeq}. As before, it deals with the evolution in the whole space $\R^n$ of a generic hypersurface, even only immersed, hence possibly with self--intersections. It is then straightforward to adapt the same arguments to our case, when the ambient is the flat torus $\T^n$ and the hypersurfaces are the boundaries of the sets $E_t$, as in Definition~\ref{def:SDsol}, getting a (unique) surface diffusion flow in a positive time interval $[0,T)$, for every initial smooth set $E_0\subseteq\T^n$.

\begin{thm}\label{th:EMS0}
Let $\psi_0:M\to\R^n$ be a smooth and compact, immersed hypersurface. Then, there exists a unique smooth surface diffusion flow $\psi:[0, T)\times M\to\R^n$, starting from $M_0=\psi_0(M)$ and solving system~\eqref{sdf2}, for some maximal time of existence $T>0$.\\
Moreover, such flow and the maximal time of existence depend continuously on the $C^{2,\alpha}$ norm of the initial hypersurface.
\end{thm}

As an easy consequence, we have the following proposition (analogous to Theorem~\ref{th:EscNis}), better suited for our setting.

\begin{proposition}\label{th:EMS1}
Let $E\subseteq\T^n$ be a smooth set and $N_\varepsilon$ a tubular neighborhood of $\pa E$, as in formula~\eqref{tubdef}. Then, for every $\alpha\in (0,1)$ and $M\in(0,\eps/2)$ small enough, there exists $T=T(E,M,\alpha)>0$ such that if $E_0\in \mathfrak{C}^{2,\alpha}_M(E)$ there exists a unique smooth surface diffusion flow, starting from $E_0$, in the time interval $[0, T)$.
\end{proposition}

In the same paper~\cite{escmaysim}, Escher, Mayer and Simonett also showed that if the initial set $E_0$ is in $\mathfrak{C}^{2,\alpha}_M(B)$, where $B\subseteq\R^n$ is a ball with the same volume and $M$ is small enough (that is, $E_0$ is $C^{2,\alpha}$--close to the ball $B$), then the smooth flow $E_t$ exists for every time and smoothly converges to a translate of the ball $B$.\\
The analogous result for the ({\em unmodified}, that is, with $\gamma=0$) Mullins--Sekerka flow, was proved by Escher and Simonett in~\cite{EscherSi4} (moving by their previous work~\cite{EscherSi2}), generalizing to any dimension the two dimensional case shown by Chen in~\cite{XChen}.

The next section will be devoted to present the generalization by Acerbi, Fusco, Julin and Morini in~\cite{AcFuMoJu} (in dimensions two and three) of this stability result for the surface diffusion and modified Mullins--Sekerka flow, to every strictly stable critical set (as it is every ball for the Area functional under a volume constraint, by direct check -- see the last section).

We conclude mentioning another interesting result by Elliott and Garcke~\cite{EllGar} (which is not present in literature for the modified Mullins--Sekerka flow, up to our knowledge) is that if the initial curve $E_0$ in $\R^2$ of the surface diffusion flow is closed to a circle, then the flow $E_t$ exists for all times and converges, up to translations, to a circle in the plane with the same volume. This is clearly related to the fact that the unique bounded strictly stable critical sets for the Area functional under a volume constraint in the plane $\R^2$ are the disks (see the last section).

\section{Global existence and asymptotic behavior around a strictly stable critical set}\label{globalex}

In this section we show the proof by Acerbi, Fusco, Julin and Morini in~\cite{AcFuMoJu}, in dimensions two and three of the toric ambient, that if the ``initial'' sets is ``close enough'' to a strictly stable critical set of the respectively relative functional, then the surface diffusion and the modified Mullins--Sekerka flow exist for all times and smoothly converge to a translate of $E$. Heuristically, this shows that a strictly stable critical set is in a way like the equilibrium configuration of a system at the bottom of a potential well ``attracting'' the close enough smooth sets.

We will deal here with the (more difficult) case of dimension three. When the dimension is two, the ``exponents'' in the functional spaces involved in the estimates (in particular the ones in the interpolation inequalities, which are very dimension--dependent) change but the same proof still works (roughly speaking, we have the necessary ``compactness'' of the sequences of hypersurfaces -- see Lemma~\ref{w52conv} and~\ref{w32conv}), modifying suitably the statements. If the dimension of the toric ambient is larger than three, the analogous (mostly, interpolation) estimates are too weak to conclude and this proof does not work. It is indeed a challenging open problem to extend these results to such higher dimensions.
    
For both flows, we will have a subsection with the necessary technical lemmas and then one with the proof of the main theorem. Moreover, for the modified Mullins--Sekerka flow, we also briefly discuss the ``Neumann case'', in Subsection~\ref{Neucase}.

\subsection{The modified Mullins--Sekerka flow -- Preliminary lemmas}\ \vskip.3em

In order to simplify the notation, for a smooth set $E_t\subseteq\T^n$ we will write $\nu_t$ and $\partial_{\nu_t}$ in place of  $\nu_{E_t}$ and $\partial_{\nu_{E_t}}$, $w_t$ for the function $w_{E_t}\in H^1(\T^n)$ uniquely defined by problem~\eqref{WE}. Moreover, we will also denote with $v_t$ the smooth potential function $v_{E_t}$ associated to $E_t$ by formula~\eqref{potential}.

\smallskip

We start with the following lemma holding in all dimensions.

\begin{lem}[Energy identities] \label{calculations}
Let $E_t\subseteq\T^n$ be a modified Mullins--Sekerka flow as in Definition~\eqref{MSF def}. Then, the following 
identities hold:
\begin{equation}
\label{der of J}
\frac{d}{dt} J(E_t) = - \int_{\T^n} |\nabla w_t|^2\, dx\,,
\end{equation}
and
\begin{equation}
\label{der of dw}
\frac{d}{dt}\,\frac{1}{2} \int_{\T^n} |\nabla w_t|^2\, dx= -\Pi_{E_t}\bigl([\pa_{\nu_t}w_t\vphantom{^{^4}}]\bigr)
+ \frac{1}{2}\int_{\partial E_t} \bigl(\partial_{\nu_t} w^+_t+ \partial_{\nu_t} w_t^-\bigr) [\partial_{\nu_t} w_t]^2 \, d \mu_t \,,
\end{equation}
where $\Pi_{E_t}$ is the quadratic form defined in formula~\eqref{Pieq}.
\end{lem}
\begin{proof}
Let $\psi_t$ the smooth family of maps describing the flow as in 
formula~\eqref{msf2}. By formula~\eqref{dermu2}, where $X$ is the smooth (velocity) vector field $X_t=\frac{\partial\psi_t}{\partial t}=[\pa_{\nu_t}w_t]\nu_t$ along $\pa E_t$, hence $X_\tau=X_t-\langle X_t \vert \nu_t\rangle\nu_t=0$ (as usual $\nu_t$ is the outer normal unit vector of $\partial E_t$), following the computation in the proof of Theorem~\eqref{first var}, we have 
\begin{equation}
\frac{d}{dt} J(E_t)= \int_{\partial E_t} (\HHH_t+ 4 \gamma v_t) \langle X_t |\nu_t \rangle \, d\mu_t= \int_{\pa E_t} w_t[\pa_{\nu_t}w_t]\, \dmu_t=- \int_{\T^n} \abs{\nabla w_t}^2 \, dx\,,
\end{equation}
where the last equality follows integrating by parts, as $w_t$ is harmonic in $\T^n\setminus \pa E_t$. This establishes relation~\eqref{der of J}.

In order to get identity~\eqref{der of dw}, we compute
\begin{align}
\frac{d}{dt}\,\frac{1}{2} \int_{E_t} |\nabla w_t^-|^2\, dx
&= \frac{1}{2} \int_{\partial E_{t}} |\nabla^{\T^n}\! w^-_t |^2 \left\langle X_t \vert \nu_t\right\rangle\, d \mu_t+\frac{1}{2} \int_{E_t} \frac{d}{dt}|\nabla w_t^-|^2\, dx\\
&= \frac{1}{2} \int_{\partial E_{t}} |\nabla^{\T^n}\! w^-_t |^2 [ \partial_{\nu_t} w_t ] \, \dmu + \int_{E_{t}} \nabla \partial_t{w}^-_t \nabla w^-_t \, \dmu_t\\
&= \frac{1}{2} \int_{\partial E_{t}} |\nabla^{\T^n}\! w^-_t |^2 [ \partial_{\nu_t} w_t ] \, \dmu + \int_{\partial E_{t}} \partial_t{w}^-_t \partial_{\nu_t} w^-_t \, \dmu_t, \,\label{calculation}
\end{align}
where we interchanged time and space derivatives and we applied the divergence theorem, taking into account that $w_t^-$ is harmonic in $E_t$.\\
Then, we need to compute $\partial_t{w}^-_t$ on $\pa E_t$. We know that
$$
w_t^-=\HHH_t+4\gamma v_t
$$
on $\pa E_t$, hence, (totally) differentiating in time this equality, we get
$$
\partial_t{w}^-_t + \bigl\langle \nabla^{\T^n}\! w^-_t\bigl\vert X_t\bigr\rangle = \partial_t\HHH_t+ 4\gamma \partial_t{v}_t + 4\gamma \bigl\langle\nabla^{\T^n}\! v_t\bigl\vert X_t\bigr\rangle\,,
$$
that is,
\begin{align*}
\partial_t{w}^-_t + [\partial_{\nu_t} w_t]\partial_{\nu_t}w^-_t=&\,\partial_t\HHH_t+ 4\gamma \partial_t{v}_t + 4\gamma [\partial_{\nu_t} w_t]\partial_{\nu_t}v_t\\
=&\,-|B_t|^2 [\partial_{\nu_t} w_t]-\Delta_t[\partial_{\nu_t} w_t]+ 4\gamma \partial_t{v}_t + 4\gamma [\partial_{\nu_t} w_t]\partial_{\nu_t}v_t\,,
\end{align*}
where we used computation~\eqref{derH}.\\
Therefore from equations~\eqref{calculation} and~\eqref{v'} we get
\begin{align}
\frac{d}{dt} \frac{1}{2} \int_{E_t} |\nabla w_t^-|^2\, dx
=&\,-\int_{\partial E_t} \partial_{\nu_t} w^-_t\, \Delta_t [\partial_{\nu_t} w_t]\, d \mu_t - \int_{\partial E_t}\partial_{\nu_t} w_t^-\, |B_t|^2 \, [\partial_{\nu_t} w_t] \, d \mu_t \\
&\,+8\gamma \int_{\partial E_t}\int_{\partial E_t} G(x,y) \,\partial_{\nu_t} w_t^-(x)\, [\partial_{\nu_t} w_t](y) \, d \mu_t(x) d \mu_t(y) \\
&\,+4\gamma \int_{\partial E_t} \partial_{\nu_t} v_t\,\partial_{\nu_t} w_t^-\, [\partial_{\nu_t} w_t] \, d \mu_t \\
&\,+\frac{1}{2} \int_{\partial E_{t}} |\nabla^{\T^n}\! w_t^- |^2 [\partial_{\nu_t} w_t] \, d \mu_t - \int_{\partial E_t} (\partial_{\nu_t} w_t^-)^2 [\partial_{\nu_t} w_t] \, d \mu_t\,.\label{interior}
\end{align}
Computing analogously for $w^+_t$ in $E^c$ and adding the two results, we get
\begin{align}
\frac{d}{dt} \frac{1}{2} \int_{\T^n} |\nabla w_t|^2\, dx
=&\,\int_{\partial E_t} [\partial_{\nu_t} w_t]\, \Delta_t [\partial_{\nu_t} w_t]\, d \mu_t + \int_{\partial E_t}|B_t|^2 \, [\partial_{\nu_t} w_t]^2 \, d \mu_t \\
&\,-8\gamma \int_{\partial E_t}\int_{\partial E_t} G(x,y) \,[\partial_{\nu_t} w_t](x)\, [\partial_{\nu_t} w_t](y) \, d \mu_t(x) d \mu_t(y) \\
&\,-4\gamma \int_{\partial E_t} \partial_{\nu_t} v_t\, [\partial_{\nu_t} w_t]^2 \, d \mu_t \\
&\,+\int_{\partial E_t} \bigl((\partial_{\nu_t} w^+_t)^2 - (\partial_{\nu_t} w_t^-)^2 \bigr)[\partial_{\nu_t} w_t] \, d \mu_t \\
&\,-\frac{1}{2} \int_{\partial E_{t}}\bigl (|\nabla^{\T^n}\! w^+_t |^2 - |\nabla^{\T^n}\! w^-_t |^2\bigr) [\partial_{\nu_t} w_t] \, d \mu_t\\
=&\,- \Pi_{E_t}\bigl([\pa_{\nu_t}w_t]\bigr)+ \frac{1}{2}\int_{\partial E_t} \bigl(\partial_{\nu_t} w^+_t+ \partial_{\nu_t} w_t^-\bigr) [\partial_{\nu_t} w_t]^2 \, d \mu_t \,,
\end{align}
where we integrated by parts the very first term of the right hand side, recalled Definition~\eqref{Pieq} and in the last step we used the identity
$$
|\nabla^{\T^n}\! w^+_t |^2 - |\nabla^{\T^n}\! w^-_t |^2 = (\partial_{\nu_t} w^+_t)^2 - (\partial_{\nu_t} w_t^-)^2 = (\partial_{\nu_t} w_t^+ + \partial_{\nu_t} w_t^- ) [\partial_{\nu_t} w_t]\,.
$$
Hence, also equation~\eqref{der of dw} is proved.
\end{proof}

\smallskip

{\em From now on, we restrict ourselves to the three--dimensional case, that is, we will consider smooth subsets of $\T^3$ with boundaries which then are smooth embedded ($2$--dimensional) surfaces. As  we said at the beginning of the section, this is due to the dependence on the dimension of several of the estimates that follow.}

\medskip

{\em In the estimates in the following series of lemmas, we will be interested in having uniform constants for the families $\mathfrak{C}^{1,\alpha}_M(F)$, given a smooth set $F\subseteq\T^n$ and a tubular neighborhood $N_\eps$ of $\pa F$ as in formula~\eqref{tubdef}, for any $M\in(0,\eps/2)$ and $\alpha\in(0,1)$. This is guaranteed if the constants in the Sobolev, Gagliardo--Nirenberg interpolation and Calder\'on--Zygmund inequalities, relative to all the smooth hypersurfaces $\pa E$ boundaries of the sets $E\in\mathfrak{C}^{1,\alpha}_M(F)$, are uniform, as it is proved in detail in~\cite{DDM}.}

\medskip

{\em We remind that in all the inequalities, the constants $C$ may vary from one line to another.}

\medskip

The next lemma provides some boundary estimates for harmonic functions.
\begin{lem}[Boundary estimates for harmonic functions]\label{harmonic estimates}
Let $F\subseteq\T^3$ be a smooth set and $E\in\mathfrak{C}^{1,\alpha}_M(F)$. Let $f \in C^\alpha(\partial E)$ with zero integral on $\partial E$ and let $u \in H^{1}(\T^3)$ be the (distributional) solution of 
\begin{equation*}
-\Delta u = f \mu\bigr|_ {\partial E}
\end{equation*} 
with zero integral on $\T^3$. Let $u^- = u|_{E}$ and $u^+ = u|_{E^c}$ and assume that $u^-$ and $u^+$ are of class $C^1$ up to the boundary $\pa E$. Then, for every $1<p<+\infty$ there exists a constant $C=C(F,M,\alpha,p)>0$, such that: 
\begin{enumerate}
\item[{\em (i)}] \ \vspace{-.5cm}
$$\| u \|_{L^p(\partial E)} \leq C \|f\|_{L^p(\partial E)}$$

\item[{\em (ii)}] \ \vspace{-.5cm}
\begin{equation*}
\| \partial_{\nu_E} u^+\|_{L^2(\partial E)} + \| \partial_{\nu_E} u^-\|_{L^2(\partial E)} \leq C \|u \|_{H^1(\partial E)}
\end{equation*} 
\item[{\em (iii)}] \ \vspace{-.5cm}
\begin{equation*}
\| \partial_{\nu_E} u^+\|_{L^p(\partial E)} + \| \partial_{\nu_E} u^-\|_{L^p(\partial E)} \leq C \|f\|_{L^p(\partial E)}
\end{equation*}
\item[{\em (iv)}] \ \vspace{-.5cm}
$$
\|u\|_{C^{0, \beta}(\pa E)}\leq C \|f\|_{L^{p}(\pa E)}
$$ 
for all $\beta\in \bigl(0, \frac{p-2}{p}\bigr)$, with $C$ depending also on $\beta$.
\end{enumerate}
Moreover, if $f \in H^1(\partial E)$, then for every $2\leq p< +\infty$ there exists a constant $C=C(F,M,\alpha,p)>0$, such that 
\begin{equation*}
\| f\|_{L^p(\partial E)} \leq C \|f\|_{H^1(\partial E)}^{{(p-1)}/{p}}\, \|u \|_{L^2(\partial E)}^{1/p}\,.
\end{equation*}
\end{lem}
\begin{proof} We are not going to underline it every time, but it is easy to check that all the constants that will appear in the proof will depend with only on $F$, $M$, $\alpha$ and sometimes $p$, recalling the previous discussion about the ``uniform'' inequalities holding for the families of sets $\mathfrak{C}^{1,\alpha}_M(F)$.

\smallskip

{(i)} Recalling Remark~\ref{rm:potential}, we have
\begin{equation*}
u(x) = \int_{\partial E} G(x,y)f(y) \, d\mu(y).
\end{equation*} 
It is well known that it is always possible to write $G(x,y) = h(x-y) + r(x-y)$ where $h:\R\to\R$ is smooth away from $0$, one--periodic and $h(t)=\frac{1}{4\pi|t|}$ in a neighborhood of $0$, while $r:\R\to\R$ is smooth and one--periodic.
The conclusion then follows since
for $v(x)= \int_{\partial E} \frac{f(y)}{|x-y|} \, d\mu(y)$ there holds
\begin{equation*}
 \|v\|_{L^p(\partial E)} \leq C \|f\|_{L^p(\partial E)}\,,
\end{equation*}
with $C=C(F,M,\alpha,p)>0$.

\smallskip

{(ii)} We are going to adapt the proof of~\cite{JK} to the periodic setting.
First observe that since $u$ is harmonic in $E\subseteq\T^3$ we have 
\begin{equation}
\label{harmonic lemma 12}
\Div\left( 2\langle \nabla u\vert x\rangle \nabla u -|\nabla u|^2x + u \nabla u \right) = 0.
\end{equation}
Moreover, there exist constants $r>0$, $C_0$ and $N\in\N$, depending only on $F$, $M$, $\alpha$, such that we may cover $\partial E$ with $N$ balls $B_r(x_k)$, with every $x_k\in F$ and
\begin{equation}
\label{harmonic lemma 1}
\frac{1}{C_0} \leq \langle x \vert \nu_E(x)\rangle \leq C_0 \qquad\text{ for $x\in \pa E \cap B_{2r}(x_k)$}\,.
\end{equation}
for every that $E\in\mathfrak{C}^{1,\alpha}_M(F)$.\\
If then $0 \leq \varphi_k \leq 1$ is a smooth function with compact support in $B_{2r}(x_k)$ such that $\varphi_k \equiv 1$
in $B_r(x_k)$ and $|\nabla \varphi_k| \leq C/r$, by integrating the function
\begin{equation*}
\Div\left( \varphi_k \left(2\langle \nabla u\vert x\rangle \nabla u -|\nabla u|^2x + u \nabla u \right) \right) 
\end{equation*}
in $E$ and using equality~\eqref{harmonic lemma 12}, we get
\begin{align}
\int_E \bigl\langle \nabla \varphi_k \vert \,  2 &\left \langle \nabla u\vert x\right \rangle \nabla u -|\nabla u|^2x + u \nabla u\bigr\rangle \,dx\\
&=\int_E \Div\bigl(\varphi_k (2 \left \langle \nabla u\vert x\right \rangle \nabla u -|\nabla u|^2x + u \nabla u)\bigr)  \,dx \\
&= \int_{\pa E} \bigl(2\varphi_k \langle \nabla^{\T^3}\!\!u \vert x \rangle\pa_{\nu_E}u - \varphi_k |\nabla^{\T^3}\!\! u |^2 \langle x \vert \nu_E \rangle + \varphi_ku\pa_{\nu_E} u\bigr)  \,d\mu,
\end{align}
hence,
\begin{align}
\int_E \bigl\langle \nabla \varphi_k \vert  2& \bigl\langle \nabla^{\T^3}\!\! u\vert x\bigr\rangle \nabla u -|\nabla u|^2x + u \nabla u \bigr\rangle  \,dx - \int_{\pa E} \varphi_ku\pa_{\nu_E} u^-  \,\dmu- 2 \int_{\pa E} \varphi_k  \langle \nabla u\vert x \rangle  \pa_{\nu_E} u^-  \,\dmu \\
=&\,-\int_{\pa E} \varphi_k |\nabla^{\T^3}\!\! u^- |^2 \langle x \vert \nu_E \rangle\,d\mu
+2 \int_{\pa E} \varphi_k  |\pa_{\nu_E}u^-|^2\langle x|\nu_E\rangle  \,\dmu \\
=&\,\int_{\pa E} \varphi_k |\pa_{\nu_E} u^-|^2 \langle x \vert \nu_E \rangle\,\dmu - \int_{\pa E} \varphi_k |\nabla u|^2\langle x |\nu_E\rangle\,\dmu\,.
\end{align}
Using the Poincar\'e inequality on the torus $\T^3$ (recall that $u$ has zero integral) and estimate~\eqref{harmonic lemma 1}, this inequality implies
\begin{align}
\int_{\partial E \cap B_r(x_k)} |\pa_{\nu_E} u|^2\, d \mu &\leq C \int_{\partial E} (u^2 + | \nabla u|^2) \, d \mu + C\int_{\T^3} (u^2 +|\nabla u|^2) \, dx\\
&\leq C \int_{\partial E} (u^2 + | \nabla u|^2) \, d \mu + C\int_{\T^3} |\nabla u|^2 \, dx\,.
\end{align}
Putting together all the above estimates and repeating the argument on $E^c$, we get
\begin{equation*}
\int_{\partial E } (|\pa_{\nu_E} u^-|^2 + |\pa_{\nu_E} u^+|^2)\, d \mu \leq C \int_{\partial E} (u^2 + | \nabla u|^2) \, d \mu + C\int_{\T^3} |\nabla u|^2 \, dx\,.
\end{equation*}
The thesis then follows by observing that 
\begin{equation*}
\int_{\T^3} |\nabla u|^2 \, dx = \int_{\partial E} u (\pa_{\nu_E} u^- - \pa_{\nu_E} u^+) \, d \mu\,. 
\end{equation*}

\smallskip

{(iii)} Let us define
\begin{equation*}
Kf(x) = \int_{\partial E}\bigl \langle\nabla^{\T^3}_x\! G(x,y)| \nu_E(x)\bigr\rangle f(y) \, d\mu(y)\,.
\end{equation*}
We want to show that
\begin{equation}
\label{flow lemma 1}
\|Kf\|_{L^p(\partial E)} \leq C \|f\|_{L^p(\partial E)}.
\end{equation}
By the decomposition recalled at the point (i), we have $\nabla^{\T^3}_x\!G(x,y)=\nabla^{\T^3}_x\![h(x-y)] + \nabla^{\T^3}_x\![r(x-y)]$, where
$\nabla^{\T^3}_x\![h(x-y)]=-\frac{1}{4\pi}\frac{x-y}{|x-y|^3}$, for $|x-y|$ small enough and $\nabla^{\T^3}_x\![r(x-y)]$ is smooth.
Thus, by a standard partition of unity argument we may localize the estimate and reduce to show that if $\varphi \in C^{1,\alpha}_c(\R^{2})$ and $U \subseteq \R^{2}$ is a bounded domain setting
$\Gamma = \{(x', \varphi(x')) \, : \, x' \in U\}\subseteq\R^3$ and 
$$
Tf(x) = \int_\Gamma \frac{ \langle x-y\vert \nu_E(x)\rangle }{|x-y|^{3}} f(y) \, d\mu(y)
$$
for every $x\in \Gamma$, where $\nu_E$ is the ``upper'' normal to the graph $\Gamma$, then $Tf(x)$ is well 
defined at every $x \in \Gamma$ and 
\begin{equation*}
\|Tf\|_{L^p(\Gamma)} \leq C \|f\|_{L^p(\Gamma)}\,.
\end{equation*}
In order to show this we observe that we may write
\begin{equation*}
Tf(x) = \int_U \frac{ \varphi(x') - \varphi(y') - \left\langle \nabla\varphi(x')\vert x'-y'\right\rangle }{(|x'-y'|^2+ [\varphi(x') -\varphi(y')]^2 )^{3/2}} f(y', \varphi(y')) \, dy'.
\end{equation*} 
where we used the fact that 
$$
\Gamma=\{(x',y') \, : \, y'-\varphi(x')=F(x',y')=0\}
$$
and then that
$$
\nu_E= \frac{\nabla F}{|\nabla F|}=\frac{(-\nabla \varphi(x'),1)}{\sqrt{1+|\nabla \varphi(x')|^2}}.
$$
Therefore,
$$
|Tf(x)|\leq C \int_U \frac{ |x'-y'|^{1+ \alpha}}{(|x'-y'|^2+ [\varphi(x') -\varphi(y')]^2 )^{3/2}}|f(y', \varphi(y'))| \, dy' \leq C \int_U \frac{ |f(y', \varphi(y'))|}{|x'-y'|^{2-\alpha}} \, dy'.
$$
Thus, inequality~\eqref{flow lemma 1} follows from a standard convolution estimate.\\
For $x \in E$ we have
\begin{equation*}
\nabla u(x) = \int_{\partial E} \nabla^{\T^3}_x\!G(x,y)f(y) \, d\mu(y),
\end{equation*}
hence, for $x \in \partial E$ there holds
\begin{equation*}
\langle \nabla u(x - t\nu_E(x)) \vert \, \nu_E(x)\rangle = \int_{\partial E} \bigl\langle\nabla^{\T^3}_x\!G(x -t \nu_E(x),y) \vert \, \nu_E(x)\bigr\rangle f(y) \, d\mu(y).
\end{equation*}
We claim that 
\begin{equation}\label{flow lemma 2}
{\pa_{\nu_E} u^-}(x)=\lim_{t \to 0+} \left\langle\nabla u(x - t\nu_E(x)) \vert \, \nu_E(x)\right\rangle = Kf(x) +\frac{1}{2}f(x), 
\end{equation}
for every $x \in \partial E$, then the result follows from inequality~\eqref{flow lemma 1} and this limit, together with the analogous identity for $\pa_{\nu_E} u^+(x)$.\\
To show equality~\eqref{flow lemma 2} we first observe that 
\begin{align}
&\int_{\partial E} \bigl\langle\nabla^{\T^3}_x\!G(x,y) \vert \,\nu_E(y)\bigr\rangle \, d\mu(y) = 1- \vol(E) \quad\quad\, \text{if $x\in E\setminus\pa E$} \label{flow lemma divergence0}\\
& \int_{\partial E} \bigl\langle \nabla^{\T^3}_x\!G(x ,y) \vert \, \nu_E(y)\bigr\rangle \, d\mu(y) = {1}/{2}- \vol(E)\quad \text{if $x\in \pa E$}. \label{flow lemma divergence} 
\end{align}
Indeed, using Definition~\eqref{potential}, we have 
\begin{align}
\Delta v_E(x) =&\,\int_E \Delta_x G(x,y) \, dy - \int_{E^c} \Delta_x G(x,y) \, dy \\
=&\,-2 \int_{\pa E} \bigl\langle\nabla^{\T^3}_x\!G(x,y)\vert \, \nu_E(y) \bigr\rangle \, \dmu(y)\\
=&\, 2\vol(E) - 1-u_E(x),
\end{align}
then,
\begin{equation}
\int_{\pa E} \bigl \langle\nabla^{\T^3}_x\! G(x,y) \vert \, \nu_E(y) \bigr\rangle\,d\mu(y) =  1/2-\vol(E) +u_E(x)/2,
\end{equation}
which clearly implies equation~\eqref{flow lemma divergence0}. Equality~\eqref{flow lemma divergence} instead follows by an approximation argument, after decomposing the Green function as at the beginning of the proof of point (i), $G(x,y) = h(x-y) + r(x-y)$, with $h(t)=\frac{1}{4\pi|t|}$ in a neighborhood of $0$ and $r:\R\to\R$ a smooth function.\\
Therefore, we may write, for $x\in\pa E$ and $t>0$ (remind that $\nu_E$ is the {\em outer} unit normal vector, hence $x - t\nu_E(x)\in E$),
\begin{align}
\langle\nabla u(x - t\nu_E(x)) \vert \nu_E(x)\rangle 
=&\,\int_{\partial E} \bigl\langle\nabla^{\T^3}_x\!G(x - t\nu_E(x),y) \vert \nu_E(x)\bigr\rangle (f(y)-f(x)) \, d\mu(y) \\
&\,+ f(x) \int_{\partial E}\bigl\langle \nabla_x^{\T^3}\!G(x - t\nu_E(x),y) \vert \nu_E(x)-\nu_E(y)\bigr\rangle \, d\mu(y) \\
& \,+ f(x)(1-\vol(E))\,,\label{flow lema long}
\end{align}
by equality~\eqref{flow lemma divergence0}.\\
Let us now prove that 
\begin{align*}
\lim_{t \to 0^+} \int_{\partial E} &\,\bigl\langle\nabla_x^{\T^3}\!G(x- t\nu_E(x),y) \,\vert\, \nu_E(x)\bigr\rangle (f(y)-f(x)) \, d\mu(y)\\
&= \int_{\partial E} \bigl\langle\nabla^{\T^3}_x\!G(x,y)\, \vert\, \nu_E(x)\bigr\rangle (f(y)-f(x)) \, d\mu(y),
\end{align*}
observing that since $\pa E$ is of class $C^{1,\alpha}$ then for $|t|$ sufficiently small we have
\begin{equation}\label{vesa1}
|x-y-t\nu_E(x)|\geq\frac12|x-y|\qquad\text{for all $y\in\pa E$}\,.
\end{equation}
Then, in view of the decomposition of $\nabla_xG$ above, it is enough show that 
\begin{equation*}
\lim_{t \to 0^+} \int_{\partial E}\frac{\langle x -y - t\nu_E(x) \, \vert \, \nu_E(x)\rangle}{|x-y-t\nu_E(x)|^3} (f(y)-f(x)) \, d\mu(y)= \int_{\partial E} \frac{\langle x -y \, \vert \, \nu_E(x)\rangle}{|x-y|^3} (f(y)-f(x)) \, d\mu(y)\,,
\end{equation*}
which follows from the dominated convergence theorem, after observing that due to the $\alpha$--H\"older continuity of $f$ and 
to inequality~\eqref{vesa1}, the absolute value of both integrands can be estimated from above by $C/|x-y|^{2-\alpha}$ for some constant $C>0$.\\
Arguing analogously, we also get
$$
\lim_{t \to 0^+} \int_{\partial E}\bigl\langle \nabla_x^{\T^3}\!G(x - t\nu_E(x),y) \vert \nu_E(x)-\nu_E(y)\bigr\rangle \, d\mu(y)=\int_{\partial E}\bigl\langle \nabla_x^{\T^3}\!G(x,y) \vert \nu_E(x)-\nu_E(y)\bigr\rangle \, d\mu(y)\,.
$$
Then, letting $t \to 0^+$ in equality~\eqref{flow lema long}, for every $x\in\pa E$, we obtain
\begin{align*}
\lim_{t\to 0^+}\langle\nabla u(x - t\nu_E(x)) \vert \nu_E(x)\rangle
=&\,\int_{\pa E} \bigl \langle \nabla^{\T^3}_x\!G(x,y)\vert \nu_E(x) \bigr\rangle (f(y)-f(x)) \, \dmu(y)\\
&\,+ f(x) \int_{\pa E} \bigl\langle\nabla^{\T^3}_x\!G(x,y)|\nu_E(x)-\nu_E(y)\bigr\rangle\,\dmu(y)+ f(x)(1-\vol(E))\\
=&\,\int_{\pa E} \bigl \langle \nabla^{\T^3}_x\!G(x,y)\vert \nu_E(x) \bigr\rangle f(y)\, \dmu(y)\\
&\,-f(x) \int_{\pa E} \bigl\langle\nabla^{\T^3}_x\!G(x,y)|\nu_E(y)\bigr\rangle\,\dmu(y)+ f(x)(1-\vol(E))\\
=&\,Kf(x)+ f(x)(\vol(E)-1/2) + f(x)(1-\vol(E))\\
=&\,Kf(x)+\frac{1}{2}f(x),
\end{align*}
where we used equality~\eqref{flow lemma divergence}, then limit~\eqref{flow lemma 2} holds and the thesis follows.

\smallskip

{(iv)} Fixed $p>2$ and $\beta\in (0, \frac{p-2}{p})$, as before, due to the properties of the Green's function, it is sufficient to establish the statement for the function 
$$
 v(x) =\int_{\pa E}\frac{f(y)}{|x-y|}\, d\mu(y)\,.
$$
For $x_1$, $x_2\in \pa E$ we have
$$
|v(x_1)-v(x_2)|\leq \int_{\pa E}|f(y)|\frac{\big||x_1-y|-|x_2-y|\big|}{|x_1-y|\, |x_2-y|}\, d\mu(y)\,.
$$
In turn, by an elementary inequality, we have
$$
\frac{\big||x_1-y|-|x_2-y|\big|}{|x_1-y|\, |x_2-y|}\leq C(\beta)\frac{\big||x_1-y|^{1-\beta}+|x_2-y|^{1-\beta}\big|}{|x_1-y|\, |x_2-y|}|x_1-x_2|^\beta\,,
$$
thus, by H\"older inequality we have
\begin{align*}
|v(x_1)-v(x_2)| &\leq C(\beta) \int_{\pa E}|f(y)|\frac{\big||x_1-y|^{1-\beta}+|x_2-y|^{1-\beta}\big|}{|x_1-y|\, |x_2-y|}\, d\mu(y)\,\, |x_1-x_2|^\beta \\
&\leq C'(\beta)\|f\|_{L^p} |x_1-x_2|^\beta\,, 
\end{align*}
where we set 
$$
C'(\beta)=2C(\beta)\biggl(\sup_{z_1,\, z_2\in\pa E} \int_{\pa E}\frac{1}{|z_1-y|^{\beta p'}\, |z_2-y|^{p'}}\, d\mu(y)\biggr)^{1/{p'}}\,,
$$
with $p'=p/(p-1)$.

\smallskip

\noindent For the second part of the lemma, we start by observing that
$$
\norma{f}_{L^2(\pa E)}\leq C\norma{f}_{H^1(\pa E)}^{1/{2}}\norma{f}^{1/{2}}_{H^{-1}(\pa E)}.
$$
If $p>2$ we have, by Gagliardo--Nirenberg interpolation inequalities (see~\cite[Theorem 3.70]{Aubin}),
\begin{equation*}
\| f\|_{L^p(\partial E)} \leq C \|f\|_{H^1(\partial E)}^{^{{(p-2)}/{p}}} \|f\|_{L^2(\partial E)}^{2/p}.
\end{equation*}
Therefore, by combining the two previous inequalities we get that, for $p\geq2$, there holds 
\begin{equation*}
\| f\|_{L^p(\partial E)} \leq C \|f\|_{H^1(\partial E)}^{^{(p-1)/{p}}} \|f\|_{H^{-1}(\partial E)}^{1/p}.
\end{equation*}
Hence, the thesis follows once we show
\begin{equation*}
\|f \|_{H^{-1}(\partial E)} \leq C \|u\|_{L^2(\partial E)}.
\end{equation*}
To this aim, let us fix $\varphi \in H^1(\partial E)$ and with a little abuse of notation denote its harmonic extension to $\T^3$ still by $\varphi$. Then, by integrating by parts twice and by point (ii), we get 
\begin{align}
\int_{\partial E} \varphi f \, d \mu=&\,
- \int_{\pa E} \varphi \Delta u \, \dmu\\
=&\, - \int_{\partial E} u [\pa_{\nu_E} \varphi] \, d \mu \\
\leq&\, \|u \|_{L^2(\partial E)} \bigl\|[\pa_{\nu_E} \varphi]\bigr \|_{L^2(\partial E)} \\
\leq&\, \|u \|_{L^2(\partial E)} \bigl( \|\pa_{\nu_E} \varphi^+ \|_{L^2(\partial E)} + \|\pa_{\nu_E} \varphi^- \|_{L^2(\partial E)} \bigr) \\
\leq&\, C \|u\|_{L^2(\partial E)} \|\varphi \|_{H^1(\partial E)}.
\end{align} 
Therefore,
\begin{equation*}
\| f\|_{H^{-1}(\partial E)} = \sup_{\|\varphi\|_{H^1(\partial E)}\leq 1} \int_{\partial E} \varphi f \, d \mu \leq C \|u \|_{L^2(\partial E)}
\end{equation*}
and we are done.
\end{proof}
For any smooth set $E\subseteq\T^3$, the fractional Sobolev space $W^{s,p}(\pa E)$, usually obtained via local charts and partitions of unity, has an equivalent definition considering directly the Gagliardo $W^{s,p}$--seminorm of a function $f\in L^p(\pa E)$, for $s\in(0,1)$, as follows
$$
[f]_{W^{s,p}(\pa E)}^p=\int_{\pa E}\int_{\pa E}\frac{|f(x)-f(y)|^p}{|x-y|^{2+sp}}\,d\mu(x)d\mu(y)
$$
and setting $\Vert f\Vert_{W^{s,p}(\pa E)}=\Vert f\Vert_{L^p(\pa E)}+[f]_{W^{s,p}(\pa E)}$ (we refer to~\cite{AdamsFournier,Dem,NePaVa,RuSi} for details). As it is customary, we set $[f]_{H^s(\pa E)}=[f]_{W^{s,2}(\pa E)}$ and $H^s(\pa E)=W^{s,2}(\pa E)$.

Then, it can be shown that for all the sets $E\in\mathfrak{C}^{1,\alpha}_M(F)$, given a smooth set $F\subseteq\T^3$ and a tubular neighborhood $N_\eps$ of $\pa F$ as in formula~\eqref{tubdef}, for any $M\in(0,\eps/2)$ and $\alpha\in(0,1)$, the constants giving the equivalence between this norm above and the ``standard'' norm of $W^{s,p}(\pa E)$ can be chosen to be uniform, independent of $E$. Moreover, as for the ``usual'' (with integer order) Sobolev spaces, all the constants in the embeddings of the fractional Sobolev spaces are also uniform for this family. This is related to the possibility, due to the closeness in $C^{1,\alpha}$ and the graph representation, of ``localizing'' and using partitions of unity ``in a single common way'' for all the smooth hypersurfaces $\pa E$ boundaries of the sets $E\in\mathfrak{C}^{1,\alpha}_M(F)$, see~\cite{DDM} for details.

Then, we have the following technical lemma. 

\begin{lem}\label{nicola1}
Let $F\subseteq\T^3$ be a smooth set and $E\in\mathfrak{C}^{1,\alpha}_M(F)$. For every $\beta\in[0,1/2)$, there exists a constant $C=C(F,M,\alpha,\beta)$ such that if $f\in H^{1/2}(\pa E)$ and $g\in W^{1,4}(\pa E)$, then
$$
[fg]_{H^{{1}/{2}}(\pa E)}\leq C[f]_{H^{{1}/{2}}(\pa E)}\|g\|_{L^\infty(\pa E)}+C\|f\|_{L^{\frac{4}{1+\beta}}(\pa E)}\|g\|_{L^\infty(\pa E)}^\beta\|\nabla g\|_{L^4(\pa E)}^{1-\beta}\,.
$$
\end{lem}
\begin{proof}
We estimate with H\"older inequality, noticing that $6\beta/(1+\beta)<2$, as $\beta\in[0,1/2)$, hence there exists $\delta>0$ such that $(6\beta+\delta)/(1+\beta)<2$, 
\begin{align*}
[fg]_{H^{{1}/{2}}(\pa E)}^2
\leq&\,2[f]_{H^{{1}/{2}}(\pa E)}^2\|g\|_{L^\infty(\pa E)}^2+2\int_{\pa E}\int_{\pa E}|f(y)|^2\frac{|g(x)-g(y)|^2}{|x-y|^3}\,d\mu(x)d\mu(y)\\
\leq &\,2[f]_{H^{{1}/{2}}(\pa E)}^2\|g\|_{L^\infty(\pa E)}^2\\
&\,+C\int_{\pa E}\int_{\pa E}\frac{|f(y)|^2}{|x-y|^{3\beta+\delta/2}}\frac{|g(x)-g(y)|^{2(1-\beta)}}{|x-y|^{3(1-\beta)-\delta/2}}\Vert g\Vert_{L^\infty(\pa E)}^{2\beta}\,d\mu(x)d\mu(y)\\
\leq &\,2[f]_{H^{{1}/{2}}(\pa E)}^2\|g\|_{L^\infty(\pa E)}^2\\
&\,+C\Bigl(\int_{\pa E}{|f(y)|^{\frac{4}{1+\beta}}}\int_{\pa E}\frac{1}{|x-y|^{\frac{6\beta+\delta}{1+\beta}}}\,d\mu(x)d\mu(y)\Bigr)^{(1+\beta)/2}\Vert g\Vert_{L^\infty(\pa E)}^{2\beta}\\
&\,\phantom{+C\quad}\cdot\Bigl(\int_{\pa E}\int_{\pa E}\frac{|g(x)-g(y)|^4}{|x-y|^{6-\frac{\delta}{1-\beta}}}\,d\mu(x)d\mu(y)\Bigr)^{(1-\beta)/2}\\
\leq &\,2[f]_{H^{{1}/{2}}(\pa E)}^2\|g\|_{L^\infty(\pa E)}^2\\
&\,+C\Bigl(\int_{\pa E}{|f(y)|^{\frac{4}{1+\beta}}}\,d\mu(y)\Bigr)^{(1+\beta)/2}\Vert g\Vert_{L^\infty(\pa E)}^{2\beta}\,
[g]_{W^{1-\frac{\delta}{4(1-\beta)},4}(\pa E)}^{2(1-\beta)}\\
\leq &\,2[f]_{H^{{1}/{2}}(\pa E)}^2\|g\|_{L^\infty(\pa E)}^2+C\|f\|_{L^{\frac{4}{1+\beta}}(\pa E)}^2\|g\|_{L^\infty(\pa E)}^{2\beta}\|\nabla g\|_{L^4(\pa E)}^{2(1-\beta)}\,.
\end{align*}
Hence the thesis follows noticing that all the constants $C$ above depend only on $F$, $M$, $\alpha$ and $\beta$, by the previous discussion, before the lemma.
\end{proof}

As a corollary we have the following estimate.

\begin{lem}\label{5.2}
Let $F\subseteq\T^3$ be a smooth set and $E\in\mathfrak{C}^{1,\alpha}_M(F)$. Then, for $M$ small enough, there holds
$$
\|\psi_E\|_{W^{5/2,2}(\pa F)}\leq C(F,M,\alpha)\big(1+\|\HHH\|_{H^{1/2}(\pa E)}^2\big)\,,
$$
where $\HHH$ is the mean curvature of $\pa E$ and the function $\psi_E$ is defined by formula~\eqref{front}.
\end{lem}

\begin{proof} By a standard localization/partition of unity/straightening argument, we may reduce ourselves to the case where the function $\psi_E$ is defined in a disk $D\subseteq\R^2$ and $\|\psi_E\|_{C^{1,\alpha}(D)}\leq M$. Fixed a smooth cut--off function $\varphi$ with compact support in $D$ and equal to one on a smaller disk $D'\subseteq D$, we have 
\begin{equation}\label{noia1}
\Delta(\varphi\psi_E)-\frac{\nabla^2(\varphi\psi_E)\nabla\psi_E \nabla\psi_E}{1+|\nabla\psi_E|^2}=\varphi\HHH\sqrt{1+|\nabla\psi_E|^2}+R(x,\psi_E,\nabla\psi_E)\,,
\end{equation}
where the remainder term $R(x, \psi_E, \nabla \psi_E)$ is a smooth Lipschitz function. Then, using Lemma~\ref{nicola1} with $\beta=0$ and recalling that $\|\psi_E\|_{C^{1,\alpha}(D)}\leq M$, we estimate
\begin{align}
[\Delta(\varphi\psi_E)]_{H^{{1}/{2}}(D)}\leq C(F,M,\alpha)\Bigl(&\,M^2[\nabla^2(\varphi\psi_E)]_{H^{{1}/{2}}(D)}+[\HHH]_{H^{{1}/{2}}(\pa E)}(1+\norma{\nabla\psi_E}_{L^\infty(D)})\\
&\,+\norma{\HHH}_{L^4(\pa E)}(1+\norma{\psi_E}_{W^{2,4}(D)})+1+\norma{\psi_E}_{W^{2,4}(D)}\Bigr)\,.
\end{align}
We now use the fact that, by a simple integration by part argument, if $u$ is a smooth function with compact support in $\R^2$, there holds
$$
[\Delta u]_{H^{{1}/{2}}(\R^2)}=[\nabla^2u]_{H^{{1}/{2}}(\R^2)}\,,
$$
hence,
\begin{align}
[\nabla^2(\varphi \psi_E)]_{H^{{1}/{2}}(D)}&\,=[\Delta (\varphi\psi_E)]_{H^{{1}/{2}}(D)}\\
&\,\leq C(F,M,\alpha)\Big(M^2[\nabla^2(\varphi\psi_E)]_{H^{{1}/{2}}(D)}
+[\HHH]_{H^{{1}/{2}}(\pa E)}(1+\norma{\nabla\psi_E}_{L^\infty(D)})\\
&\,\phantom{\leq C(F,M,\alpha)\bigl(\,}+\norma{\HHH}_{L^4(\pa E)}(1+\norma{\psi_E}_{W^{2,4}(D)})+1+\norma{\psi_E}_{W^{2,4}(D)}\Big),
\end{align}
then, if $M$ is small enough, we have
\beq \label{aub}
[\nabla^2(\varphi\psi_E)]_{H^{{1}/{2}}(D)}\leq C(F,M,\alpha)(1+ \norma{\HHH}_{H^{1/2}(\pa E)} )(1+\norma{\mathrm{Hess}\,\psi_E}_{L^4(D)}),
\eeq
as 
\beq\label{aub2}
\norma{\HHH}_{L^4(\pa E)} \leq C(F,M,\alpha) \norma{\HHH}_{H^{1/2}(\pa E)},
\eeq
where we used the continuous embedding of $H^{1/2}(\pa E)$ in $L^4(\pa E)$ (see for instance Theorem~$6.7$ in~\cite{NePaVa}, with $q=4$, $s=1/2$ and $p=2$).\\
By the Calder\'on--Zygmund estimates (holding uniformly for every hypersurface $\pa E$, with $E\in\mathfrak{C}^{1,\alpha}_M(F)$, see~\cite{DDM}),
\begin{equation}\label{Cald-Zyg}
\norma{\mathrm{Hess} \, \psi_E}_{L^{4}(D)} \leq C(F,M,\alpha)(\norma{\psi_E}_{L^4(D)}+ \norma{\Delta \psi_E}_{L^4(D)})
\end{equation}
and the expression of the mean curvature 
\begin{equation}\label{exprmeancurv}
\HHH= \frac{\Delta \psi_E}{\sqrt{1+ |\nabla \psi_E|^2}} - \frac{\mathrm{Hess} \, \psi_E (\nabla \psi_E \nabla \psi_E)}{(\sqrt {1 + |\nabla \psi_E|^2})^3} \, ,
\end{equation}
we obtain 
\begin{align}
\norma{\Delta \psi_E}_{L^4(D)} &\leq 2M \norma{\HHH}_{L^4(\pa E)} + M^2 \norma{\mathrm{Hess}\, \psi_E}_{L^4(D)} \\
&\leq 2M \norma{\HHH}_{L^4(\pa E)} + C(F,M,\alpha)M^2(\norma{\psi_E}_{L^4(D)} + \norma{\Delta \psi_E}_{L^4(D)}) \, . \label{normadeltapsi1}
\end{align}
Hence, possibly choosing a smaller $M$, we conclude
\begin{equation}
\norma{\Delta \psi_E}_{L^4(D)} \leq C(F,M,\alpha) (1 + \norma{\HHH}_{L^4(\pa E)})\leq C(F,M,\alpha)(1+ \norma{\HHH}_{H^{1/2}(\pa E)}), \label{normadeltapsi2}
\end{equation}
again by inequality~\eqref{aub2}.\\
Thus, by estimate~\eqref{Cald-Zyg}, we get
\beq\label{eqcar10030}
\norma{\mathrm{Hess}\,\psi_E}_{L^4(D)} \leq C(F,M,\alpha)(1+ \norma{\HHH}_{H^\frac12(\pa E)}),
\eeq
and using this inequality in estimate~\eqref{aub},
\beq
[\nabla^2(\varphi\psi_E)]_{H^{{1}/{2}}(D)}\leq C(F,M,\alpha)(1+ \norma{\HHH}_{H^\frac12(\pa E)})^2,
\eeq
hence,
\beq
[\nabla^2 \psi_E]_{H^{1/2}(D')}\leq C(F,M,\alpha)(1+ \norma{\HHH}_{H^\frac12(\pa E)})^2 \leq C(F,M,\alpha)(1+ \norma{\HHH}_{H^\frac12(\pa E)}^2).
\eeq
The inequality in the statement of the lemma then easily follows by this inequality, estimate~\eqref{eqcar10030} and $\norma{\psi_E}_{C^{1,\alpha}(D)} \leq M$, with a standard covering argument.
\end{proof}

We are now ready to prove the last lemma of this section.

\begin{lem}[Compactness]\label{w52conv}
Let $F\subseteq\T^3$ be a smooth set and $E_n\subseteq \mathfrak{C}^{1,\alpha}_M(F) $ a sequence of smooth sets such that 
$$
\sup_{n\in\N}\int_{\T^3}|\nabla w_{E_n}|^2\, dx<+\infty\,,
$$
where $w_{E_n}$ are the functions associated to $E_n$ by problem~\eqref{WE}.\\
Then, if $\alpha\in(0,1/2)$ and $M$ is small enough, there exists a smooth set $F'\in \mathfrak{C}^1_M(F)$ such that, up to a (non relabeled) subsequence, $E_n\to F'$ in $W^{2,p}$ for all $1\leq p<4$ (recall the definition of convergence of sets at the beginning of Subsection~\ref{stabsec}).\\
Moreover, if
$$
\int_{\T^3}|\nabla w_{E_n}|^2\, dx\to 0\,,
$$
then $F'$ is critical for the volume--constrained nonlocal Area functional $J$ and the convergence $E_n\to F'$ is in $W^{5/2,2}$.
\end{lem}
\begin{proof} Throughout all the proof we write $w_n$, $\HHH_n$, and $v_n$ instead of $w_{E_n}$, $\HHH_{\pa E_n}$, and $v_{E_n}$, respectively. Moreover, we denote by $\widehat w_n = \fint_{\T^3}w_n\,dx$ and we set $\widetilde w_n=\fint_{\pa E_n}w_n\,d\mu_n$ and $\widetilde{\HHH}_n=\fint_{\pa E_n}\HHH_n\,d\mu_n$.

First, we recall that 
\begin{equation}\label{w521}
w_n=\HHH_n+4\gamma v_n \quad\text{on }\pa E_n\qquad\text{and}\qquad \sup_{n\in\N}\,\|v_n\|_{C^{1,\alpha}(\T^3)}<+\infty\,,
\end{equation}
by standard elliptic estimates. 
We want to show that 
\begin{equation}\label{eqcar15000}
\norma{w_n-\widetilde w_n}^2_{H^{1/2}(\pa E_n)}\leq\|w_n-\widehat w_n\|^2_{H^{1/2}(\pa E_n)}.
\end{equation}
To this aim, we recall that for every constant $a$
\begin{equation}
\norma{w_n - a}_{L^2(\pa E_n)}^2 = \norma{w_n}^2_{L^2(\pa E_n)} + a^2 \A(\pa E_n) - 2 a \int_{\pa E_n} w_n \, d\mu_n
\end{equation}
then, 
\begin{equation}
\frac{d}{da} \norma{w_n -a}_{L^2(\pa E_n)}^2= 2a \A(\pa E_n) - 2 \int_{\pa E_n} w_n \, d\mu_n .
\end{equation}
The above equality vanishes if and only if $a= \fint _{\pa E_n} w_n \, d\mu_n$, hence, 
$$
\norma{w_n -\widetilde w_n}_{L^2(\pa E_n)} = \min_{a\in\R}\norma{w_n -a}_{L^2(\pa E_n)}
$$
and inequality~\eqref{eqcar15000} follows by the definition of $\Vert\cdot\Vert_{H^{1/2}(\pa E_n)}$ and the observation on the Gagliardo seminorms just before Lemma~\ref{nicola1}.\\
Then, from the {\em trace inequality} (see~\cite{Ev}), which holds with a ``uniform'' constant $C=C(F,M,\alpha)$, for all the sets $E\in\mathfrak{C}^{1,\alpha}_M(F)$ (see~\cite{DDM}), we obtain
\begin{equation}\label{w522}
\|w_n-\widetilde w_n\|^2_{H^{1/2}(\pa E_n)}\leq\|w_n-\widehat w_n\|^2_{H^{1/2}(\pa E_n)}\leq C\int_{\T^3}|\nabla w_n|^2\, dx<C<+\infty
\end{equation}
with a constant $C$ independent of $n\in\N$.\\
We claim now that 
\begin{equation}\label{claim1111}
\sup_{n\in\N}\,\|\HHH_n\|_{H^{1/2}(\pa E_n)}<+\infty.
\end{equation}
To see this note that by the uniform $C^{1,\alpha}$--bounds on $\pa E_n$, we may find a fixed solid cylinder of the form $C=D\times(-L,L)$, with $D\subseteq\R^{2}$ a ball centered at the origin and functions $f_n$, with
\begin{equation}\label{w523}
\sup_{n\in\N}\|f_n\|_{C^{1,\alpha}(\overline D)}<+\infty\,,
\end{equation}
such that $\pa E_n\cap C=\{(x',x_n)\in D\times(-L,L):\, x_n= f_n(x')\}$ with respect to a suitable coordinate frame (depending on $n\in\N$).  Then,
\begin{equation}
\int_{D}(\HHH_n-\widetilde{\HHH}_n)\, dx'+ \widetilde{\HHH}_n\,{\mathrm{Area}}(D)= \int_{D}\Div\Bigl(\frac{\nabla_{x'} f_n}{\sqrt{1+|\nabla_{x'} f_n|^2}}\Bigr)\, dx'=\int_{\pa D}\frac{\nabla_{x'} f_n}{\sqrt{1+|\nabla_{x'} f_n|^2}}\cdot \frac{x'}{|x'|}\, d\sigma\,, \label{intHn}
\end{equation}
where $\sigma$ is the canonical (standard) measure on the circle $\pa D$.\\
Hence, recalling the uniform bound~\eqref{w523} and the fact that $\|\HHH_n- \widetilde{\HHH}_n\|_{H^{1/2}(\pa E_n)}$ are equibounded thanks to inequalities~\eqref{w521} and~\eqref{w522}, we get that ${\widetilde{\HHH}_n}$ are also equibounded (by a standard ``localization'' argument, ``uniformly'' applied to all the hypersurfaces $\pa E_n$). Therefore, the claim~\eqref{claim1111} follows.\\
By applying the Sobolev embedding theorem on each connected component of $\partial F$, we have that 
$$
\norma{\HHH_n}_{L^p(\pa E_n)} \leq C \norma{\HHH_n}_{H^\frac{1}{2}(\pa E_n)} <C<+\infty\qquad \text{for all $p \in [1,4]$.}
$$
for a constant $C$ independent of $n\in\N$.\\
Now, by means of Calder\'on--Zygmund estimates, it is possible to show (see~\cite{DDM}) that there exists a constant $C>0$ depending only on $F$, $M$, $\alpha$ and $p>1$ such that for every $E\in \mathfrak{C}^{1,\alpha}_M(F)$, there holds 
\begin{equation}\label{CZG}
\norma{B}_{L^p(\pa E)} \leq C(1+  \norma{\HHH}_{L^p(\pa E)})\,.
\end{equation}
Then, if we write 
$$
\pa E_n =\{y+\psi_n(y)\nu_F(y):\, y\in \pa F\}\,,
$$ 
we have $\sup_{n\in\N}\|\psi_n\|_{W^{2,p}(\pa F)}<+\infty$, for all $p \in [1,4]$. Thus, by the Sobolev compact embedding $W^{2,p}(\pa F)\hookrightarrow C^{1,\alpha}(\pa F)$, up to a subsequence (not relabeled), there exists a set $F'\in \mathfrak{C}^{1,\alpha}_M(F)$ such that 
$$
\psi_n\to \psi_{F'} \text{ in $C^{1,\alpha}(\pa F)$}\quad\text{and}\quad v_n\to v_{F'}\text{ in $C^{1,\beta}(\T^3)$ }
$$
for all $\alpha\in (0,1/2)$ and $\beta\in (0,1)$.\\
From estimate~\eqref{claim1111} and Lemma~\ref{5.2} (possibly choosing a smaller $M$), we have then that the functions $\psi_n$ are bounded in $W^{5/2,2}(\pa F)$. Hence, possibly passing to another subsequence (again not relabeled), we conclude that $E_n \to F'$ in $W^{2,p}$ for every $p\in[1,4)$, by the Sobolev compact embedding (see for instance Theorem~$6.7$ in~\cite{NePaVa}, with $q\in[1,4)$, $s=1/2$ and $p=2$, applied to ${\mathrm{Hess}}\, \psi_n$).

If moreover we have
$$
\int_{\T^3}|\nabla w_n|^2\, dx\to 0 \, ,
$$
then, the above arguments yield the existence of $\lambda\in \R$ and a subsequence (not relabeled) such that $w_n\big(\cdot + \psi_n(\cdot)\nu_F(\cdot)\big)\to \lambda$ in $H^{1/2}(\pa F)$.
In turn,
$$
\HHH_n\big(\cdot + \psi_n(\cdot)\nu_F(\cdot)\big)\to \lambda-4\gamma v_{F'}\big(\cdot + \psi_{F'}(\cdot)\nu_F(\cdot)\big)=
\HHH\big(\cdot + \psi_{F'}(\cdot)\nu_F(\cdot)\big)
$$
in $H^{1/2}(\pa F)$, where $\HHH$ is the mean curvature of $F'$. Hence $F'$ is critical.\\
To conclude the proof we then only need to show that $\psi_n$ converge to $\psi =\psi_{F'}$ in $W^{5/2,2}(\pa F)$.\\
Fixed $\delta>0$, arguing as in the proof of Lemma~\ref{5.2}, we reduce ourselves to the case where the functions $\psi_n$ are defined on a disk $D\subseteq\R^2$, are bounded in $W^{5/2,2}(D)$, converge in $W^{2,p}(D)$ for all $p\in[1,4)$ to $\psi\in W^{5/2,2}(D)$ and  $\|\nabla\psi\|_{L^\infty(D)}\leq\delta$. 
Then, fixed a smooth cut--off function $\varphi$ with compact support in $D$ and equal to one on a smaller disk $D'\subseteq D$, we have
\begin{align*}
\frac{\Delta(\varphi\psi_n)}{\sqrt{1+|\nabla\psi_n|^2}}-\frac{\Delta(\varphi\psi)}{\sqrt{1+|\nabla\psi|^2}} =&\,(\nabla^2(\varphi\psi_n)-\nabla^2(\varphi\psi))\frac{\nabla\psi \nabla\psi}{(1+|\nabla\psi|^2)^{3/2}}\\
&\,+ \nabla^2(\varphi\psi_n)\Bigl(\frac{\nabla\psi_n\nabla\psi_n}{(1+|\nabla\psi_n|^2)^{3/2}}-\frac{\nabla\psi \nabla\psi}{(1+|\nabla\psi|^2)^{3/2}}\Bigr)\\
&\,+\varphi(\HHH_n-\HHH)+R(x,\psi_n,\nabla\psi_n)-R(x,\psi,\nabla\psi)\,,
\end{align*}
where $R$ is a smooth Lipschitz function. Then, using Lemma~\ref{nicola1} with $\beta\in(0,1/2)$, an argument similar to the one in the proof of Lemma~\ref{5.2} shows that
\begin{align*}
 &\bigg[\frac{\Delta(\varphi\psi_n)}{\sqrt{1+|\nabla\psi_n|^2}}-\frac{\Delta(\varphi\psi)}{\sqrt{1+|\nabla\psi|^2}} \bigg]_{H^{1/2}(D)}\leq C(M)\Bigl(\delta^2[\nabla^2(\varphi\psi_n)-\nabla^2(\varphi\psi)]_{H^{1/2}(D)}\\
& \qquad+\|\nabla^2(\varphi\psi_n)-\nabla^2(\varphi\psi)\|_{L^{\frac{4}{1+\beta}}(D)}\|\nabla\psi\|_{L^\infty(D)}^\beta\|\nabla^2\psi\|_{L^4(D)}^{1-\beta}+\\
& \qquad+\,[\nabla^2(\varphi\psi_n)]_{H^{1/2}(D)}\|\nabla\psi_n-\nabla\psi\|_{L^\infty(D)}\\
& \qquad+\|\nabla^2(\varphi\psi_n)\|_{L^{\frac{4}{1+\beta}}(D)}\|\nabla\psi_n-\nabla\psi\|_{L^\infty(D)}^\beta(\|\nabla^2\psi_n\|_{L^4(D)}+\|\nabla^2\psi\|_{L^4(D)})^{1-\beta}\\
& \qquad+\|\HHH_n-\HHH\|_{H^{1/2}(D)}+\|\psi_n-\psi\|_{W^{2,2}(D)}\Bigr)\,.
\end{align*}
Using Lemma~\ref{nicola1} again to estimate $[\Delta(\varphi\psi_n)-\Delta(\varphi\psi)]_{H^{1/2}(D)}$ with the seminorm on the left hand side of the previous inequality and arguing again as in the proof of Lemma~\ref{5.2}, we finally get
$$
[\nabla^2(\varphi\psi_n)-\nabla^2(\varphi\psi)]_{H^{1/2}(D)}\leq C(M)\Bigl(\|\psi_n-\psi\|_{W^{2,\frac{4}{1+\beta}}(D)}
+\|\nabla\psi_n-\nabla\psi\|_{L^\infty(D)}^\beta+\|\HHH_n-\HHH\|_{H^{1/2}(D)}\Bigr)\,,
$$
hence,
$$
[\nabla^2\psi_n-\nabla^2\psi]_{H^{1/2}(D')}\leq C(M)\Bigl(\|\psi_n-\psi\|_{W^{2,\frac{4}{1+\beta}}(D')}
+\|\nabla\psi_n-\nabla\psi\|_{L^\infty(D')}^\beta+\|\HHH_n-\HHH\|_{H^{1/2}(D')}\Bigr)\,,
$$
from which the conclusion follows, by the first part of the lemma with $p=4/(1+\beta)<4$ and a standard covering argument.
\end{proof}

\subsection{The modified Mullins--Sekerka flow -- The main theorem}\ \vskip.3em

We are ready to prove the long time existence/stability result.
 
\begin{thm}\label{existence}
Let $E\subseteq\T^n$ be a smooth strictly stable critical set for the nonlocal Area functional under a volume constraint and $N_\eps$ (with $\eps<1$) a tubular neighborhood of $\pa E$, as in formula~\eqref{tubdef}. For every $\alpha\in (0,1/2)$ there exists $M>0$ such that, if $E_0$ is a smooth set in $\mathfrak{C}^{1,\alpha}_M(E)$ satisfying $\vol( E_0)= \vol( E )$ and
$$ 
\int_{\T^3} \vert \nabla w_{E_0 }\vert^2\,dx \leq M\,
$$
where $w_0=w_{E_0}$ is the function relative to $E_0$ as in problem~\eqref{WE}, then the unique smooth solution $E_t$ of the modified Mullins--Sekerka flow (with parameter $\gamma\geq 0$) starting from $E_0$, given by Theorem~\ref{th:EscNis}, is defined for all $t\geq0$. Moreover, $E_t\to E+\eta$ exponentially fast in $W^{5/2,2}$ as $t\to +\infty$ (recall the definition of convergence of sets at the beginning of Subsection~\ref{stabsec}), for some $\eta\in \R^3$, with the meaning that the functions 
$\psi_{\eta, t} : \pa E+ \eta \to \R$ representing $\pa E_t$ as ``normal graphs'' on $\pa E + \eta$, that is,
$$
\pa E_t= \{ y+ \psi_{\eta,t} (y) \nu_{E+\eta}(y) \, : \, y \in \pa E+\eta \},
$$
satisfy
$$
\Vert \psi_{\eta, t}\Vert_{W^{5/2,2}(\pa E + \eta)}\leq Ce^{-\beta t},
$$
for every $t\in[0,+\infty)$, for some positive constants $C$ and $\beta$.
\end{thm}

\begin{remark}\label{existence+}
With some extra effort, arguing as in the proof of Theorem~5.1 in~\cite{FusJulMor18} (last part -- see also Theorem~4.4 in the same paper), it can be shown that the convergence of $E_t\to E+\eta$ is actually smooth (see also Remark~\ref{existence2+}). Indeed, by means of standard parabolic estimates and interpolation (and Sobolev embeddings) the exponential decay in $W^{5/2,2}$ implies analogous estimates in $C^k$, for every $k\in\N$,
$$
\Vert \psi_{\eta, t}\Vert_{C^k(\pa E + \eta)}\leq C_ke^{-\beta_k t},
$$
for every $t\in[0,+\infty)$, for some positive constants $C_k$ and $\beta_k$.
\end{remark}

\begin{remark}\label{closedness}
We already said that the property of a set $E_0$ to belong to $\mathfrak{C}^{1,\alpha}_M(E)$ is a ``closedness'' condition in $L^1$ of $E_0$ and $E$ and in $C^{1,\alpha}$ of their boundaries. The extra condition in the theorem on the $L^2$--smallness of the gradient of $w_0$ (see the second part of Lemma~\ref{w52conv} and its proof) implies that the quantity $\HHH_0+4\gamma v_0$ on $\pa E_0$ is ``close'' to be constant, as it is the analogous quantity for the set $E$ (or actually for any critical set). Notice that this is a second order condition for the boundary of $E_0$, in addition to the first order one $E_0\in \mathfrak{C}^{1,\alpha}_M(E)$.
\end{remark}

\begin{proof}[Proof of Theorem~\ref{existence}]
Throughout the whole proof $C$ will denote a constant depending only on $E$, $M$ and $\alpha$, whose value may vary from line to line.
 
Assume that the modified Mullins--Sekerka flow $E_t$ is defined for $t$ in the maximal time interval $[0,T(E_0))$, where $T(E_0)\in (0,+\infty]$ and let the moving boundaries $\pa E_t$ be represented as ``normal graphs'' on $\pa E$ as
$$
\pa E_t= \{ y+ \psi_t(y) \nu_{E}(y) \, : \, y \in \pa E\},
$$
for some smooth functions $\psi_t:\pa E\to \R$. As before we set $\nu_t=\nu_{E_t}$, $v_t=v_{E_t}$ and $w_t=w_{E_t}$.\\
We recall that, by Theorem~\ref{th:EscNis}, for every $F\in \mathfrak{C}^{2,\alpha}_M(E)$, the flow is defined in the time interval $[0, T)$, with $T=T(E,M,\alpha)>0$.\\
We show the theorem for the smooth sets $E_0\subseteq\T^3$ satisfying
\begin{equation}
\vol(E_0\triangle E)\leq M_1,\quad\|\psi_0\|_{C^{1,\alpha}(\pa E)}\leq M_2\quad\text{and}\quad \int_{\T^3} |\nabla w_0|^2\, dx\leq M_3\,,
\end{equation}
for some positive constants $M_1,M_2,M_3$, then we get the thesis by setting $M=\min\{M_1,M_2,M_3\}$.\\
For any set $F\in \mathfrak{C}^{1,\alpha}_{M}(E)$ we introduce the following quantity
\begin{equation}\label{D(F)0}
D(F)=\int_{F\Delta E}d(x, \pa E)\, dx=\int_F d_E\, dx-\int_E d_E\, dx,
\end{equation}
where $d_E$ is the signed distance function defined in formula~\eqref{sign dist}. We observe that 
$$
\vol(F\Delta E)\leq C\|\psi_F\|_{L^1(\pa E)} \leq C\|\psi_F\|_{L^2(\pa E)}
$$
for a constant $C$ depending only on $E$ and, as $F\subseteq N_\eps$, 
\begin{equation}\label{D(F)0bis}
D(F)\leq \int_{F\Delta E}\eps\, dx\leq \eps\vol(F\Delta E).
\end{equation}
Moreover,
\begin{align}
\|\psi_F\|_{L^2(\pa E)}^2&=2\int_{\pa E} \int_0 ^{|\psi_F(y)|} t \, dt\,d\mu(y) \\
&=2\int_{\pa E} \int_0^{|\psi_F(y)|} d(L(y,t),\pa E) \, dt\,d\mu(y) \\
&=2\int_{E\Delta F} d(x,\pa E)\,JL^{-1}(x)\,dx\\
&\leq C D(F).
\end{align}
where $L:\partial E\times (-\eps,\eps)\to N_\eps$ the smooth diffeomorphism defined in formula~\eqref{eqcar410} and $JL$ its Jacobian. As we already said, the constant $C$ depends only on $E$ and $\eps$. This clearly implies
\begin{equation}\label{D(F)}
\vol(F\Delta E)\leq C\|\psi_F\|_{L^1(\pa E)} \leq C\|\psi_F\|_{L^2(\pa E)}\leq C\sqrt{D(F)}\,.
\end{equation}
Hence, by this discussion, the initial smooth set $E_0\in\mathfrak{C}^{1,\alpha}_M(E)$ satisfies
$D(E_0)\leq M\leq M_1$ (having chosen $\eps<1$).\\
By rereading the proof of Lemma~\ref{w52conv}, it follows that for $M_2,M_3$ small enough, if $\Vert\psi_F\Vert_{C^{1,\alpha}(\pa E)}\leq M_2$ and
\begin{equation}\label{ex-de02}
\int_{\T^3} |\nabla w_{F}|^2\, dx \leq M_3\,,
\end{equation}
then, 
\begin{equation}\label{eqcar50001}
\|\psi_F\|_{W^{2,3}(\pa E)}\leq \omega(\max\{M_2,M_3\})\,,
\end{equation}
where $s\mapsto\omega(s)$ is a positive nondecreasing function (defined on $\R$) such that $\omega(s)\to 0$ as $s\to 0^+$. Hence,   
\begin{equation}\label{eqcar50003}
\|\nu_F\|_{W^{1,3}(\pa F)}\leq \omega'(\max\{M_2,M_3\})\,,
\end{equation}
for a function $\omega'$ with the same properties of $\omega$. Both $\omega$ and $\omega'$ only depend on $E$ and $\alpha$, for $M$ small enough.
 
We split the proof of the theorem into steps. 
 
\smallskip

\noindent \textbf{Step ${\mathbf 1}$} ({Stopping--time})\textbf{.}\\ Let $\overline T\leq T(E_0)$ be the maximal time such that 
\begin{equation}\label{Tprimo}
\vol(E_t\triangle E)\leq 2M_1,\quad\|\psi_t\|_{C^{1,\alpha}(\pa E)}\leq 2M_2\quad\text{and}\quad \int_{\T^3} |\nabla w_t|^2\, dx\leq 2M_3\,,
\end{equation}
for all $t\in [0, \overline T)$. Hence, 
\begin{equation}\label{eqcar50005}
\|\psi_t\|_{W^{2,3}(\pa E)}\leq \omega(2\max\{M_2,M_3\})\,
\end{equation}
for all $t\in [0, \overline T')$, as in formula~\eqref{eqcar50001}. Note that such a maximal time is clearly positive, by the hypotheses on $E_0$.\\
We claim that by taking $M_1,M_2,M_3$ small enough, we have $\overline T=T(E_0)$.
 
\smallskip

\noindent \textbf{Step ${\mathbf 2}$} ({Estimate of the translational component of the flow})\textbf{.}\\ We want to see that there exists a small constant $\theta>0$ such that
\begin{equation} \label{not a translation}
\min_{\eta\in\OO_E}\big\|\,[\pa_{\nu_t}w_t]- \langle\eta \, \vert \, \nu_t\rangle \big\|_{L^2(\pa E_t)}\geq \theta\bigl\|[\pa_{\nu_t}w_t]\bigr\|_{L^2(\pa E_t)}\qquad\text{for all }t\in [0, \overline T)\,,
\end{equation}
where $\OO_E$ is defined by formula~\eqref{OOeq}.\\
If $M$ is small enough, clearly there exists a constant $C_0=C_0(E,M,\alpha)>0$ such that, for every $i\in\II_E$, we have $\Vert \langle e_i|\nu_t\rangle\Vert_{L^2(\pa E_t)}\geq C_0>0$, holding $\Vert \langle e_i|\nu_E\rangle\Vert_{L^2(\pa E)}>0$. It is then easy to show that the vector $\eta_t\in\OO_E$ realizing such minimum is unique and satisfies
\begin{equation} \label{not a translation 2}
[\partial_{\nu_t} w_t] = \langle \eta_t\, \vert \nu_t \rangle + g,
\end{equation} 
where $g\in L^2(\pa E_t)$ is a function $L^2$--orthogonal (with respect to the measure $\mu_t$ on $\partial E_t$) to the vector subspace of $L^2(\pa E_t)$ spanned by $\langle e_i|\nu_t\rangle$, with $i\in\II_E$, where $\{e_1,\dots,e_3\}$ is the orthonormal basis of $\R^3$ given by Remark~\ref{rembase}. Moreover, the inequality 
\begin{equation}\label{eqcar50002}
|\eta_t|\leq C\bigl\|[\pa_{\nu_t}w_t]\bigr\|_{L^2(\partial E_t)}
\end{equation}
holds, with a constant $C$ depending only on $E$, $M$ and $\alpha$.\\
We now argue by contradiction, assuming $\|g\|_{L^2(\pa E_t)} < \theta \bigl\|[\pa_{\nu_t}w_t]\bigr\|_{L^2(\pa E_t)}$. First, by formula~\eqref{first var} and the translation invariance of the functional $J$, we have
$$
0=\frac{d}{ds}J(E_t+s\eta_t)\biggl|_{s=0}=\int_{\pa E_t}(\HHH_{t}+4\gamma v_t)\langle \eta_t \, \vert \, \nu_t\rangle\, d\mu_t=\int_{\pa E_t}w_t\langle \eta_t\, \vert \, \nu_t\rangle\, d\mu_t\,.
$$
It follows that, by multiplying equality~\eqref{not a translation 2} by $w_t-\widehat w_t$, with $\widehat w_t=\fint_{\T^3}w_t\, dx$ and integrating over $\partial E_t$, we get
\begin{align*}
\int_{\T^3}|\nabla w_t|^2 \, dx =&\,-\int_{\pa E_t}w_t [\pa_{\nu_t}w_t]\, d\mu_t\\
=&\, -\int_{\pa E_t}(w_t-\widehat w_t) [\pa_{\nu_t}w_t]\, d\mu_t\\
=&\,-\int_{\pa E_t}(w_t-\widehat w_t) g\, d\mu_t\\
\leq&\,\,\theta \| w_t - \widehat w_t \|_{L^2(\partial E_t)} \bigl\|[\pa_{\nu_t}w_t]\bigr\|_{L^2(\partial E_t)}.
\end{align*}
Note that in the second and the third equality above we have used the fact that $[\pa_{\nu_t}w_t]$ and $\nu_t$ have zero integral on $\pa E_t$.\\
By the trace inequality (see~\cite{Ev}), we have
\begin{equation}\label{tracecar}
\|w_t-\widehat w_t\|^2_{L^2(\pa E_t)}\leq\|w_t-\widehat w_t\|^2_{H^{1/2}(\pa E_t)}\leq C\int_{\T^3}|\nabla w_t|^2\, dx\,,
\end{equation}
hence, by the previous estimate, we conclude
\begin{equation}
\int_{\T^3}|\nabla w_t|^2 \, dx \leq C\theta^2\bigl\|[\pa_{\nu_t}w_t]\bigr\|_{L^2(\partial E_t)}^2.\label{trans}
\end{equation}
Let us denote with $f:\T^3\to\R$ the harmonic extension of $\langle \eta_t\, \vert \, \nu_t\rangle $ to $\T^3$, we then have
\begin{equation} \label{bound for harm ext}
\|\nabla f\|_{L^2(\T^3)} \leq C \|\langle\eta_t\,\vert\, \nu_t\rangle\|_{H^{{1}/{2}}(\partial E_t)} \leq C |\eta_t| \|\nu_t\|_{W^{1,3}(\pa E_t)}\leq 
C \bigl\|[\pa_{\nu_t}w_t]\bigr\|_{L^2(\partial E_t)}\,,
\end{equation}
where the first inequality comes by standard elliptic estimates (holding with a constant $C=C(E,M,\alpha)>0$, see~\cite{DDM} for details), the second is trivial and the last one follows by inequalities~\eqref{eqcar50003} and~\eqref{eqcar50002}.\\
Thus, by equality~\eqref{not a translation 2} and estimates~\eqref{trans} and~\eqref{bound for harm ext}, we get
\begin{align}
\|\langle \eta_t\,\vert \, \nu_t\rangle\|^2_{L^2(\pa E_t)}
&=\int_{\pa E_t}[\pa_{\nu_t}w_t]\langle \eta_t\,\vert\nu_t\rangle\, d\mu\\
&= - \int_{\T^3} \langle \nabla w_t\,\vert \, \nabla f\rangle \, dx\\
&\leq \left(\int_{\T^3} |\nabla w_t|^2 \,dx\right)^{1/2} \left(\int_{\T^3} |\nabla f|^2\, dx\right)^{1/2} \\
&\leq C \theta \bigl\|[\pa_{\nu_t}w_t]\bigr\|^2_{L^2(\partial E_t)}\,. 
\end{align}
If then $\theta>0$ is chosen so small that $C\theta+\theta^2 <1$ in the last inequality, then we have a contradiction with equality~\eqref{not a translation 2} and the fact that $\|g\|_{L^2(\pa E_t)}<\theta \bigl\|[\pa_{\nu_t}w_t]\bigr\|_{L^2(\pa E_t)}$, as they imply (by $L^2$--orthogonality) that
$$
\|\langle \eta_t\,\vert \, \nu_t\rangle\|^2_{L^2(\pa E_t)}>(1-\theta^2)\bigl\|[\pa_{\nu_t}w_t]\bigr\|_{L^2(\pa E_t)}^2\,.
$$
All this argument shows that for such a choice of $\theta$ condition~\eqref{not a translation} holds.\\
By Propositions~\ref{2.6} and~\ref{prop:nocrit}, there exist positive constants $\sigma_\theta$ and $\delta$ with the following properties: for any set $F\in \mathfrak{C}^{1,\alpha}_M(E)$ such that  
$\|\psi_F\|_{W^{2,3}(\pa E)}<\de$, there holds
\begin{equation}\label{de03}
\Pi_F(\varphi)\geq \sigma_\theta\| \varphi\|_{H^1(\pa F)}^2
\end{equation}
for all $\varphi\in \Htilde^1(\pa F)$ such that $\min_{\eta\in\OO_E}\|\varphi-\langle\eta\,\vert\,\nu_F\rangle\|_{L^2(\pa F)}\geq \theta\|\varphi\|_{L^2(\pa F)}$ and if $E'$ is critical, $\vol(E')=\vol(E)$ with $\|\psi_{E'}\|_{W^{2,3}(\pa E)}<\de$, then 
\begin{equation}\label{de04}
E'=E+\eta
\end{equation}
for a suitable vector $\eta\in \R^3$. We then assume that $M_2,M_3$ are small enough such that
\begin{equation}\label{de05}
\omega(2\max\{M_2,M_3\})<\de/2\,
\end{equation}
where $\omega$ is the function introduced in formula~\eqref{eqcar50001}.

\smallskip

\noindent \textbf{Step ${\mathbf 3}$} ({The stopping time $\overline T$ is equal to the maximal time $T(E_0)$})\textbf{.}\\ We show now that, by taking $M_1,M_2,M_3$ smaller if needed, we have $\overline T=T(E_0)$.\\
By the previous point and the suitable choice of $M_2,M_3$ made in its final part, formula~\eqref{not a translation} holds, hence we have
$$
\Pi_{E_t}\bigl([\pa_{\nu_t}w_t]\bigr)\geq \sigma_\theta\bigl\|[\pa_{\nu_t}w_t]\bigr\|_{H^1(\pa E)}^2\qquad \text{ for all $t\in [0, \overline T)$.}
$$
In turn, by Lemma~\ref{calculations} we may estimate 
\begin{equation*}
\frac{d}{dt} \left(\frac{1}{2} \int_{\T^3} |\nabla w_t|^2\, dx \right) \leq-\sigma_\theta \bigl\|[\pa_{\nu_t}w_t]\bigr\|_{H^1(\pa E_t)}^2+ \frac{1}{2}\int_{\partial E_t} (\partial_{\nu_t} w^+_t+ \partial_{\nu_t} w_t^-) [\partial_{\nu_t} w_t]^2 \, d \mu_t
\end{equation*}
for every $t \leq \overline T$.\\
It is now easy to see that
\begin{equation}\label{Lapw}
\Delta w_t= [\partial_{\nu_t} w_t]\mu_t\,,
\end{equation}
then, by point (iii) of Lemma~\ref{harmonic estimates}, we estimate the last term as
\begin{equation*}
\int_{\partial E_t} (\partial_{\nu_t} w^+_t+ \partial_{\nu_t} w_t^-) [\partial_{\nu_t} w_t]^2 \, d \mu_t\leq C \int_{\partial E_t} (|\partial_{\nu_t} w^+_t|^3 + |\partial_{\nu_t} w^-_t|^3) \, d \mu_t\leq C \int_{\partial E_t}\bigl| [\partial_{\nu_t} w_t]\bigr|^3 \, d \mu_t\,,
\end{equation*}
thus, the last estimate in the statement of Lemma~\ref{harmonic estimates} implies
\begin{equation*}
\bigl\|[\pa_{\nu_t}w_t]\bigr\|_{L^3(\partial E_t)} \leq C \bigl\|[\pa_{\nu_t}w_t]\bigr\|_{H^1(\partial E_t)}^{2/3} \|w_t - \widehat w_t\|_{L^2(\partial E_t)}^{1/3}.
\end{equation*}
Therefore, combining the last three estimates, we get 
\begin{align}
\frac{d}{dt} \int_{\T^3} |\nabla w_t|^2\, dx
\leq&\,-2\sigma_\theta \bigl\|[\pa_{\nu_t}w_t]\bigr\|_{H^1(\pa E_t)}^2+C\|w_t - \widehat w_t\|_{L^2(\partial E_t)} \bigl\|[\pa_{\nu_t}w_t]\bigr\|_{H^1(\pa E_t)}^2\nonumber\\
\leq&\,-\sigma_\theta\bigl\|[\pa_{\nu_t}w_t]\bigr\|_{H^1(\pa E_t)}^2,\label{quasi}
\end{align}
for every $t\in[0,\overline{T})$, where in the last inequality we used the trace inequality~\eqref{tracecar} 
\begin{equation}\label{tttt}
\Vert w_t - \widehat w_t \Vert^{2}_{L^2(\pa E_t)} \leq \Vert w_t - \widehat w_t \Vert^{2}_{H^{{1}/{2}}(\pa E_t)} \leq C \int_{\T^3} |\nabla w_t|^2 \, dx \leq 2CM_3,
\end{equation}
possibly choosing a smaller $M$ such that $2CM_3<\sigma_{\theta}$.\\
This argument clearly says that the quantity $\int_{\T^3} |\nabla w_t|^2\, dx$ is nonincreasing in time, hence, if $M_2,M_3$ are small enough, the inequality $\int_{\T^3} |\nabla w_t|^2\, dx\leq 2M_3$ is preserved during the flow. If we assume by contradiction that $\overline T< T(E_0)$, then it must happen that $\vol(E_{\overline{T}}\triangle E)=2M_1$ or $\|\psi_{\overline T}\|_{C^{1,\alpha}(\pa E)}=2M_2$. Before showing that this is not possible, we prove that actually the quantity $\int_{\T^3} |\nabla w_t|^2\, dx$ decreases (non increases) exponentially.\\
Computing as in the previous step,
\begin{align}
\int_{\T^3} |\nabla w_t|^2\, dx 
&=- \int_{\partial E_t} w_t [\partial_{\nu_t} w_t] \, d \mu_t\\
&= - \int_{\partial E_t} (w_t - \widehat w_t) [\partial_{\nu_t} w_t] \, d \mu_t \\
&\leq \| w_t - \widehat w_t \|_{L^2(\partial E_t)} \bigl\|[\pa_{\nu_t}w_t]\bigr\|_{L^2(\partial E_t)}\\
&\leq  C\left(\int_{\T^3}|\nabla w_t|^2\, dx\right)^{1/2}\bigl\|[\pa_{\nu_t}w_t]\bigr\|_{L^2(\partial E_t)},
\end{align}
where we used again the trace inequality~\eqref{tracecar}. Then,
\begin{equation*}
\int_{\T^3}|\nabla w_t|^2 \, dx \leq C\bigl\|[\pa_{\nu_t}w_t]\bigr\|_{L^2(\partial E_t)}^2 \leq C\norma{[\pa_{\nu_t} w_t]}^2_{H^1(\pa E_t)},
\end{equation*}
and combining this inequality with estimate~\eqref{quasi}, we obtain 
$$
\frac{d}{dt} \int_{\T^3} |\nabla w_t|^2\, dx\leq-c_0 \int_{\T^3} |\nabla w_t|^2\, dx,
$$
for every $t \leq \overline T$ and for a suitable constant $c_0\geq 0$. Integrating this differential inequality, we get
\begin{equation}\label{expfinalmente}
\int_{\T^3} |\nabla w_{ t}|^2\, dx \leq e^{-c_0 t}\int_{\T^3} |\nabla w_0|^2\, dx \leq M_3 e^{-c_0 t}\leq M_3\,,
\end{equation}
for every $t \leq \overline T$.\\
Then, we assume that $\vol(E_{\overline{T}}\triangle E)=2M_1$ or $\|\psi_{\overline T}\|_{C^{1,\alpha}(\pa E_{\overline{T}})}=2M_2$. Recalling formula~\eqref{D(F)0} and denoting by $X_t$ the velocity field of the flow (see Definition~\ref{def:smoothflow} and the subsequent discussion), we compute
\begin{align}
\frac{d}{dt}D(E_t)&=\frac{d}{dt}\int_{E_t} d_E\, dx= \int_{E_t}\Div(d_E X_t)\, dx= \int_{\pa E_t}d_E\langle X_t\vert\nu_t\rangle\, d\mu_t\\
&= \int_{\pa E_t}d_E [\pa_{\nu_t}w_t]\, d\mu_t=- \int_{\T^3}\left\langle\nabla h\,\vert\, \nabla w_t\right\rangle\, dx\,,
\end{align}
where $h$ denotes the harmonic extension of $d_E$ to $\T^3$. Note that, by standard elliptic estimates and the properties of the signed distance function $d_E$, we have
$$
\|\nabla h\|_{L^2(\T^3)}\leq C\|d_E\|_{C^{1,\alpha}(\pa E)}\leq C=C(E)\,,
$$
then, by the previous equality and formula~\eqref{expfinalmente}, we get
\begin{equation}\label{Deqcar}
\frac{d}{dt} D(E_t) \leq C \|\nabla w_t\|_{L^2(\T^3)}\leq C\sqrt{M_3\,} e^{-c_0t/2}\,,
\end{equation}
for every $t \leq \overline T$.
By integrating this differential inequality over $[0, \overline T)$ and recalling estimate~\eqref{D(F)}, we get
\begin{equation}\label{step33}
\vol(E_{\overline T}\triangle E)\leq C\|\psi_{\overline T}\|_{L^2(\pa E_{\overline{T}})}\leq C\sqrt{D(E_{\overline T})}\leq C\sqrt{D(E_0)+C\sqrt{M_3}}\leq C\sqrt[4]{M_3}\,,
\end{equation}
as $D(E_0)\leq M_1$, provided that $M_1,M_3$ are chosen suitably small. This shows that 
$\vol(E_{\overline{T}}\triangle E)=2M_1$ cannot happen if we chose $C\sqrt[4]{M_3}\leq M_1$.\\
By arguing as in Lemma~\ref{w52conv} (keeping into account inequality~\eqref{Tprimo} and 
formula~\eqref{eqcar50001}), we can see that the $L^2$--estimate~\eqref{step33} implies a $W^{2,3}$--bound on $\psi_{\overline T}$ with a constant going to zero, keeping fixed $M_2$, as $\int_{\T^3} |\nabla w_{\overline{T}}|^2\, dx\to0$, hence, by estimate~\eqref{expfinalmente}, as $M_3\to0$. Then, by Sobolev embeddings, the same holds for $\|\psi_{\overline T}\|_{C^{1,\alpha}(\pa E_{\overline{T}})}$, hence, if $M_3$ is small enough, we have a contradiction with $\|\psi_{\overline T}\|_{C^{1,\alpha}(\pa E_{\overline{T}})}=2M_2$.\\
Thus, $\overline T=T(E_0)$ and 
\begin{equation}\label{finaldecay}
\vol(E_t\triangle E)\leq C\sqrt[4]{M_3}\,,\quad\quad\|\psi_t\|_{C^{1,\alpha}(\pa E_t)}\leq 2M_2\,,\quad\quad\int_{\T^3} |\nabla w_{t}|^2\, dx \leq  M_3e^{-c_0 t}\,,
\end{equation}
for every $t\in[0, T(E_0))$, by choosing $M_1,M_2,M_3$ small enough.

\smallskip

\noindent \textbf{Step ${\mathbf 4}$} ({Long time existence})\textbf{.}\\ We now show that, by taking $M_1,M_2,M_3$ smaller if needed, we have $T(E_0)=+\infty$, that is, the flow exists for all times.\\
We assume by contradiction that $T(E_0)<+\infty$ and we recall that, by estimate~\eqref{quasi} and the fact that $\overline T=T(E_0)$, we have
$$
\frac{d}{dt}\int_{\T^3} |\nabla w_t|^2\, dx +\sigma_\theta\bigl\|[\pa_{\nu_t}w_t]\bigr\|_{H^1(\pa E_t)}^2\leq 0
$$
for all $t\in [0,T(E_0))$. Integrating this differential inequality over the interval 
$$
\left[T(E_0)-{T}/2,T(E_0)-{T}/4\right]\,,
$$
where $T$ is given by Theorem~\ref{th:EscNis}, as we said at the beginning of the proof, we obtain
\begin{equation*}
\sigma_{{\theta}}\int_{T(E_0)-T/2}^{T(E_0)-T/4}\bigl\|[\pa_{\nu_t}w_t]\bigr\|_{H^1(\pa E_t)}^2\, dt\leq 
\int_{\T^3} |\nabla w_{T(E_0)-\frac{T}2}|^2\, dx- \int_{\T^3} |\nabla w_{T(E_0)-\frac{T}4}|^2\, dx\leq M_3\,,
\end{equation*}
where the last inequality follows from estimate~\eqref{finaldecay}. Thus, by the mean value theorem there exists $\overline t\in \left(T(E_0)-{T}/2,T(E_0)-{T}/4 \right)$ such that
$$
\bigl\|[\pa_{\nu_{\overline t}}w_{\overline t}] \bigr\|_{H^1(\pa E_t)}^2\leq \frac{4M_3}{T\sigma_\theta}\,.
$$
Note that for any smooth set $F\subseteq\T^3$, we have $\Vert v_{F}\Vert_{C^1(\T^3)}\leq L$, for some ``absolute'' constant $L$ and that $w_F$ is constant, then, since $H^1(\pa E_{\overline t})$ embeds into $L^p(\pa E_{\widehat t})$ for all $p>1$, by Lemma~\ref{harmonic estimates}, we in turn infer that
\begin{align}
[\HHH_{\overline t}(\cdot + \psi_{\overline t}(\cdot)&\,\nu_E(\cdot))-\HHH_E]^2_{C^{0,\alpha}(\pa E)}\\
\leq&\,C [w_{\overline t}(\cdot + \psi_{\overline t}(\cdot)\nu_E(\cdot))-w_E]^2_{C^{0,\alpha}(\pa E)}\\
&\,+ C [v_{\overline t}(\cdot + \psi_{\overline t}(\cdot)\nu_E(\cdot))-v_{\overline t}]^2_{C^{0,\alpha}(\pa E)} + C[v_{\overline t} - v_E]^2_{C^{0,\alpha}(\pa E)}\\
\leq&\,C [w_{\overline t}]^2_{C^{0,\alpha}(\pa E_{\overline t})}\norma{\psi_{\overline t}}^2_{C^{1,\alpha}(\pa F)} 
+ C L^2 \norma{\psi_{\overline t}}^2_{C^{1,\alpha}(\pa F)} + C \norma{u_{\overline t}-u_E}^2_{L^2(\T^3)}\\
\leq&\, C \frac{M_3}{T \sigma_\theta} + C L^2 \norma{\psi_{\overline t}}^2_{C^{1,\alpha}(\pa E)} + C\vol(E_{\overline t} \triangle E)^2 \, ,
\end{align} 
where $[\cdot ]_{C^{0,\alpha}(\pa E_{\overline t})}$ and $[\cdot]_{C^{0,\alpha}(\pa E)}$ stand for the $\alpha$--H\"older seminorms on $\pa E_{\overline t}$ and $\pa E$, respectively and remind that $v_{\overline t},v_E$ are the potentials, defined by formula~\eqref{potential1}, associated to 
$u_{\overline{t}}= \chi_{\text{\raisebox{-.5ex}{$\scriptstyle E_{\overline{t}}$}}} - \chi_{\text{\raisebox{-.5ex}{$\scriptstyle \T^n \setminus E_{\overline{t}}$}}}$ and $u_E= \chi_{\text{\raisebox{-.5ex}{$\scriptstyle E$}}} - \chi_{\text{\raisebox{-.5ex}{$\scriptstyle \T^n \setminus E$}}}$.\\ 
By means of Schauder estimates (as Calder\'on--Zygmund inequality implied estimate~\eqref{CZG}), it is possible to show (see~\cite{DDM}) that there exists a constant $C>0$ depending only on $E$, $M$, $\alpha$ and $p>1$ such that for every $F\in \mathfrak{C}^{1,\alpha}_M(E)$, choosing even smaller $M_1,M_2,M_3$, there holds 
\begin{equation}\label{CZGS}
\norma{B}_{C^{0,\alpha}(\pa F)} \leq C(1+  \norma{\HHH}_{C^{0,\alpha}(\pa F)})\,.
\end{equation}
Hence, by the above discussion, we can conclude that $E_{\overline t}\in \mathfrak{C}^{2,\alpha}_M(E)$. Therefore, the maximal time of existence of the classical solution starting from $E_{\overline t}$ is at least $T$, which means that the flow $E_t$ can be continued beyond $T(E_0)$, which is a contradiction.

\smallskip

\noindent \textbf{Step ${\mathbf 5}$} ({Convergence, up to subsequences, to a translate of $E$})\textbf{.}\\ Let $t_n\to +\infty$, then, by estimates~\eqref{finaldecay}, the sets $E_{t_n}$ satisfy the hypotheses of Lemma~\ref{w52conv}, hence, up to a (not relabeled) subsequence we have that there exists a critical set $E'\in \mathfrak{C}^{1,\alpha}_M(E)$ such that $E_{t_n}\to E'$ in $W^{5/2,2}$. Due to formulas~\eqref{eqcar50001} and~\eqref{de05} we also have $\|\psi_{E'}\|_{W^{2,3}(\pa E)}\leq \delta$ and $E'=E+\eta$ for some (small) $\eta\in \R^3$ (equality~\eqref{de04}).

\smallskip

\noindent \textbf{Step ${\mathbf 6}$} ({Exponential convergence of the full sequence})\textbf{.}\\ Consider now
$$
D_\eta(F)=\int_{F\Delta (E+\eta)}\mathrm{dist\,}(x, \pa E+\eta)\, dx\,.
$$
The very same calculations performed in Step~$3$ show that 
\begin{equation}\label{step6}
\Bigl\vert\frac{d}{dt} D_\eta(E_t)\Bigr\vert \leq C \|\nabla w_t\|_{L^2(\T^3)}\leq C\sqrt{M_3} e^{-{c_0}t/2}
\end{equation}
for all $t\geq0$, moreover, by means of the previous step, it follows $\lim_{t\to +\infty} D_\eta(E_t)=0$. In turn, by integrating this differential inequality and writing 
$$
\pa E_t=\{y+\psi_{\eta, t}(y)\nu_{E+\eta}(y): y\in \pa E+\eta\}\,,
$$
we get
\begin{equation}\label{step61}
\|\psi_{\eta, t}\|_{L^2(\pa E+\eta)}^2\leq C D_\eta(E_t)\leq\int_t^{+\infty}C\sqrt{M_3} e^{-c_0s/2}\, ds\leq  C\sqrt{M_3} e^{-c_0t/2}\,.
\end{equation}
Since by the previous steps $\|\psi_{\eta, t}\|_{W^{2,3}(\pa E+\eta)}$ is bounded, we infer from this inequality and interpolation estimates that also $\|\psi_{\eta, t}\|_{C^{1,\beta}(\pa E+\eta)}$ decays exponentially for all $\beta\in (0,1/3)$. Then, setting $p = \frac{2}{1 -\beta}$, we have, by estimates~\eqref{step61} and~\eqref{D(F)} (and standard elliptic estimates),
\begin{align}
\|v_t-v_{E+\eta}\|_{C^{1,\beta}(\T^3)} &\leq C \|v_t-v_{E+\eta}\|_{W^{2,p}(\T^3)} \leq C \|u_t-u_{E+\eta}\|_{L^{p}(\T^3)} \\
&\leq C \vol(E_t \triangle (E+\eta))^{{1}/{p}}\leq C \|\psi_{\eta, t}\|_{L^2(\pa E+\eta)}^{{1}/{p}}\\
&\leq CM_3^{{1}/{4p}} e^{-{c_0t}/{4p}t}\label{vdecay}
\end{align}
for all $\beta\in (0,1/3)$.
Denoting the average of $w_t$ on $\pa E_t$ by $\overline w_t$, as by estimates~\eqref{tracecar} and~\eqref{expfinalmente} (recalling the argument to show 
inequality~\eqref{eqcar15000}), we have that 
\begin{align}
\|w_t\big(\cdot + \psi_{\eta, t}(\cdot)\nu_{E+\eta}(\cdot)\big)-\overline w_t\|_{H^{1/2}(\pa E+\eta)}
&\leq C\|w_t-\overline w_t\|_{H^{1/2}(\pa E_t)}\|\psi_{\eta, t}\|_{C^1(\pa E+\eta)}\\
& \leq C\|\nabla w_t\|_{L^2(\T^3)}\\
&\leq C\sqrt{M_3} e^{-{c_0t}/2}\,.
\end{align}
It follows, taking into account inequality~\eqref{vdecay}, that 
\begin{equation}\label{quasiHdecay}
\bigl\|[\HHH_t\big(\cdot + \psi_{\eta, t}(\cdot)\nu_{E+\eta}(\cdot)\big)-\overline \HHH_t]
-[\HHH_{\pa E+\eta}-\overline \HHH_{\pa E+\eta}]\bigr\|_{H^{1/2}(\pa E+\eta)}\to0 
\end{equation}
exponentially fast, as $t\to+\infty$, where $\overline \HHH_t$ and $\overline \HHH_{\pa E+\eta}$ stand for the averages of $\HHH_t$ on $\pa E_t$ and of $\HHH_{\pa E+\eta}$ on $\pa E+\eta$, respectively.\\
Since $E_t\to E+\eta$ (up to a subsequence) in $W^{{5}/{2},2}$, it is easy to check that 
$\vert\overline \HHH_{t}-\overline \HHH_{\pa E+\eta}\vert\leq C\|\psi_{\eta, t}\|_{C^1(\pa E+\eta)}$ which decays exponentially, therefore, thanks to 
limit~\eqref{quasiHdecay}, we have
$$
\bigl\|\HHH_t\big(\cdot + \psi_{\eta, t}(\cdot)\nu_{E+\eta}(\cdot)\big) -\HHH_{\pa E+\eta}\bigr\|_{H^{1/2}(\pa E+\eta)}\to0
$$
exponentially fast.\\
The conclusion then follows arguing as at the end of Step~$4$.

\end{proof}

\subsection{A brief overview of the Neumann case}\label{Neucase}

Let $\Omega$ be a smooth bounded open subset of $\R^n$.
As before we consider the nonlocal Area functional
\beq\label{N.1}
J_N(E)=\A_{\Omega}(\pa E)+\gamma\int_{\Omega}|\nabla v_E|^2\, dx\,,
\eeq
for every $E\subseteq \Omega$ with $\pa E \cap \pa \Omega=$~\varempty, where $\gamma\geq 0$ is a real parameter and $v_E$ is the potential defined as follows, similarly to problem~\eqref{potential},
$$
\begin{cases}
-\Delta v_E=u_E-m\quad &\text{ in }\Omega\vspace{1pt}\\
\displaystyle\,\,\frac{\pa v_E}{\pa\nu_E}=0\quad &\text{ on } \pa\Omega\vspace{5pt}\\
\displaystyle\int_\Omega v_E\,dx=0\\
\end{cases}
$$
with $m=\fint_\Omega u_E\,dx$, $u_E= \chi_{\text{\raisebox{-.5ex}{$\scriptstyle E$}}} - \chi_{\text{\raisebox{-.5ex}{$\scriptstyle \Omega\setminus E$}}}$ and $\nu_E$ the outer unit normal to $E$.\\
As in formula~\eqref{G1}, we have
$$
\int_\Omega |\nabla v_E|^2\, dx=\int_\Omega\int_\Omega G(x,y)u_E(x)u_E(y)\, dxdy\,,
$$
where $G$ is the (distributional) solution of
$$
\begin{cases}
-\Delta_x G(x,y)=\delta_y -\frac1{\vol(\Omega)}\quad &\text{ for every $x\in\Omega$}\vspace{1pt}\\
\displaystyle \langle\nabla_x G(x,y)\vert \nu_E(x)\rangle=0\quad &\text{ for every $x\in\pa\Omega$}\vspace{5pt}\\
\displaystyle\int_\Omega G(x,y)\,dx=0
\end{cases}
$$
for every $y\in\Omega$.\\
Note that, unlike the ``periodic'' case (when the ambient is the torus $\T^n$), the functional $J_N$ is not translation invariant, therefore several arguments simplify.
The calculus of the first and second variations of $J_N$, under a volume constraint, is exactly the same as for $J$, then we say that a smooth set $E \subseteq \Omega$, with $\pa E \cap \pa \Omega =$~\varempty, is a \emph{critical set}, if it satisfies
the Euler--Lagrange equation 
$$
\HHH+4\gamma v_E=\lambda \qquad\text{on $\partial E$,}
$$
for a constant $\lambda \in \R$, instead, since $J_N$ is not translation invariant, the spaces $T(\pa E)$, $T^\perp (\pa E)$, and the decomposition~\eqref{decomp} are no longer needed and, defining the same quadratic form $\Pi_E$ as in formula~\eqref{Pieq}, we say that a smooth critical set $E$ is \emph{strictly stable} if 
$$
\Pi_E(\varphi)>0\qquad\text{for all $\varphi\in \Htilde^1(\pa E)\setminus \{0\}$. }
$$
Naturally, $E\subseteq \Omega$ is a \emph{local minimizer} if there exists a $\delta \geq 0$ such that 
$$
J_N(F)\geq J_N(E),
$$
for all $F\subseteq \Omega$, $\pa F \cap \pa\Omega=$~\varempty, $\vol(F)=\vol(E)$ and $\vol(E\triangle F)\leq \delta$.
Then, as in the periodic case, we have  a local minimality result with respect to  small $W^{2,p}$--perturbations. Precisely, the following (cleaner) counterpart to Theorem~\ref{W2pMin} holds (see also~\cite{JuPi}).

\begin{thm}\label{MinNeum}
Let $p>\max\{2, n-1\}$ and $E\subseteq\Omega$ a smooth strictly stable critical set for the nonlocal Area functional $J_N$ (under a volume constraint) with $N_\eps$ a tubular neighborhood of $E$ as in formula~\eqref{tubdef}. Then there exist constants $\delta,C>0$ such that
$$
J_N(F)\geq J_N(E)+C[\vol(E \triangle F)]^2\,,
$$
for all smooth sets $F\subseteq \T^n$ such that $\vol(F)=\vol(E)$, $\vol(F\triangle E)<\delta$, $\pa F \subseteq N_{\varepsilon}$ and
\begin{equation}
\pa F= \{y+\psi(y)\nu_E(y)\, : \, y \in \pa E\},
\end{equation}
for a smooth $\psi$ with $\norma{\psi}_{W^{2,p}(\pa E)} < \delta$.\\
As a consequence, $E$ is a $W^{2,p}$--local minimizer of $J_N$ (as defined above). Moreover, if $F$ is $W^{2,p}$--close enough to $E$ and $J_N(F)=J_N(E)$, then $F=E$, that is, $E$ is locally the unique $W^{2,p}$--local minimizer.
\end{thm}
\begin{proof}[Sketch of the proof.]
Following the line of proof of Theorem~\ref{W2pMin}, since the functional is not translation invariant we do not need Lemma~\ref{Lemma 3.8} and
inequality~\eqref{3.38}, proved in Step~2 of the proof of such theorem, simplifies to 
$$
\inf\Bigl\{\Pi_F(\varphi):\, \varphi\in \Htilde^1(\pa F)\,, \|\varphi\|_{H^1(\pa F)}=1\Bigr\}\geq\frac{m_0}2\,,
$$
where $m_0$ is the constant defined in formula~\eqref{m0}. The proof of this inequality then goes exactly as there.\\
Coming to Step~3 of the proof of Theorem~\ref{W2pMin}, we do not need inequality~\eqref{assumption}, thus we do not need to replace $F$ by a suitable translated set $F-\eta$. Instead, we only need to observe that inequality~\eqref{3.46} is still satisfied. The rest of the proof remains unchanged.
\end{proof}

The short time existence and uniqueness Theorem~\ref{th:EscNis}, proved in~\cite{EsNi} in any dimension, holds also in the ``Neumann case" for the modified Mullins--Sekerka flow with parameter $\gamma\geq0$, obtained (as in Definition~\ref{MSF def}) by letting the outer normal velocity $V_t$ of the moving boundaries given by
\begin{equation}\label{MSnl2}
V_t= [\partial_{\nu_t} w_{t}] \quad\text{ on $\partial E_t$ for all $t\in [0, T)$,}
\end{equation}
where $\nu_t=\nu_{E_t}$ and $w_t=w_{E_t}$ is the unique solution in $H^1(\Omega)$ of the problem
\begin{equation}\label{WE2}
\begin{cases}
\Delta w_{E_t}=0 & \text{in }\Omega\setminus \pa E_t\\
w_{E_t}= \HHH + 4\gamma v_{E_t} & \text{on } \, \partial E_t,
\end{cases}
\end{equation}
with $v_{E_t}$ the potential defined above and, as before, $[\pa_{\nu_t} w_t]$ is the jump of the outer normal derivative of $w_{E_t}$ on $\pa E_t$. 

Then, we conclude by stating the following analogue of Theorem~\ref{existence} (taking into account~Remark~\ref{existence+}).

\begin{thm}\label{mainN}
Let $\Omega$ be an open smooth subset of $\R^3$ and let $E\subseteq\Omega$ be a smooth strictly stable critical set for the nonlocal Area functional under a volume constraint, with $\pa E \cap \pa \Omega =$~{\rm{\varempty}} and $N_\eps$ (with $\eps<1$) a tubular neighborhood of $\pa E$, as in formula~\eqref{tubdef}. Then, for every $\alpha\in (0,1/2)$ there exists $M>0$ such that, if $E_0$ is a smooth set in $\mathfrak{C}^{1,\alpha}_M(E)$ satisfying $\vol( E_0)= \vol( E )$ and
$$
 \int_{\Omega}|\nabla w_{E_0}|^2\, dx\leq M\,
$$
where $w_0=w_{E_0}$ is the function relative to $E_0$ as in problem~\eqref{WE} (with $\Omega$ in place of $\T^3\setminus\pa E$), then, the unique smooth solution $E_t$ to  the Mullins--Sekerka flow (with parameter $\gamma\geq 0$) starting from $E_0$, given by Theorem~\ref{th:EscNis}, is  defined  for all $t\geq0$. Moreover,  $E_t\to E$ exponentially fast in $C^k$,as $t\to +\infty$, for every $k\in\N$, with the meaning that the functions 
$\psi_{\eta, t} : \pa E \to \R$ representing $\pa E_t$ as ``normal graphs'' on $\pa E$, that is,
$$
\pa E_t= \{ y+ \psi_{\eta,t} (y) \nu_{E+\eta}(y) \, : \, y \in \pa E\},
$$
satisfy, for every $k\in\N$,
$$
\Vert \psi_{\eta, t}\Vert_{C^k(\pa E + \eta)}\leq C_ke^{-\beta_k t},
$$
for every $t\in[0,+\infty)$, for some positive constants $C_k$ and $\beta_k$.
\end{thm}

The proof of this result is similar to the one of Theorem~\ref{existence} and actually it is simpler since we do not need the argument used in Step~2 of such proof, where we controlled the translational component of the flow. Note also that in the statement of Proposition~\ref{2.6}, in this case, inequality~\eqref{2.13} holds for all $\varphi\in\Htilde^1(\pa F)$.  Finally, observe that under the hypotheses of Proposition~\ref{prop:nocrit} we may actually conclude that $E^\prime=E$, that is, there are no other critical sets close to $E$.

\subsection{The surface diffusion flow -- Preliminary lemmas}\ \vskip.3em

As for the modified Mullins--Sekerka flow, we start with the technical lemmas for the global existence result.

\begin{lem}[Energy identities] \label{calculations2}
Let $E_t\subseteq\T^n$ be a surface diffusion flow. Then, the  following 
identities hold:
\begin{equation}
\label{der of A}
\frac{d}{dt} \A(\pa E_t) = -  \int_{\pa E_t} \abs{\nabla\HHH_t}^2\, \dmu_t\,,
\end{equation}
and
\begin{align}
\frac{d}{dt}\frac{1}{2}\int_{\pa E_t} \abs{ \nabla\HHH_t}^2\, \dmu_t= & -\Pi_{E_t}(\Delta_t\HHH_t )
-\int_{\partial E_t}  B_t (\nabla\HHH_t, \nabla\HHH_t)  \Delta_t\HHH_t \, \dmu_t\\
&+ \frac{1}{2}\int_{\partial E_t}\HHH_t \abs{\nabla\HHH_t}^2  \Delta_t\HHH_t\, \dmu_t\,, \label{der of DH}
\end{align}
where $\Pi_{E_t}$ is the quadratic form defined in formula~\eqref{Pieq} (with $\gamma=0$).
\end{lem}
\begin{proof}
Let $\psi_t$ the smooth family of maps describing the flow as in 
formula~\eqref{sdf2}. By formula~\eqref{dermu2}, where $X$ is the smooth (velocity) vector field $X_t=\frac{\partial\psi_t}{\partial t}=(\Delta_t\HHH_t)\nu_{E_t}$ along $\pa E_t$, hence $X_\tau=X_t-\langle X_t \vert \nu_{E_t}\rangle\nu_{E_t}=0$ (as usual $\nu_{E_t}$ is the outer normal unit vector of $\partial E_t$), following computation~\eqref{local}, we have 
\begin{align}
\frac{d}{dt} \A(\pa E_t)& = \frac{d}{dt} \int_{\pa E_t}{\dmu_t}\\
&= \int_{\pa E_t} (\Div X_\tau+\HHH_t  \langle X \vert  \nu_{E_t} \rangle ) \,  d\mu_t\\
&= \int_{\pa E_t}\HHH_t\Delta_t\HHH_t\, \dmu_t\\
&=- \int_{\pa E_t} \abs{\nabla\HHH_t}^2 \, \dmu_t \, ,
\end{align}
where the last equality follows integrating by parts. This establishes 
relation~\eqref{der of A}.

In order to get relation~\eqref{der of DH} we also need the time derivatives of the evolving metric and of the mean curvature of $\pa E_t$, that we already computed in 
formulas~\eqref{derg2},~\eqref{1Acalc} and~\eqref{derH} (where the function $\varphi$ in this case is equal to $\Delta_t\HHH_t$ and $X_\tau=0$), that is,
\begin{equation}
\frac{\partial g_{ij}}{\partial t}= 2h_{ij}\Delta_t\HHH_t\qquad\text{ and }\qquad 
\frac{\pa g^{ij}}{\pa t} = - 2  h^{ij}\Delta_t\HHH_t\,,
\end{equation}
\begin{equation}\label{3bisbis}
\frac{\pa\HHH_t}{\pa t}=-  \abs{B_t}^2\Delta_t\HHH_t  - \Delta_t \Delta_t\HHH_t
\end{equation}
Then, we compute
\begin{align}
\frac{d}{dt}\frac12  \int_{\pa E_t} \abs{\nabla\HHH_t}^2 \, \dmu_t
=&\,\frac12  \int_{\pa E_t}\HHH_t\abs{\nabla\HHH_t}^2 \,\Delta_t\HHH_t\,\dmu_t
-\int_{\pa E_t} h^{ij}\nabla_i\HHH_t\nabla_j\HHH_t\,\Delta_t\HHH_t\, d\mu_t\\
&\,-\int_{\pa E_t} g^{ij}\nabla_i\HHH_t\nabla_j\bigl(\abs{B}^2\Delta_t\HHH_t+\Delta_t \Delta_t\HHH_t\bigr)\, d\mu_t\\
=&\,\frac12  \int_{\pa E_t}\HHH_t\abs{\nabla\HHH_t}^2 \,\Delta_t\HHH_t\,\dmu_t-\int_{\pa E_t} B(\nabla\HHH_t,\nabla\HHH_t)\,\Delta_t\HHH_t\, d\mu_t\\
&\,+\int_{\pa E_t} \abs{B_t}^2(\Delta_t\HHH_t)^2\, d\mu_t
+\int_{\pa E_t} \Delta_t\HHH_t\,\Delta_t \Delta_t\HHH_t\, d\mu_t\\
=&\,\frac12  \int_{\pa E_t}\HHH_t\abs{\nabla\HHH_t}^2 \,\Delta_t\HHH_t\,\dmu_t-\int_{\pa E_t} B_t(\nabla\HHH_t,\nabla\HHH_t)\,\Delta_t\HHH_t\, d\mu_t\\
&\,+\int_{\pa E_t} \abs{B_t}^2(\Delta_t\HHH_t)^2\, d\mu_t
-\int_{\pa E_t} \vert\nabla\Delta_t\HHH_t\vert^2\, d\mu_t\,,
\end{align}
which is formula~\eqref{der of DH}, recalling the definition of $\Pi_{E_t}$ in formula~\eqref{Pieq}.
\end{proof}

{\em From now on, as before due to the dimension--dependence of the estimates that follow, we restrict ourselves to the three--dimensional case.}

\medskip

The following lemma is an easy consequence of Theorem~3.70 in~\cite{Aubin}, with $j=0$, $m=1$, $n=2$ and $r=q=2$, taking into account the previous discussion.

\begin{lem}[Interpolation  on boundaries]\label{interpolation}
Let $F\subseteq\T^3$ be a smooth set. In the previous notations, for every $p\in[2,+\infty)$ there exists a constant $C=C(F,M,\alpha,p)>0$ such that for every set $E \in \mathfrak{C}^{1,\alpha}_M(F)$ and $g \in H^1(\partial E)$, we have
$$
\norma{g}_{L^p(\partial E)} \leq C ( \norma{\nabla g}_{L^2(\pa E)}^{\theta} \norma{g}_{L^2(\pa E)}^{1-\theta} + \norma{g}_{L^2(\pa E)} )\,,
$$
with $\theta=1-2/p$.\\
Moreover, the following Poincar\'e inequality holds
$$
\norma{g-\overline{g}}_{L^p(\pa E)} \leq C \norma{\nabla g}_{L^2(\partial E)} \, , 
$$
where $\overline{g}(x)= \fint_{\Gamma} g\, \dmu $, if $x$ belongs to a connected component  $\Gamma$ of $\partial E$.
\end{lem}

Then, we have the following mixed ``analytic--geometric'' estimate. 

\begin{lem}[$H^2$--estimates on boundaries]\label{laplacian}
Let $F \subseteq \T^3$ be a smooth set. Then there exists a constant $C=C(F,M,\alpha,p)>0$ such that if $E\in \mathfrak{C}^{1,\alpha}_M(F)$ and $f\in H^1(\pa E)$ with $\Delta f\in L^2(\pa E)$, then $f\in H^2(\pa E)$ and  
$$
\norma{\nabla^2 f}_{L^2(\partial E)}\leq C  \norma{\Delta f}_{L^2(\pa E)}(1+  \norma{\HHH}_{L^4(\pa E)}^2) \, .
$$
\end{lem}
\begin{proof}
We first claim that the following inequality holds,
\begin{equation}
\label{laplacian 1}
\int_{\partial E} \abs{\nabla ^2 f}^2 \, \dmu \leq  \int_{\partial E} \abs{ \Delta f}^2 \, d\mu  + C\int_{\partial E} \abs{B}^2 \abs{\nabla f}^2  \, \dmu \, .
\end{equation}
Indeed, if we integrate by parts the left--hand side, we obtain (the Hessian of a function is symmetric)
$$
\int_{\pa E} g^{ik}g^{jl}\nabla^2_{ij}f \nabla^2_{kl}f \, \dmu = - \int_{\pa E}g^{ik}g^{jl} \nabla_k \nabla_j\nabla_if \nabla_l f \, \dmu \, .
$$
Hence, interchanging the covariant derivatives and integrating by parts, we get
\begin{align}
- \int_{\pa E}g^{ik}g^{jl} \nabla_k \nabla_j\nabla_if \nabla_l f \, \dmu
=&\,- \int_{\pa E} g^{ik}g^{jl}\nabla_j \nabla_k\nabla_if \nabla_l f \, \dmu\\
&\,- \int_{\pa E} g^{ik}g^{jl}R_{kjip}g^{ps}\nabla_sf\nabla_l f \, \dmu \\
=&\,- \int_{\pa E} g^{jl}\nabla_j \Delta f \nabla_l f\, \dmu-\int_{\pa E} \Ric(\nabla f, \nabla f)\, \dmu  \\
=&\,\int_{\pa E} \abs{\Delta f}^2 \, \dmu+\int_{\pa E}\bigl[\abs{B}^2 |\nabla f|^2-\HHH B (\nabla f, \nabla f)\bigr]\, \dmu\\
\leq  & \, \int_{\pa E} \abs{\Delta f}^2 \, \dmu + C\int_{\pa E} \abs{B}^2 \abs{\nabla f}^2 \, \dmu\, ,
\end{align}
thus, inequality~\eqref{laplacian 1} holds (in the last passage we applied Cauchy--Schwarz inequality and the well known relation $|\HHH|\leq\sqrt{2}|B|$, then $C=1+\sqrt{2}$).

We now estimate the last term in formula~\eqref{laplacian 1} by means of 
Lemma~\ref{interpolation} (which is easily extended to vector valued functions $g:\pa E\to\R^m$) with $g=\nabla f$ and $p=4$:
\begin{align}
\int_{\partial E} \abs{B}^2\abs{\nabla f}^2  \, \dmu &\leq \norma{B}_{L^4(\pa E)}^2  \norma{\nabla  f}_{L^4(\pa E)}^2 \\
&\leq  C \norma{B}_{L^4(\pa E)}^2\bigl( \norma{\nabla^2 f}_{L^2(\pa E)}^{1/2}  \norma{\nabla f}_{L^2(\pa E)}^{1/2}  + \norma{\nabla  f}_{L^2(\pa E)}\bigr)^2\\
&\leq  C \norma{B}_{L^4(\pa E)}^2\bigl( \norma{\nabla^2 f}_{L^2(\pa E)}
\norma{\nabla f}_{L^2(\pa E)}  + \norma{\nabla  f}_{L^2(\pa E)}^2\bigr)\,.
\end{align}
Hence, expanding the product on the last line, using Peter--Paul (Young) inequality on the first term of such expansion and ``adsorbing'' in the left hand side of inequality~\eqref{laplacian 1} the small fraction of the term $\norma{\nabla^2 f}_{L^2(\pa E)}^2$ that then appears, we obtain
\begin{align}
\norma{\nabla^2 f}_{L^2(\pa E)}^2 &\leq  C ( \norma{\Delta f}_{L^2(\pa E)}^2 +  \norma{\nabla  f}_{L^2(\pa E)}^2 ( \norma{B}_{L^4(\pa E)}^2 +  \norma{B}_{L^4(\pa E)}^4 ) ) \nonumber \\
&\leq  C ( \norma{\Delta f}_{L^2(\pa E)}^2 +  \norma{\nabla f}_{L^2(\pa E)}^2 ( 1+  \norma{B}_{L^4(\pa E}^4 ) ) \,. \label{aY}
\end{align}
By the fact that $\Delta f$ has zero average on each connected component of $\pa E$, there holds
\begin{align}
\norma{\nabla  f}_{L^2(\pa E)}^2  &= -\int_{\partial E}  f \Delta f \, \dmu\\
&= -  \int_{\pa E}(f- \overline f)  \Delta f \, \dmu \nonumber \\
&\leq \norma{f- \overline f}_{L^2(\pa E)}\norma{\Delta f }_{L^2(\pa E)}\\
&\leq C \norma{\nabla  f}_{L^2(\pa E)} \norma{\Delta f }_{L^2(\pa E)}\,,\label{poin}
\end{align}
where we used Lemma~\ref{interpolation} again, hence,
\begin{equation}
\norma{\nabla  f}_{L^2(\pa E)}\leq C \norma{\Delta f }_{L^2(\pa E)}\,.\label{poin2}
\end{equation}
Thus, from inequality~\eqref{aY}, we deduce  
\begin{equation}\label{eq10010}
\norma{\nabla ^2 f}_{L^2(\pa E)}^2 \leq  C \norma{\Delta f}_{L^2(\pa E)}^2 ( 1 + \norma{B}_{L^4(\pa E)}^4 )\, .
\end{equation}
Now, by means of Calder\'on--Zygmund estimates, it is possible to show (see~\cite{DDM}) that there exists a constant $C>0$ depending only on $F$, $M$, $\alpha$ and $q>1$ such that for every $E\in \mathfrak{C}^{1,\alpha}_M(F)$, there holds 
\begin{equation}\label{CZG2}
\norma{B}_{L^q(\pa E)} \leq C(1+  \norma{\HHH}_{L^q(\pa E)})\,.
\end{equation}
Then, since it is easy to check that also all the other constant in the previous inequalities (and the ones coming from Lemma~\ref{interpolation} also) depend only on $F$, $M$, $\alpha$ and $p$, if $E\in \mathfrak{C}^{1,\alpha}_M(F)$, substituting this estimate, with $q=4$, in formula~\eqref{eq10010}, the thesis of the lemma follows.
\end{proof}
The following lemma provides a crucial ``geometric interpolation'' that will be needed in the proof of the main theorem.
\begin{lem}[Geometric interpolation]\label{nasty}
Let $F \subseteq \T^3$ be a smooth set. Then there exists a constant $C=C(F,M,\alpha)>0$ such that the following estimates holds
$$
\int_{\partial E} \abs{B}\abs{\nabla\HHH}^2 \abs{\Delta\HHH} \, \dmu
\leq C \norma{\nabla\Delta\HHH}_{L^2(\pa E)}^2 \,  \norma{\nabla \HHH}_{L^2(\pa E)}\, (1+ \norma{\HHH}_{L^6(\pa E)}^3 )\,,
$$
for every $E\in \mathfrak{C}^{1,\alpha}_M(F)$.
\end{lem}
\begin{proof}
First, by a standard application of H\"older inequality, we have
$$
\int_{\partial E} \abs{B} \abs{\nabla\HHH}^2 \abs{\Delta\HHH} \, \dmu \leq \norma{\Delta\HHH}_{L^3(\pa E)} \Bigl(  \int_{\partial E} \abs{B}^\frac{3}{2}\abs{\nabla\HHH}^3 \, \dmu \Bigl)^{2/3}.
$$
Then, using the Poincar\'e inequality stated in Lemma~\ref{interpolation} and the fact that $\Delta\HHH$ has zero average on each connected component of $\pa E$, we get 
$$
 \norma{\Delta\HHH}_{L^3(\pa E)} \leq C  \norma{ \nabla\Delta\HHH }_{L^2(\pa E)}.
$$
Now, we use H\"older inequality again
$$
 \Bigl(  \int_{\partial E} \abs{B}^\frac{3}{2}\abs{\nabla\HHH}^3 \, \dmu \Bigr)^{2/3} \leq  \Bigl(  \int_{\partial E}\abs{\nabla\HHH}^{4} \, \dmu \Bigr)^{1/2}\Bigl(  \int_{\partial E}\abs{B}^{6} \, \dmu \Bigr)^{1/6} \, ,
$$
and we apply Lemma~\ref{interpolation} with $p=4$,
$$
\Bigl(  \int_{\partial E}\abs{\nabla\HHH}^{4} \, \dmu \Bigr)^{1/2} \leq C  ( \norma{\nabla^2\HHH}_{L^2(\pa E)}  \norma{\nabla\HHH}_{L^2(\pa E)}  + \norma{\nabla\HHH}_{L^2(\pa E)}^2  )\,.
$$
Combining all these inequalities, we conclude
$$
\int_{\partial E} \abs{B} \abs{\nabla\HHH}^2 \abs{\Delta\HHH} \, \dmu \leq C\norma{\nabla\Delta\HHH }_{L^2(\pa E)}  \,  \norma{B}_{L^6(\pa E)} \,  \norma{\nabla\HHH}_{L^2(\pa E)}(\norma{\nabla^2\HHH}_{L^2(\pa E)}  + \norma{\nabla\HHH}_{L^2(\pa E)})\,.
$$
By Lemma~\ref{laplacian} and estimate~\eqref{poin2}, with $\HHH$ in place of $f$, the right--hand side of the previous inequality can be bounded from above by 
\begin{equation}\label{stima}
C \norma{\nabla\Delta\HHH }_{L^2(\pa E)}  \, \norma{B}_{L^6(\pa E)}  \,   \norma{\Delta\HHH}_{L^2(\pa E)} \,  \norma{\nabla\HHH}_{L^2(\pa E)}  \, (1  + \norma{H}_{L^4(\pa E)}^2).
\end{equation}
Hence, using again Poincar\'e inequality and estimate~\eqref{CZG2} with $q=6$, we have
$$
\norma{\Delta\HHH}_{L^2(\pa E)} \leq C  \norma{\nabla\Delta\HHH }_{L^2(\pa E)}
$$
and 
$$
\norma{B}_{L^6(\pa E)} \leq C(1+  \norma{\HHH}_{L^6(\pa E)})\,. 
$$
Finally, using this relations and H\"older inequality, we obtain the thesis
$$
\int_{\partial E} \abs{B} \abs{\nabla\HHH}^2 \abs{\Delta\HHH} \, \dmu \leq C \norma{\nabla\Delta\HHH }_{L^2(\pa E)}^2  \, \norma{\nabla\HHH}_{L^2(\pa E)} \, (1+  \norma{\HHH}^3_{L^6(\pa E)})\,.
$$
\end{proof}

We now remind that since $\partial E$ can be disconnected (as in the case of lamellae), the Poincar\'e inequality could fail for $\partial E$. However, if $E$ is sufficiently close to a stable critical set then it is true for the mean curvature of $\pa E$.

\begin{lem}[Geometric Poincar\'e inequality]\label{lm:geopoinc}
Fixed $p>2$ and a smooth strictly stable critical set $F\subseteq\T^3$, let $\delta>0$ be the constant provided by Proposition~\ref{2.6}, with $\theta=1$. Then, for $M$ small enough, there exists a constant $C=C(F,M,\alpha,p)>0$ such that 
\begin{equation}\label{geopoinc}
\int_{\pa E} \abs{\HHH-\overline{\HHH}}^2\, \dmu \leq C\int_{\pa E}\abs{\nabla\HHH}^2\, \dmu \,,
\end{equation}
for every set $E\in \mathfrak{C}^{1,\alpha}_M(F)$ such that $\pa E\subseteq N_\eps$ with
\begin{equation}
\pa E= \{y + \psi (y) \nu_F(y) \, : \, y \in \pa F \}\, ,
\end{equation}
for a smooth function $\psi$ with $\norma{\psi}_{W^{2,p}(\pa F)}<\delta$.
\end{lem}
\begin{proof}
Since 
$$
\int_{\pa E}(\HHH -\overline\HHH)\nu_E \, \dmu=0 \, ,
$$
there holds
$$
\int_{\pa E} \abs{\HHH- \overline\HHH - \langle \eta \vert  \nu_E \rangle }^2 \, \dmu =\norma{\HHH - \overline\HHH }_{L^2(\pa E)}^2+\int_{\pa E} \langle \eta \vert  \nu_E \rangle^2 \, \dmu\geq\norma{\HHH - \overline\HHH }_{L^2(\pa E)}^2
$$
for all $\eta \in\R^3$. Choosing $M<\delta$, we may then apply Proposition~\ref{2.6} with $\theta=1$ and $\varphi=\HHH-\overline\HHH$, obtaining
\begin{equation}
\sigma_1 \int_{\pa E}\abs{\HHH-\overline\HHH}^2\, \dmu \leq \int_{\pa E}\abs{\nabla\HHH}^2\, \dmu-\int_{\pa E}\abs{B}^2 \abs{\HHH-\overline\HHH}^2\, \dmu \leq \int_{\pa E}\abs{\nabla\HHH}^2\, \dmu\,.
\end{equation}
\end{proof}

The following lemma is straightforward.

\begin{lem}\label{5.1}
Let $E \subseteq \T^3$ be a smooth set. If $f\in H^1(\pa E)$ and $g\in W^{1,4}(\pa E)$, then
$$
\norma{\nabla(fg)}_{L^2(\pa E)}\leq C\norma{\nabla f}_{L^2(\pa E)}\|g\|_{L^\infty(\pa E)}+C\|f\|_{L^4(\pa E)}\|\nabla g\|_{L^4(\pa E)}\,,
$$
for a constant $C$ independent of $E$.
\end{lem}
\begin{proof}
We estimate with Cauchy--Schwarz inequality,
\begin{align*}
\norma{\nabla(fg)}_{L^2(\pa E)}^2\leq&\,2\norma{\nabla f}_{L^2(\pa E)}^2\norma{g}_{L^\infty(\pa E)}^2
+2\int_{\pa E}|f|^2|\nabla g|^2\,d\mu\\
\leq&\, 2\norma{\nabla f}_{L^2(\pa E)}^2\|g\|_{L^\infty(\pa E)}^2+2\|f\|_{L^4(\pa E)}^2\|\nabla g\|_{L^4(\pa E)}^2\,,
\end{align*}
hence the thesis follows.
\end{proof}

As a consequence, we prove the following result.

\begin{lem}\label{5.2sdf}
Let $F\subseteq\T^3$ be a smooth set and $E\in\mathfrak{C}^{1,\alpha}_M(F)$. Then, for $M$ small enough, there holds
$$
\|\psi_E\|_{W^{3,2}(\pa F)}\leq C(F,M,\alpha)(1+\|\HHH\|_{H^1(\pa E)}^2)\,,
$$
where $\HHH$ is the mean curvature of $\pa E$ (the function $\psi_E$ is defined by formula~\eqref{front}).
\end{lem}
\begin{proof}
As we do in Lemma~\ref{5.2}, by a standard localization/partition of unity/straightening argument, we may reduce ourselves to the case where the function $\psi_E$ is defined in a disk $D\subseteq\R^2$ and $\|\psi_E\|_{C^{1,\alpha}(D)}\leq M$. Fixed a smooth cut--off function $\varphi$ with compact support in $D$ and equal to one on a smaller disk $D'\subseteq D$, we have again relation~\eqref{noia1} (see also~\cite{Man}). \\
Then, using Lemma~\ref{5.1} and recalling that $\|\psi_E\|_{C^{1,\alpha}(D)}\leq M$, we estimate
\begin{align}
\norma{\nabla\Delta(\varphi\psi_E)}_{L^2(D)}\leq C(F,M,\alpha)\bigl(&\,M^2\norma{\nabla^3(\varphi\psi_E)}_{L^2(D)}+\Vert\nabla\HHH\Vert_{L^2(\pa E)}(1+\norma{\nabla\psi_E}_{L^\infty(D)})\\
&\,+\norma{\HHH}_{L^4(\pa E)}(1+\norma{\psi_E}_{W^{2,4}(D)})+1+\norma{\psi_E}_{W^{2,4}(D)}\bigr)\,.
\end{align}
We now use the fact that, by a simple integration by part argument, if $u$ is a smooth function with compact support in $\R^2$, there holds
$$
\Vert\nabla\Delta u\Vert_{L^2(\R^2)}=\Vert\nabla^3u\Vert_{L^2(\R^2)}\,,
$$
hence, 
\begin{align*}
\Vert\nabla^3(\varphi\psi_E)\Vert_{L^2(D)}
&\,=\Vert\nabla\Delta(\varphi\psi_E)\Vert_{L^2(D)}\\
&\,\leq C(F,M,\alpha)\bigl(M^2\norma{\nabla^3(\varphi\psi_E)}_{L^2(D)}+\Vert\nabla\HHH\Vert_{L^2(\pa E)}(1+\norma{\nabla\psi_E}_{L^\infty(D)})\\
&\,\phantom{\leq C(F,M,\alpha)\bigl(\,}+\norma{\HHH}_{L^4(\pa E)}(1+\norma{\psi_E}_{W^{2,4}(D)})+1+\norma{\psi_E}_{W^{2,4}(D)}\bigr)\,,
\end{align*}
then, if $M$ is small enough, we have
\begin{equation}\label{eqcar10020}
\Vert\nabla^3(\varphi\psi_E)\Vert_{L^2(D)}
\leq C(F,M,\alpha)(1+\Vert\HHH\Vert_{H^1(\pa E)})(1+\norma{\mathrm{Hess}\,\psi_E}_{L^4(D)})\,,
\end{equation}
as
\begin{equation}\label{eqcar10040}
\Vert\HHH\Vert_{L^4(\pa E)}\leq C(F,M,\alpha)\Vert\HHH\Vert_{H^1(\pa E)}\,,
\end{equation}
by Theorem~3.70 in~\cite{Aubin}.\\ 
By the Calder\'on--Zygmund estimates (holding uniformly for every hypersurface $\pa E$, with $E\in\mathfrak{C}^{1,\alpha}_M(F)$, see~\cite{DDM}), we have again the inequality~\eqref{Cald-Zyg} and the most useful estimation~\eqref{normadeltapsi1}.

Hence, possibly choosing a smaller $M$, we conclude (as in inequality~\eqref{normadeltapsi2})
\begin{equation}
\norma{\Delta \psi_E}_{L^4(D)} \leq C(F,M,\alpha) (1 + \norma{\HHH}_{L^4(\pa E)}) \leq C(F,M,\alpha) (1 + \norma{\HHH}_{H^1(\pa E)})\,,
\end{equation}
again by inequality~\eqref{eqcar10040}.\\
Thus, by estimate~\eqref{Cald-Zyg}, we get
\begin{equation}\label{eqcar10030sdf}
\norma{\mathrm{Hess} \, \psi_E}_{L^{4}(D)}\leq C(F,M,\alpha) (1 + \norma{\HHH}_{H^1(\pa E)})\,,
\end{equation}
and using this inequality in estimate~\eqref{eqcar10020}, 
$$
\Vert\nabla^3(\varphi\psi_E)\Vert_{L^2(D)}
\leq C(F,M,\alpha)(1+\Vert\HHH\Vert_{H^1(\pa E)})^2\,,
$$
hence,
$$
\Vert\nabla^3\psi_E\Vert_{L^2(D')}
\leq C(F,M,\alpha)(1+\Vert\HHH\Vert_{H^1(\pa E)})^2\leq C(F,M,\alpha)(1+\Vert\HHH\Vert_{H^1(\pa E)}^2)\,.
$$
The inequality in the statement of the lemma then easily follows by this inequality, estimate~\eqref{eqcar10030sdf} and $\Vert\psi_E\Vert_{C^{1,\alpha}(D)}\leq M$, with a standard covering argument.
\end{proof}

Now, we state a compactness result whose proof is very close in spirit to the proof of Lemma~\ref{w52conv}, however we present it explicitly in order to show how the lemmas above come differently into play.

\begin{lem}[Compactness]\label{w32conv}
Let $F\subseteq\T^3$ be a smooth set and $E_n\subseteq \mathfrak{C}^{1,\alpha}_M(F) $ a sequence of smooth sets such that 
$$
\sup_{n\in\N}\,\int_{\pa E_n}|\nabla \HHH_n|^2\, d\mu_n<+\infty\,.
$$
Then, if $\alpha\in(0,1/2)$ and $M$ is small enough, there exists a smooth set $F'\in \mathfrak{C}^1_M(F)$ such that, up to a (non relabeled) subsequence, $E_n\to F'$ in $W^{2,p}$ for all $1\leq p<+\infty$.\\
Moreover, if inequality~\eqref{geopoinc} holds for every set $E_n$ with a constant $C$ independent of $n$ and 
$$
\int_{\pa E_n}|\nabla\HHH_n|^2\, d\mu_n\to 0\,,
$$
then $F'$ is critical for the volume--constrained Area functional $\A$ and the convergence $E_n\to F'$ is in $W^{3,2}$.
\end{lem}

\begin{proof}
We first claim that 
\begin{equation}\label{claim1111sdf}
\sup_{n\in\N}\,\|\HHH_n\|_{H^{1}(\pa E_n)}<+\infty.
\end{equation}
We set $\widetilde{\HHH}_n=\fint_{\pa E_n}\HHH_n\,d\mu_n$, then, by the ``geometric'' Poincar\'e inequality of Lemma~\ref{lm:geopoinc}, which holds with a ``uniform'' constant $C=C(F,M,\alpha)$, for all the sets $E\in\mathfrak{C}^{1,\alpha}_M(F)$ (see~\cite{DDM}), if $M$ is small enough, we have
\begin{equation}
\|\HHH_n-\widetilde{\HHH}_n\|^2_{H^{1}(\pa E_n)}\leq\sup_{n\in\N}\,\int_{\pa E_n}|\nabla \HHH_n|^2\, d\mu_n<C<+\infty
\end{equation}
with a constant $C$ independent of $n\in\N$.\\
Then, we note that, as in Lemma~\ref{w52conv}, by the uniform $C^{1,\alpha}$--bounds on $\pa E_n$, we may find a solid cylinder of the form $C=D\times(-L,L)$, with $D\subseteq\R^{2}$ a ball centered at the origin and functions $f_n$, with
\begin{equation}\label{w32}
\sup_{n\in\N}\|f_n\|_{C^{1,\alpha}(\overline D)}<+\infty\,,
\end{equation}
such that $\pa E_n\cap C=\{(x',x_n)\in D\times(-L,L):\, x_n= f_n(x')\}$ with respect to a suitable coordinate frame (depending on $n\in\N$). Hence, recalling the formula~\eqref{intHn}, the uniform bound~\eqref{w32} and the fact that $\|\HHH_n- \widetilde{\HHH}_n\|_{H^{1}(\pa E_n)}$ are equibounded, we get that ${\widetilde{\HHH}_n}$ are also equibounded (by a standard ``localization'' argument, ``uniformly'' applied to all the hypersurfaces $\pa E_n$). Therefore, the claim~\eqref{claim1111sdf} follows.\\
By applying the Sobolev embedding theorem on each connected component of $\partial F$, we have that 
$$
\norma{\HHH_n}_{L^p(\pa E_n)} \leq C \norma{\HHH_n}_{H^{1}(\pa E_n)} <C<+\infty\qquad \text{for all $p \in [1,+\infty)$.}
$$
for a constant $C$ independent of $n\in\N$.\\
Now, as before, we obtain
$$
\norma{B}_{L^p(\pa E)} \leq C(1+  \norma{\HHH}_{L^p(\pa E)})\,.
$$
for every $E\in \mathfrak{C}^{1,\alpha}_M(F)$ with a uniform constant $C$. Then, if we write 
$$
\pa E_n =\{y+\psi_n(y)\nu_F(y):\, y\in \pa F\}\,,
$$ 
we have $\sup_{n\in\N}\|\psi_n\|_{W^{2,p}(\pa F)}<+\infty$, for all $p \in [1,+\infty)$.\\
Thus, by the Sobolev compact embedding $W^{2,p}(\pa F)\hookrightarrow C^{1,\alpha}(\pa F)$, up to a subsequence (not relabeled), there exists a set $F'\in \mathfrak{C}^{1,\alpha}_M(F)$ such that 
$$
\psi_n\to \psi_{F'} \text{ in $C^{1,\alpha}(\pa F)$,}
$$
for all $\alpha\in (0,1/2)$.\\
From estimate~\eqref{claim1111sdf} and Lemma~\ref{5.2sdf} (possibly choosing a smaller $M$), we have then that the functions $\psi_n$ are bounded in $W^{3,2}(\pa F)$. Hence, possibly passing to another subsequence (again not relabeled), we conclude that $E_n \to F'$ in $W^{2,p}$ for every $p\in[1,+\infty)$, by the Sobolev compact embeddings.\\
For the second part of the lemma, we first observe that if
$$
\int_{\pa E_n}|\nabla\HHH_n|^2\, d\mu_n\to 0\,,
$$
then there exists $\lambda\in \R$ and a subsequence $E_n$ (not relabeled) such that 
$$
\HHH_n\big(\cdot + \psi_n(\cdot)\nu_F(\cdot)\big)\to \lambda=
\HHH\big(\cdot + \psi_{F'}(\cdot)\nu_F(\cdot)\big)
$$
in $H^{1}(\pa F)$, where $\HHH$ is the mean curvature of $F'$. Hence $F'$ is critical.\\
To conclude the proof we only need to show that $\psi_n$ converge to $\psi =\psi_{F'}$ in $W^{3,2}(\pa F)$.\\
Fixed $\delta>0$, arguing as in the proof of Lemma~\ref{5.2sdf}, we reduce ourselves to the case where the functions $\psi_n$ are defined on a disk $D\subseteq\R^2$, are bounded in $W^{3,2}(D)$, converge in $W^{2,p}(D)$ for all $p\in[1,+\infty)$ to $\psi\in W^{3,2}(D)$ and  $\|\nabla\psi\|_{L^\infty(D)}\leq\delta$. 
Then, fixed a smooth cut--off function $\varphi$ with compact support in $D$ and equal to one on a smaller disk $D'\subseteq D$, we have
\begin{align*}
\frac{\Delta(\varphi\psi_n)}{\sqrt{1+|\nabla\psi_n|^2}}-\frac{\Delta(\varphi\psi)}{\sqrt{1+|\nabla\psi|^2}} =&\,(\nabla^2(\varphi\psi_n)-\nabla^2(\varphi\psi))\frac{\nabla\psi \nabla\psi}{(1+|\nabla\psi|^2)^{3/2}}\\
&\,+ \nabla^2(\varphi\psi_n)\Bigl(\frac{\nabla\psi_n\nabla\psi_n}{(1+|\nabla\psi_n|^2)^{3/2}}-\frac{\nabla\psi \nabla\psi}{(1+|\nabla\psi|^2)^{3/2}}\Bigr)\\
&\,+\varphi(\HHH_n-\HHH)+R(x,\psi_n,\nabla\psi_n)-R(x,\psi,\nabla\psi)\,,
\end{align*}
where $R$ is a smooth Lipschitz function.\\
Then, using Lemma~\ref{5.1}, an argument similar to the one of the proof of Lemma~\ref{5.2sdf} shows that
\begin{align*}
\bigg \Vert \nabla\Big( \frac{\Delta(\varphi\psi_n)}{\sqrt{1+|\nabla\psi_n|^2}}-\frac{\Delta(\varphi\psi)}{\sqrt{1+|\nabla\psi|^2}} \Big ) \bigg \Vert_{L^2(D)}
& \, \leq C(M)\big(\delta^2\norma{\nabla^3(\varphi\psi_n)-\nabla^3(\varphi\psi)}_{L^2(D)}\\
& \,+\|\nabla^2(\varphi\psi_n)-\nabla^2(\varphi\psi)\|_{L^4 (D)}\|\nabla^2\psi\|_{L^4(D)} \\& \, +
\|\nabla^3(\varphi\psi_n)\|_{L^2(D)}\|\nabla\psi_n-\nabla\psi\|_{L^\infty(D)}\\
& \,+\|\nabla^2(\varphi\psi_n)\|_{L^4(D)}(\|\nabla^2\psi_n\|_{L^4}+\|\nabla^2\psi\|_{L^4(D)})\\
& \,
+\|\nabla\HHH_n-\nabla\HHH\|_{L^2(D)}+\|\psi_n-\psi\|_{W^{2,4}(D)}\big)\,.
\end{align*}
Being $\HHH$ constant, that is $\nabla\HHH=0$, by using Lemma~\ref{5.1} again and arguing as in the proof of Lemma~\ref{5.2sdf}, we finally get
$$
\norma{\nabla^3(\varphi\psi_n)-\nabla^3(\varphi\psi)}_{L^2(D)}\leq C(M)\big(\|\psi_n-\psi\|_{W^{2,4}(D)}
+\|\nabla\psi_n-\nabla\psi\|_{L^\infty(D)}+\| \nabla\HHH_n\|_{L^2(D)}\big)\,,
$$
hence,
$$
\norma{\nabla^3\psi_n-\nabla^3\psi}_{L^2(D')}\leq C(M)\big(\|\psi_n-\psi\|_{W^{2,4}(D)}+\|\nabla\psi_n-\nabla\psi\|_{L^\infty(D)}+\| \nabla\HHH_n\|_{L^2(D)}\big)\,,
$$
from which the conclusion follows, by the first part of the lemma and a standard covering argument.

\end{proof}

\subsection{The surface diffusion flow -- The main theorem}\ \vskip.3em

We now show the global existence result for the surface diffusion flow, whose proof is very similar to the one of Theorem~\ref{existence}. However, in order to make it clear, we present it in a detailed way.

\begin{thm}\label{existence2}
Let $E\subseteq\T^3$ be a strictly stable critical set for the Area functional under a volume constraint and let $N_\eps$ be a tubular neighborhood of $\pa E$, as in formula~\eqref{tubdef}. For every $\alpha\in (0,1/2)$ there exists $M>0$ such that, if $E_0$ is a smooth set in $C^{1,\alpha}_M(E)$ satisfying $\vol( E_0)= \vol( E )$ and
$$
\int_{\pa E_0} \vert \nabla \HHH_0\vert^2\, \dmu_0 \leq M,
$$
then the unique smooth solution $E_t$ of the surface diffusion flow starting from $E_0$, given by Proposition~\ref{th:EMS1}, is defined for all $t\geq0$. Moreover, $E_t\to E+\eta$ exponentially fast in $W^{3,2}$ as $t\to +\infty$ (recall the definition of convergence of sets in Subsection~\ref{stabsec}), for some $\eta\in \R^3$, with the meaning that the functions 
$\psi_{\eta, t} : \pa E+ \eta \to \R$ representing $\pa E_t$ as ``normal graphs'' on $\pa E+ \eta$, that is,
$$
\pa E_t= \{ y+ \psi_{\eta,t} (y) \nu_{E+\eta}(y) \, : \, y \in \pa E+\eta \},
$$
satisfy
$$
\Vert \psi_{\eta, t}\Vert_{W^{3,2}(\pa E + \eta)}\leq Ce^{-\beta t} \, ,
$$
for every $t \in [0, +\infty)$, for some positive constants $C$ and $\beta$.
\end{thm}

\begin{remark}\label{existence2+}
The convergence of $E_t\to E+\eta$ is actually smooth, that is, for every $k\in\N$, there holds
$$
\Vert \psi_{\eta, t}\Vert_{C^k(\pa E + \eta)}\leq C_ke^{-\beta_k t},
$$
for every $t\in[0,+\infty)$, for some positive constants $C_k$ and $\beta_k$. This is a particular case of Theorem~5.1 in~\cite{FusJulMor18}, proved by means of standard parabolic estimates and interpolation (and Sobolev embeddings), using the exponential decay in $W^{3,2}$, analogously to the modified Mullins--Sekerka flow (Remark~\ref{existence+}).
\end{remark}

\begin{remark}
The extra condition in the theorem on the $L^2$--smallness of the gradient of $\HHH_0$ (see the second part of Lemma~\ref{w32conv} and its proof) implies that the mean curvature of $\pa E_0$ is ``close'' to be constant, as it is for the set $E$ or actually for any critical set (recall Remark~\ref{closedness}).
\end{remark}

\begin{proof}[Proof of Theorem~\ref{existence2}]
As in proof of Theorem~\ref{existence}, $C$ will denote a constant depending only on $E$, $M$ and $\alpha$, whose value may vary from line to line.

Assume that the surface diffusion flow $E_t$ is defined for $t$ in the maximal time interval $[0,T(E_0))$, where $T(E_0)\in (0,+\infty]$ and let the moving boundaries $\pa E_t$ be represented as ``normal graphs'' on $\pa E$ as
$$
\pa E_t= \{ y+ \psi_t(y) \nu_{E}(y) \, : \, y \in \pa E\} \, ,
$$
for some smooth functions $\psi_t:\pa E\to \R$.\\
We recall that, by Proposition~\ref{th:EMS1}, for every $F\in \mathfrak{C}^{2,\alpha}_M(E)$, the flow is defined in the time interval $[0, T)$, with $T=T(E,M,\alpha)>0$.\\
As before, we show the theorem for the smooth sets $E_0\subseteq\T^3$ satisfying
\begin{equation}\label{initial_2}
\vol(E_0\Delta E)\leq M_1,\quad\|\psi_0\|_{C^{1,\alpha}(\pa E)}\leq M_2\quad\text{and}\quad \int_{\pa E_0} |\nabla \HHH_0|^2\, d\mu_0\leq M_3\,,
\end{equation}
for some positive constants $M_1,M_2,M_3$, then we get the thesis by setting $M=\min\{M_1,M_2,M_3\}$.\\
For any set $F\in \mathfrak{C}^{1,\alpha}_{M}(E)$, we define quantity in~\eqref{D(F)0} and by the same arguments we obtain estimation~\eqref{D(F)}.\\
Hence, by this discussion, the initial smooth set $E_0\in\mathfrak{C}^{1,\alpha}_M(E)$ satisfies
$D(E_0)\leq M\leq M_1$ (having chosen $\eps<1$).\\
By rereading the proof of Lemma~\ref{w32conv}, it follows that for $M_2,M_3$ small enough, if 
$$
\Vert\psi_F\Vert_{C^{1,\alpha}(\pa E)}\leq M_2
$$
and
\begin{equation}\label{ex-de02_2}
\int_{\pa F} |\nabla \HHH|^2\,\dmu \leq M_3\,,
\end{equation}
then 
\begin{equation}\label{eqcar50001_2}
\|\psi_F\|_{W^{2,6}(\pa E)}\leq \omega(\max\{M_2,M_3\})\,,
\end{equation}
where $s\mapsto\omega(s)$ is a positive nondecreasing function (defined on $\R$) such that $\omega(s)\to 0$ as $s\to 0^+$. This clearly implies 
\begin{equation}\label{eqcar50003_2}
\|\nu_F\|_{W^{1,6}(\pa F)}\leq \omega'(\max\{M_2,M_3\})\,,
\end{equation}
for a function $\omega'$ with the same properties of $\omega$ (also in this case, $\omega$ and $\omega'$ only depend on $E$ and $\alpha$, for $M$ small enough). Moreover, thanks to Lemma~\ref{lm:geopoinc}, there exists $C>0$ such that, choosing $M_2,M_3$ small enough, in order that $\omega(\max\{M_2,M_3\})$ is small enough, we have
\begin{equation}\label{de02bis_2}
\int_{\pa F}\abs{\HHH-\overline \HHH}^2\, \dmu \leq C\int_{\pa F}|\nabla \HHH|^2\, \dmu \, ,
\end{equation}
where, as usual, $\overline \HHH$ is the average of $\HHH$ over $\pa F$.

We again split the proof of the theorem into steps.

\smallskip

\noindent \textbf{Step ${\mathbf 1}$} ({Stopping--time})\textbf{.}\\ Let $\overline T\leq T(E_0)$ be the maximal time such that 
\begin{equation}\label{Tprimo_2}
\vol(E_t\Delta E)\leq 2M_1,\quad\|\psi_t\|_{C^{1,\alpha}(\pa E)}\leq 2M_2\quad\text{and}\quad \int_{\pa E_t} |\nabla \HHH_t|^2\, \dmu_t\leq 2M_3\,,
\end{equation}
for all $t\in [0, \overline T)$. Hence, 
\begin{equation}\label{eqcar50005_2}
\|\psi_t\|_{W^{2,6}(\pa E)}\leq \omega(2\max\{M_2,M_3\})\,
\end{equation}
for all $t\in [0, \overline T)$, as in formula~\eqref{eqcar50001_2}.\\
As before, we claim that by taking $M_1,M_2,M_3$ small enough, we have $\overline T=T(E_0)$.

\smallskip
\noindent \textbf{Step ${\mathbf 2}$} ({Estimate of the translational component of the flow})\textbf{.}\\ We want to show that there exists a small constant $\theta>0$ such that
\begin{equation} \label{not a translation_2}
\min_{\eta\in\OO_E} \norma{ \Delta \HHH_t- \langle\eta , \nu_t\rangle }_{L^2(\pa E_t)}\geq \theta\norma{\Delta \HHH_t}_{L^2(\pa E_t)}\qquad\text{for all }t\in [0, \overline T)\,,
\end{equation}
where $\OO_F$ is defined by formula~\eqref{OOeq}.\\
If $M$ is small enough, clearly there exists a constant $C_0=C_0(E,M,\alpha)>0$ such that, for every $i\in\II_E$, we have $\Vert \langle e_i,\nu_t\rangle\Vert_{L^2(\pa E_t)}\geq C_0>0$, holding $\Vert \langle e_i,\nu_E\rangle\Vert_{L^2(\pa E)}>0$. It is then easy to show that the vector $\eta_t\in\OO_E$ realizing such minimum is unique and satisfies
\begin{equation} \label{not a translation 2_2}
\Delta \HHH_t = \langle \eta_t , \nu_t \rangle + g,
\end{equation} 
where $g\in L^2(\pa E_t)$ is chosen as in relation~\eqref{not a translation 2}. Moreover, the inequality 
\begin{equation}\label{eqcar50002_2}
|\eta_t|\leq C \norma{\Delta \HHH_t}_{L^2(\partial E_t)}
\end{equation}
holds, with a constant $C$ depending only on $E$, $M$ and $\alpha$.\\
We now argue by contradiction, assuming $\|g\|_{L^2(\pa E_t)} < \theta \norma{\Delta \HHH_t}_{L^2(\pa E_t)}$.\\
First we recall that $\Delta \HHH_t$ has zero average. Then, setting $\overline \HHH=\fint_{\pa E_t}\HHH\, \dmu_t$,  and recalling relation~\eqref{de02bis_2}, we get 
\begin{align}
\norma{\HHH_t -\overline \HHH_t }_{L^2(\pa E_t)}^2 & \leq C \int_{\pa E_t}  \abs{\nabla \HHH_t}^2\, \dmu_t \nonumber \\
&= -C\int_{\pa E_t} \HHH_t\Delta \HHH_t \,\dmu_t\\
&= -C\int_{\pa E_t}  \Delta \HHH_t  (\HHH_t-\overline \HHH_t )\, \dmu_t \nonumber \\
&\leq C \norma{\HHH_t -\overline \HHH_t }_{L^2(\pa E_t)}\norma{\Delta \HHH_t}_{L^2(\pa E_t)} \,.\label{extraeq_2}
\end{align}
Hence, we conclude
\begin{equation}
\norma{\HHH_t -\overline \HHH_t }_{L^2(\pa E_t)}\leq C\norma{\Delta \HHH_t}_{L^2(\pa E_t)}\,.\label{sfiguz_2}
\end{equation}
Since, there holds 
$$
\int_{\pa E_t}\HHH_t \, \nu_t\, \dmu_t=\int_{\pa E_t}\nu_t\, \dmu_t=0 \,,
$$
by multiplying relation~\eqref{not a translation 2_2} by $\HHH_t-\overline \HHH_t$, integrating over $\pa E_t$, and using inequality~\eqref{sfiguz_2},  we get
\begin{align}
\Bigl \rvert \int_{\partial E_t}  (\HHH_t -\overline \HHH_t)\Delta \HHH_t   \, \dmu_t  \Bigl \lvert &= 
\Bigl \rvert \int_{\partial E_t}  (\HHH_t -\overline \HHH_t)g  \, \dmu_t  \Bigl \rvert \\
&< \theta \norma{ \HHH_t -\overline \HHH_t}_{L^2(\partial E_t)} \norma{ \Delta \HHH_t }_{L^2(\pa E_t)}\\
&\leq C\theta\norma{ \Delta \HHH_t }_{L^2(\pa E_t)}^2 \, .
\end{align}
Recalling now estimate~\eqref{eqcar50002_2}, as $g$ is orthogonal to $\langle \eta_t, \nu_t \rangle$, computing as in the first three lines of formula~\eqref{extraeq_2}, we have
\begin{align}
\norma{ \langle \eta_t, \nu_t \rangle }^2_{L^2(\pa E_t)}&=\int_{\pa E_t}\Delta \HHH_t \langle \eta_t,\nu_t \rangle \, \dmu_t\\
&= - \int_{\partial E_t}   \langle \nabla \HHH_t ,   \nabla \langle \eta_t, \nu_t\rangle \rangle \, \dmu_t\\
&\leq\abs{\eta_t}\norma{\nabla\nu_t}_{L^2(\pa E_t)} \norma{ \nabla \HHH_t }_{L^2(\pa E_t)}\\
&\leq C \norma{\nabla\nu_t}_{L^2(\pa E_t)} \norma{ \Delta \HHH_t}_{L^2(\pa E_t)} \Bigl\vert\int_{\partial E_t} (\HHH_t -\overline \HHH_t)\Delta \HHH_t \, \dmu_t\, \Bigr\vert^{1/2}\\
&\leq C\sqrt{\theta}\norma{\nabla\nu_t}_{L^2(\pa E_t)}\norma{ \Delta \HHH_t}_{L^2(\pa E_t)}^2\\
&\leq C \sqrt{\theta}\norma{\Delta \HHH_t}_{L^2(\pa E_t)}^2 \,,
\end{align}
where in the last inequality we estimated $\norma{\nabla\nu_t}_{L^2(\pa E_t)}$ with $C\norma{\psi_t}_{W^{2,6}(\pa E_t)}$ and we used inequality~\eqref{eqcar50005_2}.\\
If then $\theta>0$ is chosen so small that $C\sqrt{\theta}+\theta^2 <1$ in the last inequality, then we have a contradiction with equality~\eqref{not a translation 2_2} and the fact that $\|g\|_{L^2(\pa E_t)}<\theta \norma{\Delta \HHH_t}_{L^2(\pa E_t)}$, as they imply (by $L^2$--orthogonality) that
$$
\|\langle \eta_t,  \nu_t\rangle\|^2_{L^2(\pa E_t)}>(1-\theta^2)\norma{\Delta \HHH_t}_{L^2(\pa E_t)}^2\,.
$$
All this argument shows that for such a choice of $\theta$ condition~\eqref{not a translation_2} holds.\\
Then, we can conclude as in Step~$2$ of Theorem~\ref{existence}, by replacing the $W^{2,3}$--norm on $\pa E$ with the $W^{2,6}$--norm on the same boundary.
\smallskip

\noindent \textbf{Step ${\mathbf 3}$} ({The stopping time $\overline T$ is equal to the maximal time $T(E_0)$})\textbf{.}\\ We show now that, by taking $M_1,M_2,M_3$ smaller if needed, we have $\overline T=T(E_0)$.\\
By the previous point and the suitable choice of $M_2,M_3$ made in its final part, formula~\eqref{not a translation_2} holds, hence we have
$$
\Pi_{E_t}(\Delta \HHH_t)\geq \sigma_\theta \norma{\Delta \HHH_t}_{H^1(\pa E_t)}^2\qquad \text{ for all $t\in [0, \overline T)$.}
$$
In turn, by Lemma~\ref{calculations2} and~\ref{nasty} we may estimate 
\begin{align}
\frac{d}{dt}\frac{1}{2}\int_{\partial E_t}{ \abs{\nabla \HHH_t}^2\, \dmu_t}
 \leq &  - \sigma_{\theta} \norma{ \Delta \HHH_t}^2_{H^1(\pa E_t)}  + \int_{\partial E_t} \abs{B} \abs{\nabla \HHH_t }^2 \abs{\Delta \HHH_t} \, \dmu_t \\
  \leq & \,  - \sigma_{\theta} \norma{\Delta \HHH_t}^2_{H^1(\pa E_t)} \\& \, + C\norma{ \nabla (\Delta \HHH_t)}_{L^2(\pa E_t)}^2 \norma{\nabla \HHH_t}_{L^2(\pa E_t)} (1+ \norma{\HHH_t}_{L^6(\pa E_t)}^3) \\
 {\leq} & \,  - \sigma_{\theta} \norma{\Delta \HHH_t}^2_{H^1(\pa E_t)} \\
&\,  + C \sqrt{M_3}\norma{ \nabla (\Delta \HHH_t)}_{L^2(\pa E_t)}^2  (1+ \norma{\HHH_t}_{L^6(\pa E_t)}^3) \\
{\leq} & \,   - \sigma_{\theta} \norma{\Delta \HHH_t}^2_{H^1(\pa E_t)}\\
&\,+ C \sqrt{M_3}\norma{\Delta \HHH_t}_{H^1(\pa E_t)}^2(1+C\omega(\max\{M_2,M_3\}))\label{eqcar6000_2}
\end{align}
for every $t \leq \overline T$, where in the last step we used relations~\eqref{Tprimo_2} and~\eqref{eqcar50005_2}.\\
Noticing that from formulas~\eqref{extraeq_2} and~\eqref{sfiguz_2} it follows
\begin{equation}\label{sfiguz2_2}
\norma{ \nabla \HHH_t}_{L^2(\pa E_t)}\leq C \norma{\Delta \HHH_t}_{L^2(\pa E_t)}\leq C \norma{\Delta \HHH_t}_{H^1(\pa E_t)}\,,
\end{equation} 
keeping fixed $M_2$ and choosing a suitably small $M_3$, we conclude
$$
\frac{d}{dt}\int_{\partial E_t} \abs{\nabla \HHH_t}^2\, \dmu_t\leq - \frac{\sigma_{\theta}}2 \norma{\Delta \HHH_t}^2_{H^1(\pa E_t)}\leq - c_0 \norma{\nabla \HHH_t}^2_{L^2(\pa E_t)}\, .
$$
This argument clearly says that the quantity $\int_{\pa E_t} |\nabla \HHH_t|^2\, \dmu_t$ is nonincreasing in time, hence, if $M_2,M_3$ are small enough, the inequality $\int_{\pa E_t } |\nabla \HHH_t|^2\, \dmu_t\leq M_3$ is preserved during the flow. As before, if we assume by contradiction that $\overline T< T(E_0)$, then it must happen that $\vol(E_{\overline{T}}\Delta E)=2M_1$ or $\|\psi_{\overline T}\|_{C^{1,\alpha}(\pa E)}=2M_2$.\\
Before showing that this is not possible, we prove that actually the quantity $\int_{\pa E_t} |\nabla \HHH_t|^2\, \dmu_t$ decreases (non increases) exponentially. Indeed, integrating the differential inequality above and recalling proprieties~\eqref{initial_2}, we obtain
\begin{equation}\label{exp_2}
\int_{\partial E_t} \abs{\nabla \HHH_t}^2\, \dmu_t \leq e^{-c_0 t}\int_{\partial E_0} \abs{\nabla \HHH_{\pa E_0}}^2\, \dmu_0 \leq M_3  e^{-c_0 t} \leq M_3
\end{equation}
for every $t \leq \overline T$.
Then, we assume that $\vol(E_{\overline{T}}\Delta E)=2M_1$ or $\|\psi_{\overline T}\|_{C^{1,\alpha}(\pa E_{\overline{T}})}=2M_2$. Recalling formula~\eqref{D(F)0} and denoting by $X_t$ the velocity field of the flow (see Definition~\ref{def:smoothflow} and the subsequent discussion), we compute
\begin{align}
\frac{d}{dt}D(E_t)&=\frac{d}{dt}\int_{E_t} d_E\, dx= \int_{E_t}\Div(d_E X_t)\, dx= \int_{\pa E_t}d_E\langle X_t, \nu_t\rangle\, d\mu_t\\
&=\int_{\pa E_t}d_E\, \Delta \HHH_t\, \dmu_t = \int_{\pa E_t}\langle \nabla d_E, \nabla \HHH_t \rangle \,\dmu_t\\
&\leq C \norma{\nabla \HHH_t}_{L^2(\pa E_t)}\leq C\sqrt{M_3} \mathrm{e}^{-c_0 t/2}\,,
\end{align}
for all $t \leq \overline T$, where the last inequality clearly follows from inequality~\eqref{exp_2}.\\
\\
By integrating this differential inequality over $[0, \overline T)$ and recalling estimate~\eqref{D(F)}, we get
\begin{align}
\vol(E_{\overline T}\Delta E) & \leq C\|\psi_{\overline T}\|_{L^2(\pa E_{\overline{T}})}\leq C\sqrt{D(E_{\overline T})}\nonumber  \\& \leq C\sqrt{D(E_0)+C\sqrt{M_3}}\leq C\sqrt[4]{M_3}\,,\label{step33_2}
\end{align}
as $D(E_0)\leq M_1$, provided that $M_1,M_3$ are chosen suitably small. This shows that 
$\vol(E_{\overline{T}}\Delta E)=2M_1$ cannot happen if we chose $C\sqrt[4]{M_3}\leq M_1$.\\
By arguing as in Lemma~\ref{w32conv} (keeping into account inequality~\eqref{Tprimo_2} and 
formula~\eqref{eqcar50001_2}), we can see that the $L^2$--estimate~\eqref{step33_2} implies a $W^{2,6}$--bound on $\psi_{\overline T}$ with a constant going to zero, keeping fixed $M_2$, as $\int_{\pa E_t} |\nabla \HHH_{\overline{T}}|^2\, d\mu_t \to0$, hence, by estimate~\eqref{exp_2}, as $M_3\to0$. Then, by Sobolev embeddings, the same holds for $\|\psi_{\overline T}\|_{C^{1,\alpha}(\pa E_{\overline{T}})}$, hence, if $M_3$ is small enough, we have a contradiction with $\|\psi_{\overline T}\|_{C^{1,\alpha}(\pa E_{\overline{T}})}=2M_2$.\\
Thus, $\overline T=T(E_0)$ and 
\begin{equation}\label{finaldecay_2}
\vol(E_t\Delta E) \leq C\sqrt[4]{M_3}\,,\quad \|\psi_t\|_{C^{1,\alpha}(\pa E_t)}\leq 2M_2\,, \quad \int_{\pa E_t} |\nabla \HHH_{t}|^2\, \dmu_t \leq  M_3e^{-c_0 t}\,,
\end{equation}
for every $t\in[0, T(E_0))$, by choosing $M_1,M_2,M_3$ small enough.

\smallskip

\noindent \textbf{Step ${\mathbf 4}$} ({Long time existence})\textbf{.}\\ We now show that, by taking $M_1,M_2,M_3$ smaller if needed, we have $T(E_0)=+\infty$, that is, the flow exists for all times.\\
We assume by contradiction that $T(E_0)<+\infty$ and we notice that, by computation~\eqref{eqcar6000_2} and the fact that $\overline T=T(E_0)$, we have
$$
\frac{d}{dt}\int_{\pa E_t} |\nabla\HHH_t|^2\, d\mu_t +\sigma_\theta\|\Delta\HHH_t\|_{H^1(\pa E_t)}^2\leq 0
$$
for all $t\in [0,T(E_0))$. Integrating this differential inequality over the interval $\left[T(E_0)-{T}/2,T(E_0)-{T}/4\right]$, where $T$ is given by Proposition~\ref{th:EMS1}, as we said at the beginning of the proof, we obtain
\begin{align}
\sigma_{{\theta}}\int_{T(E_0)-T/2}^{T(E_0)-T/4}\|\Delta\HHH_t\|_{H^1(\pa E_t)}^2\, dt\leq&\, 
\int_{\pa E_{T(E_0)-\frac{T}2}} |\nabla\HHH |^2\, d\mu_{T(E_0)-\frac{T}2}\\
&\,- \int_{\pa E_{T(E_0)-\frac{T}4}} |\nabla\HHH|^2\, d\mu_{T(E_0)-\frac{T}4}\\
&\leq M_3\,,
\end{align}
where the last inequality follows from estimate~\eqref{finaldecay_2}.
Thus, by the mean value theorem there exists $\overline t\in \left(T(E_0)-{T}/2,T(E_0)-{T}/4 \right)$ such that
$$
\|\Delta\HHH_{\overline{t}}\|_{H^1(\pa E_{\overline{t}})}^2\leq \frac{4M_3}{T\sigma_\theta}\,.
$$
Then, by Lemma~\ref{laplacian}
\begin{align}
\norma{\nabla^2 \HHH_{\overline{t}}}_{L^2(\pa E_{\overline{t}})}^2
\leq&\,C\norma{\Delta\HHH_{\overline{t}}}^2_{L^2(\pa E_{\overline{t}})}(1+\norma{\HHH_{\overline{t}}}^4_{L^4(\pa E_{\overline{t}})})\\
\leq&\,CM_3(1+\omega^4(2\max\{M_2,M_3\}))\,
\end{align}
where in the last inequality we also used the curvature bounds provided by formula~\eqref{eqcar50005_2}. In turn, for $p\in\R$ large enough, we get
\begin{equation}
[\HHH_{\overline{t}}]^2_{C^{0,\alpha}(\pa E_{\overline{t}})}\leq C\norma{\nabla \HHH_{\overline{t}}}^2_{L^p(\pa E_{\overline{t}})}\leq C \norma{\nabla \HHH_{\overline{t}}}^2_{H^1(\pa E_{\overline{t}})}\leq CM_3(M_2,M_3)\,,
\end{equation}
where $[ \cdot ]_{C^{0, \alpha}(\pa E_{\overline{t}})}$ stands for the $\alpha$--H\"older seminorm on $\pa E_{\overline{t}}$ and in the last inequality we used the previous estimate.\\
Then, arguing as in Step~$4$ of Theorem~\ref{existence}, it is possible to show that flow $E_t$ exists beyond $T(E_0)$, which is a contradiction.

\smallskip

\noindent \textbf{Step ${\mathbf 5}$} ({Convergence, up to subsequences, to a translate of $F$})\textbf{.}\\ Let $t_n\to +\infty$, then, by estimates~\eqref{finaldecay_2}, the sets $E_{t_n}$ satisfy the hypotheses of Lemma~\ref{w32conv}, hence, up to a (not relabeled) subsequence we have that there exists a critical set $E'\in \mathfrak{C}^{1,\alpha}_M(E)$ such that $E_{t_n}\to E'$ in $W^{3,2}$. Due to formulas~\eqref{eqcar50001_2} (and estimation~\eqref{de05}, that also holds in this case) we have $\|\psi_{E'}\|_{W^{2,6}(\pa E)}\leq \delta$ and $E'=E+\eta$ for some (small) $\eta\in \R^3$.

\smallskip

\noindent \textbf{Step $\mathbf 6$} (Exponential convergence of the full sequence)\textbf{.}\\
Consider now
$$
D_\eta(F)=\int_{F\Delta (E+\eta)}\mathrm{dist\,}(x, \pa E+\eta)\, dx\,.
$$
The very same calculations performed in Step~$3$ show that 
\begin{equation}\label{step6_2}
\Bigl\vert\frac{d}{dt} D_\eta(E_t)\Bigr\vert \leq C \|\nabla\HHH_t\|_{L^2(\pa E_t)}\leq C\sqrt{M_3} e^{-{c_0}t/2}
\end{equation}
for all $t\geq0$, moreover, by means of the previous step, it follows $\lim_{t\to +\infty} D_\eta(E_t)=0$. In turn, by integrating this differential inequality and writing 
$$
\pa E_t=\{y+\psi_{\eta, t}(y)\nu_{E+\eta}(y): y\in \pa E+\eta\}\,,
$$
we get
\begin{equation}\label{step61_2}
\|\psi_{\eta, t}\|_{L^2(\pa E+\eta)}^2\leq C D_\eta(E_t)\leq\int_t^{+\infty}C\sqrt{M_3} e^{-c_0s/2}\, ds\leq  C\sqrt{M_3} e^{-c_0t/2}\,.
\end{equation}
Since by the previous steps $\norma{\psi_{\eta, t}}_{W^{2,6}(\pa E+\eta)}$ is bounded, we infer from this inequality, Sobolev embeddings and standard interpolation estimates that also $\norma{\psi_{\eta, t}}_{C^{1,\beta}(\pa E+\eta)}$ decays exponentially for $\beta\in (0,2/3)$.\\
Denoting the average of $\HHH_t$ on $\pa E_t$ by $\overline\HHH_t$, as by estimates~\eqref{extraeq_2} and~\eqref{exp_2}, we have that 
\begin{align}
\|\HHH_t(\cdot + \psi_{\eta, t}(\cdot)\nu_{E+\eta}(\cdot))&-\overline \HHH_t\|_{H^{1}(\pa E+\eta)}\\
&\leq C\|\HHH_t-\overline \HHH_t\|_{H^{1}(\pa E_t)}\|\psi_{\eta, t}\|_{C^1(\pa E+\eta)}\\
& \leq C\|\nabla \HHH_t\|_{L^2(\pa E_t)}\\
&\leq C\sqrt{M_3} e^{-{c_0t}/2}\,.
\end{align}
It follows that
\begin{equation}\label{quasiHdecay_2}
\|[\HHH_t(\cdot + \psi_{\eta, t}(\cdot)\nu_{E+\eta}(\cdot))-\overline \HHH_t]
-[\HHH_{\pa E+\eta}-\overline \HHH_{\pa E+\eta}]\|_{H^{1}(\pa E+\eta)}\to0 
\end{equation}
exponentially fast, as $t\to+\infty$, where $\overline \HHH_{\pa E+\eta}$ stands for the average of $\HHH_{\pa E+\eta}$ on $\pa E+\eta$.\\
Since $E_t\to E+\eta$ (up to a subsequence) in $W^{3,2}$, it is easy to check that 
$\vert\overline \HHH_{t}-\overline \HHH_{\pa E+\eta}\vert\leq C\|\psi_{\eta, t}\|_{C^1(\pa E+\eta)}$ which decays exponentially, therefore, thanks to 
limit~\eqref{quasiHdecay_2}, we have
$$
\|\HHH_t(\cdot + \psi_{\eta, t}(\cdot)\nu_{E+\eta}(\cdot)) -\HHH_{\pa E+\eta}\|_{H^{1}(\pa E+\eta)}\to0
$$
exponentially fast.\\
The conclusion then follows arguing as at the end of Step~$4$ of Theorem~\ref{existence}.
\end{proof}

\section{The classification of the stable critical sets}\label{classification}

In this final section, we are going to discuss the classes of smooth sets to which Theorems~\ref{existence} and~\ref{existence2} can be applied, hence, ``dynamically exponentially stable'' for the modified Mullins--Sekerka and surface diffusion flow. Much is known for the stable and strictly stable  critical sets $E\subseteq\T^n$ (or of $\R^n$) of the Area functional (hence, for the {\em unmodified} Mullins--Sekerka and surface diffusion flows), characterized by having constant mean curvature $\HHH$ and satisfying respectively
$$
\Pi_E(\varphi)=\int_{\pa E}\bigl(|\nabla  \varphi|^2- \varphi^2|B|^2\bigr)\, d\mu\geq 0
$$
for every  $\varphi \in \Htilde^1(\pa E)=\bigl\{\varphi \in H^1(\partial E)\, : \, \int_{\partial E} \varphi \,\dmu = 0 \bigr\}$ and
$$
\Pi_E(\varphi)=\int_{\pa E}\bigl(|\nabla  \varphi|^2- \varphi^2|B|^2\bigr)\, d\mu>0
$$
for every  $\varphi \in \Tort(\pa E)=\bigl \{\varphi \in H^1(\pa E) \, : \,\int_{\pa E} \varphi \, \dmu = 0 \,\,\text{ and }\, \int_{\pa E} \varphi \nu_E \, \dmu = 0\bigr \}$, according to Definition~\ref{str stab}. Instead, considerably less can be said for the ``nonlocal case'', relative to the modified (with $\gamma>0$) Mullins--Sekerka flow, for which in the above formulas we need to consider analogously the positivity properties of form 
\begin{align}
\Pi_E(\varphi)=&\, \int_{\pa E}\bigl(|\nabla  \varphi|^2- \varphi^2|B|^2\bigr)\, d\mu+8\gamma \int_{\pa E}  \int_{\pa E}G(x,y)\varphi(x)\varphi(y)\,d\mu(x)\,d\mu(y)\nonumber\\
&\,+4\gamma\int_{\pa E}\pa_{\nu_E} v_E \varphi ^2\,d\mu\,,
\end{align}
on the critical sets $E$ (in $\T^n$ or in domains of $\R^n$, with ``Neumann conditions'' at the boundary) satisfying $\HHH+4\gamma v_E=0$ on $\pa E$.

Concentrating for a while on the Area functional, we observe that it is easy to see that (by a dilation/contraction argument) any strictly stable smooth critical set must be connected, but actually, being the normal velocity of the surface diffusion flow at every point defined by the {\em local} quantity $\Delta\HHH$, it follows that Theorem~\ref{existence2} can be applied also to finite unions of boundaries of strictly stable critical sets (see~\cite{FusJulMor18} and the Figure~\ref{figurauno} below). Moreover, by the very definition above, if $\pa E$ in $\T^n$ is composed by flat pieces, hence its second fundamental form $B$ is identically zero, the set $E$ is critical and stable and with a little effort, actually strictly stable. It is a little more difficult to show that any ball in any dimension $n\in\N$ is strictly stable (it is obviously a critical set), which is connected to the study of the eigenvalues of the Laplacian on the sphere $\SSS^{n-1}$, see~\cite[Theorem~5.4.1]{groemer}, for instance. The same then holds for all the ``cylinders'' $\R^k\times\SSS^{n-k-1}\subseteq\R^n$, bounding $E\subseteq\T^n$ after taking their quotient by the same equivalence relation defining $\T^n$, determined by the standard integer lattice of $\R^n$.

Notice that if $n=2$, it follows that the only bounded strictly stable critical sets of the (in this case) {\em Length} functional in the plane are the disks and in $\T^2$ they are the disks and the ``strips'' with straight borders. This is clearly in agreement with the  two--dimensional convergence/stability result of Elliott and Garcke~\cite{EllGar}, mentioned at the end of Section~\ref{msfsdf}.

In the three--dimensional case, a first classification of the smooth stable ``periodic'' critical sets for the volume--constrained Area functional, was given by Ros in~\cite{Ros}, where it is shown that in the flat torus $\T^3$,  they are {\em balls}, {\em $2$--tori}, {\em gyroids} or {\em lamellae}.

\medskip

\begin{figure}[H]
\begin{center}
\includegraphics[scale=0.6]{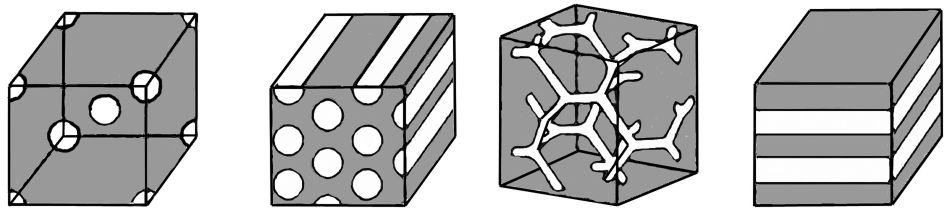}
\caption{From left to right: balls, $2$--tori, gyroids and lamellae.}
\label{figurauno}
\end{center}
\end{figure}

\noindent Notice that, despite their name, the {\em lamellae} are (after taking the quotient) parallel planar $2$--tori and the {\em $2$--tori} are quotients of circular cylinders in $\R^3$. As we said, with the balls, these surfaces are actually strictly stable, while in~\cite{Gr,GrWo,Ross} the authors established the strict stability of gyroids only in some cases. To give an example, we refer to~\cite{GrWo} where Grosse--Brauckmann and Wohlgemuth showed the strictly stability of the gyroids that are fixed with respect to translations. We remind that the gyroids, that were discovered by the crystallographer Schoen in the $1970$ (see~\cite{Schoen}), are the unique non--trivial embedded members of the family of the Schwarz P surfaces and then conjugate to the D surfaces, that are the simplest and most well--known triply--periodic minimal surfaces (see~\cite{Ross}). 

For the case $\gamma >0$, that is, for the nonlocal Area functional, a complete classification of the stable periodic structures is instead, up to now, still missing. 

It is worth to mention what is shown in~\cite{AcFuMo} about the minimizers of $J$. The authors proved that if a horizontal strip $L$ is the unique global minimizer of the Area functional in $\T^n$, then it is also the unique global minimizer of the nonlocal Area functional under a volume constraint, provided that $\gamma>0$ is sufficiently small. Precisely, the following result holds.

\begin{thm}\label{globalminofJ}
Assume that $L\subseteq\T^n$ is the unique, up to rigid motions, global minimizer of the Area functional, under a volume constraint. Then the same set is also the unique global minimizer of the nonlocal Area functional~\eqref{area}, provided that $\gamma>0$ is sufficiently small.
\end{thm}

This theorem then allows to conclude that the global minimizers are lamellae in several cases in low dimensions (two and three), for suitable parameters $\gamma$ and volume constraint. Moreover, in~\cite{AcFuMo}, it is also shown that lamellae with multiple strips are local minimizers of the functional $J$, if the number of strips is large enough.

Finally, we conclude by citing the papers~\cite{ChSt,CicaLeo,Cristof,MorStern,RenWei6,RenWei5,RenWei4,RenWei3,RenWei2,RenWei1} with related and partial results on the classification problem which is at the moment fully open.

\section*{Acknowledgments }
We wish to thank Nicola Fusco for many discussions about his work on the topic and several suggestions. We also thank the anonymous referee for the careful reading and several suggestions.  

\section*{Conflict of interest}
The authors  declare no conflict of interest.

\providecommand{\bysame}{\leavevmode\hbox to3em{\hrulefill}\thinspace}

\end{document}